\documentclass[english]{amsart}

\usepackage[all]{xy} 
\usepackage{mathtools}
\usepackage{amssymb}
\usepackage{amsmath}
\usepackage{amsfonts}
\usepackage{multirow, comment}
\usepackage{fdsymbol}
\usepackage{makebox}

\newtheorem{theorem}{Theorem}[subsection]
\newtheorem{proposition}[theorem]{Proposition}
\newtheorem{corollary}[theorem]{Corollary}
\newtheorem{lemma}[theorem]{Lemma}

\newtheorem{itheorem}{Theorem}
\newtheorem{icorollary}[itheorem]{Corollary}

\theoremstyle{definition}
\newtheorem{remark}[theorem]{Remark}

\newtheorem{problem}{Question}

\newcommand{\begpf}{\noindent{\bf Proof.}\enspace}
\newcommand{\epf}{{\ifhmode\unskip\nobreak\hfil\penalty50 \hskip1em
\else\nobreak\fi \nobreak\mbox{}\hfil\mbox{$\square$} \parfillskip=0pt
\finalhyphendemerits=0 \par\vskip5pt}}

\newcommand{\ol}{\overline}

\newcommand{\wt}{\widetilde}

\newcommand{\A}{\mathbf{A}}

\newcommand{\C}{\mathbf{C}}
\newcommand{\D}{\mathbf{D}}
\newcommand{\F}{\mathbf{F}}
\newcommand{\G}{\mathbf{G}}

\newcommand{\Q}{\mathbf{Q}}
\newcommand{\R}{\mathbf{R}}
\newcommand{\BS}{\mathbf{S}}
\newcommand{\Z}{\mathbf{Z}}
\newcommand{\ints}{\mathbf{Z}}
\renewcommand{\P}{\mathbf{P}}

\newcommand{\f}{\mathbf{f}}

\DeclareMathOperator{\GL}{GL}

\DeclareMathOperator{\Spec}{Spec}

\newcommand{\Ind}{\mathrm{Ind}}
\newcommand{\gr}{\mathrm{gr}}
\newcommand{\fil}{\mathrm{Fil}}

\newcommand{\crys}{\mathrm{crys}}
\newcommand{\Ext}{\mathrm{Ext}}

\newcommand{\Aut}{\mathrm{Aut}}
\renewcommand{\hom}{\mathrm{Hom}}
\newcommand{\End}{\mathrm{End}}
\newcommand{\sym}{\mathrm{Sym}}
\newcommand{\Frob}{\mathrm{Frob}}
\newcommand{\Ver}{\mathrm{Ver}}

\newcommand{\Res}{\mathrm{Res}}
\newcommand{\Shom}{{\mathcal{H}om}}
\newcommand{\im}{\mathrm{im}}
\newcommand{\Lie}{\mathrm{Lie}}
\newcommand{\Nm}{\mathrm{Nm}}
\newcommand{\Tr}{\mathrm{Tr}}
\newcommand{\dr}{\mathrm{dR}}

\newcommand{\gal}{\mathrm{Gal}}
\newcommand{\tr}{\mathrm{tr}}

\renewcommand{\th}{\mathrm{th}}

\newcommand{\univ}{\mathrm{univ}}

\renewcommand{\det}{\mathrm{det}}
\newcommand{\SPEC}{\mathbf{Spec}}
\newcommand{\st}{\mathrm{st}}
\newcommand{\red}{\mathrm{red}}
\newcommand{\coker}{\mathrm{coker}}
\newcommand{\cris}{\mathrm{cris}}

\newcommand{\CO}{\mathcal{O}}
\newcommand{\CF}{\mathcal{F}}
\newcommand{\CV}{\mathcal{V}}
\newcommand{\CG}{\mathcal{G}}

\newcommand{\CL}{\mathcal{L}}
\newcommand{\CI}{\mathcal{I}}
\newcommand{\CJ}{\mathcal{J}}
\newcommand{\CM}{\mathcal{M}}
\newcommand{\CH}{\mathcal{H}}
\newcommand{\CK}{\mathcal{K}}
\newcommand{\CA}{\mathcal{A}}

\newcommand{\CC}{\mathcal{C}}
\newcommand{\CE}{\mathcal{E}}
\newcommand{\CR}{\mathcal{R}}
\newcommand{\CP}{\mathcal{P}}

\newcommand{\ga}{\mathfrak{a}}
\newcommand{\gb}{\mathfrak{b}}

\newcommand{\gm}{\mathfrak{m}}

\newcommand{\uhp}{\mathfrak{H}}

\newcommand{\Qbar}{\overline{\Q}}
\newcommand{\Fpbar}{\overline{\F}_p}
\newcommand{\Qpbar}{\overline{\Q}_p}

\newcommand{\Ybar}{{\overline{Y}}}

\newenvironment{psmallmatriks}
  {\left(\begin{smallmatrix}}
  {\end{smallmatrix}\right)}

\begin{document}

\title[A mod $p$ Jacquet--Langlands relation]{A mod $p$ Jacquet--Langlands relation and Serre filtration via the geometry of Hilbert modular varieties: Splicing and dicing}

\author{Fred Diamond}
\email{fred.diamond@kcl.ac.uk}
\address{Department of Mathematics,
King's College London, WC2R 2LS, UK}
\author{Payman Kassaei}
\email{payman.kassaei@kcl.ac.uk}
\address{Department of Mathematics,
King's College London, WC2R 2LS, UK}
\author{Shu Sasaki}
\email{s.sasaki.03@cantab.net}
\address{School of Mathematical Sciences, Queen Mary University of London, E1 4NS, UK}

\thanks{F.D.~was partially supported by EPSRC Grant EP/L025302/1. S.S.~was partially supported by the DFG, SFB/TR45 and Leverhulme Trust Research Project Grant RPG-2018-401.}
%
\date{November 2019}

\begin{abstract}
We consider Hilbert modular varieties in characteristic~$p$ with Iwahori level at~$p$
and construct a geometric Jacquet--Langlands relation showing that the irreducible components 
are isomorphic to products of projective bundles over quaternionic Shimura varieties of level prime to~$p$.
We use this to establish a relation between mod~$p$ Hilbert and quaternionic modular forms that 
reflects the representation theory of $\GL_2$ in characteristic~$p$ and generalizes a result of
Serre for classical modular forms.  Finally we study the fibres of the degeneracy map to level prime
to~$p$ and prove a  cohomological vanishing result that is used to associate Galois representations
to mod~$p$ Hilbert modular forms.

\medskip

\noindent {\sc R\'esum\'e.} {\em Relation de Jacquet--Langlands et filtration de Serre modulo $p$ via la g\'eometrie des 
vari\'et\'es modulaires de Hilbert : \'Epissage et d\'ecoupage.}\,
On consid\`ere les vari\'et\'es modulaires de Hilbert en caract\'eristique~$p$ de niveau Iwahori en~$p$,
et on construit une relation g\'eom\'etrique de Jacquet--Langlands qui montre que
les composantes irr\'eductibles sont isomorphes \`a des produits de fibr\'es projectifs sur des vari\'et\'es de
Shimura quaternioniques de niveau premier \`a~$p$.  On l'utilise pour \'etablir une relation entre les formes
modulaires de Hilbert et quaternioniques modulo~$p$ qui refl\`ete la th\'eorie des repr\'esentations de $\GL_2$
en caract\'eristique~$p$, et qui g\'en\'eralise un r\'esultat de Serre pour les formes modulaires classiques.
Enfin, on \'etudie les fibres de l'application naturelle vers la vari\'et\'e de niveau premier \`a~$p$, et
on d\'emontre un r\'esultat d'annulation de cohomologie qui est utilis\'e dans la construction des
repr\'esentations galoisiennes associ\'ees aux formes modulaires de Hilbert modulo~$p$.

\end{abstract}

\maketitle

\section{Introduction}
\subsection{Overview} Suppose that $N\ge 4$ is an integer and $p$ is a prime not dividing $N$, and let
$X_1(N;p)$ denote the modular curve associated to the group $\Gamma_1(N)\cap \Gamma_0(p)$.
According to a fundamental result of Deligne and Rapoport~\cite[V.1]{DR},
the curve $X_1(N;p)$ has a semistable model
over $\Z[1/N]$ whose reduction mod $p$ is a union of two irreducible components,
each isomorphic to the reduction of $X_1(N)$, the modular curve associated to $\Gamma_1(N)$ .  The model is defined by viewing
$X_1(N;p)$ as parametrizing isogenies of degree $p$ between elliptic curves
(with a point of order $N$), the components in characteristic $p$ are described by whether the isogeny
or its dual has connected kernel, and they cross at the points corresponding
to supersingular elliptic curves. In a similar vein, they describe (in~\cite[V.2]{DR}) a semistable model
over $\Z[\zeta_p,1/N]$ for the modular curve $X_1(Np)$ associated to $\Gamma_1(Np)$; its structure
in characteristic $p$ underpins an elegant relation, due to Serre (see \cite[Thm.~12.8.8]{KM} and \cite[\S8]{Gr}),
between mod $p$ modular forms of weight $2$ with respect to $\Gamma_1(Np)$ and those of weights ranging
from $2$ to $p+1$ with respect to $\Gamma_1(N)$.

The analogous situation becomes more complicated for Hilbert modular varieties, i.e., the
Shimura varieties associated to $\Res_{F/\Q}\GL_2$ where $F$ is a totally real number field.
Recall that these varieties have dimension $d = [F:\Q]$, and can be viewed as (coarse)
moduli spaces for certain $d$-dimensional abelian varieties with additional structure.
We restrict our attention to the case where $p$ is unramified in $F$, and consider the
setting, analogous to the one above, of Hilbert modular varieties with Iwahori level at $p$.
In \cite{P}, Pappas defined models for these varieties over $\Z_{(p)}$, which he proved were flat
local complete intersections of relative dimension $d$.  The geometry of their reduction
mod $p$ has been studied by various authors, as will be discussed below. 
In brief, some collections
of components can be identified with reductions of Hilbert modular varieties of level prime to $p$,
as in the classical case, but there are also ``intermediate'' components with no such
description.   Our first main result is a geometric Jacquet--Langlands relation that describes them
as products of projective bundles over quaternionic Shimura varieties of level prime to $p$.

In \cite{HU2} Helm proves results of a similar nature in the case of unitary Shimura varieties which imply the following:  collections of components in a Hilbert modular variety with Iwahori level at $p$ are related via  {\it Frobenius factors} to a product of projective bundles over  quaternionic Shimura varieties of level prime to $p$. While some of our geometric methods are inspired by Helm's work, we would like to point out an essential difference in the results:  the quaternion algebras appearing in our work are {\it different}, leading to the existence of isomorphisms on the nose (as opposed to Frobenius factors). Apart from facilitating new applications, we consider the merit of these results to be the naturality of the relationships they establish between the mod $p$  geometry of Shimura varieties associated to different reductive groups.

There is also an essential difference between our method and the one in \cite{HU2}: our construction of the 
above-mentioned isomorphisms is more direct in that it does not involve the degeneracy (or forgetful) map to the prime-to-$p$ level.  Indeed, the image of a point 
corresponding to an isogeny of abelian varieties is constructed by ``splicing"  the Dieudonn\'e modules of the isogeny's source and target.
As a result, we believe our method to be more amenable to generalization to Shimura varieties associated to higher rank groups.

Our second main set of results concerns a generalization of Serre's relation between mod $p$ modular forms of weight $2$ and level $Np$ and those of weight $k \in [2,p+1]$ and level prime to $p$.  More precisely, we obtain a filtration on the space of mod $p$ Hilbert modular forms of parallel weight $2$ and pro-$p$ Iwahori level at $p$, and identify the graded pieces with spaces of quaternionic modular forms of level prime to $p$ and weight (components) in $[2,p+1]$.  We accomplish this by combining our geometric Jacquet-Langlands relation with an analysis of dualizing sheaves by a method we call ``dicing".  The motivation for the result, discussed further below, comes from the relation between algebraic and geometric Serre weights explored in \cite{DS}.   In fact, this relation was a critical clue to our understanding of the more canonical  choice of quaternion algebras.

Our final set of geometric results centers on the degeneracy map from the Hilbert modular variety of Iwahori level at $p$ to level prime to $p$ (i.e., the one intervening in Helm's approach as opposed to ours).  In particular, we apply techniques from crystalline Dieudonn\'e theory to determine the precise structure of its fibres (restricted to irreducible components). These results complement those given in the Key Lemma in \cite{GK}, and are expected to have applications in the context of that work. We give a different application in this paper: combining this with our method of dicing dualizing sheaves, we prove a cohomological vanishing result which is a key ingredient in the construction in  \cite{DS} of Galois representations associated to (non-paritious) Hilbert modular eigenforms in characteristic $p$.

\subsection{Geometric Jacquet--Langlands}
We first introduce the notation for our main objects of study: certain mod $p$ Shimura varieties and automorphic bundles.
Their precise definitions are given in \S\S2--4, along with various technical results that will be needed later in the paper.

We fix a totally real field $F$ and a prime $p$ unramified in $F$.
We also fix embeddings $\Qbar \to \Qpbar$ and $\Qbar \to \C$,
and let $\Theta$ denote the set of embeddings $F \to \Qbar$.
We can thus identify $\Theta$ with the sets of homomorphisms:
$$\Theta_\infty := \{ F \to \R\},\quad \Theta_p := 
\{F \to \Qpbar\}\quad \mbox{and}\quad \overline{\Theta}_p 
 := \{O_F/p \to \Fpbar\}.$$
We let $\phi$ denote the Frobenius automorphism of $\Fpbar$.
Furthermore, we have a natural decomposition $\overline{\Theta}_p
 = \coprod_{v|p} \Theta_v$, where $\theta \in \Theta_v$ if and only
 if $\theta$ factors through $\CO_F/v$, and each $\Theta_v$ is an
 orbit in $\overline{\Theta}_p$ under $\theta \mapsto \phi \circ \theta$.

Let $G = \Res_{F/\Q}(\GL_2)$, and let $U$ be a sufficiently small open compact
subgroup of $G(\A_\f) =  \GL_2(\A_{F,\f})$ containing $\GL_2(\CO_{F,p})$.
The Hilbert modular variety with complex points
$$\GL_2(F) \backslash ((\C - \R)^\Theta \times \GL_2(\A_{F,\f})) / U$$
has a canonical integral model over $\Z_{(p)}$, which we denote $Y_U(G)$, and
write $\Ybar = Y_U(G) \times_{\Z_p} \Fpbar$ for its geometric special fibre.
Similarly for any non-empty $\Sigma \subset \Theta$ of even cardinality, consider
the quaternion algebra $B = B_\Sigma$ over $F$ ramified at precisely the
set of infinite places corresponding to $\Sigma$, and let $G_\Sigma$
denote the algebraic group over $\Q$ defined by $G_\Sigma(R) = (B\otimes_\Q R)^\times$.
Choosing a maximal order $\CO_B$ of $B$ and an isomorphism 
$\widehat{\CO}_B \cong M_2(\widehat{\CO}_F)$, we can identify $U$ with an
open compact subgroup of $G_\Sigma(\A_\f) \cong G(\A_\f)$ and consider
the quaternionic Shimura variety with complex points
$$B^\times \backslash ((\C - \R)^{\Theta - \Sigma} \times (B \otimes \A_{\f})^\times) / U.$$
We let $Y_U(G_\Sigma)$ denote its canonical integral model, defined over the localization
at a prime over $p$ in its reflex field; denote its geometric special fibre by
$\Ybar_\Sigma$, and  let $\Ybar_\emptyset = \Ybar$.  Thus $\Ybar_\Sigma$
is a smooth variety over $\Fpbar$ of dimension $|\Theta - \Sigma|$,
which is proper if and only if $\Sigma \neq \emptyset$.

For each $\theta \in \Theta$, one can also define a rank two automorphic vector bundle
$\CV_\theta$ on $\Ybar_\Sigma$, together with a line bundle $\omega_\theta \subset \CV_\theta$
whenever $\theta \not\in \Sigma$.  We let $\delta_\theta$ denote the line
bundle $\wedge^2_{\CO_{\Ybar_\Sigma}}\!\!\CV_\theta$ on $\Ybar_\Sigma$.
For $(k,\ell) \in \ints^\Theta \times \ints^\Theta$ with $k_\theta \ge 2$ for all
$\theta \in \Sigma$, we define the automorphic bundle of weight $(k,\ell)$ on $\Ybar_\Sigma$
as
$$\CA_{k,\ell} = \left(\bigotimes_{\theta \not\in \Sigma} 
\delta_\theta^{\ell_\theta}\omega_\theta^{k_\theta} \right)
\bigotimes \left(\bigotimes_{\theta \in \Sigma} 
\delta_\theta^{\ell_\theta}\sym^{k_\theta - 2}\CV_\theta\right).
$$
The space of mod $p$ modular forms of weight $(k,\ell)$ and level $U$ with respect to $G_\Sigma$
is then defined as $H^0(\Ybar_\Sigma,\CA_{k,\ell})$, under the assumption $F \neq \Q$.
(For $F = \Q$, one has to extend the line bundle to the cusps in order to recover the
usual notion.)  The bundles are equipped with a natural action of $G_\Sigma(\A_\f^{(p)})$
 defined compatibly with its action on the varieties $\Ybar_\Sigma$ for varying $U$, yielding a Hecke action
on the spaces of forms.

Restrict attention now to the case $G = \Res_{F/\Q}(\GL_2)$ and consider the open compact
subgroup
$$U_0(p) = \{g\in U\,|\, g_p \equiv \begin{psmallmatriks} * & * \\ 0 & * \end{psmallmatriks}\bmod p\CO_{F,p}\}.$$
We let $Y_{U_0(p)}(G)$ denote the model over $\ints_{(p)}$ defined by Pappas in~\cite{P}
for the Hilbert modular variety of level $U_0(p)$, and let $\Ybar_0(p)$ denote its geometric
special fibre.  Pappas studied the local structure of $\Ybar_0(p)$ and proved that it is a flat local complete intersection of relative dimension $d=[F:\Q]$. 
The global geometry of $\Ybar_0(p)$ and its degeneracy map to $\Ybar$ were studied in \cite{Stamm} in the case $d=2$, and for general $d$ in 
\cite{GK}, through the introduction of a stratification.
A similar stratification was considered in \cite{HU2}  in the context of related unitary Shimura varieties, and
a generalization to the context of Hilbert modular varieties with $p$ ramified in $F$ was studied in \cite{ERX}.

The stratification  on $\Ybar_0(p)$ is given by closed subvarieties indexed by pairs $(I,J)$ of
subsets of $\Theta = \overline{\Theta}_p$ satisfying $(\phi^{-1} I) \cup J= \Theta$. The $d$-dimensional strata have the form $\Ybar_0(p)_{I,J}$ where 
$I  = \{\phi\circ\theta\,|\,\theta\not\in J\}$, which we denote simply by $\Ybar_0(p)_J$.
Thus $\Ybar_0(p) = \bigcup_{J \subset \Theta} \Ybar_0(p)_J$ where each $\Ybar_0(p)_J$
is a smooth $d$-dimensional variety over $\Fpbar$, and each irreducible component of
$\Ybar_0(p)$ lies in $\Ybar_0(p)_J$ for a unique $J \subset \Theta$.
We can then state our geometric Jacquet--Langlands relation (proved in \S\ref{sec:JL}) as follows:
 
\begin{itheorem}  \label{thm:JLiso} For each $J \subset \overline{\Theta}_p$
and sufficiently small open compact subgroup $U$
of  $\GL_2(\A_{F,\f})$ containing $\GL_2(O_{F,p})$,
there is a Hecke-equivariant isomorphism
$$\Ybar_0(p)_J \stackrel{\sim}{\longrightarrow} 
\prod_{\theta\in\Sigma} \P_{\Ybar_\Sigma}(\CV_\theta),$$
where  the product is a fibre product over $\Ybar_\Sigma$,
and $\Sigma = \Sigma_J \subset \Theta_\infty$ corresponds under the identification
$\Theta_\infty = \overline{\Theta}_p$ to
$\{\theta \in J \,|\, \phi\circ\theta \not\in  J\} \cup \{\theta\not\in J \,|\,\phi\circ\theta \in J\}.$
\end{itheorem}

As we mentioned earlier, a similar result was proved in the context of related unitary Shimura varieties by Helm (\cite[Thm.~5.10]{HU2})
though the morphism constructed by Helm is only proved to be bijective on points, i.e., a  ``Frobenius factor'' in the terminology of \cite{HU2}.  
 Using $'$ to denote the analogous unitary Shimura varieties, Helm's approach is to relate the $\Ybar'_0(p)_J$ to products of $\P^1$-bundles
 over lower-dimensional strata of $\Ybar'$, 
and to relate those in turn to products of $\P^1$-bundles over lower-dimensional Shimura varieties.
We had initially hoped to use Helm's result  for the application we had in mind, despite the presence of Frobenius factors in his construction. 
Indeed we were encouraged by the fact that results of Tian--Xiao~\cite{TX} remove the Frobenius factor from the latter step.
While the Frobenius factor is intrinsic to the former step (see  Theorem~\ref{thm:fibre} below), a more serious problem
was that  the set of ramified places for the quaternion algebra provided by the results in \cite{HU2} does not match the set
$\Sigma_J$ determining the vector bundles $\CV_\theta$.   
This led us to the consideration of different quaternion algebras: in fact the ones that emerge
naturally from our method of ``splicing" described below, which is more direct and 
bypasses the projections to strata of $\Ybar'$.

To prove Theorem~\ref{thm:JLiso}, we first prove the analogous result
in the context of related unitary Shimura varieties (as in \cite{TX}), so that the quaternionic
Shimura varieties are replaced by ones which are moduli spaces for abelian varieties.
Denoting the special fibres of the corresponding unitary Shimura varieties by  $\Ybar'_0(p)$
and $\Ybar_\Sigma'$, the stratum $\Ybar'_0(p)_J$ parametrizes
$p$-isogenies $A \to B$ such that the induced morphism on Dieudonn\'e modules
satisfies conditions determined by $J$.  The idea is to define
morphisms to (projective bundles over) $\Ybar_\Sigma'$ by splicing
the Dieudonn\'e modules of $A$ and $B$ to obtain an abelian variety $C$
corresponding to a point of $\Ybar_\Sigma'$.  We then prove this yields
an isomorphism analogous to the one we want, and
explain how to transfer the result to the Hilbert/quaternionic setting to
obtain Theorem~\ref{thm:JLiso} using results in \S\ref{sec:ShimVar}.  We remark that 
a key idea that enables us to obtain such clean results
in comparison to \cite{HU2} and \cite{TX}  lies in exploiting the possibility
of allowing $A$ and $B$ to play symmetric roles.

\subsection{The Serre filtration} \label{subsec:i.fil}
We now describe in more detail the application of Theorem~\ref{thm:JLiso} we had in mind,
and carried out in \S\ref{sec:fil}.
Recall that we wish to generalize a result of Serre
relating mod $p$ modular forms of weight $2$ and level $\Gamma_1(Np)$
to mod $p$ modular forms of weights $k \in [2,p+1]$ and level $\Gamma_1(N)$
(see~\cite[Thm.~12.8.8]{KM}).  More precisely, let $X_1(Np)$ denote the semistable model
over $R = \ints_{(p)}[\mu_p]$ for the compact modular curve $X_1(Np)$ (as in \cite[\S7]{Gr})
and let $\CK$ denote its dualizing sheaf.  Then $H^0(X_1(Np), \CK)$ is a lattice
over $R$ in the space of weight two cusp forms with respect to $\Gamma_1(Np)$, and
tensoring over $R$ with $\Fpbar$ yields $H^0(\overline{X}_1(Np), \overline{\CK})$, where 
$\overline{X}_1(Np)$ is the special fibre of $X_1(Np)$ and
its dualizing sheaf $\overline{\CK}$ is its sheaf of regular (Rosenlicht) differentials.
The space $H^0(\overline{X}_1(Np), \overline{\CK})$
decomposes as a direct sum of eigenspaces with respect to the natural
action of $(\ints/p\ints)^\times$.  Writing the characters $(\ints/p\ints)^\times \to
\F_p^\times$ as $\chi_m: a \mapsto a^m$ for $m = 1,2,\ldots,p-1$, Serre's result,
as refined by Gross (see Propositions~8.13 and~8.18 of~\cite{Gr}),
gives a Hecke-equivariant exact sequence
$$0 \to H^0(\overline{X} , \delta^m\omega^{p+1-m}(-C))
        \to H^0(\overline{X}_1(Np),\overline{\CK})^{\chi_m}
        \to H^0(\overline{X}, \omega^{m+2}(-C)) \to 0,$$
 where $\overline{X}$ is the reduction of $X_1(N)$ and $\omega^k(-C)$ is
 the line bundle whose sections are (mod $p$)
 cusp forms of weight $k$ with respect to $\Gamma_1(N)$, and $\delta$
 is a trivial bundle whose presence has the effect of twisting the action of the Hecke operator
 $T_q$ by $q$.
 
The above exact sequence can be viewed as a geometric counterpart to the
one arising in the cohomology of the modular curve $\Gamma_1(N)\backslash\uhp$
with coefficients in local systems associated to the right $\GL_2(\F_p)$-modules
in the exact sequence:
$$0 \to \det^m \otimes \sym^{p-1-m}\F_p^2
        \to \Ind_P^{\GL_2(\F_p)} (1\otimes \chi_m)
        \to \sym^{m}\F_p^2 \to 0,$$
where $P$ is the subgroup of upper-triangular matrices and $1 \otimes \chi_m$
is the character sending $\begin{psmallmatriks} a & b \\ 0 & d \end{psmallmatriks}$
to $\chi_m(a) = d^m$.
From the point of view of Serre weight conjectures for
Galois representations, and in particular the relation between algebraic and geometric
notions of Serre weights as explored in \cite{DS}, it is natural to seek a generalization
of Serre's result to the context of Hilbert modular forms.  The desired result should
describe the space of mod $p$ forms of parallel weight $2$ and character $\chi$ with
respect to pro-$p$-Iwahori level at $p$ in terms of spaces of mod $p$ forms of level
prime to $p$ and weights corresponding to the Jordan--Holder factors of the right representation
$\Ind_P^{\GL_2(\CO_F/p\CO_p)} \chi$, where $P$ is again the subgroup of upper-triangular
matrices and $\chi:P \to \Fpbar^\times$ is any character, which by twisting easily reduces
to the case of characters the form
$\begin{psmallmatriks} a & b \\ 0 & d \end{psmallmatriks} \mapsto \chi(d)$ where
$\chi$ is a character on $(\CO_F/p\CO_F)^\times$.

To that end, we maintain the notation from the discussion before Theorem~\ref{thm:JLiso} and now consider
$$U_1(p) = \{g\in U\,|\, g_p \equiv \begin{psmallmatriks} * & * \\ 0 & 1 \end{psmallmatriks}\bmod p\CO_{F,p}\}.$$
Then Pappas provides us with a model for $Y_{U_1(p)}(G)$ which is finite and flat over 
$Y_{U_0(p)}(G)$, and hence Cohen--Macaulay over $\Z_{(p)}$.  Let $\Ybar_1(p)$ denote its
geometric special fibre, and let $\CK$ (resp.~$\overline{\CK}$) denote the dualizing sheaf
on $Y_{U_1(p)}(G)$ (resp.~$\Ybar_1(p)$).  Via the Kodaira--Spencer isomorphism
 $H^0(Y_{U_1(p)}(G),\CK)$ can be identified with a lattice over $\Z_{(p)}$ in the space of Hilbert modular
forms of parallel weight $2$ and level $U_1(p)$ over $\Q$, and we view  $H^0(\overline{Y}_1(p),\overline{\CK})$
as the space of mod $p$ forms of parallel weight $2$ and level $U_1(p)$.
 The natural action of the group
$U_0(p)/U_1(p) \cong (\CO_F/p\CO_F)^\times$ on $H^0(\Ybar_1(p),\overline{\CK})$ yields a 
decomposition
$$H^0(\Ybar_1(p),\overline{\CK}) = \bigoplus_{\chi} H^0(\Ybar_1(p),\overline{\CK})^\chi$$
into eigenspaces for the characters $\chi:  (\CO_F/p\CO_F)^\times \to \Fpbar^\times$.

Before stating the second main result, we recall two basic facts from the representation
theory of $\GL_2(\CO_F/p\CO_F)$.  Firstly, its irreducible 
(right) representations over $\Fpbar$ are precisely those of the form:
$$V_{m,n} = \bigotimes_{\theta \in \overline{\Theta}_p}  
\det^{m_\theta} \otimes \sym^{n_\theta} V_\theta $$
for $(m,n) \in \ints^\Theta \times \ints^\Theta$ with
$0 \le m_\theta, n_\theta \le p-1$ for all $\theta$
and $m_\theta < p-1$ for some $\theta$ in each $\Theta_v$, and
$V_\theta = \Fpbar^2$ with $\GL_2(\CO_F/p\CO_F)$ acting via $\theta$.
Secondly  (see \cite[\S2]{BP}), there is a decreasing filtration
$$ 0 \subset \fil^dV_\chi
        \subset \fil^{d-1}V_\chi \subset \cdots \subset
         \fil^1 V_\chi \subset \fil^0V_\chi = V_\chi $$
on the representation $V_\chi = \Ind_P^{\GL_2(\CO_F/p\CO_p)} \chi$ 
such that the graded part $\gr^j V_\chi$ has the form $\bigoplus_{|J| = j} V_{\chi,J}$, where each $V_{\chi,J}$ 
is either irreducible or zero.  We then prove the following theorem in \S\ref{sss:fun.proof}
(see \S\ref{sss:fun.Hecke} for elaboration on the meaning of Hecke-equivariance):

\begin{itheorem} \label{thm:JLss}
For each sufficiently small open compact subgroup $U$
of  $\GL_2(\A_{F,\f})$ containing $\GL_2(O_{F,p})$,
there is a Hecke-equivariant spectral sequence
$$ E_1^{j,i} = \bigoplus_{|J| = j} H^{i+j}(\Ybar_{\Sigma_J} , \CA_{\chi,J})
         \Longrightarrow H^{i+j}(\Ybar_1(p), \overline{\CK})^\chi,$$
where $\CA_{\chi,J} = \CA_{n+2,m}$ (resp.~$0$)
if $V_{\chi,J} \cong V_{m,n}$ (resp.~$0$).
\end{itheorem}

We thus obtain the following generalization of Serre's filtration,
where the graded pieces take the same form as in the classical case,
except that spaces of Hilbert modular forms are in general replaced
by the quaternionic ones to which they correspond via Jacquet--Langlands.
\begin{icorollary} \label{cor:JLfil}
For each sufficiently small open compact subgroup $U$
of  $\GL_2(\A_{F,\f})$ containing $\GL_2(O_{F,p})$,
there is a Hecke-equivariant decreasing filtration of length $d+1$ on
$H^0(\Ybar_1(p),\overline{\CK})^\chi$, together with a Hecke-equivariant
inclusion:
$$ \gr^j \left(H^0(\Ybar_1(p), \overline{\CK})^\chi\right)  \hookrightarrow
     \bigoplus_{|J|=j}  H^0(\Ybar_{\Sigma_J}, \CA_{\chi,J})$$
for $j=0,1,\ldots, d$.
\end{icorollary}

In order to prove Theorem~\ref{thm:JLss}, we consider the direct image of $\CK$
under the projection $\Ybar_1(p) \to \Ybar_0(p)$, decompose it
into line bundles $\CK_\chi$ on $\Ybar_0(p)$ under the action of
$U_0(p)/U_1(p)$, and define a filtration on $\CK_\chi$
by restricting to the strata.  Using the fact that the $\Ybar_0(p)_J$
are obtained by successively bisecting $\Ybar_0(p)$ into local complete
intersections, we obtain a description of the graded
pieces of the filtration in terms of line bundles $\CK_{\chi,J}$
on the $\Ybar_0(p)_J$.  Our method of ``dicing" the dualizing sheaf may be of independent interest, and is used again
in the proof of Theorem~\ref{thm:H1} below.
The proof of Theorem~\ref{thm:JLss}
is then completed by determining the line bundles
to which the $\CK_{\chi,J}$ correspond under the isomorphism of
Theorem~\ref{thm:JLiso}. 

\subsection{Degeneracy fibres} \label{sec:i.df}
Our final set of results, the subject of \S\ref{sec:fibre}, concerns
the degeneracy map $\overline{Y}_0(p) \to \overline{Y}$, or
more precisely, its restriction to the stratum 
$\overline{Y}_0(p)_J$ (maintaining the above notation).  This restriction is known (see \cite{GK}) to factor through a pointwise bijective
morphism $\xi_J$ from $\overline{Y}_0(p)_J$ to a product of $\P^1$-bundles
over a lower-dimensional stratum in $\overline{Y}$; thus $\xi_J$ is a Frobenius
factor in the sense of \cite{HU2}.
We make this more precise by showing that $\xi_J$ is a factor of the Frobenius
itself (rather than a power), and we go on to determine the precise structure
of the fibres of  $\overline{Y}_0(p)_J  \to \overline{Y}$.  In particular we prove
the following (see Theorem~\ref{thm:fibres} and the subsequent discussion
for an even more precise version):

\begin{itheorem}  \label{thm:fibre}  If $Z$ is a non-empty fibre of the morphism
$\overline{Y}_0(p)_J  \to \overline{Y}$ over a closed point of $\overline{Y}$, then
$Z$ is isomorphic to $(\P_{\Fpbar}^1)^r \times (\Spec(\Fpbar[t]/(t^p)))^s$, where 
$r = |\Sigma_J|/2$ and $s = |J| - r$.
\end{itheorem}

Our approach to proving Theorem~\ref{thm:fibre}  relies heavily on crystalline Dieudonn\'e theory.
In particular, we use the full faithfulness of the Dieudonn\'e crystal functor over smooth
bases, due to Berthelot--Messing~\cite{BM3}, in order to obtain the Frobenius factorization, which
we then use to determine the local structure of  the fibre $Z$.
In order to determine the global structure of $Z$, we show that certain pointwise relations between the
Dieudonn\'e modules of $A$ and $B$ (where $A \to B$ is a universal isogeny) in fact
arise from isomorphisms of crystals over $Z$.

We remark that the problem of describing the fibres of $\overline{Y}_0(p)_J  \to \overline{Y}$
is also considered in \cite[\S4.9]{ERX}, where a weaker result than Theorem~\ref{thm:fibre} is used in their approach to constructing
Hecke operators at primes dividing $p$.  Related results, complementary to ones in this paper, 
are obtained in \cite{GK} where the degeneracy morphism is studied before restriction to the strata.
Further motivation for such analysis of the degeneracy map is provided by its applications to 
$p$-adic analytic continuation of Hilbert modular forms via the dynamics of Hecke operators at primes over $p$,
as in~\cite{payman} and~\cite{Shu}.

In this paper, we combine (the more precise version of) Theorem~\ref{thm:fibre} with the method of dicing introduced in 
\S\ref{sec:fil}\footnote{The results of \S\ref{sec:fibre} are, however, independent of those in \S\ref{sec:JL}.} in order to prove
the following cohomological vanishing result. 

\begin{itheorem} \label{thm:H1}
If $\pi:Y_{U_1(p)}(G) \to Y_U(G)$ is the natural projection
(of models over $\Z_{(p)}$) and $\CK$ is the dualizing sheaf on $Y_{U_1(p)}(G)$,
then $R^i\pi_*\CK = 0$ for all $i > 0$.
\end{itheorem}
We note that Theorem~\ref{thm:H1} (for $i=1$) is a crucial ingredient in the proof of Theorem~6.1.1
of~\cite{DS}, which associates Galois representations to mod $p$
Hilbert modular eigenforms of arbitrary weight.

\subsection{Questions} \label{sec:i.q}
We close the Introduction by listing several questions and directions
for further research that are suggested by our work.

\begin{problem} \label{q:quaternionic}
Is there a more general framework for Theorems~\ref{thm:JLiso} and~\ref{thm:JLss}
where the group $G = \Res_{F/\Q}\GL_2$ is replaced by the one associated to a
quaternion algebra over $F$ (or an even more general reductive group), and the 
representation $\Ind\,\chi$ is replaced by any tamely ramified (or even more general) type?
For a totally definite quaternion algebra, where the associated Shimura
varieties are zero-dimensional, the analogues of the theorems are essentially tautologies, while
the case of Shimura curves is related to the work of Newton--Yoshida~\cite{NY}.  
\end{problem}

\begin{problem}   The flipping and twisting of weights that appear in the computation
of the line bundles in the proof of Theorem~\ref{thm:JLss} perfectly reflect the same
phenomena in the computation of the Jordan--H\"older factors of the principal series types.
Can one give a more conceptual explanation for this synchronized gymnastics?
\end{problem}

\begin{problem}  Note that Corollary~\ref{cor:JLfil} only produces an injection.
The obstruction to proving that it is an isomorphism comes from terms of the form
$H^1(\Ybar_{\Sigma_J} , \CA_{\chi,J})$ in the spectral sequence, and one can
construct examples where these do not vanish.  If there is in fact a non-trivial cokernel,
can one at least prove that the Hecke action on it is ``Eisenstein''?
\end{problem}

\begin{problem}  A Hilbert modular eigenform $f$ of parallel weight $2$ and level $U_1(p)$, with coefficients
in a finite extension $\CO$ of $\Z_p$, determines a rank one submodule of the space of sections of the dualizing sheaf
on $Y_{U_1(p)}(G)$ over $\CO$, and hence a one-dimensional subspace of $H^0(\overline{Y}_1(p),\CK)^\chi$
for some $\chi$.  Motivated by the conjectures of \cite{DS}, one can ask if its position in the filtration and inclusion
of Corollary~\ref{cor:JLfil} is determined by the local Galois 
representations\footnote{Together with the $U_v$-eigenvalue if $f$ is old at $v$.} 
$\rho_f|_{G_{F_v}}$ for $v|p$, or more precisely the invariants $v_\theta$ for $\theta \in \Theta_v$ associated
to $\rho_f|_{G_{F_v}}$ as in the formulation of Breuil's Lattice Conjecture~\cite[Conj.~1.2]{B}
(proved by Emerton--Gee--Savitt~\cite{EGS}).
\end{problem}

\subsection{Acknowledgments}
The authors would like to thank D.\ Helm for several useful conversations.
We have already described how our results on geometric Jacquet--Langlands are inspired
by those in his paper~\cite{HU2}, but we should also remark that this circle of ideas has deeper roots
in Zink's work~\cite{Z}, Serre's letters\footnote{Written in 1987 and 1989.}~\cite{serre}, Ribet's seminal paper~\cite{ribet},
Pappas's thesis~\cite{PappasPhD} and work of Ghitza~\cite{Gh} and Nicole~\cite{Nic}.
It is a pleasure to acknowledge that some of the seeds for this paper were in fact planted by Pappas's description
of the results in his thesis to one of the authors (F.D.) at Columbia in the 1990's; they were only germinated
in the last few years by ideas diffusing from the geometric Serre weight conjectures in \cite{DS}.  

We also learned much from Y.~Tian and L.~Xiao's paper \cite{TX} and are grateful to the authors
for responding to several questions about it.

We are also grateful to C.~Breuil for instructive comments on a preliminary version of this paper,
and to the referees and to G.~Micolet for their careful reading of parts of the paper, leading to several minor
corrections and improvements to the exposition.

Finally one of the authors (S.S.) would like to thank V. Pa\v sk\=unas for moral support and DFG/SFB for financial support
 whilst research pertaining to this paper was carried out at Universit\"at Duisburg-Essen.


\section{Shimura varieties}\label{sec:ShimVar}
\subsection{Hilbert modular varieties}\label{sec:hmv}
We begin with a slight variant, based on \cite[\S2]{DS},
of the standard construction (e.g., in \cite[\S1]{rap})
of integral canonical models for Hilbert modular varieties.
The approach presented here to defining the moduli problem provides the
most natural and convenient framework for our results.

\subsubsection{The Shimura datum}\label{sss:hmv.sd}
We maintain the basic notation from the Introduction, so $F$ is a totally real field
of degree $d = [F:\Q]$ in which $p$ is unramified, and we identify 
 $\Theta = \{\theta: F \to \Qbar\}$ with $\Theta_\infty = \{F\to \R\}$, 
 $\Theta_p = \{F \to \Qpbar\}$ and $\overline{\Theta}_p = \{\CO_F \to \Fpbar\}$
 via  fixed embeddings $\Qbar \to \C$ and $\Qbar \to \Qpbar$.
  We let $\A_F$ denote the adeles of $F$ (omitting the subscript for $F = \Q$),
 $\A_{F,\f} = F \otimes \widehat{\Z}$ the finite adeles, and $\A_{F,\f}^{(p)} = F \otimes \widehat{\Z}^{(p)}$
 the prime-to-$p$ finite adeles.
 
 Let $G = \Res_{F/\Q}\GL_2$ and $\BS = \Res_{\C/\R}\G_m$,
 and consider the Shimura datum $(G, [h] )$, where
 $[h]$ is the $G(\R)$-conjugacy class of the homomorphism $h:\BS \to G_\R$
 defined on $\R$-points by the map $\C^\times \to \GL_2(\R)^\Theta$
sending $x+iy \mapsto \begin{psmallmatrix}x&y\\ -y &x\end{psmallmatrix}_{\theta \in \Theta}$.
Let $U$ be an open compact subgroup of $G(\A_\f) = \GL_2(\A_{F,\f})$
containing $\GL_2(\CO_{F,p})$, which we assume is sufficiently small
(as in \cite{DS}).  Write $U = U^p U_p$, where $U^p \subset \GL_2(\A_{F,\f}^{(p)})$
and $U_p = \GL_2(\CO_{F,p})$.  The associated Shimura variety (of level $U$)
is the Hilbert modular variety with complex points
\begin{equation} \label{eqn:HMVC}  \GL_2(F) \backslash ((\C - \R)^\Theta \times \GL_2(\A_{F,\f})) / U
= \GL_2(F)_+ \backslash (\uhp^\Theta \times \GL_2(\A_{F,\f})) / U, \end{equation}
where $\uhp$ is the complex upper-half plane and $\GL_2(F)_+$ denotes
the subgroup of elements of $\GL_2(F)$ with totally positive determinant.

\subsubsection{The moduli problem}\label{sss:hmv.mp}
Consider the functor $\widetilde{Y}_U(G)$ that associates
to a $\Z_{(p)}$-scheme $S$ the set of isomorphism classes of
tuples $\underline{A} = (A,\iota,\lambda,\eta)$, where
\begin{itemize}
\item $A$ is an abelian scheme over $S$ of dimension $d$;
\item $\iota$ is an embedding $\CO_F \to \End_S(A)$ such that $\Lie(A/S)$ is,
locally on $S$, free of rank one over $\CO_F \otimes \CO_S$;
\item $\lambda$ is a prime-to-$p$ quasi-polarization\footnote{By a {\em prime-to-$p$
quasi-polarization}, we mean a quasi-isogeny $\lambda$ such that, locally on $S$, $n\lambda$
is a polarization of degree prime to $p$ for some positive integer $n$ prime to $p$.}
such that $\iota$ is invariant under the associated Rosati involution;
\item $\eta$ is a level $U^p$-structure on $A$, 
i.e., for a choice of geometric point $\overline{s}_i$ on each connected component $S_i$ of $S$,
the data of a $\pi_1(S_i,\overline{s}_i)$-invariant $U^p$-orbit of
$\widehat{\CO}_F^{(p)} = \CO_F\otimes \widehat{\Z}^{(p)}$-linear isomorphisms
$$\eta_i :(\widehat{\CO}_F^{(p)})^2 \to T^{(p)}(A_{\overline{s}_i}),$$
where $T^{(p)}$ denotes the product over $\ell \neq p$ of the $\ell$-adic Tate modules,
and $g \in U^p$ acts on $\eta_i$ by pre-composing with right multiplication by $g^{-1}$.
\end{itemize}

Note that for any connected $S$ and $(A,\iota,\lambda,\eta)$ as above, there is a unique 
$\epsilon \in  (\A_{F,\f}^{(p)})^\times/(\det U^p)(\widehat{\Z}^{(p)})^\times$ such that
the following diagram commutes
$$\xymatrix{ (\widehat{\CO}_F^{(p)})^2  \times (\widehat{\CO}_F^{(p)})^2
\ar[r]^-{\wedge^2}  \ar[d]_{(\eta,\eta)}  &
\A_{F,\f}^{(p)} 
\ar[d]^{\epsilon\otimes\zeta}  \\
T^{(p)}(A_{\overline{s}})  \times T^{(p)}(A_{\overline{s}}) 
\ar[d]_{(1,\lambda)}  &
\A_{F,\f}^{(p)} (1)
\ar[d]^{\Tr_{F/\Q}} \\
T^{(p)}(A_{\overline{s}})  \times (\Q \otimes T^{(p)}(A^\vee_{\overline{s}}) )
\ar[r]^-{\mathrm{Weil}} &
\A_{\f}^{(p)} (1)}$$
for some compatible system of prime-to-$p$ roots of unity, i.e., some
isomorphism $\zeta: \widehat{\Z}^{(p)} \to \widehat{\Z}^{(p)}(1)$,
where the top arrow is the standard $\CO_F$-bilinear alternating pairing sending $((a,b),(c,d))$ to $ad-bc$.
For each $\epsilon  \in (\A_{F,\f}^{(p)})^\times/(\det U^p)(\widehat{\Z}^{(p)})^\times$,
the corresponding subfunctor $\widetilde{Y}^\epsilon_U(G)$ is 
representable by a smooth quasi-projective scheme over $\Z_{(p)}$ (in fact
a PEL Shimura variety associated to the preimage of $\G_m$ under $\det: G \to \Res_{F/\Q}\G_m$).
It follows that $\widetilde{Y}_U(G)$ is representable by an infinite disjoint union of 
smooth quasi-projective schemes over $\Z_{(p)}$, which we also denote by $\widetilde{Y}_U(G)$.
Thus $\widetilde{Y}_U(G)$ is smooth over $\Z_{(p)}$; in particular it is locally of finite type.
We remark that the complex points of $\widetilde{Y}_U(G)$ are in canonical (holomorphic)
bijection with the double quotient
$$\mathrm{SL}_2(\CO_{F,(p)}) \backslash (\uhp^\Theta \times \GL_2(\A_{F,\f}^{(p)})) / U^p.$$

\subsubsection{Descent and Hecke action}\label{sss:hmv.dha}
We have an action of $\CO_{F,(p),+}^\times$ on $\widetilde{Y}_{U}(G)$ defined
by $\theta_\mu(A,\iota,\lambda,\eta) = (A,\iota,\mu\lambda,\eta)$ for 
$\mu \in \CO_{F,(p),+}^\times$.  Note that $(U\cap \CO_F^\times)^2$ acts
trivially, and \cite[Lemma~2.4.1]{DS} shows that the resulting action by
$\CO_{F,(p),+}^\times/(U\cap \CO_F^\times)^2$ is free (for sufficiently small $U$).
Furthermore
$\widetilde{Y}_{U}(G)$ is the union of the orbits of finitely many $\widetilde{Y}^\epsilon_U(G)$
and the stabilizer of each $\widetilde{Y}^\epsilon_U(G)$ is the finite group
$(\CO_{F,+}^\times \cap \det(U))/(U\cap \CO_F^\times)^2$, so the quotient of $\widetilde{Y}_{U}(G)$
by this action is a smooth quasi-projective scheme over $\Z_{(p)}$, which we 
denote by $Y_{U}(G)$.

We now define an action\footnote{We shall, throughout the paper, only consider locally Noetherian schemes,
and here we simply mean an action on the projective system, without concern for whether the system is
representable by a (locally Noetherian) scheme.}
 of $G(\A_{\f}^{(p)})$ on $\varprojlim_{U} Y_{U}(G)$.
Suppose that $U_1$ and $U_2$ are open compact subgroups as above, and
that $g \in \GL_2(\A_\f^{(p)})$ is such that $g^{-1}U_1g \subset U_2$.
Let $\underline{A}_1 = (A_1,\iota_1,\lambda_1,\eta_1)$ denote the universal
abelian variety over $\widetilde{Y}_{U_1}(G)$.
Let $A'$ denote the abelian variety over $S = \widetilde{Y}_{U_1}(G)$ which is (prime-to-$p$) isogenous
to $A$ and satisfies
$$T^{(p)}(A'_{\overline{s}_i})  = \eta_{1,i}((\widehat{\CO}_F^{(p)})^2 g^{-1})  $$
for all $i$ (indexing the connected components of $S$).
Then $A'$ inherits an $\CO_F$-action $\iota'$ from the identification
$\End_S(A_1)\otimes \Q = \End_S(A') \otimes \Q$ induced by the
canonical quasi-isogeny $\pi \in \hom_S(A_1,A') \otimes \Z_{(p)}$,
which furthermore induces an isomorphism
$\Lie(A_1/S) \to \Lie(A'/S)$ compatible with the $\CO_F$-actions.
Note also that the quasi-polarization $\lambda'$ on $A'$ defined by
$\lambda_1 = \pi^\vee \circ \lambda' \circ \pi$ induces the same
Rosati involution as $\lambda_1$.  Moreover
$\eta' = \eta_1\circ r_{g^{-1}}$ (where $r_{g^{-1}}$ denotes right multiplication by $g^{-1}$),
defines a level $U_2$-structure
on $A'$.  We then define $\tilde{\rho}_g: \widetilde{Y}_{U_1}(G) \to \widetilde{Y}_{U_2}(G)$
by $\underline{A}_1 \mapsto \underline{A}' = (A',\iota',\lambda',\eta')$, i.e., if $\underline{A}_2$
is the universal abelian variety over $\widetilde{Y}_{U_2}(G)$, then
$\tilde{\rho}_g^*(\underline{A}_2) = \underline{A}'$.
Since $\tilde{\rho}_g$ commutes with the action of $\CO_{F,(p),+}^\times$, it descends to
a morphism $\rho_g: Y_{U_1}(G) \to Y_{U_2}(G)$.  Furthermore the morphisms
$\tilde{\rho}_g$ and $\rho_g$ are finite and \'etale.
Finally if $h$ and $U_3$ (again sufficiently small) are such that
$h^{-1}U_2 h\subset U_3$, then $\tilde{\rho}_g^*\tilde{\rho}_h^*(\underline{A}_3) \cong
\tilde{\rho}_{gh}^*(\underline{A}_3)$, so that $\tilde{\rho}_h\circ \tilde{\rho}_g = \tilde{\rho}_{gh}$
and $\rho_h\circ \rho_g = \rho_{gh}$, giving the desired action of
$G(\A_{\f}^{(p)})$ on $\varprojlim_{U} Y_{U}(G)$ (the limit being taken
under the maps $\rho_1$ for $U_1 \subset U_2$).

Proceeding as in the case of the related PEL Shimura varieties (or alternatively deducing
the analogous results from the PEL setting), one finds that $Y_U(G)$ defines the canonical model over $\Q$
for the Hilbert modular variety in (\ref{eqn:HMVC}) in the sense of \cite{deligne} (see also \cite[\S12]{milne}),
and hence that the $Y_U(G)$ define a system of integral canonical models in the sense of 
\cite{kisin}.

\subsection{Unitary Shimura varieties} \label{sec:usv}
In this section, we recall the construction of integral canonical models for certain unitary
Shimura varieties, partly following~\cite[\S3]{TX}, but treating the non-PEL setting as we
did in \S\ref{sec:hmv}.

\subsubsection{The Shimura data}\label{sss:usv.sd}
Let $\Sigma$ be any subset of $\Theta_\infty$ of even cardinality, and let $B = B_\Sigma$ denote
the quaternion algebra over $F$ ramified at precisely the places in $\Sigma$; thus
$B$ is split at all finite places, and $B = \GL_2(F)$ if $\Sigma$ is empty.  We let $G_\Sigma$
denote the algebraic group over $\Q$ defined by $G_\Sigma(R) = (B\otimes R)^\times$.
For $\Sigma \neq \emptyset$, we choose isomorphisms $B\otimes_{F,\theta}\R \cong M_2(\R)$
for each $\theta \not\in \Sigma$ and define $h_\Sigma:\BS \to G_{\Sigma,\R}$ by
sending $x+iy \mapsto \begin{psmallmatrix}x&y\\ -y &x\end{psmallmatrix}_{\theta \not\in \Sigma}$.
In \S\ref{sec:qsv} we will recall how canonical models for the Shimura varieties associated to
the Shimura datum $(G_\Sigma,[h_\Sigma])$ can be described in terms of the ones for
certain related unitary groups, which we now define.

Choose a quadratic CM extension $E$ of $F$ such that all primes $v$ of $F$ dividing $p$
split in $E$, and let $c$ denote the non-trivial element of $\gal(E/F)$.  We let
$\Theta_E = \{\tau: E \to \Qbar\}$, which we identify with 
$\Theta_{E,\infty} = \{E\to \C\}$, $\Theta_{E,p} = \{E \to \Qpbar\}$ and
$\overline{\Theta}_{E,p} = \{\CO_F \to \Fpbar\}$.  Choose a subset $\widetilde{\Sigma}$
of $\Theta_E$ mapping bijectively to $\Sigma$ under $\tau \mapsto \tau|_F$, and
  define $s_{\tilde{\theta}} \in \{0,1,2\}$ for $\theta \in \Theta_E$ by:
\begin{itemize}
\item $s_{\tau} = 1$ if $\tau|_F \not\in\Sigma$;
\item $s_{\tau} = 0$ and $s_{\tau^c} = 2$ if $\tau \in\widetilde{\Sigma}$.
\end{itemize}

We let $T_F$ denote the torus $\Res_{F/\Q}\G_m$, and similarly let 
$T_E = \Res_{E/\Q}\G_m$.  We define $G_\Sigma'$ to be the quotient $(G_\Sigma \times T_E)/Z$,
where $Z = T_F$ is embedded in $G_\Sigma \times T_E$ via $x \mapsto (x,x^{-1})$.
We define $i_{\widetilde{\Sigma}}: \BS \to T_{E,\R}$ as the homomorphism which on
$\R$-points is the composite
$$\C^\times \to \prod_{\theta\in \Sigma} \C^\times  \cong
\prod_{\theta\in \Sigma} (E\otimes_{F,\theta} \R)^\times  \hookrightarrow
\prod_{\theta \in \Theta} (E\otimes_{F,\theta} \R)^\times,$$ 
where the first map is diagonal, the second is defined on the $\theta$-component
by the embedding $\tau^c$ for the $\tau \in \widetilde{\Sigma}$ such that $\tau|_F = \theta$,
and the third is the natural inclusion.  We define
$h_{\widetilde{\Sigma}}':\BS \to G'_{\Sigma,\R}$ as the composite of $(h_\Sigma,i_{\widetilde{\Sigma}})$
with the projection $(G_{\Sigma} \times T_{E})_{\R} \to G'_{\Sigma,\R}$.  
We let $L_\Sigma$ (resp.~$L_{\widetilde{\Sigma}}$) denote the reflex field of 
the Shimura datum $(G_\Sigma,[h_\Sigma])$ (resp.~$(G'_\Sigma,[h'_{\widetilde{\Sigma}}])$,
i.e., the fixed field in $\Qbar$ of the stabilizer in $\Aut\Qbar$ of $\Sigma$ (resp.~$\widetilde{\Sigma}$).

\subsubsection{The moduli problem}\label{sss:usv.mp}
Let $D = D_\Sigma = E \otimes_F B_\Sigma$, and let $u \mapsto \overline{u}$
denote the anti-involution on $D$
defined by $c \otimes \iota$, where $\iota$ is the standard anti-involution.
We can then identify $G_\Sigma'$ with the algebraic group defined by
$$G_\Sigma'(R) = \{\, g \in D\otimes R \,|\, g\overline{g} \in (F\otimes R)^\times \,\}.$$

Let $\CO_D$ be an order of $D$ such that  $\CO_{D,p} = \CO_E \otimes_{\CO_F} \CO_{B,p}$ 
for a maximal order $\CO_{B,p}$ of $B_p$.   We choose an element $\delta \in D^\times$ such
that 
\begin{itemize}
\item $\delta \in \CO_{D,p}^\times$,
\item $\overline{\delta} = - \delta$,
\item the bilinear form on $D \otimes \R$ defined by
  $$(v,w) \mapsto \Tr_{E/\Q}(\tr_{D/E}(v\overline{h}'_{\widetilde{\Sigma}}(i)  \overline{w} \delta))$$
is positive definite.
\end{itemize}
(Thus our $\delta$ is $\delta \sqrt{\mathfrak{d}}$ in the notation of \cite[Lemma~3.8]{TX}.)

We define an anti-involution $u \mapsto u^*$ on $D$ by $u^* = \delta^{-1}\overline{u}\delta$,
and a pairing $\psi_E:D\times D \to E$ by
$$\psi_E(v,w) = \tr_{D/E}(v\overline{w}\delta) =  \tr_{D/E}(v\delta w^*),$$
so $\psi_E$ satisfies $\psi_E(w,v) = -\psi_E(v,w)^c$ and $\psi_E(u v,w) = \psi_E(v,u^*w)$
for all $u,v,w\in D$.  We also
define $\psi_F:D \times D \to F$ by $\psi_F  = \Tr_{E/F}\circ\psi_E$, so $\psi_F$ is alternating
and satisfies $\psi_F(uv,w) = \psi_F(v,u^*w)$; in particular $\psi_F$ is $F$-bilinear.

Now suppose that $U'$ is a sufficiently small open compact subgroup of $G_\Sigma'(\A_\f)$ containing
the image $U'_p$ of $\CO_{B,p}^\times \times \CO_{E,p}^\times$ (under the natural map
$G_\Sigma \times T_E \to G_{\Sigma}'$); as usual write $U' = (U')^pU'_p$ 
where $(U')^p \subset G'(\A_{\f}^{(p)})$. We will define an integral canonical model for
the Shimura variety $Y_{U'}(G'_{\Sigma})$ as a quotient of a representable moduli problem
in a manner similar to that for Hilbert modular varieties in \S\ref{sec:hmv}.

Let $\CO$ denote the localization of $\CO_{L_{\widetilde{\Sigma}}}$ at the prime over $p$
determined by the embedding $\Qbar \to \Qpbar$, and consider the functor that associates to
an $\CO$-scheme $S$ the set of isomorphism classes of tuples
$\underline{A} = (A,\iota,\lambda,\eta,\epsilon)$, where
\begin{itemize}
\item $A$ is an abelian scheme over $S$ of dimension $4d$ (where $d = [F:\Q]$);
\item $\iota$ is an embedding $\CO_D \to \End_S(A)$ such that for all $\alpha \in \CO_E$,
the characteristic polynomial of $\iota(\alpha)$ on $\Lie(A/S)$ is 
$$\prod_{{\tau} \in \Theta_{E}}  (x - {\tau}(\alpha))^{2s_{{\tau}}}
 \quad\in \quad \CO[x].$$
\item $\lambda$ is a prime-to-$p$ quasi-polarization  whose associated
Rosati involution is compatible with $l \mapsto l^*$ on $D$;
\item $(\eta,\epsilon)$ is a level $U'$-structure on $(A,\iota,\lambda)$, i.e., a $\pi_1(S_i,\overline{s}_i)$-invariant
$(U')^p$-orbit of pairs $(\eta_i,\epsilon_i)$ for a geometric point $\overline{s}_i$ on each connected component $S_i$ of $S$, where $\eta_i$ is an 
$\widehat{\CO}_D^{(p)} = \CO_D\otimes \widehat{\Z}^{(p)}$-linear isomorphism
$\widehat{\CO}_D^{(p)} \to T^{(p)}(A_{\overline{s}_i})$ and $\epsilon_i$ is an
$\A_{F,\f}^{(p)}$-linear isomorphism $\A_{F,\f}^{(p)} \to \A_{F,\f}^{(p)}(1)$ such
that the following diagram commutes:
$$\xymatrix{ \widehat{\CO}_D^{(p)} \times \widehat{\CO}_D^{(p)}
\ar[r]^{\psi_F}  \ar[d]_{(\eta_i,\eta_i)}  &
\A_{F,\f}^{(p)}
\ar[d]^{\epsilon_i}  \\
T^{(p)}(A_{\overline{s}_i})  \times T^{(p)}(A_{\overline{s}_i}) 
\ar[d]_{(1,\lambda)}  &
\A_{F,\f}^{(p)}(1)
\ar[d]^{\Tr_{F/\Q}} \\
T^{(p)}(A_{\overline{s}_i})  \times (\Q \otimes T^{(p)}(A^\vee_{\overline{s}_i}) )
\ar[r]^-{\mathrm{Weil}} &
\A_{\f}^{(p)}(1),}$$
where $g \in (U')^p$ on $\eta_i$ is by precomposing with right multiplication by $g^{-1}$,
and on $\epsilon_i$ by multiplication by $\nu(g)^{-1}$ where $\nu: G_\Sigma' \to T_F$
is defined by $g \mapsto g\overline{g}$.
\end{itemize}

As in \S\ref{sss:hmv.mp} for $U'$ sufficiently small, the functor is representable by an
infinite disjoint union of smooth quasi-projective schemes over $\CO$, but these are now
Shimura varieties for the preimage of $\G_m$ under $\nu$, indexed by the classes of
$\epsilon \bmod  \nu((U')^p)(\widehat{\Z}^{(p)})^\times$.  We denote the representing
$\CO$-scheme by $\widetilde{Y}_{U'}(G'_{\Sigma})$

\subsubsection{Descent and Hecke action}\label{sss:usv.dha}
In order to define the descent data to obtain the Shimura variety from $\widetilde{Y}_{U'}(G'_{\Sigma})$,
we need the following rigidity lemma.
\begin{lemma}  \label{lem:automorphisms}
Suppose that $k$ is an algebraically closed field,
$\underline{A} = (A,\iota,\lambda,\eta,\epsilon)$ is a $k$-point of
$\widetilde{Y}_{U'}(G')$, $\mu \in \CO_{F,(p),+}^\times$, and $\alpha \in \Aut_k(A)$ induces
an isomorphism from $(A,\iota,\mu\lambda,\eta,\mu\epsilon)$ to $\underline{A}$.
Then $\alpha = \iota(\gamma)$ for some $\gamma \in U' \cap E^\times \subset \CO_E^\times$
such that $\mu = \Nm_{E/F}(\gamma)$ (assuming $U'$ is sufficiently small).
\end{lemma}
\begpf  This is presumably standard, but here is a sketch.
We wish to prove that $\alpha = \iota(\gamma)$ for some $\gamma \in E$, from which
 the rest follows easily, so suppose this is not the case.
Since $\alpha$ is compatible with the $\CO_D$-action on $A$
and $D \cong M_2(E)$, $A$ is isogenous to $B \times B$ for some
$B$ with $E' = E(\alpha) \subset \End(B) \otimes \Q$ (identifying $E$
with its image in $\End(A)$ and $\End(B)$).  It then follows from
the classification of endomorphism algebras of abelian varieties
(since $\dim(B) = [E:\Q]$) that $E'$ is a quadratic CM
extension of $E$.  Moreover since $\alpha$ is an automorphism, it is
a unit in an order in $E'$, so $\alpha \in \CO_{E'}^\times$.

Next note that the commutativity of
$$\xymatrix{ A
\ar[r]^{\mu\lambda}  \ar[d]_{\alpha}  &
A^\vee\\
A\ar[r]^{\lambda}  &
A^\vee  \ar[u]_{\alpha^\vee}  }$$
shows that the image of $\alpha$ under the $\lambda$-Rosati involution
is $\mu \alpha^{-1} \in E'$, and hence that 
$\Nm_{E'/F'}(\alpha) = \alpha\overline{\alpha} = \mu$, where
$F'$ is the totally real subfield of $E'$.  Since
$\CO_E^\times$ is a finite index subgroup of 
$\{\,\beta \in \CO_{E'}^\times \,|\, \Nm_{E'/F'}(\beta) \in \CO_F^\times\,\}$,
it follows that $\alpha^r \in \CO_E^\times$ for some $r > 0$.  
One can then proceed exactly as in the proof of \cite[Lemma~2.4.1]{DS} for
example (with $E'$ and $E$ replacing $K$ and $F$).
\epf

Now consider the action of $\CO_{F,(p),+}^\times$ on $\widetilde{Y}_{U'}(G_{\Sigma}')$ defined
by $\theta_\mu(A,\iota,\lambda,\eta, \epsilon) = (A,\iota,\mu\lambda,\eta,\mu\epsilon)$ for 
$\mu \in \CO_{F,(p),+}^\times$.  Note that $\Nm_{E/F}(U' \cap E^\times)$ acts
trivially, and the preceding lemma shows that the resulting action by
$\CO_{F,(p),+}^\times/\Nm_{E/F}(U' \cap E^\times)$ is free.  Arguing as in \S\ref{sss:hmv.dha}
we see that the quotient of $\widetilde{Y}_{U'}(G_{\Sigma}')$
by this action is a smooth quasi-projective scheme over $\CO$, which we 
denote by $Y_{U'}(G_{\Sigma}')$.

We define an action of $G_\Sigma'(\A_{\f}^{(p)})$ on $\varprojlim_{U'}Y_{U'}(G'_\Sigma)$
as in \S\ref{sss:hmv.dha}.
Suppose that $g \in G_\Sigma'(\A_{\f}^{(p)})$ and that $U_1'$ and $U_2'$ are sufficiently small open
compact subgroups as above satisfying $g^{-1}U_1' g\subset U_2'$.
Let $\underline{A}_1 = (A_1,\iota_1,\lambda_1,\eta_1,\epsilon_1)$ denote the universal abelian variety over $\widetilde{Y}_{U_1'}(G_\Sigma')$.
Let $A'$ denote the abelian variety over $\widetilde{Y}_{U_1'}(G_\Sigma')$ which is (prime-to-$p$) isogenous
to $A$ and satisfies
$$T^{(p)}(A'_{\overline{s}_i})  = \eta_{1,i}(\widehat{\CO}_D^{(p)} g^{-1})  $$
for all $i$ (indexing connected components).
Then $A'$ inherits an $\CO_D$-action $\iota'$ from the identification
$\End_S(A_1)\otimes \Q = \End_S(A') \otimes \Q$ induced by the
canonical quasi-isogeny $\pi \in \hom_S(A_1,A') \otimes \Z_{(p)}$,
which furthermore induces an isomorphism
$\Lie(A_1/S) \to \Lie(A'/S)$ compatible with the $\CO_D$-actions.
Moreover the quasi-polarization $\lambda'$ on $A'$ defined by
$\lambda_1 = \pi^\vee \circ \lambda' \circ \pi$ induces the same
Rosati involution as $\lambda_1$.
Letting $\eta' = \eta_1\circ r_{g^{-1}}$ (where $r_{g^{-1}}$ denotes right multiplication by $g^{-1}$)
and $\epsilon' = \nu(g)^{-1}\epsilon_1$,  we find that $(\eta',\epsilon')$ defines a level $U_2'$-structure
on $(A',\iota',\lambda')$.  We then define $\tilde{\rho}_g: \widetilde{Y}_{U_1'}(G_\Sigma') \to \widetilde{Y}_{U_2'}(G_\Sigma')$
by $\underline{A}_1 \mapsto \underline{A}' = (A',\iota',\lambda',\eta',\epsilon')$, i.e., if $\underline{A}_2$
is the universal abelian variety over $ \widetilde{Y}_{U_2'}(G_\Sigma')$, then
$\tilde{\rho}_g^*(\underline{A}_2) = \underline{A}'$.
Since $\tilde{\rho}_g$ commutes with the action of $\CO_{F,(p),+}^\times$, it descends to
a morphism $\rho_g: Y_{U_1'}(G_\Sigma') \to Y_{U_2'}(G_\Sigma')$.  Furthermore the morphisms
$\tilde{\rho}_g$ and $\rho_g$ are finite and \'etale.
Finally if $h$ and $U_3'$ (again sufficiently small) are such that
$h^{-1}U_2' h\subset U_3'$, then $\tilde{\rho}_g^*\tilde{\rho}_h^*(\underline{A}_3) \cong
\tilde{\rho}_{gh}^*(\underline{A}_3)$, so that $\tilde{\rho}_h\circ \tilde{\rho}_g = \tilde{\rho}_{gh}$
and $\rho_h\circ \rho_g = \rho_{gh}$, giving the desired action of
$G_\Sigma'(\A_{\f}^{(p)})$ on $\varprojlim_{U'} Y_{U'}(G_\Sigma')$ (the limit being taken
under the maps $\rho_1$ for $U_1' \subset U_2'$).

Proceeding again as in (or deducing from) the PEL case, one finds that the
$Y_{U'}(G_\Sigma')$ define a system of canonical models over $\CO$ for the
Shimura varieties associated to  $(G_\Sigma',[h_{\widetilde{\Sigma}}'])$.
In particular the complex manifold associated to $Y_{U'}(G_\Sigma')$ is identified with
$$G_\Sigma'(\Q) \backslash ((\C - \R)^{\Theta - \Sigma} \times G_\Sigma'(\A_{\f})) / U'
= G_\Sigma'(\Q)_+ \backslash (\uhp^{\Theta - \Sigma} \times G_\Sigma'(\A_{\f})) / U',$$
where $G_\Sigma'(\Q)_+$ denotes the subgroup of $G_\Sigma'(\Q)$ consisting of
$\gamma$ such that $\nu(\gamma)$ is totally positive.  For clarity, we remark that
$$\widetilde{Y}_{U'}(G_\Sigma')(\C) = G_\Sigma^1(\Z_{(p)}) \backslash (\uhp^{\Theta - \Sigma} \times G_\Sigma'(\A_{\f}^{(p)})) / (U')^p,$$
where $G_\Sigma^1 = \ker(\nu:G_\Sigma' \to T_F)$.

\subsection{Quaternionic Shimura varieties}  \label{sec:qsv}
We now recall and make explicit the way in which canonical models for the Shimura
varieties associated to $G_\Sigma$ can be described in terms of those for $G'_\Sigma$.
In particular, for suitably chosen levels, a Shimura variety associated
to $G_\Sigma \times T_E$ provides a common finite \'etale cover such that the projection
maps induce isomorphisms on geometric connected components.

\subsubsection{Tori} \label{sss:qsv.tori}
We first consider various tori that intervene in the relation.

Let $T'$ denote the abelian quotient of $G_\Sigma'$, which we may identify with the
quotient $(T_F\times T_E)/T_F$, where $T_F$ is embedded in the product via
$x \mapsto (x^2,x^{-1})$.  Let $\nu':G_\Sigma' \to T'$ denote
the natural projection, so that $\nu'$ is induced by the map $G \times T_E \to T_F \times T_E$
defined  by $(g,y) \mapsto (\det(g),y)$ where $\det$ is the reduced norm.
Note that the inclusion $T_F \to T'$ induced by 
$x \mapsto (x,1)$ is split by the projection $(x,y) \mapsto x \Nm_{E/F}y$, so that
$$T' \cong T_F \times (T_E/T_F)  \cong T_F \times T_E^1,$$
where $T_E^1$ is the kernel of $\Nm_{E/F}:T_E \to T_F$, the isomorphisms being defined by
$$ (x,y)T_F \leftrightarrow (xyy^c, yT_F) \leftrightarrow (xyy^c, y/y^c).$$
 
We recall the description of canonical models over $\CO$ for zero-dimensional Shimura varieties
associated to the tori $T_E$, $T_F$ and $T'$. More precisely, consider the Shimura data
$(T_E,i_{\widetilde{\Sigma}})$, $(T_F,i_\Sigma)$ and $(T',i'_{\widetilde{\Sigma}})$, where
$i_{\widetilde{\Sigma}}$ was defined in \S\ref{sss:usv.sd}, $i_\Sigma =  \det\circ h_\Sigma$ 
sends $z \in \BS(\R) = \C^\times$ to $(z\overline{z})_{\theta \not\in\Sigma} \in T_F(\R) = 
\prod_{\theta\in \Theta} \R^\times$, and $i'_{\widetilde{\Sigma}}$ is the composite of
$(i_\Sigma,i_{\widetilde{\Sigma}})$ with the natural projection. We continue to suppress the
Deligne homomorphism from the notation for the associated Shimura varieties, denoting them
simply $Y_{V_E}(T_E)$ (for open compact $V_E \subset \A_{E,\f}^\times$), etc.

Assuming $\CO_{E,p}^\times \subset V_E$ and writing $V_E = V_E^p \CO_{E,p}^\times$
with $V_E^p \subset ( \A_{E,\f}^{(p)})^\times$,
the geometric points of $Y_{V_E}(T_E)$ are canonically identified with the finite set
$$C_{V_E} = (\A_{E,\f}^{(p)})^\times/\CO_{E,(p)}^{\times} V^p_E = 
E^\times \backslash\A_{E,\f}^\times / V_E.$$
As a scheme over $\CO$, it can be characterized by descent from the canonical isomorphism
$$ Y_{V_E}(T_E) \times_{\CO} \CO_M \cong \coprod_{c \in C_{V_E}} \Spec \CO_M,$$
where $M$ is an abelian extension of the reflex field $L_{\widetilde{\Sigma}}$ and the action of 
$\gal(M/L_{\widetilde{\Sigma}})$ is given by the Shimura reciprocity law for $i_{\widetilde{\Sigma}}$
(see \cite[2.7]{TX}).  In particular,
since $\CO_{E,p}^\times \subset V_E$, we can assume $M$ is unramified at the primes
of $L_{\widetilde{\Sigma}}$ dividing $p$, so  $Y_{V_E}(T_E)$ is \'etale over $\CO$.

Similarly for $V_F = V_F^p \CO_{F,p}^\times$ with $V_F^p \subset ( \A_{F,\f}^{(p)})^\times$,
the geometric points of $Y_{V_F}(T_F)$ are canonically identified with the finite set
$$C_{V_F} = (\A_{F,\f}^{(p)})^\times/\CO_{F,(p),+}^{\times} V^p_F 
= F_+^\times\backslash \A_{F,\f}^\times / V_F$$
with Galois action determined by Shimura reciprocity for $i_\Sigma$.  Note however that we will view
$Y_{V_F}(T_F)$ as a finite \'etale scheme over $\CO$ rather than a localization of the
ring of integers of its reflex field $L_\Sigma$.
Likewise for $V' \subset T'( \A_\f^{(p)})$ containing the image of
$\CO_{F,p}^\times \times \CO_{E,p}^\times$,
$Y_{V'}(T')$ is a finite \'etale $\CO$-scheme whose geometric points are identified with
$C_{V'} = T'(\Q)_+\backslash T'(\A_{\f})/ V'$, where $T'(\Q)_+ =( F_+^\times \times E^\times)/F^\times$.

Suppose now that $U$ is an open compact subgroup of $G_\Sigma(\A_{F,\f})$ containing $\CO_{B,p}^\times$.
We say that $V_E$ is {\em sufficiently small relative to} $U$ if the following hold:
\begin{itemize}
\item $E^\times \cap V_E^1 = \{1\}$, where $V_E^1 :=  \{\,y/y^c\,|\,y\in V_E\,\}  \subset \A_{E,\f}^\times$;
\item $V_E \cap \A_{F,\f}^\times \subset U$;
\item $\Nm_{E/F}(V_E) \subset \det(U)$.
\end{itemize}
Note that $V_E$ can be chosen sufficiently small to satisfy the first condition
(independently of $U$ and of level prime to $p$) since $\Nm_{E/F}: \CO_E^\times \to \CO_F^\times$
has finite kernel.  Indeed this finiteness ensures we can choose an open subgroup
$U_E \subset \A_{E,\f}^\times$ such that if $\alpha \in E^\times \cap U_E$
and $\Nm_{E/F}(\alpha) = 1$, then $\alpha = 1$, and then
choose $V_E$ contained in the preimage of $U_E$ under $y \mapsto y/y^c$.
It follows that we can choose $V_E$ (of level prime to $p$) satisfying all three conditions.

\subsubsection{Complex points}  \label{sss:qsv.c}
We next describe some explicit relations among the complex points
of the relevant Shimura varieties.    Note that the commutative diagram
$$ \begin{array}{ccc}
G_\Sigma \times T_E  & \longrightarrow & G'_\Sigma \\
\downarrow && \downarrow \\
T_F \times T_E & \longrightarrow & T' \end{array}$$
is compatible with the Deligne homomorphisms $(h_\Sigma,i_{\widetilde{\Sigma}})$, 
$h'_{\widetilde{\Sigma}}$, $(i_\Sigma,i_{\widetilde{\Sigma}})$ and $i'_{\widetilde{\Sigma}}$, and
therefore gives rise to commutative diagrams of Shimura varieties whenever
the corresponding open compact subgroups satisfy the evident compatibilities.
In particular for sufficiently small open compact subgroups $U \subset G_\Sigma(\A_\f)$, $V_E \subset \A_{E,\f}^\times$
and $U' \subset G_\Sigma'(\A_\f)$ such that $U'$ contains the image of $U \times V_E$
under the natural projection, we get the commutative diagram
\begin{equation} \label{eqn:complex}
\begin{array}{ccc}  (G_\Sigma(\Q)_+ \backslash 
 (\uhp^{\Theta - \Sigma} \times G_\Sigma(\A_{\f}) ) / U) \times C_{V_E}
 &\longrightarrow& G'_\Sigma(\Q)_+ \backslash (\uhp^{\Theta-\Sigma} \times G'_\Sigma(\A_\f))/U'\\
\downarrow&&\downarrow \\
C_{\det(U)} \times C_{V_E}  &\longrightarrow& C_{\nu'(U')}
\end{array}
\end{equation}
on complex points of the associated Shimura varieties,
where $G_\Sigma(\Q)_+$ is the subgroup of $G_\Sigma(\Q)$ consisting of elements
whose reduced norm is totally positive.

\begin{lemma}  \label{lem:cartesian}
If $V_E$ is sufficiently small relative to $U$ and $U'$ is the image
of $U \times V_E$, then the preceding diagram is Cartesian.
\end{lemma}
\begpf 
The lemma immediately reduces to the claim that the following diagram is Cartesian:
$$\begin{array}{ccc}  G_\Sigma(\Q)_+ \backslash 
 (\uhp^{\Theta - \Sigma} \times G_\Sigma(\A_{\f}) ) / U
 &\longrightarrow& G'_\Sigma(\Q)_+ \backslash (\uhp^{\Theta-\Sigma} \times G'_\Sigma(\A_\f))/U'\\
\downarrow&&\downarrow \\
C_{\det(U)}  &\longrightarrow& C_{\nu'(U')}
\end{array}
$$
Recall that $C_{\det(U)} = F_+^\times\backslash \A_{F,\f}^\times / \det(U)$
and $C_{\nu'(U')} = T'(\Q)_+ \backslash T'(\A_\f) / \nu'(U')$ where 
$T'(\Q)_+ = (F_+^\times \times E^\times)/F^\times$.    Note in particular
that the inclusion $F^\times_+ \to T'(\Q)_+$ is split  by $(x,y) \mapsto x\Nm_{E/F}y$,
as is the inclusion $\det(U) \to \nu'(U')$ under our assumptions that
$U'$ is generated by $U \times V_E$ and $\Nm_{E/F}(V_E)\subset \det(U)$.
Therefore the bottom row of the diagram is injective.

Identifying $G_\Sigma(\A_\f)$ and $T_E(\A_\f) = \A_{E,\f}^\times$ with their
images in $G'_\Sigma(\A_\f)$, we remark that
$$G'_\Sigma(\A_\f) = \A_{E,\f}^\times G_\Sigma(\A_\f) \quad\mbox{and} \quad
    G_\Sigma'(\Q)_+ = E^\times G_\Sigma(\Q)_+.$$
Indeed viewing $G'_\Sigma(\Q)_+ \subset D^\times$, note that if $\gamma \in G'_\Sigma(\Q)_+$,
then $\gamma^c\gamma^\iota \in F_+^\times$ and $\gamma\gamma^\iota \in E^\times$,
so $\alpha:= \gamma^{-1}\gamma^c\in E^\times$.  Moreover
$\Nm_{E/F}(\alpha) = 1$, so $\alpha = \beta^{-1} \beta^c$ for some $\beta \in E^\times$, and
we can write $\gamma = \beta\delta$ where $\delta = \beta^{-1}\gamma \in G_\Sigma(\Q)_+$.
Similarly we have $G_\Sigma'(\Q_q) = E_q^\times G_\Sigma(\Q_q)$ for all primes $q$ and 
$\{g \in \CO_{D,q}^\times\,|\,g\overline{g} \in \CO_{F,q}^\times\}
   = \CO_{E,q}^\times\CO_{B,q}^\times$ for all but finitely many $q$, 
so that $G_\Sigma'(\A_\f) = \A_{E,\f}^\times G_\Sigma(\A_\f)$.

We consider separately the cases where $B$ is indefinite and $B$ is totally definite.
Recall that if $B$ is indefinite, then the vertical arrows induce bijections
on sets of connected components.  So to prove the lemma, it suffices to
prove the top row restricts to an isomorphism on each connected component.
We must therefore show that if $g \in G_\Sigma(\A_\f)$, then the map
$$\Gamma_g \backslash \uhp^{\Theta-\Sigma}  \to \Gamma'_g\backslash \uhp^{\Theta-\Sigma}$$
is an isomorphism, where $\Gamma_g = G_\Sigma(\Q)_+ \cap gUg^{-1}$ and
$\Gamma'_g := G_\Sigma'(\Q)_+ \cap gU'g^{-1}$.
We show that  under our hypotheses, the inclusion 
$\Gamma_g \subset \Gamma'_g$ is in fact an equality.
Suppose then that $\gamma \in \Gamma'_g$, and write $\gamma = \beta\delta$
for some $\beta \in E^\times$, $\delta \in G_\Sigma(\Q)_+$.  Then $\beta\delta = hy$
for some $h \in U$, $y \in V_E$, which implies that 
$\beta/\beta^c = y/y^c\in E^\times \cap V_E^1 = \{1\}$.
We therefore have $\beta = \beta^c\in F^\times$ and 
$y = y^c\in \A_{F,\f}^\times \cap V_E \subset U$, so that $\gamma \in \Gamma_g$.

Suppose now that $B$ is totally definite.  Now $\Sigma = \Theta$,
$G_\Sigma(\Q)_+ = G_\Sigma(\Q)$ and $G'_\Sigma(\Q)_+ = G_\Sigma'(\Q)$, making the top arrow
$$G_\Sigma(\Q) \backslash G_\Sigma(\A_\f)/U  \to G'_\Sigma(\Q)\backslash G'_\Sigma(\A_\f)/U'.$$
We must show that this map is injective, and that its image consists precisely of
the classes $G_\Sigma'(\Q)gU'$ such that $g \in G'_\Sigma(\A_\f)$ and
$T'(\Q)_+ \nu'(g) \nu'(U')$ is in the image of the bottom arrow.  

Suppose then that $h,h' \in G_\Sigma(\A_\f)$ and $h' \in G'_\Sigma(\Q)hU'$.  Since
$G_\Sigma'(\Q) = E^\times G_\Sigma(\Q)$ and $U' = UV_E$, we can write
$h' = \beta\delta h gy$ for some $\beta \in E^\times$, $\delta \in G_\Sigma(\Q)$,
$g\in U$ and $y \in V_E$, so $\beta y = \beta^c y^c$.
As in the indefinite case, we deduce that $\beta = \beta^c \in F^\times$
and $y = y^c \in U$, so that $h' \in G_\Sigma(\Q) hU$, proving the injectivity.

Finally suppose that $g \in G'_\Sigma(\A_\f)$ is such that $T'(\Q)_+ \nu'(g) \nu'(U')$
is in the image of the bottom arrow.   Since
$G_\Sigma'(\A_\f) = \A_{E,\f}^\times G_\Sigma(\A_\f)$, we can write $g = yh$ for
some $y \in \A_{E,\f}^\times$ and $h \in G_\Sigma(\A_\f)$.  The condition on
$\nu'(g)$ implies $y \in E^\times \A_{F,\f}^\times V_E$, so we can
write $y \in E^\times x V_E$ for some $x \in \A_{F,\f}^\times$, and conclude
that $G_\Sigma'(\Q) g U' = G_\Sigma'(\Q) xh U'$ is in the image of the top arrow.
\epf

\subsubsection{Construction of models}\label{sss:qsv.mod}
Under the hypotheses of Lemma~\ref {lem:cartesian}, we define 
$$Y_{U\times V_E}(G_\Sigma \times T_E) = 
Y_{U'}(G_\Sigma') \times_{Y_{\nu'(U')}(T')}(Y_{\det(U)}(T_F) \times_{\CO} Y_{V_E}(T_E)).$$

Suppose now that
$g \in G_\Sigma(\A_{\f}^{(p)})$ and that $U_1$ and $U_2$ are sufficiently small open
compact subgroups (of level prime to $p$) in
$G_\Sigma(\A_{\f})$ satisfying $g^{-1}U_1g \subset U_2$.  
For $y \in (\A_{E,\f}^{(p)})^\times$ and $V_E$ (of level prime to $p$)
sufficiently small relative to $U_1$ (hence also $U_2$), we have the
commutative diagram
$$\xymatrix{ Y_{U_1'}(G'_{\Sigma})  \ar[rd]  \ar[dd]_{\rho_{gy}} &&
Y_{\det(U_1)}(T_F) \times_\CO Y_{V_E}(T_E)
  \ar[ld]  \ar[dd]^{(\det g, y)} \\
  & Y_{\nu'(U_1')}(T') \ar[dd]^{\nu'(gy)}& \\
Y_{U_2'}(G'_{\Sigma})  \ar[rd]  &&
Y_{\det(U_2)}(T_F) \times_\CO Y_{V_E}(T_E)
  \ar[ld]  \\
  & Y_{\nu'(U_2')}(T')&}$$
yielding a morphism $\rho_{(g,y)}: Y_{U_1\times V_E}(G_\Sigma \times T_E) 
\to Y_{U_2\times V_E}(G_\Sigma \times T_E)$.

The morphisms $\rho_{(g,y)}$ satisfy the usual compatibilities.
In particular for $U = U_1 = U_2$, the automorphisms $\rho_{(1,y)}$
define a free action of the group $C_{V_E} =  (\A_{E,\f}^{(p)})^\times/\CO_{E,(p)}^{\times} V^p_E$
on $Y_{U\times V_E}(G_\Sigma \times T_E)$.
We define $Y_U(G_\Sigma)$ as the quotient by this action.  Note that this quotient is independent
of the chosen $V_E$.  Indeed if $V_E$ and $V_E'$ are of level prime to $p$ and sufficiently small
relative to $U$, then so is $V_E \cap V_E'$, so we can assume $V_E' \subset V_E$, in which
case $Y_{U\times V_E}(G_\Sigma \times T_E)$ can be identified with the quotient of 
$Y_{U\times V'_E}(G_\Sigma \times T_E)$ by the action of $\ker(C_{V_E'} \to C_{V_E})$,
compatibly with the action of $C_{V_E}$.  Note also that
$Y_{U\times V_E}(G_\Sigma \times T_E)$ is smooth and quasi-projective over $\CO$
and hence so is $Y_U(G_\Sigma)$.  Moreover the natural projections induce
an isomorphism 
\begin{equation} \label{eqn:product}
Y_{U\times V_E}(G_\Sigma \times T_E) \stackrel{\sim}{\longrightarrow}
 Y_U(G_\Sigma) \times_{\CO} Y_{V_E}(T_E).
\end{equation}

Returning to the general case of $g$, $U_1$, $U_2$ 
satisfying $g^{-1}U_1g \subset U_2$, the morphism $\rho_{(g,1)}$ is compatible
with the actions of $C_{V_E}$ on the $Y_{U_i\times V_E}(G_\Sigma \times T_E)$,
hence descends to a morphism
$$\rho_g : Y_{U_1} (G_\Sigma) \longrightarrow Y_{U_2} (G_\Sigma)$$
(independent of the choice of $V_E$).    Furthermore if $h$ and $U_3$ 
(again sufficiently small) are such that
$h^{-1}U_2 h\subset U_3$, then $\rho_h\circ \rho_g = \rho_{gh}$, 
so we obtain an action of
$G_\Sigma(\A_{\f}^{(p)})$ on $\varprojlim_{U} Y_{U}(G_\Sigma)$.
Note that under the isomorphism of
(\ref{eqn:product}), the morphism $\rho_{(g,y)}$ corresponds to
$(\rho_g,y)$ on the product.

From the fact that the $Y_{U'}(G'_\Sigma)$ define canonical models over $L_{\widetilde{\Sigma}}$
for the Shimura varieties associated to $(G_\Sigma', [h_{\widetilde{\Sigma}}'])$,
it follows from their construction that so do the $Y_{U\times V_E}(G_\Sigma \times T_E)$
with respect to the Shimura data $(G_\Sigma \times T_E, [(h_\Sigma,i_{\widetilde{\Sigma}})])$,
and hence that the $Y_{U}(G_\Sigma)$ define a system of integral canonical models
with respect to $(G_\Sigma,[h_\Sigma])$.  Note that we view $Y_U(G_\Sigma)$ as a scheme over $\CO$,
a localization of $L_{\widetilde{\Sigma}}$ rather than the reflex field $L_\Sigma$.   We shall not consider
its descent to $L_\Sigma$; on the contrary we need to extend scalars further in order to obtain our main results later.
To that end it will be convenient to work over $W = W(\Fpbar)$, and write $S_W$ for the base-change to $W$
of an $\CO$-scheme $S$.  The identity element of $C_{V_E}$ then defines a section
$\Spec W \to Y_{V_E}(T_E)_W \cong \coprod_{c\in C_{V_E}} \Spec W$.  We thus obtain
a morphism
$$Y_U(G_\Sigma)_W  \to Y_U(G_\Sigma)_W \times_W Y_{V_E}(T_E)_W
      \to Y_{U \times V_E}(G_\Sigma \times T_E)_W \to Y_{U'}(G_\Sigma')_W$$
for any $U$, $V_E$, $U'$ as in Lemma~\ref{lem:cartesian}.  We have the
following immediate consequence of the lemma and discussion above:

\begin{lemma}  \label{lem:restriction}  For $U$, $V_E$ and $U'$ as in Lemma~\ref{lem:cartesian},
the diagram
$$\begin{array}{ccc}
Y_U(G_\Sigma)_W &\longrightarrow& Y_{U'}(G_\Sigma')_W \\
\downarrow&&\downarrow \\
C_{\det(U)} & \longrightarrow & C_{\nu'(U')}
\end{array}
$$
is Cartesian (where $C_{\det(U)}$ and $C_{\nu'(U')}$ are viewed as schemes over $W$),
and identifies $Y_U(G_\Sigma)_W$ with an open and closed subscheme 
of $Y_{U'}(G_\Sigma')_W$.  Moreover the inclusion is compatible with the Hecke action in
the sense that the diagram
$$\begin{array}{ccc}
Y_{U_1}(G_\Sigma)_W &\longrightarrow& Y_{U_1'}(G_\Sigma')_W \\
{\scriptstyle\rho_g}\downarrow&&\downarrow{\scriptstyle \rho_g} \\
Y_{U_2}(G_\Sigma)_W &\longrightarrow& Y_{U_2'}(G_\Sigma')_W \\
\end{array}
$$
commutes for any $g$, $U_1$, $U_2$ satisfying $g^{-1}U_1g \subset U_2$
(and sufficiently small $V_E$ relative to $U_1$, with $U_i'$ the image  of
$U_i\times V_E$ for $i=1,2$).
\end{lemma}

\subsubsection{Relation of Hilbert and unitary definitions} \label{sec:comp}
Note that we have now given two definitions of the canonical model $Y_U(G)$ for 
$G = \Res_{F/\Q}\GL_2$ over $\CO = \Z_{(p)}$: the first as a Hilbert modular variety 
in \S\ref{sss:hmv.dha}, the second by taking $\Theta = \emptyset$ in the construction in \S\ref{sss:qsv.mod}. 
These are necessarily isomorphic by uniqueness
of canonical models, but we will need to describe the isomorphism in terms of the
moduli problems appearing in the two definitions.  

For clarity we write $Y_U(G_\emptyset)$ for the model defined using $G'_\emptyset$,
for which we choose $M_2(\CO_E)$ for $\CO_D$
and $\delta = \begin{psmallmatrix}0&1 \\ -1 & 0 \end{psmallmatrix}$.  We first define a morphism
\begin{equation}\label{eqn:comp}
\tilde{i}: \widetilde{Y}_U(G) \longrightarrow \widetilde{Y}_{U'}(G_\emptyset')\end{equation}
for any sufficiently small $U \subset G(\A_\f)$ and $U' \subset G_\emptyset'(\A_\f)$ of level prime
to $p$ such that $U'$  contains the image of $U$.

Let $(A,\iota,\lambda,\eta)$ be the universal object over $\widetilde{Y}_U(G)$.
Define $A' = A\otimes_{\CO_F} \CO_E^2$ with $M_2(\CO_E)$-action
$\iota'$ defined by left-multiplication on $\CO_E^2$.
Then $(A')^\vee$ is canonically isomorphic to 
$ A^\vee \otimes_{\CO_F} \hom_{\CO_F}(\CO_E^2 , \CO_F)$
(with $\alpha \in \CO_F$ acting on $A^\vee$ as $\iota(\alpha)^\vee$).
Define the quasi-polarization $\lambda'$ on $A'$ as the tensor product
of $\lambda$ with the $\CO_F$-linear homomorphism
$\CO_E^2 \to \hom_{\CO_F}(\CO_E^2 , \CO_F)$ induced by
the pairing $(\alpha,\beta) \mapsto \Tr_{E/F}(\alpha\overline{\beta})$.
We define  the level structure on $A'$ so that for each of the geometric points
$\overline{s}_i$, $\eta'_i$ is the unique isomorphism
$$M_2(\widehat{\CO}_E^{(p)}) \longrightarrow
T^{(p)}(A'_{\overline{s}_i}) 
 = T^{(p)}(A_{\overline{s}_i}) \otimes_{\widehat{\CO}_F^{(p)}} (\widehat{\CO}_E^{(p)})^2$$
such that $\eta_i'\left((a,b)\left(\begin{array}{c}c\\d\end{array}\right)\right)
    = (\eta_i(a,b))\otimes \left(\begin{array}{c}c\\d\end{array}\right)$
for all $a,b\in \widehat{\CO}_F^{(p)}$, $c,d\in \widehat{\CO}_E^{(p)}$.
Finally we let $\epsilon' = \zeta\otimes \epsilon$, where $\zeta$ and $\epsilon$
are defined in the discussion following the definition of the functor $\widetilde{Y}_U(G)$.
It is straightforward to check that $(A',\iota',\lambda',\eta',\epsilon')$ defines a tuple
over $\widetilde{Y}_{U}(G)$ represented by $\widetilde{Y}_{U'}(G_\emptyset')$, i.e.,
a morphism $\tilde{i}: \widetilde{Y}_U(G) \longrightarrow \widetilde{Y}_{U'}(G_\emptyset')$.

The morphism $\tilde{i}$ is clearly compatible with action of $\CO_{F,(p),+}^\times$,
hence descends to a morphism $i:{Y}_U(G) \longrightarrow {Y}_{U'}(G_\emptyset')$.
Suppose now that $V_E$ is an open compact subgroup of $\A_{E,\f}^\times$
contained in $U'$ and of level prime to $p$.  Since $\Sigma = \emptyset$,
the Galois action on $C_{V_E}$ defined by Shimura reciprocity is trivial, so
we may identify $Y_{V_E}(T_E)$ with $\coprod_{C_{V_E}} \Spec (\Z_{(p)})$
as a scheme over $\Z_{(p)}$ and extend $i$ to a morphism
$$i': Y_U(G) \times Y_{V_E}(T_E)=  \coprod_{C_{V_E}} Y_U(G)
\longrightarrow Y_{U'}(G'_\emptyset)$$
by setting $i' = \rho_y \circ i$ on the component represented by $y \in (\A_{E,\f}^{(p)})^\times$.

It is straightforward to check that the morphism induced by $i'$ 
on complex points is the same as the top line of (\ref{eqn:complex}).
In particular the diagram
$$\begin{array}{ccc}
Y_U(G) \times Y_{V_E}(T_E)& \longrightarrow& Y_{U'}(G'_\emptyset)\\
\downarrow && \downarrow \\
Y_{\det{U}}(T_F) \times Y_{V_E}(T_E) & \longrightarrow & Y_{\nu'(U')}
\end{array}$$
commutes, since it does so on complex points.  Thus if $V_E$ is sufficiently small
relative to $U$, we obtain a morphism
\begin{equation} \label{eqn:product2}
Y_U(G) \times Y_{V_E}(T_E) \longrightarrow Y_{U\times V_E}(G_\emptyset \times T_E)
\end{equation}
where the target is the fibre product defined after Lemma~\ref{lem:cartesian}.
Furthermore the morphisms of (\ref{eqn:product2}) for varying $U$ and $V_E$
are compatible with the action of 
$G(\A_\f^{(p)}) \times (\A_{E,\f}^{(p)})^\times$ in the usual sense,
as can be deduced either from the definitions or the corresponding assertion
on complex points.  Since (\ref{eqn:product2}) is an isomorphism on complex points,
it follows from uniqueness of canonical models that it is in fact an isomorphism.
Combined with (\ref{eqn:product}), we obtain isomorphisms
$$Y_U(G) \times Y_{V_E}(T_E)  \longrightarrow Y_U(G_\emptyset) \times Y_{V_E}(T_E)$$
compatible with the action of $G(\A_\f^{(p)}) \times (\A_{E,\f}^{(p)})^\times$.
Taking quotients by the action of $C_{V_E}$ we obtain the desired isomorphism
$Y_U(G) \to Y_U(G_\emptyset)$.  Furthermore the isomorphisms are compatible
with the Hecke action in the usual sense, and its composite with the inclusion of Lemma~\ref{lem:restriction}
is precisely the (base-change to $W$) of the morphism $i: Y_U(G) \to Y_{U'}(G'_\emptyset)$
defined above.

\section{Automorphic vector bundles}

\subsection{Construction of automorphic vector bundles} \label{sec:autobun}
In this section we will define automorphic vector bundles on the special fibres of the Shimura
varieties $Y_U(G_\Sigma)$ (resp.~$Y_{U'}(G_\Sigma')$) for sufficiently small $U$ (resp.~$U'$)
of level prime to $p$.  

\subsubsection{The Hilbert modular setting}\label{sss:avb.hmf}
We begin with the case of $G = G_\emptyset = \Res_{F/\Q}\GL_2$,
proceeding as in \cite{DS}.
We assume $\F \subset \Fpbar$ is sufficiently large to contain the image of 
$\overline{\theta}:\CO_F \to \Fpbar$ for all $\theta \in \Theta$
(for example take $\F = \Fpbar$), and set $\Ybar = Y_U(G)_{\F}$ and $S = \widetilde{Y}_U(G)_\F$.
We assume that $U$ is sufficiently small that, in addition to the usual sense,
we have $\alpha - 1 \in p\CO_F$ for all $\alpha \in U \cap \CO_F^\times$., i.e. {\em $p$-neat}
in the terminology of \cite[Def.~3.2.3]{DS}

Suppose that $\underline{A} = (A,\iota,\lambda,\eta) \in \widetilde{Y}_U(G)(S)$ and
let $s:A \to S$ denote the structure morphism.  Let $\widetilde{\CV} = \CH^1_{\dr}(A/S)$
and consider the exact sequence
$$0 \to s_*\Omega^1_{A/S}  \to \wt{\CV} \to R^1s_*\CO_A \to 0$$
of vector bundles on $S$, of rank $d$, $2d$ and $d$ respectively.
The action of $\CO_F$ on $A$ induces one on each of the vector bundles,
making them sheaves of $\CO_F \otimes \CO_S$-modules.
The assumption on $\iota$ implies that $s_*\Omega^1_{A/S}$ is locally
(on $S$) free of rank one, the hypotheses on $\lambda$ ensure
the same is true of $R^1s_*\CO_A \cong \Lie(A^\vee/S)$, and it follows that
$\CV$ is locally free of rank two.  Decomposing 
$\CO_F \otimes \CO_S \cong  \bigoplus_{\theta \in \Theta} \CO_S$,
we obtain a corresponding decompositions
$\widetilde{\CV} = \bigoplus_{\theta} \widetilde{\CV}_{\theta}$,
$s_*\Omega^1_{A/S} =  \bigoplus_{\theta} \widetilde{\omega}_{\theta}$ and
$R^1s_*\CO_A =  \bigoplus_{\theta} \widetilde{\upsilon}_{\theta}$,
where each $\widetilde{\omega}_\theta$ and $\widetilde{\upsilon}_\theta$ 
is a line bundle on $S$, $\widetilde{\CV}_\theta$ is a vector bundle of rank two,
and we have an exact sequence
$$0\; \to\; \widetilde{\omega}_{\theta} \;\to \;
\widetilde{\CV}_{\theta} \;\to\; \widetilde{\upsilon}_{\theta} \;\to\; 0.$$

We now define descent data on $\widetilde{\CV}$ with respect to the covering $S \to \Ybar$.
Recall that the covering group is $\CO_{F,(p),+}^\times/(U \cap \CO_F^\times)^2$,
and that for $\mu \in \CO_{F,(p),+}^\times$, the action of $\mu$ on $S$ is
defined by $\theta_\mu(A,\iota,\lambda,\eta) = (A,\iota,\mu\lambda,\eta)$,
so $\theta_\mu^*A$ is canonically $\CO_F$-linearly isomorphic to $A$.  
We thus obtain a canonical $\CO_F\otimes \CO_S$-linear isomorphism
$\theta_\mu^*\widetilde{\CV} \to \widetilde{\CV}$, 
hence $\CO_S$-linear isomorphisms 
$\theta_\mu^*\widetilde{\CV}_{\theta} \to \widetilde{\CV}_{\theta}$
for all $\theta$, inducing isomorphisms
$\theta_\mu^*\widetilde{\omega}_{\theta} \to \widetilde{\omega}_{\theta}$
and $\theta_\mu^*\widetilde{\upsilon}_{\theta} \to \widetilde{\upsilon}_{\theta}$.
Furthermore if $\mu = \alpha^2$ for some $\alpha \in U \cap \CO_F^\times$ 
(i.e., $\theta_\mu$ is the identity map), then the canonical isomorphism
$A \to \theta_\mu^*A = A$ is $\iota(\alpha^{-1})$,
so the resulting automorphism of $\widetilde{\CV}$ is $\iota(\alpha^{-1})^*$, which is
the identity since our hypotheses on $U$ imply that $\alpha \equiv 1 \bmod p\CO_F$.
We thus obtain an action of the covering group on $\widetilde{\CV}$ over its
action on $S$, inducing actions on the rank two vector bundles $\widetilde{\CV}_{\theta}$
as well as the line bundles $\widetilde{\omega}_{\theta}$ and $\widetilde{\upsilon}_{\theta}$.
Since $S$ is a disjoint union of finite \'etale schemes over $\Ybar$, the actions define effective
descent data and hence vector bundles on $\Ybar$ fitting in an exact sequence
$$0\; \to\; {\omega}_{\theta} \;\to \;
{\CV}_{\theta} \;\to\; {\upsilon}_{\theta} \;\to\; 0.$$

Suppose now that $U_1$, $U_2 \subset G(\A_{\f})$ are sufficiently small 
(in the sense above) and of level prime to $p$,
and $g \in G(\A_{\f}^{(p)})$ is such that $g^{-1}U_1g \subset U_2$.
Let $\Ybar_i = Y_{U_i}(G)_{\F}$, $S_i = \widetilde{Y}_U(G)_\F$, etc., for $i=1,2$.
We thus have the morphism 
$\tilde{\rho}_g:S_1 \to S_2$ lying over $\rho_g:\Ybar_1 \to \Ybar_2$,
and the canonical quasi-isogeny 
$\pi_g \in \hom_{S_1}(A_1,\tilde{\rho}_g^*A_2) \otimes \Z_{(p)}$
induces an $\CO_F\otimes \CO_{S_1}$-linear
isomorphism $\pi_g^*: \tilde{\rho}_g^*\widetilde{\CV}_2 \to \widetilde{\CV}_1$, and
hence isomorphisms $\pi_g^*: \tilde{\rho}_g^*\widetilde{\CV}_{2,\theta} \to 
\widetilde{\CV}_{1,\theta}$ for each $\theta$.
It is straightforward to check that $\pi_g^*$ is compatible with the descent data,
hence defines an isomorphism $\rho_g^*\CV_{2,\theta} \to 
\CV_{1,\theta}$, and that these induce isomorphisms $\rho_g^*\omega_{2,\theta} \to 
\omega_{1,\theta}$ and $\rho_g^*\upsilon_{2,\theta} \to 
\upsilon_{1,\theta}$, all of which we also denote by $\pi_g^*$.
Finally if $h \in G(\A_{\f}^{(p)})$ and $U_3$ is as above
with $h^{-1}U_2h \subset U_3$, then
the relation $\pi_{gh} = \tilde{\rho}_g^*(\pi_h)\circ \pi_g$ 
ensures that $\pi_{gh}^* = \pi_g^*\circ \rho_g^*(\pi_h^*)$ for the sheaves
on $S_1$, and hence on $\Ybar_1$.  

\subsubsection{The unitary setting}\label{sss:avb.umf}
We now proceed similarly to define automorphic bundles on special fibres of
Shimura varieties for $G'_\Sigma$.  We assume $\F$ is sufficiently large as before,
and now let $\Ybar = Y_{U'}(G'_\Sigma)_{\F}$ and $S= \widetilde{Y}_{U'}(G'_\Sigma)_\F$
where $U'$ (of level prime to $p$) is sufficiently small that Lemma~\ref{lem:automorphisms} 
holds and $\alpha - 1 \in p\CO_F$ for all $\alpha \in U' \cap \CO_E^\times$.
Now let $s:A \to S$ denote the pull-back of the universal abelian variety
over $\widetilde{Y}_{U'}(G_\Sigma')$, let $\widetilde{\CV} = \CH^1_{\dr}(A/S)$ and
consider the exact sequence
$$0 \to s_*\Omega^1_{A/S}  \to \wt{\CV} \to R^1s_*\CO_A \to 0$$
of vector bundles on $S$, of rank $4d$, $8d$ and $4d$ respectively.

The left action of $\CO_D$ on $A$ induces a right action on each of
the bundles, making them sheaves of right $\CO_D \otimes \CO_S$-modules.
Fix an isomorphism $\CO_{B,p} \cong M_2(\CO_{F,p})$.  
Recall that $\CO_{D,p} = \CO_E \otimes \CO_{B,p}$, so we obtain an isomorphism
$\CO_{D,p} \cong M_2(\CO_{E,p})$, and hence 
$$\CO_D \otimes \CO_S \cong M_2(\CO_E \otimes \CO_S)  = 
\bigoplus_{\tau \in \Theta_{E}} M_2(\CO_S).$$
We thus obtain corresponding decompositions
$\widetilde{\CV} = \bigoplus_{\tau} \widetilde{\CV}_{\tau}$,
$s_*\Omega^1_{A/S} =  \bigoplus_{\tau} \widetilde{\omega}_{\tau}$ and
$R^1s_*\CO_A =  \bigoplus_{\tau} \widetilde{\upsilon}_{\tau}$,
with each factor inheriting the structure of a sheaf of right $M_2(\CO_S)$-modules.
The assumption on $\iota$ implies that $\widetilde{\omega}_\tau$ is locally
free of rank $2s_\tau$.  Since $R^1s_*\CO_A \cong \Lie(A^\vee/S)$ as
$\CO_E\otimes \CO_S$-modules with $\alpha \in \CO_E$ acting as 
$\Lie(\iota(\alpha)^\vee)$, it follows from the hypotheses on $\lambda$ that 
$\widetilde{\upsilon}_\tau$ is locally free of rank $2s_{\tau^c}$, and so
$\widetilde{\CV}_\tau$ is locally free of rank $2s_\tau + 2s_{\tau^c} = 4$. 
Let $e_0$ denote the idempotent
$\begin{psmallmatrix} 1&0\\0&0\end{psmallmatrix} \in M_2(\CO_S)$, and let
$\widetilde{\CV}_{{\tau}}^0  = \widetilde{\CV}_{{\tau}} e_0$, a rank two vector
bundle on $S$ since 
$\widetilde{\CV}_\tau \cong \widetilde{\CV}_{{\tau}}^0 \oplus \widetilde{\CV}_{{\tau}}^0$.
Similarly if $\tau  \in \Theta_E$ is such that $\tau|_F \not\in \Sigma$,
then $s_\tau = s_{\tau^c} = 1$, so we have the line bundles 
$\widetilde{\omega}_\tau^0 = \widetilde{\omega}_{{\tau}} e_0$,
$\widetilde{\upsilon}_\tau^0 = \widetilde{\upsilon}_{{\tau}} e_0$, and
an exact sequence
$$0\; \to\; \widetilde{\omega}^0_{\tau} \;\to \;
\widetilde{\CV}^0_{\tau} \;\to\; \widetilde{\upsilon}^0_{\tau} \;\to\; 0.$$
We also define the line bundles
$\widetilde{\delta}_{\tau} = \wedge^2_{\CO_S}\widetilde{\CV}^0_\tau$
for all $\tau \in \Theta_E$, canonically isomorphic to 
$\widetilde{\omega}^0_{\tau} \otimes_{\CO_S}\widetilde{\upsilon}^0_{\tau}$
if $\tau|_F\not\in\Sigma$.

We define descent data on $\widetilde{\CV}$ with respect to the covering $S \to \Ybar$
similarly to the case $G = G_\emptyset$.  Under our assumptions on $U'$, the covering
group is again $\CO_{F,(p),+}^\times/(U' \cap \CO_F^\times)^2$.
Now $\mu \in \CO_{F,(p),+}^\times$ acts via
$\theta_\mu(A,\iota,\lambda,\eta,\epsilon) = (A,\iota,\mu\lambda,\eta,\mu\epsilon)$.
so $\theta_\mu^*A$ is canonically $\CO_D$-linearly isomorphic to $A$ over $S$,
yielding $\CO_S$-linear isomorphisms 
$\theta_\mu^*\widetilde{\CV}^0_{\tau} \to \widetilde{\CV}^0_{\tau}$ for all $\tau$.
Again our assumptions imply this isomorphism is the identity if $\mu \in (U'\cap \CO_F)^2$,
so we obtain vector bundles $\CV^0_\tau$ on $\Ybar$ by descent, and we let 
$\delta_\tau = \wedge^2_{\CO_\Ybar}\CV^0_\tau$.  Similarly if
$\tau|_F \not\in \Sigma$, then we obtain line bundles $\omega^0_\tau$, $\upsilon^0_\tau$
on $\Ybar$, an exact sequence
$$0\; \to\; {\omega}^0_{\tau} \;\to \;
{\CV}^0_{\tau} \;\to\; {\upsilon}^0_{\tau} \;\to\; 0,$$
and an isomorphism 
$\delta_\tau \cong 
{\omega}^0_{\tau} \otimes_{\CO_\Ybar}{\upsilon}^0_{\tau}$.

Suppose now that $U'_1$, $U'_2 \subset G_\Sigma'(\A_{\f})$ are sufficiently small 
(in the sense above) and of level prime to $p$,
and $g \in G_\Sigma'(\A_{\f}^{(p)})$ is such that $g^{-1}U'_1g \subset U'_2$.  
Letting $\Ybar_i = Y_{U_i'}(G_\Sigma')_\F$, $S_i = \widetilde{Y}_{U_i'}(G_\Sigma')_\F$, etc
as in the case of $G_\emptyset$, we have the morphism $\tilde{\rho}_g:S_1 \to S_2$
over $\rho_g:\Ybar_1 \to \Ybar_2$, and the quasi-isogeny $\pi_g$ induces isomorphisms
$\pi_g^*: \tilde{\rho}_g^*\widetilde{\CV}^0_{2,\tau} \to \widetilde{\CV}^0_{1,\tau}$ descending
to isomorphisms $\rho_g^*\CV^0_{2,\tau} \to \CV^0_{1,\tau}$.  Similarly we obtain
isomorphisms $\rho_g^*\omega^0_{2,\tau} \to \omega^0_{1,\tau}$,
$\rho_g^*\upsilon^0_{2,\tau} \to \upsilon^0_{1,\tau}$ if $\tau|_F \not\in \Sigma$,
and $\rho_g^*\delta_{2,\tau} \to \delta_{1,\tau}$ for all $\tau$,
all of which we denote $\pi_g^*$.
Finally if $h \in G'_\Sigma(\A_{\f}^{(p)})$ and $U'_3$ is as above
with $h^{-1}U'_2h \subset U'_3$, then we have $\pi_{gh}^* = \pi_g^*\circ \rho_g^*(\pi_h^*)$.

\subsubsection{The quaternionic setting}\label{sss:avb.qmf}
We now turn to the definition of the automorphic bundles in the case of $G_\Sigma$ for
arbitrary $\Sigma$ (of even cardinality). Now we assume $\F = \Fpbar$, and that 
$U \subset G_\Sigma(\A_\f^{(p)}) $ is sufficiently small that we can choose $V_E$
sufficiently small relative to $U$ (see \S\ref{sss:qsv.tori}
so that $U' = UV_E \subset G'_\Sigma(\A_\f^{(p)})$
is sufficiently small in the sense of \S\ref{sss:avb.umf}.

We decompose $\Theta =  \overline{\Theta}_p = \coprod_{v|p} \Theta_v$, where $\theta \in \Theta_v$
if and only if $\theta$ factors through $\CO_F/v$. For each prime $v$ of $F$ dividing $p$, we 
choose a prime $\tilde{v}$ of $E$ dividing $v$, and let 
$$\widetilde{\Theta} = \coprod_{v|p} \Theta_{E,\tilde{v}} \subset 
\overline{\Theta}_{E,p} = \Theta_E$$
(where $\tau \in \Theta_{E,\tilde{v}}$ if and only if $\tau$ factors through $\CO_E/\tilde{v}$).
Note that $\widetilde{\Theta}$ maps bijectively to $\Theta$ under $\tau \mapsto \tau|_F$;
we do not assume $\widetilde{\Sigma} \subset \widetilde{\Theta}$.  For each $\theta \in \Theta$,
we let $\tilde{\theta}$ denote the extension of $\theta$ in $\widetilde{\Theta}$. 

Let $\Ybar = Y_U(G_\Sigma)_{\F}$, and consider the inclusion $i:\Ybar \to \Ybar' =  Y_{U'}(G_\Sigma')_{\F}$
obtained from the one in Lemma~\ref{lem:restriction}.  For each $\theta \in \Theta$, we define
the rank two vector bundle $\CV_\theta$ on $\Ybar$ to be $i^*\CV^0_{\tilde{\theta}}$.
Similarly we define the line bundle $\delta_\theta =  i^*\omega^0_{\tilde{\theta}}
= \wedge^2_{\CO_\Ybar}\CV_\theta$, and if $\theta \not\in\Sigma$, we have  
$\omega_\theta = i^*\omega^0_{\tilde{\theta}}$, $\upsilon_\theta = i^*\upsilon^0_{\tilde{\theta}}$
on $\Ybar$, the exact sequence
$$0\; \to\; {\omega}_{\theta} \;\to \;
{\CV}_{\theta} \;\to\; {\upsilon}_{\theta} \;\to\; 0,$$
and the isomorphism $\delta_\theta \cong 
{\omega}_{\theta} \otimes_{\CO_\Ybar}{\upsilon}_{\theta}$.

The vector bundles $\CV_\theta$, $\omega_\theta$ and $\upsilon_\theta$
are independent of the choice of $V_E$.  Moreover if $U_1,U_2 \subset G_\Sigma(\A_\f^{(p)})$
are as above and $g \in G_\Sigma(\A_\f^{(p)})$ is such that $g^{-1}U_1g \subset U_2$, then
we can choose $V_E$ as above for $U_1$ and $U_2$, let $U_i' = U_i V_E$,
$\Ybar_i = Y_{U_i}(G_\Sigma)_\F$, etc. for $i=1,2$,
and define the isomorphism $\pi_g^*: \rho_g^*\CV_{\theta,2} \to \CV_{\theta,1}$ to be the pull-back
via $i_1:\Ybar_1 \to \Ybar_1'$ of the one so denoted on $\Ybar_1'$.  
We similarly define isomorphisms $\rho_g^*\omega_{\theta,2} \to \omega_{\theta,1}$ 
and $\rho_g^*\upsilon_{\theta,2} \to \upsilon_{\theta,1}$  for $\theta\not\in\Sigma$,
and $\rho_g^*\delta_{\theta,2} \to \delta_{\theta,1}$ for all $\theta\in \Theta$.
As usual, it is straightforward to check that these are independent on the choice of $V_E$,
and that $\pi_{gh}^* = \pi_g^*\circ \rho_g^*(\pi_h^*)$ for $h$ and $U_3$ satisfying
$h^{-1}U_2h \subset U_3$. 

Note that we have given two definitions of the bundles $\CV_\theta$, $\delta_\theta$,
$\omega_\theta$ and $\upsilon_\theta$ on $\Ybar = Y_U(G_\Sigma)_{\Fpbar}$ for sufficiently
small $U$ in the case $\Sigma = \emptyset$, one by descent from $\widetilde{\CV} = \CH^1_\dr(A/S)$
where $A$ is the universal abelian variety on $S = \widetilde{Y}_U(G)_{\Fpbar}$,
the other by restriction and descent from $\widetilde{\CV}'=\CH^1_\dr(A'/S')$ where $A'$ is
the universal abelian variety on $S' = \widetilde{Y}_U(G'_\emptyset)_{\Fpbar}$.
Writing $\tilde{i}:S \to S'$ for the base-change of the morphism in (\ref{eqn:comp}),
we see from its construction that we have a canonical $\CO_E \otimes \CO_S$-linear
isomorphism
$$\tilde{i}^*(\widetilde{\CV}'e_0) \cong \widetilde{\CV}\otimes_{\CO_F}\CO_E,$$
compatible with the actions of $\CO_{F,(p),+}^\times$ on $\widetilde{\CV}$ and
$\widetilde{\CV}'e_0$ over its actions on $S$ and $S'$. It follows that the two vector
bundles on $\Ybar$ denoted $\CV_\theta$ are canonically isomorphic, hence also
for those denoted $\delta_\theta$, and we see similarly the same holds for
$\omega_\theta$ and $\upsilon_\theta$.  Furthermore, the isomorphisms are
compatible with those denoted $\pi_g^*$.

\begin{remark}  Note that choices of the field $E$, set of places $\widetilde{\Sigma}$ and isomorphisms
$\CO_{B,p} \cong M_2(\CO_{F,p})$ are implicit in our definition of the automorphic bundles
on $Y_U(G_\Sigma)_{\Fpbar}$ (for $\Sigma \neq \emptyset$).  We do not consider here the question
of the bundles' dependence on these choices or descent to a suitably defined reflex field.
\end{remark}

\subsection{Relations among vector bundles}
In this section we recall certain canonical relations among the automorphic bundles and
cotangent bundles.

\subsubsection{The Kodaira--Spencer isomorphism} \label{sec:KS}
We begin with the construction of the Kodaira--Spencer isomorphism, adapted to our setting.

Let $s:A \to S$ denote the universal abelian variety over $S = \widetilde{Y}_{U'}(G_\Sigma')_{\F}$,
where we assume $\F$ is sufficiently large and $U'$ is sufficiently small and of level prime to $p$
as in \S\ref{sss:avb.umf}.   The smooth morphisms $A \to S \to \Spec \F$ yield the exact sequence
$$0 \to s^*\Omega^1_{S/\F} \to \Omega^1_{A/\F} \to \Omega^1_{A/S} \to 0,$$
to which we apply $R^\bullet s_*$ to obtain the connecting morphism:
\begin{equation}\label{eqn:connect}
s_* \Omega^1_{A/S} \longrightarrow  R^1s_*(s^*\Omega^1_{S/\F}) \cong
   \Omega^1_{S/\F} \otimes_{\CO_S} R^1s_* \CO_A.
\end{equation}
Furthermore the morphism is right $\CO_D \otimes \CO_S$-linear (where $\CO_D$ acts
on the target via its action on $R^1s_*\CO_A$).
We may therefore apply idempotents to obtain an $\CO_S$-linear morphism
$$\widetilde{\omega}_{\tau}^0 \longrightarrow
  \Omega^1_{S/\F} \otimes_{\CO_S} \widetilde{\upsilon}_{{\tau}}^0,$$
and hence $\Shom_{\CO_S} (\widetilde{\upsilon}_{{\tau}}^0,\widetilde{\omega}_{{\tau}}^0) \to \Omega^1_{S/\F}$,
for each ${\tau}$ such that ${\tau}|_F \not\in \Sigma$.
By a standard argument, the resulting morphism
\begin{equation}\label{eqn:PELKS} \bigoplus_{\theta \in \Theta - \Sigma}  
\Shom_{\CO_S} (\widetilde{\upsilon}_{\tilde{\theta}}^0,\widetilde{\omega}_{\tilde{\theta}}^0) \longrightarrow \Omega^1_{S/\Fpbar}\end{equation}
is an isomorphism.  See in particular \cite[\S2.1.7]{KWL}, where (\ref{eqn:connect}) is rewritten in such a way that
the argument in the proof of \cite[Cor.~3.17]{TX} shows that (\ref{eqn:PELKS}) is an isomorphism.

Furthermore the isomorphism is compatible with the action of $\CO_{F,(p),+}^\times$, where the action
is defined on $\Shom_{\CO_S}(\widetilde{\upsilon}_{\tilde{\tau}}^0,\widetilde{\omega}_{\tilde{\tau}}^0)$ in
\S\ref{sss:avb.umf}, and on $\Omega^1_{S/\Fpbar}$ via the canonical isomorphism 
$\theta_\mu^*\Omega^1_{S/\Fpbar} \to \Omega^1_{S/\Fpbar}$. This compatibility follows
from the general fact that if 
$$\begin{array}{ccc}
A_1  & \stackrel{s_1}{\longrightarrow}& S_1  \\ {\scriptstyle\pi }\downarrow && \downarrow {\scriptstyle\rho}
 \\A_2 & \stackrel{s_2}{\longrightarrow} & S_2 \end{array}$$
is a commutative diagram of smooth morphisms of smooth schemes over $\F$ with 
$\pi$ and $\rho$ finite, then the resulting diagram
\begin{equation}\label{eqn:KScomm} \begin{array}{ccc}
\rho^*s_{2,*}\Omega^1_{A_2/S_2}  & \longrightarrow & \rho^*\Omega^1_{S_2/\F} \otimes_{\CO_{S_1}}  \rho^*R^1s_{2,*}\CO_{A_2} \\  \downarrow &&\downarrow \\
s_{1,*}\Omega^1_{A_1/S_1}   & \longrightarrow & 
\Omega^1_{S_1/\F} \otimes_{\CO_{S_1}} R^1s_{1,*}\CO_{A,1} \end{array}\end{equation}
commutes, where the top arrow is
$\rho^*(s_{2,*}\Omega^1_{A_2/S_2} \to \Omega^1_{S_2/\F} \otimes_{\CO_{S_2}} R^1s_{2,*}\CO_{S_2})$ and
the vertical arrows are obtained from the natural maps 
$\pi^*\Omega^1_{A_2/S_2} \to \Omega^1_{A_1/S_1}$ and $\rho^*\Omega^1_{S_2/\F} \to \Omega^1_{S_1/\F}$ by adjunction.
Therefore (\ref{eqn:PELKS}) descends to an isomorphism 
$$\bigoplus_{\theta \in \Theta - \Sigma}  
\Shom_{\CO_\Ybar} ({\upsilon}_{\tilde{\theta}}^0,{\omega}_{\tilde{\theta}}^0) 
\stackrel{\sim}{\longrightarrow} \Omega^1_{\Ybar/\Fpbar}$$
on $\Ybar = Y_{U'}(G_\Sigma')_\F$, called the Kodaira--Spencer isomorphism.
\footnote{We actually have such an isomorphism on the integral model over a sufficiently
large $\CO'$.  Note that $(F^\times \cap K)^2$ acts trivially on 
$\Shom_{\CO_{S'}} (\widetilde{\upsilon}_{\tau}^0,\widetilde{\omega}_{{\tau}}^0)$ where 
$\widetilde{\upsilon}_\tau^0$ and  $\widetilde{\omega}_\tau^0$ are defined as above, but
on $S' = \widetilde{Y}_{U'}(G_\Sigma')_{\CO'}$, so it descends to a sheaf
on $Y_{U'}(G_\Sigma')_{\CO'}$, even though $\widetilde{\upsilon}_{{\tau}}^0$ 
and $\widetilde{\omega}_{\tilde{\tau}}^0$ do not.
Also note that \cite{TX} use the analogue of $\widetilde{\omega}_{{\tau}^c}$ instead of 
$(\widetilde{\upsilon}_{{\tau}}^0)^{-1}$;
these are isomorphic, but the isomorphism depends on the polarization, so does not descend
to the Shimura varieties associated to $G_\Sigma'$.} 

Suppose now that $U \subset G_\Sigma(\A_\f^{(p)})$ is sufficiently small and
$U' = UV_E \subset G'_\Sigma(\A_\f^{(p)})$ is chosen as in \S\ref{sss:avb.qmf}.
Let $\Ybar = Y_U(G_\Sigma)_{\F}$ where $\F = \Fpbar$, and consider the inclusion 
$i:\Ybar \to \Ybar' =  Y_{U'}(G_\Sigma')_{\F}$.  The Kodaira--Spencer isomorphism defined
above on $\Ybar'$ then restricts via $i$ to give the Kodaira--Spencer isomorphism
$$\bigoplus_{\theta \not\in\Sigma}  
\Shom_{\CO_\Ybar} (\upsilon_\theta,\omega_\theta) \stackrel{\sim}{\longrightarrow} \Omega^1_{\Ybar/\F}$$
on  $\Ybar$.  Since $\delta_{\theta} = \omega_{\theta} \upsilon_{\theta}$, this can
also be written as
$$\bigoplus_{\theta \not\in\Sigma} 
\delta_{\theta}^{-1}\omega_{\theta}^2 
 \stackrel{\sim}{\longrightarrow}  \Omega^1_{\Ybar/\F},$$
 and taking determinants yields\footnote{We implicitly choose an ordering of factors when writing indexed tensor products;
 the resulting canonical isomorphisms with wedge products are only compatible up to sign with the canonical reordering
 isomorphisms on the tensor products, but we view the ordering as fixed.}
 $$\bigotimes_{\theta \not\in \Sigma}  
\delta_{\theta}^{-1}\omega_{\theta}^2 
 \stackrel{\sim}{\longrightarrow}  \CK_{\Ybar/\F},$$
 where  $\CK_{\Ybar/\F}$ is the dualizing sheaf.
Finally, using the commutativity of (\ref{eqn:KScomm}), it is straightforward to check that the above isomorphisms are independent of
the choice of $V_E$ and Hecke-equivariant in the usual sense, i.e.,
if $U_1$ and $U_2$ are sufficiently small, and 
$g \in G_\Sigma(\A_\f^{(p)})$ is such that $g^{-1}U_1 g \subset U_2$, then the diagram
$$\begin{array}{ccc}
\displaystyle\bigotimes_{\theta\not\in \Sigma}  
\rho_g^*\left(\delta_{\theta,2}^{-1}\omega_{\theta,2}^2\right)&
 \stackrel{\sim}{\longrightarrow}&
  \rho_g^* \CK_{\Ybar_2/\F}\\
  \downarrow&&\downarrow \\
\displaystyle\bigotimes_{\theta\not\in \Sigma}  
\delta_{\theta,1}^{-1}\omega_{\theta,1}^2&
 \stackrel{\sim}{\longrightarrow}&
   \CK_{\Ybar_1/\F}\end{array}$$
 commutes, where $\Ybar_i = Y_{U_i}(G_\Sigma)_\F$, etc., 
 the left vertical arrow is the tensor product of the maps $\pi_g^*$
 defined in \S\ref{sss:avb.qmf}, and the right vertical arrow is the canonical map.

\subsubsection{Frobenius relations} \label{sec:frob}
Let $s:A \to S$ denote the universal abelian variety over $S = \widetilde{Y}_{U'}(G_\Sigma')_{\F}$
as above, and consider the relative Frobenius and Verschiebung morphisms:
$$\Frob: A \to A^{(p)} \quad\mbox{and}\quad \Ver:A^{(p)} \to A$$
over $S$, where $A^{(p)}$ denotes the pull-back of $A$ along the absolute Frobenius on $S$.
We then have the canonical exact sequences
$$ 0 \to \CH \to \CH^1_\dr(A/S) \to \CI \to 0  \quad\mbox{and}\quad 
 0 \to \CI \to \CH^1_\dr(A^{(p)}/S) \to \CH \to 0$$
of vector bundles on $S$, where 
 $$\CH = \ker(\CH^1_\dr(A/S) \stackrel{\Ver^*}{\longrightarrow} \CH^1_\dr(A^{(p)}/S)
\;\;\mbox{and}\;\;
  \CI = \ker(\CH^1_\dr(A^{(p)}/S) \stackrel{\Frob^*}{\longrightarrow} \CH^1_\dr(A/S)).$$
Moreover the above morphisms are $\CO_D \otimes \CO_S$-linear, so
we can apply idempotents to obtain exact sequences
$$ 0 \to \CH^0_{{\tau}} \to \CH^1_\dr(A/S)^0_{{\tau}} \to \CI^0_{{\tau}} \to 0  \quad\mbox{and}\quad 
 0 \to \CI^0_{{\tau}} \to \CH^1_\dr(A^{(p)}/S)^0_{{\tau}} \to \CH^0_{{\tau}} \to 0$$
 for all $\tau\in \Theta_E$.

Note that $\CH^1_\dr(A^{(p)}/S)^0_{{\tau}}$ is canonically isomorphic to 
$(\CH^1_\dr(A/S)^0_{\phi^{-1}\circ{\tau}})^{(p)}$ (where $\cdot^{(p)}$ again denotes
pull-back by the absolute Frobenius on $S$), so that $\CI_\tau^0$ is dual to
$(\Lie(A/S)_{\phi^{-1}\circ\tau}^0)^{(p)}$.  It follows that 
 $\CI^0_{{\tau}}$ (resp.~$\CH_{{\tau}}^0$) has rank $s_{\phi^{-1}\circ{\tau}}$
(resp.~$2- s_{\phi^{-1}{\tau}}$).  In particular if $\phi^{-1}\circ{\tau}|_F \in \Sigma$,
then we obtain an isomorphism
$$\CH^1_\dr(A/S)^0_{{\tau}} \cong  \CH^1_\dr(A^{(p)}/S)^0_{{\tau}}$$
induced by either $\Frob^*$ or $\Ver^*$ according to whether or not
 $\phi^{-1}\circ\tau \in \widetilde{\Sigma}$.
On the other hand if $\phi^{-1}\circ{\tau}|_F \not\in \Sigma$, then we still have 
$$\wedge^2_{\CO_S} \CH^1_\dr(A/S)^0_{{\tau}} 
\cong \CH^0_{{\tau}} \otimes_{\CO_S} \CI^0_{{\tau}}
\cong \wedge^2_{\CO_S} \CH^1_\dr(A^{(p)}/S)^0_{{\tau}},$$
so we have
$$\widetilde{\delta}_\tau = 
\wedge^2_{\CO_S} \CH^1_\dr(A/S)^0_{{\tau}} 
\cong \left(\wedge^2_{\CO_S} \CH^1_\dr(A/S)^0_{\phi^{-1}\circ{\tau}}\right)^{\otimes p}
= \widetilde{\delta}_{\phi^{-1}\circ\tau}^{\otimes p}
$$
for all ${\tau} \in \Theta_{E}$.
Furthermore the isomorphism is compatible with the action of $\CO^\times_{F,(p),+}$,  so it descends to
an isomorphism $\delta_{{\tau}} \cong \delta_{\phi^{-1}\circ {\tau}}^{\otimes p}$
on $Y_{U'}(G_\Sigma')_{\F}$, which in turn restricts to give an isomorphism
$\delta_{\theta}\cong \delta_{\phi^{-1}\circ \theta}^{\otimes p}$ on 
 $Y_{U}(G_\Sigma)_{\Fpbar}$ for all $\theta \in \Theta$ and sufficiently small $U$
 (choosing suitable $U' = UV_E$).  Finally, it is straightforward to check that
 the isomorphism is independent of the choice of $V_E$ and 
 Hecke equivariant in the sense that it is compatible with the maps $\pi_g^*$
 (for $g \in G_\Sigma(\A_\f^{(p)})$ and sufficiently small $U_1$ and $U_2$
 such that $g^{-1}U_1 g \subset U_2$).

\section{Iwahori level structures}
\subsection{The Hilbert modular setting}  \label{sec:HMVIw}
In this section we recall Pappas' definition of integral models of Hilbert modular varieties
of level $U_0(p)$ and level $U_1(p)$, where $U$ is a sufficiently small open compact subgroup
of $G(\A_\f) = \GL_2(\A_{F,\f})$ of level prime to $p$, and
$$U_1(p) = \{g\in U\,|\, g_p \equiv \begin{psmallmatriks} * & * \\ 0 & 1
 \end{psmallmatriks}\bmod p\CO_{F,p}\}\;\subset \;
 U_0(p) = \{g\in U\,|\, g_p \equiv \begin{psmallmatriks} * & * \\ 0 & *
 \end{psmallmatriks}\bmod p\CO_{F,p}\}.$$
Thus writing $U = U^pU_p$ with $U_p = \GL_2(\CO_{F,p})$, we have
$U_0(p) = U^pU_{0,p}$ and $U_1(p) = U^p U_{1,p}$ where
$U_{0,p}$ is an Iwahori subgroup of $U_p$ and
$U_0(p)/U_1(p) \cong U_{0,p}/U_{1,p} \cong (\CO_F/p)^\times$.

\subsubsection{Level $U_0(p)$} \label{sss:hmv.U0p}
For $U_0(p)$, we consider the functor which associates to a $\Z_{(p)}$-scheme $S$
the set of isomorphism classes of triples $(\underline{A}_1,\underline{A}_2,f)$, where
\begin{itemize}
\item $\underline{A}_j = (A_j,\iota_j,\lambda_j,\eta_j)$ for $j=1,2$ represent $S$-points of $\widetilde{Y}_U(G)$, and
\item $f:A_1 \to A_2$ is an isogeny of degree $p^d$ such that $f\circ \iota_1(\alpha) = \iota_2(\alpha)\circ f$ for all $\alpha \in \CO_F$, $p\lambda_1 = f^\vee \circ \lambda_2 \circ f$, $\eta_2 = f \circ \eta_1$
(as $U^p$-orbits for each $\overline{s}_i$), and $H = \ker f$ decomposes as $\prod_{v|p} H_v$ where $H_v \subset A_1[v]$ has rank $p^{[F_v:\Q_p]}$ for each prime $v$ of $\CO_F$ dividing $p$.
\end{itemize}
The results of \cite{P} (see in particular Theorem~2.2.2) show that this functor is representable
by a flat local complete intersection over $\Z_{(p)}$
of constant relative dimension $d = [F:\mathbb{Q}]$.  Moreover, letting 
$\widetilde{Y}_{U_0(p)}(G)$ denote the representing
scheme, the forgetful morphism $\widetilde{Y}_{U_0(p)}(G) \to \widetilde{Y}_U(G)$
sending $(\underline{A}_1,\underline{A}_2,f)$ to $\underline{A}_1$
is projective.  

By the proof of Lemma~2.1.2 of \cite{GK}, the scheme $\widetilde{Y}_{U_0(p)}(G)$ 
also represents the functor sending a $\Z_{(p)}$-scheme $S$
the set of isomorphism classes of pairs $(\underline{A},H)$, where
\begin{itemize}
\item $\underline{A} = (A,\iota,\lambda,\eta)$ represents an $S$-point of $\widetilde{Y}_U(G)$, and
\item $H$ is a finite flat $(\CO_F/p)$-submodule scheme of $A[p]$ over $S$ which is
{\em totally isotropic} in the sense that the $\lambda$-Weil pairing $A[p] \stackrel{\sim}{\to}
A^\vee[p] = (A[p])^\vee$ induces an isomorphism
$$H \stackrel{\sim}{\longrightarrow} \ker((A[p])^\vee \to H^\vee)$$
\end{itemize}
(where we use $\cdot^\vee$ to denote Cartier duals as well as dual abelian varieties).
The natural isomorphism between the functors is defined by sending 
$(\underline{A}_1,\underline{A}_2,f)$ to $(\underline{A}_1,\ker f)$, and
we use the two descriptions of $\widetilde{Y}_{U_0(p)}(G)$ interchangeably.

The group $\CO_{F,(p),+}^\times$ acts on $\widetilde{Y}_{U_0(p)}(G)$ via its action on the
quasi-polarization $\lambda$;
 as usual the action factors through a free action of $\CO_{F,(p),+}^\times/(U\cap F^\times)^2$ 
 (note that $U\cap F^\times = U_0(p) \cap F^\times$), and we denote the quotient scheme
by $Y_{U_0(p)}(G)$.  It follows that $Y_{U_0(p)}(G)$
is also a flat local complete intersection of relative dimension $d$
over $\Z_{(p)}$ and the morphism ${Y}_{U_0(p)}(G) \to {Y}_U(G)$ is projective,
so ${Y}_{U_0(p)}(G)$ is quasi-projective over $\Z_{(p)}$.  

For open compact subgroups $U$, $V \subset G(\A_\f)$ as above, and $g \in G(\A_\f^{(p)})$
such that $g^{-1}Ug \subset V$, the usual construction affords a finite \'etale
morphism $\tilde{\rho}_g:\widetilde{Y}_{U_0(p)}(G) \to \widetilde{Y}_{V_0(p)}(G)$ descending
to ${\rho}_g:{Y}_{U_0(p)}(G) \to{Y}_{V_0(p)}(G)$, and satisfying $\tilde{\rho}_h\circ \tilde{\rho}_g
 = \tilde{\rho}_{gh}$ and hence $\rho_h\circ\rho_g  = \rho_{gh}$ for $h \in G(\A_\f^{(p)})$ and 
 sufficiently small $W$ containing $h^{-1}Vh$.  Moreover the schemes $Y_{U_0(p)}(G)$ and
 morphisms $\rho_g$ are integral models for the corresponding morphisms and varieties
 of canonical models over $\Q$ of level $U_0(p)$.
 
 \subsubsection{Level $U_1(p)$ and Raynaud group schemes}\label{sss:U1p}
For $U_1(p)$, we consider the functor which associates to an $\Z_{(p)}$-scheme $S$
the set of isomorphism classes of triples $(\underline{A},H,P)$ where $\underline{A}$ and $H$ are as above and
\begin{itemize}
\item $P \in H(S)$ is an $(\CO_F/p)$-generator of $H$ in the sense of Drinfeld--Katz--Mazur~\cite[1.10]{KM}.
\end{itemize}
The functor is represented by a $\Z_{(p)}$-scheme which we denote $\widetilde{Y}_{U_1(p)}(G)$,
and the forgetful morphism $\widetilde{Y}_{U_1(p)}(G) \to \widetilde{Y}_{U_0(p)}(G)$ is finite flat
(\cite[Thm.~2.3.3]{P}).  

Following \cite{P}, the scheme $\widetilde{Y}_{U_1(p)}(G)$ can be described explicitly in terms of the universal
$H$ over $\widetilde{Y}_{U_0(p)}(G)$.  We first recall the classification of certain finite flat group schemes
considered by Raynaud in \cite{raynaud}.  Suppose that $L$ is a number field containing the $(q-1)$-roots of unity,
where $q-1$ is a divisible by the exponent of $(O_F/p)^\times$, let $\CO$ be the localization of $\CO_L$ at the 
prime over $p$ determined by our choice of embeddings of $\Qbar$, and let $S$ be a scheme over $\CO$.
We say that a finite flat $(\CO_F/p)$-module scheme $H$ over $S$
is {\em Raynaud} if condition  ($\smallblackstar\smallblackstar$)
of \cite{raynaud} is satisfied by the $(\CO_F/v)$-vector space scheme $H[v]$ for each $v|p$.
By \cite[Thm.~1.4.1]{raynaud}, to give a Raynaud $(\CO_F/p)$-module scheme $H$ is equivalent to giving
invertible $\CO_S$-modules $\CL_\theta$ for each $\theta \in \Theta$ and
morphisms $s_\theta:\CL_\theta^{p} \to \CL_{\phi\circ\theta}$,
$t_\theta: \CL_{\phi\circ\theta}\to \CL_\theta^{ p} $ such that
$s_\theta \circ t_\theta = w_v$ for all $v|p$ and $\theta \in \Theta_v$, where
$w_v$ is a certain fixed element of $p\CO^\times$.
The $(\CO_F/p)$-module scheme corresponding
to this data is given by
$$H = \SPEC ((\sym_{\CO_S}\CL)/\CI),$$
where $\CL = \bigoplus_{\theta\in \Theta} \CL_\theta$,
$\alpha \in \CO_F$ acts on $\CL_\theta$
as the Teichmuller lift of $\overline{\theta}(\alpha)$,
and $\CI$ is the sheaf of ideals generated by the $\CO_{S}$-submodules 
$(s_\theta - 1)\CL_\theta^{ p}$ for $\theta \in \Theta$.
The comultiplication $\CO_H \to \CO_H\otimes_{\CO_S}\CO_H$
defining the group law on $H$ is given via duality by the scheme
structure on the Cartier dual $H^\vee$, which is identified with
$$ \SPEC ((\sym_{\CO_S}\CM)/\CJ),$$
where $\CM = \Shom_{\CO_S}(\CL,\CO_S) = \bigoplus_{\theta\in \Theta}\CL_\theta^{-1}$
and $\CJ$ is generated by
$(t_\theta - 1)\CL_\theta^{-p}$ for $\theta \in \Theta$.
Furthermore if $p\CO_S = 0$, then the $s_\theta$ can be identified with the morphisms
induced by $\Frob:H \to H^{(p)}$, and the $t_\theta$ with those induced by $\Ver: H^{(p)} \to H$.

We write $\CR_H$ for the direct image of $\CO_H$
under the structure morphism $H \to S$, so that $\CR_H$ is the sheaf of
$\CO_S$-algebras $(\sym_{\CO_S}\CL)/\CI$, identified as $\CO_S$-modules with
$$\bigoplus \left(\bigotimes_{\theta \in \Theta} \CL_\theta^{ m_\theta}\right),$$
where the direct sum is over $(m_\theta)_{\theta\in \Theta}$ with $0 \le m_\theta \le p-1$
for all $\theta$, and similarly let 
$$\CR_{H^\vee} = (\sym_{\CO_S}\CM)/\CJ  = \bigoplus \left(\bigotimes_{\theta \in \Theta} \CL_\theta^{-m_\theta}\right)$$
denote the direct image of $\CO_{H^\vee}$.

For later reference, we record a general result which we will use to filter
Raynaud $(\CO_F/p)$-module schemes in characteristic $p$.  
\begin{lemma}  \label{lem:slicing}
Suppose that $p\CO_S = 0$ and
$I \subset \Theta$ has the property that $t_{\theta} = 0$ for all $\theta\in I$ 
and $s_{\theta}=0$ for all $\theta\not\in I$.
Let $\CA_I$ denote the sheaf of $\CR_H$-ideals generated by the $\CL_{\theta}$
for $\theta\not\in I$, and let $C_I$ denote the
closed subscheme of $H$ defined by $\CA_I$.  Then
the scheme $C_I$ is a finite flat $(\CO_F/p)$-submodule scheme of $H$
of rank $p^{|I|}$ over $S$, and
$\Lie(C_I^\vee/S) = \bigoplus_{\theta\in I} \CL_{\theta}$ as
an $\CO_F\otimes\CO_S$-module.
\end{lemma}
\begpf   Since $s_{\theta} = 0$ for $\theta\not\in I$, it follows that
$\CA_I$ is the direct sum of the invertible $\CO_S$-submodules
$\bigotimes_{\theta\in \Theta} \CL_\theta^{m_\theta}$
of $\CR_H$ for which $m_\theta > 0$ for some $\theta\not\in I$.
In particular $\CR_H/\CA_I$ is locally free over $\CO_S$ of rank $p^{|I|}$,
and is stable under the action of $\CO_F$.

The assertion that $C_I$ is a subgroup scheme amounts to the claim that the morphism
$C_I \times_S C_I \to H$ defined by the group law on $H$ factors through $C_I$:
$$\xymatrix{ \CR_H  \ar[r]\ar@{->>}[d]  & \CR_H \otimes_{\CO_S} \CR_H \ar@{->>}[d] \\
\CR_H/\CA_I \ar@{-->}[r] & (\CR_H/\CA_I) \otimes_{\CO_S}  (\CR_H/\CA_I).}$$
Applying $\Shom_{\CO_S}(\cdot,\CO_S)$ to the diagram of morphisms of
flat $\CO_S$-modules renders the factorization equivalent to that of
$$\xymatrix{ \CA_I^\perp \otimes_{\CO_S} \CA_I^\perp
  \ar@{-->}[r]\ar@{^{(}->}[d]  & \CA_I^\perp \ar@{^{(}->}[d]  \\
 \CR_{H^\vee} \otimes_{\CO_S} \CR_{H^\vee} \ar[r] & \CR_{H^\vee},}$$
 where the bottom arrow is multiplication and $\CA_I^\perp$ is the kernel of the morphism
 $$\CR_{H^\vee}  = \Shom_{\CO_S}(\CR_H,\CO_S)  
    \to \Shom_{\CO_S}(\CA_I,\CO_S).$$
We must therefore prove that $\CA_I^\perp$ is a subalgebra of $\CR_{H^\vee}$, but
$$\CA_I^\perp = \bigoplus \left(\bigotimes_{\theta \in I}  
\CL_{\theta}^{-m_{\theta}}\right),$$
where the direct sum is over $(m_{\theta})_{\theta\in I}$
with $0 \le m_{\theta} \le p-1$ for all $\theta \in I$,
so this is immediate from the vanishing of the 
$t_{\theta}$ for $\theta \in I$.
Furthermore since $C_I^\vee = \SPEC(\CA_I^\perp)$ and the augmentation
ideal sheaf of $\CA_I^\perp$ is generated by the $\CL_{\theta}^{-1}$
for $\theta \in I$, the assertion about $\Lie(C_I^\vee/S)$ is also immediate
from this vanishing.
\epf

We now return to the discussion of integral models of Hilbert modular varieties of Iwahori level,
and let $S = \widetilde{Y}_{U_0(p)}(G)_\CO$ where $\CO$ is as above.
As in \cite[Lemma~4.2.2]{P}, the universal $(\CO_F/p)$-module scheme $H$ over $S$
is Raynaud.  Furthermore, the argument of \cite[5.1]{P} shows that $\widetilde{Y}_{U_1(p)}(G)_\CO$
can be identified with the closed
subscheme  of $H$ defined by the sheaf of ideals generated by the
$(s_v - 1)\left(\bigotimes _{\theta\in \Theta_v}
 \CL_\theta^{(p-1)}\right)$ for $v|p$,
 where $s_v = \bigotimes_{\theta\in \Theta_v} s_\theta$ is viewed as a morphism
 $\bigotimes_{\theta\in \Theta_v}\CL_\theta^{(p-1)} \to \CO_S$.

We assume as usual that $U$ is sufficiently small, and indeed $p$-neat, i.e. $\alpha - 1 \in p\CO_F$
for all $\alpha \in U \cap F^\times$ as in \S\ref{sss:avb.hmf}, or equivalently $U_1(p) \cap F^\times = U \cap F^\times$,
so the action of the group $\CO_{F,(p),+}^\times$ on $\widetilde{Y}_{U_1(p)}(G)$ via multiplication on $\lambda$ again
factors through a free action of  $\CO_{F,(p),+}^\times/(U \cap F^\times)^2$.
We denote the quotient scheme by $Y_{U_1(p)}(G)$;
thus $Y_{U_1(p)}(G)$ is finite flat over $Y_{U_0(p)}(G)$, so it is 
Cohen--Macaulay and quasi-projective of relative dimension $d$ over $\Z_{(p)}$.

Note furthermore that the hypothesis on $U$ ensures that the
canonical isomorphism $\theta_\mu^*\underline{A} \cong \underline{A}$ over $\widetilde{Y}_0(p)(G)_\CO$
is the identity on $H$ for all $\mu \in (U\cap F^\times)^2$, so $H$ descends to a finite flat
group scheme on $Y_{U_0(p)}(G)_\CO$, which we also denote by $H$.  It is also Raynaud
in the above sense; in fact the line bundles $\CL_\theta$ and sections $s_\theta$, $t_\theta$
all descend to yield the same descriptions of $H$, $H^\vee$ and $Y_{U_1(p)}(G)_\CO$
as spectra of finite flat $\CO_S$-algebras over $S = Y_{U_0(p)}(G)_\CO$.
We remark also that under the same hypothesis on $U$, the natural  action
of $(\CO_F/p\CO_F)^\times$ on $\widetilde{Y}_{U_1(p)}(G)$ defined
by $c\cdot(\underline{A},H,P) = (\underline{A},H,c\cdot P)$ (for which
$\widetilde{Y}_{U_0(p)}(G)$ is the quotient) descends to an action
of $(\CO_F/p\CO_F)^\times$ on $Y_{U_1(p)}(G)$ (with $Y_{U_0(p)}(G)$
as quotient).

As usual, for sufficiently small $U$, $V$ of level prime to $p$ and $g \in G(\A_\f^{(p)})$
such that $g^{-1}Ug \subset V$, we obtain a finite \'etale
morphism $\tilde{\rho}_g:\widetilde{Y}_{U_1(p)}(G) \to \widetilde{Y}_{V_1(p)}(G)$, descending
to ${\rho}_g:{Y}_{U_1(p)}(G) \to{Y}_{V_1(p)}(G)$ and satisfying
$\tilde{\rho}_h\circ \tilde{\rho}_g = \tilde{\rho}_{gh}$ and $\rho_h\circ\rho_g  = \rho_{gh}$
for suitable $h$ and $W$.  Furthermore the schemes $Y_{U_1(p)}(G)$ and
 morphisms $\rho_g$ are integral models for the corresponding morphisms and varieties
 of canonical models over $\Q$ of level $U_1(p)$.  

Finally, writing $(\underline{A}_U,H_U)$ (resp.~$(\underline{A}_V,H_V)$) for the universal data
over $\widetilde{Y}_{U_0(p)}(G)_\CO$ (resp.~$\widetilde{Y}_{V_0(p)}(G)_\CO$) and
$\CL_{U,\theta}$ (resp.~$\CL_{V,\theta}$) for the associated Raynaud bundles, note 
that the canonical quasi-isogeny $\pi_g$ from $A_U$ to $\tilde{\rho}_g^* A_V$
induces an isomorphism $H_U \to \tilde{\rho}_g^* H_V$ and hence, by the functoriality
of Raynaud's construction, isomorphisms $\pi_g^*:\tilde{\rho}_g^*\CL_{V,\theta} \to \CL_{U,\theta}$
satisfying the usual compatibility $\pi_{gh}^* = \pi_g^*\circ\tilde{\rho}_g^*(\pi_h^*)$.
Moreover if $V$ (and hence $U$) satisfies the additional assumption that $\alpha - 1 \in p\CO_F$ for all
$\alpha \in V \cap F^\times$, then the isomorphism $H_U \to \tilde{\rho}_g^* H_V$ descends to
one over $Y_{U_0(p)}(G)_\CO$, as do the isomorphisms $\pi_g^*$ of Raynaud bundles, now
satisfying $\pi_{gh}^* = \pi_g^*\circ{\rho}_g^*(\pi_h^*)$.

\subsection{The unitary setting} \label{sec:unitaryIw}
We now define integral models for unitary Shimura varieties  of Iwahori level
proceeding similarly to the Hilbert case.   We shall only need this for the analogue
of $U_0(p)$, and only in the case of $\Sigma = \emptyset$.

\subsubsection{The moduli problem}\label{sss:usv.U0p}
Let $G' = G'_\emptyset$, suppose as usual that $U'$ is a sufficiently small
open compact subgroup of $G'(\A_\f)$ of level prime to $p$, and define
$$ U'_0(p) = \{g\in U'\,|\, g_p \equiv \begin{psmallmatriks} * & * \\ 0 & *
 \end{psmallmatriks}\bmod p\CO_{E,p}\}.$$
Here $G'(\Q_p)$ is viewed as the subgroup $\GL_2(F_p)E_p^\times$ of $\GL_2(E_p)$,
 so if $U' = (U')^pU'_p$, then $U'_0(p) = (U')^pU'_{0,p}$ where $U'_{0,p}$ is the image of 
 $U_{0,p} \times \CO_{E,p}^\times$ in $\GL_2(E_p)$. 

Now consider the functor which associates to a $\Z_{(p)}$-scheme $S$
the set of isomorphism classes of triples $(\underline{A}_1,\underline{A}_2,f)$, where
\begin{itemize}
\item $\underline{A}_j = (A_j,\iota_j,\lambda_j,\eta_j,\epsilon_j)$ for $j=1,2$ represent $S$-points of 
$\widetilde{Y}_{U'}(G')$, and
\item $f:A_1 \to A_2$ is an isogeny of degree $p^{4d}$ such that 
$f\circ \iota_1(\alpha) = \iota_2(\alpha)\circ f$ for all $\alpha \in \CO_D$, $p\lambda_1 = f^\vee \circ \lambda_2 \circ f$, $\eta_2 = f \circ \eta_1$
(as $(U')^p$-orbits for each $\overline{s}_i$), $\epsilon_2 = p\epsilon_1$ and $H = \ker f$ decomposes as $\prod_{w|p} H_w$ 
where $H_w \subset A_1[w]$ has rank $p^{2[E_w:\Q_p]}$ for each prime $w$ of $\CO_E$ dividing $p$.
\end{itemize}

Again standard arguments show that the functor is representable by a scheme,
which we denote by $\widetilde{Y}_{U'_0(p)}(G')$, and
the forgetful morphism to $\widetilde{Y}_{U'}(G')$ sending $(\underline{A}_1,\underline{A}_2,f)$
to $\underline{A}_1$ is projective.

\subsubsection{Dieudonn\'e theory over $\Fpbar$}\label{sss:D1}
We shall need to relate $\widetilde{Y}_{U_0(p)}(G)$ and $\widetilde{Y}_{U'_0(p)}(G')$, but first we recall
some general facts from Dieudonn\'e theory.\footnote{We follow the conventions of \cite{oda}, which
differ from those of \cite{fontaine} by a Frobenius twist.  More precisely, our $\D(G)$ is the
$D(G)$ of \cite[(4.2.8.1)]{BBM}, which is canonically isomorphic to $\mathbf{M}(G)
 = \underline{M}(G)\otimes_{W,\phi} W$ where $\mathbf{M}$ (resp.~$\underline{M}$) 
 is the functor defined in \cite{oda} (resp.~\cite{fontaine}); see \cite[Thm.~4.2.14]{BBM}.
Here however we use the semilinear versions of $\Phi$ and $V$.}
If $A$ is an abelian variety over $\Fpbar$ of dimension $g$,
then the (contravariant) Dieudonn\'e module $D = \D(A[p^\infty])$ is a free $W$-module
of rank $2g$ equipped with $\phi$-semilinear endomorphism $\Phi$, induced by the
absolute Frobenius on $A$, such that $pD \subset \Phi(D)$. 
Furthermore there is a canonical (i.e., functorial in $A$) isomorphism
$D/pD = \D(A[p]) \cong H^1_\dr(A/\Fpbar)$ identifying $VD/pD$ with
$H^0(A,\Omega^1_{A/\Fpbar})$, where $V$ is the $\phi^{-1}$-semilinear endomorphism
of $D$ defined by $\Phi^{-1}p$.  Letting $D^\vee = \D(A^\vee[p^\infty])$,
we also have a canonical isomorphism
$\hom_W(D,W)  \cong D^\vee$ under which $\Phi$ and $V$ are semilinear dual, and
the induced perfect pairing $D \times D^\vee \to W$ is anti-symmetric under the
canonical isomorphism $A \cong (A^\vee)^\vee$. 

Recall also that the functor $\D$ defines a contravariant equivalence of categories between
finite (commutative) group schemes over $\Fpbar$ killed by $p$ and $\Fpbar$-vector spaces
equipped with a $\phi$-semilinear endomorphism $\Phi$ and $\phi^{-1}$-semilinear
endomorphism $V$ such that $\Phi V = V \Phi = 0$.   For any such group scheme $H$,
we have $\mathrm{rank}(H) = p^{\dim_{\Fpbar}(\D(H))}$.  Furthermore $\D$ is compatible
with formation of Cartier duals in the obvious sense, and there are canonical isomorphisms
$$\D(H)/\Phi\D(H) \cong \hom_{\Fpbar}(\Lie(H)^{(p)},\Fpbar)\quad\mbox{and}\quad
     \ker(\D(H) \stackrel{V}{\longrightarrow} \D(H))\cong \Lie(H^\vee)^{(p)}.$$
Applying $\D$ to the canonical
isomorphism $A^\vee[p]  \cong (A[p])^\vee$ gives the reduction mod $p$ of the
isomorphism $D^\vee \cong \hom_W(D,W)$ recalled above.  In particular, if
$\lambda: A \to A^\vee$ is a polarization of $A$, then the $\lambda$-Weil pairing
$A[p] \to (A[p])^\vee$ corresponds under $\D$ to the morphism
$\hom_{\Fpbar}(D/pD,\Fpbar) \to D/pD$ obtained as the reduction mod $p$
of the composite
$$ \hom_W(D,W) \stackrel{\sim}{\longrightarrow} D^\vee \longrightarrow  D.$$
Thus if $\lambda \in \hom(A,A^\vee) \otimes \Z_{(p)}$ is a prime-to-$p$
quasi-polarization, then the $\lambda$-Weil pairing on $A[p^\infty]$ induces an
isomorphism $\hom_W(D,W)\cong D$ corresponding to an alternating perfect
pairing on $D$, and whose reduction mod $p$ defines the isomorphism
$\hom_{\Fpbar}(D/pD,\Fpbar) \cong D/pD$.

\subsubsection{Relation between Hilbert and unitary settings}\label{sss:Iwcomp}
Recall from \S\ref{sec:comp} that for open compact
$U \subset G(\A_\f)$ and $U' \subset G'(\A_\f)$ containing the image of $U$,
we defined a morphism  $\tilde{i}:\widetilde{Y}_U(G) \to \widetilde{Y}_{U'}(G')$
sending $A$ to $A' = A\otimes_{\CO_F}\CO_E^2$.  Furthermore under the
hypotheses of Lemma~\ref{lem:cartesian}, the morphism $\tilde{i}$ induces
an isomorphism
$$\widetilde{Y}_U(G) \longrightarrow \widetilde{Y}_{U'}(G') \times_{C'} C$$
identifying $\widetilde{Y}_U(G)$ with a union of connected components of $\widetilde{Y}_{U'}(G')$,
where $C$ (resp.~$C'$) is the finite set $C_{\det (U)}$ (resp.~$C_{\nu'(U')}$)
viewed as a scheme over $\Z_{(p)}$.
The construction of $\tilde{i}$ clearly extends to define a morphism
$\widetilde{Y}_{U_0(p)}(G) \to \widetilde{Y}_{U'_0(p)}(G')$, which in turn induces a morphism
$$\widetilde{Y}_{U_0(p)}(G) \longrightarrow \widetilde{Y}_{U'_0(p)}(G') \times_{C'} C$$
where the maps $\widetilde{Y}_{U_0(p)}(G) \to C$ and $\widetilde{Y}_{U'_0(p)}(G) \to C'$ 
are defined by composing with the forgetful morphisms.

\begin{lemma} \label{lem:cart2}  Under the hypotheses of Lemma~\ref{lem:cartesian}, the morphism
$$\widetilde{Y}_{U_0(p)}(G) \longrightarrow \widetilde{Y}_{U'_0(p)}(G') \times_{C'} C$$
is an isomorphism.
\end{lemma}
\begpf First note that since $\widetilde{Y}_{U_0(p)}(G) \to \widetilde{Y}_{U}(G)
\to \widetilde{Y}_{U'}(G')$ is projective, so is $\widetilde{Y}_{U_0(p)}(G) \to
\widetilde{Y}_{U'_0(p)}(G)$.  Furthermore since $\widetilde{Y}_U(G) \to
\widetilde{Y}_{U'}(G)$ is injective on $S$-points for all $S$, it follows from the
definitions of the corresponding functors that
so is $\widetilde{Y}_{U_0(p)}(G) \to \widetilde{Y}_{U'_0(p)}(G)$.  Being projective and
quasi-finite, it must be finite, so in fact a closed immersion, and hence so is the morphism
of the lemma.  To prove it is an isomorphism, it therefore suffices to construct a section
$S \to \widetilde{Y}_{U_0(p)}(G)$ for each connected component $S$ of
$\widetilde{Y}_{U'_0(p)}(G') \times_{C'} C$.

Let $S$ be such a connected component, and let $(\underline{A}_1',\underline{A}_2',f')$
be the universal triple over $S$.  By assumption, the $S$-point of $\widetilde{Y}_{U'}(G')$
defined by $\underline{A}_1'$ factors through $\tilde{i}$, meaning that $\underline{A}_1'$
is defined by $A_1 \otimes_{\CO_F}\CO_E^2$ (with its additional structure) 
for some $S$-point $\underline{A}_1$ of $\widetilde{Y}_U(G)$.  Since
$\widetilde{Y}_{U'_0(p)}(G')$ is locally of finite type over $\Z_{(p)}$, we may
choose a geometric point $\overline{s} : \Spec k \to S$ such that either $k = \Fpbar$
or $k$ has characteristic  zero.

Writing $B_j$ for
$e_0A'_j$ where $e_0=\begin{psmallmatrix} 1&0\\0&0\end{psmallmatrix} \in \CO_D =
M_2(\CO_E)$ and identifying $A_j' = B_j \otimes_{\CO_E} \CO_E^2$ for $j=1,2$,
the isogeny $f'$ induces an isogeny $g:B_1 \to B_2$ such that
$f' = g \otimes_{\CO_E} \CO_E^2$.  Similarly since $e_0^* = e_0$, we have
prime-to-$p$ quasi-polarizations $\mu_j$ on the $B_j$ such that 
$p\mu_1 = g^\vee \circ \mu_2 \circ g$.  Moreover the
quasi-polarization $\mu_1$ on $B_1 = A_1 \otimes_{\CO_F} \CO_E$
is induced by $\lambda_1$ and the $\CO_F$-linear pairing on $\CO_E$ defined by
$(\alpha,\beta) \mapsto \tr_{E/F}(\alpha\overline{\beta})$.

Let $H' = \ker (f'_{\overline{s}})$.  We claim that $H' = H \otimes_{\CO_F} \CO_E^2$ for some
totally isotropic subgroup $H \subset A_{1,\overline{s}}[p]$.  
First note that $H' = I \otimes_{\CO_E} \CO_E^2$ where $I = e_0H' = \ker g_{\overline{s}}$
is a finite $(\CO_E/p\CO_E)$-submodule scheme of 
$$B_{1,\overline{s}}[p] =    (A_{1,\overline{s}}\otimes_{\CO_F} \CO_E)[p]
     =   \prod_{w|p}  (A_{1,\overline{s}} \otimes_{\CO_F} \CO_E)[w] = \prod_{w|p}B_{1,\overline{s}}[w] ,$$
the products being over primes $w$ of $\CO_E$ dividing $p$.
Furthermore we may decompose  $I = \prod_{w|p} I_w$ where each
$I_w$ is an $(\CO_E/w)$-subspace scheme 
of $B_{1,\overline{s}}[w]$ of rank $p^{[E_w:\Q_p]}$.
Note that for each prime $w$ of $\CO_F$, the isomorphism
$\CO_F/v \cong \CO_E/w$ induces an isomorphism
$A_{1,\overline{s}}[v] \cong B_{1,\overline{s}}[w]$, where $v$ is the prime
of $\CO_F$ dividing $w$, and we define $H_w$ to be the subgroup of $A_{1,\overline{s}}[v]$
corresponding to $I_w$.  We will prove that $H_w$ is isotropic and
coincides with $H_{w^c}$ for each $w$, from which the claim follows with
$H = \prod_{v} H_{\widetilde{v}}$.

Suppose first that $k = \Fpbar$ and consider the Dieudonn\'e modules
$D = \D(A_{1,\overline{s}}[p^\infty])$ and
$D_j = \D(B_{j,\overline{s}}[p^\infty])$ for $j=1,2$. Thus $D$
is an $\CO_F \otimes W$-module, hence decomposes as 
$\oplus_{\theta\in \Theta} D_{\theta}$, with $\Phi$ restricting to define
$\phi$-semilinear homomorphisms $D_{\theta} \to D_{\phi\circ\theta}$.
Similarly $D_1$ and $D_2$ are
$\CO_E \otimes W$-module, hence decompose as 
$\oplus_{\tau\in \Theta_E} D_{j,\tau}$, with $\Phi$ restricting
$\phi$-semilinear homomorphisms $D_{j,\tau} \to D_{j,\phi\circ\tau}$.
Furthermore each $D_\theta$ and $D_{j,\tau}$ is free of rank two
over $W$, and  we may identify $D_1$ with $D \otimes_{\CO_F} \CO_E$,
giving canonical isomorphisms $D_{1,\tau} \cong D_\theta$ for each 
$\tau \in \Theta_E$ restricting to $\theta \in \Theta$.
Note also that since $\Lie(B_j)_\tau$ is one-dimensional for each $\tau \in \Theta_E$,
the successive quotients in the filtrations
$$ p D_{j,\phi\circ\tau} \subset \Phi(D_{j,\tau})  \subset D_{j,\phi\circ\tau}$$
are all one-dimensional.

Now consider the inclusion $g^*:D_2 \to D_1$ induced by the isogeny $g$;
it is $\CO_E \otimes W$-linear, compatible with $\Phi$, and has cokernel isomorphic
to $\D(I)$.  Note in particular that we have the inclusions
$$ p D_{1,\tau} \subset g^*(D_{2,\tau})  \subset D_{1,\tau}$$
for each $\tau \in \Theta_E$, and since $I_w$ has rank $p^{[E_w:\Q_p]}$,
it follows that 
$$ \sum_{\tau\in \Theta_{E,w}} \dim_{\Fpbar}  D_{1,\tau}/g^*(D_{2,\tau}) =
\dim_{\Fpbar}  \D(I_w) = [E_w:\Q_p].$$
Since $g^*(\Phi(D_{2,\tau})) = \Phi(g^*(D_{2,\tau}))$ for each $\tau$, we have
$$\begin{array}{l}
\dim_{\Fpbar}(D_{1,\phi\circ\tau}/\Phi(D_{1,\tau})) + \dim_{\Fpbar}(D_{1,\tau}/g^*(D_{2,\tau}))\\
\ \quad =\dim_{\Fpbar}(D_{1,\phi\circ\tau}/\Phi(D_{1,\tau})) + \dim_{\Fpbar}(\Phi D_{1,\tau}/\Phi(g^*(D_{2,\tau})))\\
\ \quad = \dim_{\Fpbar}(D_{1,\phi\circ\tau}/g^*(D_{2,\phi\circ\tau})) + \dim_{\Fpbar}(g^*D_{2,\phi\circ\tau}/g^*(\Phi(D_{2,\tau})))\\
\ \quad =  \dim_{\Fpbar}(D_{1,\phi\circ\tau}/g^*(D_{2,\phi\circ\tau})) + \dim_{\Fpbar}(D_{2,\phi\circ\tau}/\Phi(D_{2,\tau})).\end{array}$$
Since $\dim_{\Fpbar} (D_{j,\phi\circ\tau}/\Phi(D_{j,\tau})) = 1$ for $j=1,2$, it follows that
 $\dim_{\Fpbar}  D_{1,\tau}/g^*(D_{2,\tau}) = 1$ for all $\tau$.
 We have now shown that the $(\CO_E/w)$-vector space scheme $I_w$
 has the property that $\D(I_w)$ is free of rank one over $(\CO_E/w) \otimes \Fpbar$,
 and hence $\D(H_w)$ is free of rank one over $(\CO_F/v) \otimes \Fpbar$.
 Since the pairing on $D/pD = \D(A_1[p])$ induced by the $\lambda_1$-Weil pairing
 is alternating and $(\CO_F/p\CO_F)$-bilinear, it follows that the composite
 $$\D(H_w^\vee) = \hom_{\Fpbar}(\D(H_w),\Fpbar)
   \hookrightarrow  \hom_{\Fpbar}(\D(A_{1,\overline{s}}[w]),\Fpbar)
   \stackrel{\sim}{\longrightarrow} \D(A_{1,\overline{s}}[w])  \twoheadrightarrow \D(H_w)$$
 is trivial, and hence that $H_w$ is isotropic.
 
 We must now show that $H_w = H_{w^c}$, or equivalently that $I_{w^c} = (1\otimes c)I_w$.
 We will do this by proving that each of $I_{w^c}$ and $(1\otimes c)I_w$ is the kernel of
 $$B_{1,\overline{s}}[w^c] \to (B_{1,\overline{s}}[w])^\vee \to I_w^\vee,$$
 where the first morphism induced by the $\mu_1$-Weil pairing on $B_1$.
 Since the Rosati involution associated to $\mu_1$ restricts to $c$ on
 the image of $\CO_E$ in $\End(B_{1,\overline{s}})$, the map
$B_{1,\overline{s}}[w^c] \to (B_{1,\overline{s}}[w])^\vee$ is in fact an isomorphism,
so comparing ranks, we see that it suffices to prove that each of $I_{w^c}$ and
$(1\otimes c)I_w$ is orthogonal to $I_w$.  For $(1\otimes c)I_w$, this is immediate
from the isotropy of $H_w$ and the commutativity of the diagram
$$\begin{array}{ccc}
A_1[v]  &\stackrel{\sim}{\longrightarrow} & A_1^\vee[v] \\
\wreath\downarrow && \wreath\downarrow \\
B_1[w^c]   &\stackrel{\sim}{\longrightarrow} &  B_1^\vee[w]\end{array}$$
implied by the relation between and $\lambda_1$ and $\mu_1$.
For $I_{w^c}$, this follows from the isotropy of $I$ under the $\mu_1$-Weil pairing,
which is equivalent to the vanishing of the map $\D(I^\vee) \to \D(I)$ induced by $\mu_1$.
To see this vanishing, note that the image of $\D(I^\vee)$ in 
$$\D((B_{1,\overline{s}}[p])^\vee) = \hom_{\Fpbar}(D_1/pD_1,\Fpbar) = 
\hom_W(D_1,W)\otimes_W \Fpbar$$
is the image of $L = \{\,\psi \in \hom_W(D_1,W)\,|\, \psi(g^*(D_2)) \subset pW\,\}$.
If $\psi \in L$, then $\psi\circ g^* = p\xi$ for some $\xi \in \hom_W(D_2,W)$,
and the relation $p\mu_1 = g^\vee \circ \mu_2 \circ g$ implies that 
$\mu_1^*(\psi) = g^*\mu_2^*(\xi)$ is in the image of $g^*$, hence has trivial image in $\D(I)$.
 
 This completes the proof of the claim in the case $k = \Fpbar$.
 We omit the proof in the case $k$ has characteristic zero, which is similar, but simpler,
 since one can argue using $p$-adic Tate modules instead of Dieudonn\'e modules.
 
 We have now shown that $\ker f_{\overline{s}}' = H \otimes_{\CO_F} \CO_E^2$ for some 
 totally isotropic $H \subset A_{1,\overline{s}}[p]$.
Letting $\pi: A_{1,\overline{s}} \to A_{2,\overline{s}} : = A_{1,\overline{s}}/H$ denote the 
natural projection, we may 
endow $A_{2,\overline{s}}$ with additional structure making 
 $(\underline{A}_{1,\overline{s}},\underline{A}_{2,\overline{s}},\pi)$ a $k$-point of
 $\widetilde{Y}_{U_0(p)}(G)$.  Furthermore we obtain an isomorphism 
$A_{2,\overline{s}} \otimes_{\CO_F} \CO_E^2 \cong A'_{2,\overline{s}}$
whose compatibility with $f'_{\overline{s}}$ and $\pi \otimes 1$ implies that
$\underline{A}'_{2,\overline{s}}$ corresponds to the image under $\tilde{i}$ of
the $k$-point of $\widetilde{Y}_{U}(G)$ defined by $\underline{A}_{2,\overline{s}}$.

Since $S$ is connected, it follows that the $S$-point of $\widetilde{Y}_{U'}(G')$
defined by $\underline{A}'_2$ factors through $\tilde{i}$, and hence that
$A_2'$ is isomorphic to $A_2 \otimes_{\CO_F} \CO_E^2$ with its additional
structure for some $\underline{A}_2$ corresponding to an $S$-point of
$\widetilde{Y}_U(G)$.  Moreover the fibre of $\underline{A}_2$ at $\overline{s}$
is isomorphic to the abelian variety with additional structure already denoted
$\underline{A}_{2,\overline{s}}$, and $f'_{\overline{s}} = \pi \otimes 1$ for some
$\CO_F$-linear isogeny $\pi:A_{1,\overline{s}} \to A_{2,\overline{s}}$.

We will now show that $f'  = f \otimes 1$ for some $\CO_F$-linear isogeny 
$f:A_1 \to A_2$.  To that end, we first observe that the natural maps
$$\hom_{\CO_F}(A_1,A_2) \otimes_{\CO_F} \CO_E
  \to \hom_{\CO_E}(A_1\otimes_{\CO_F}\CO_E,A_2 \otimes_{\CO_F} \CO_E)
  \to \hom_{\CO_D}(A_1',A_2')$$
are isomorphisms (the first follows for example from the fact that it has torsion-free
cokernel and becomes an isomorphism after tensoring with $\Q$, and the second
is clear).  Next note that the image of the map
$$\hom_{\CO_F}(A_1, A_2) \to \hom_{\CO_F}(A_1,A_2) \otimes_{\CO_F} \CO_E$$
defined by $f \mapsto f\otimes 1$ is precisely the set of elements fixed by the
automorphism $1 \otimes c$.  Moreover the same holds after replacing the $A_j$
by their fibres $A_{j,\overline{s}}$.  Since the natural map 
$$\hom_{\CO_F}(A_1, A_2) \to \hom_{\CO_F}(A_{1,\overline{s}}, A_{2,\overline{s}})$$
is injective and $f'_{\overline{s}} = \pi \otimes 1$ for some 
$\pi \in  \hom_{\CO_F}(A_{1,\overline{s}}, A_{2,\overline{s}})$, it follows that
$f' = f \otimes 1$ for some $f \in \hom_{\CO_F}(A_1,A_2)$.

Finally, the compatibility of $f'$ with the additional structure on $\underline{A}_1'$
and $\underline{A}_2'$ implies that of $f$ with the additional structure on
$\underline{A}_1$ and $\underline{A}_2$.  Therefore $(\underline{A}_1,\underline{A}_2,f)$
defines an $S$-point of $\widetilde{Y}_{U_0(p)}(G)$ whose composite with $\tilde{i}$
is the inclusion of $S$ in $\widetilde{Y}_{U'_0(p)}(G')$.
\epf

\subsubsection{Descent, Hecke action, and Raynaud bundles}\label{sss:Iw.dha}
We now return to the setting of arbitrary sufficiently small $U'$ of level prime to $p$ (not necessarily
satisfying the hypotheses of Lemma~\ref{lem:cartesian}).  As in \S\ref{sss:hmv.U0p}, we have a natural
action of $\CO_{F,(p),+}^\times$ on $\widetilde{Y}_{U'_0(p)}(G')$, now factoring through a free action
of $\CO_{F,(p),+}^\times/\Nm_{E/F}(U'\cap E^\times)$, and we denote the quotient scheme
by $Y_{U'_0(p)}(G')$.  Thus ${Y}_{U'_0(p)}(G') \to {Y}_{U'}(G')$ is projective, and
${Y}_{U'_0(p)}(G')$ is quasi-projective over $\Z_{(p)}$.  As usual, for open compact
subgroups $U'$, $V' \subset G'(\A_\f)$ as above, and $g \in G'(\A_\f^{(p)})$
such that $g^{-1}U'g \subset V'$, we obtain a finite \'etale
morphism $\tilde{\rho}_g:\widetilde{Y}_{U'_0(p)}(G') \to \widetilde{Y}_{V'_0(p)}(G')$ descending
to ${\rho}_g:{Y}_{U'_0(p)}(G') \to{Y}_{V'_0(p)}(G')$.  These satisfy $\tilde{\rho}_h\circ \tilde{\rho}_g
 = \tilde{\rho}_{gh}$ and hence $\rho_h\circ\rho_g  = \rho_{gh}$ for $h \in G'(\A_\f^{(p)})$ and 
 sufficiently small $W'$ containing $h^{-1}V'h$, and the schemes $Y_{U'_0(p)}(G')$ and
 morphisms $\rho_g$ are integral models for the corresponding morphisms and varieties
 of canonical models over $\Q$ of level $U'_0(p)$.
  
 As in \S\ref{sss:U1p}, let $\CO$ be the localization of $\CO_L$ at the distinguished
 prime over $p$ for a sufficiently large $L$.  
  \begin{proposition}  \label{prop:raynaud} Let $S' = \widetilde{Y}_{U'_0(p)}(G')_\CO$,
 let $H'$ denote the kernel of the universal isogeny  $A_1' \to A_2'$ on $S'$,
 and let $I = e_0H'$.  Then
 \begin{enumerate}
 \item the $\CO_E/p\CO_E$-module scheme $I$ is Raynaud;
 \item $H'$ is totally isotropic with the respect to the $\lambda_1'$-Weil pairing on $A_1'$;
 \item for any $\overline{s}:\Spec(\Fpbar) \to S'$, the Dieudonn\'e module $\D(I_{\overline{s}})$
 (resp.~$D(H'_{\overline{s}})$) is free of rank one (resp.~two) over $\CO_E\otimes \Fpbar$.
\end{enumerate}
\end{proposition}
 \begpf (1) Suppose first that we are in the setting of 
 Lemma~\ref{lem:cartesian}, so that $U' = UV_E$ for sufficiently small $U$ and $V_E$,
 and let $H$ be the kernel of the universal isogeny on $\widetilde{Y}_{U_0(p)}(G)_{\CO}$.
Since $\tilde{i}^*I = H \otimes_{\CO_F} \CO_E$, 
it follows from Lemma~\ref{lem:cart2} that the restriction of $I$ to 
$S' \times_{C'} C$ is Raynaud.  Furthermore, for any connected
component $S$ of $S'$, we may choose $h \in( \A_{E,\f}^{(p)})^\times$
so that $\widetilde{\rho}_h$ sends $S$ to $S'\times_{C'} C$,
and since the canonical quasi-isogeny $\pi_h$ from $A_1' \to \widetilde{\rho}_h^* A_1'$ induces
an isomorphism $I \to \widetilde{\rho}_h^*I$, it follows that $I$ is Raynaud on $S$.
Therefore $I$ is Raynaud on $S'$.  
Returning to the setting of arbitrary $U'$, note that we may
choose sufficiently small $U$ and $V_E$ as above so that $V' = UV_E$ is contained
in $U'$.  Considering the natural projection $\widetilde{\rho}_1:
\widetilde{Y}_{V'_0(p)}(G')_\CO \to S'$, we see that
since $\widetilde{\rho}_1^*I$ is Raynaud, so is $I$.

Parts (2) and (3) may be similarly deduced from the corresponding properties
of $H$ and $I$ on $S'\times_{C'} C$ in the setting of Lemma~\ref{lem:cart2}
(or indeed shown more directly using arguments in its proof).
\epf

From the $\CO_E/p\CO_E$-module scheme $I$ and embeddings $\tau \in \Theta_E$,
we obtain as in \S\ref{sss:U1p} the Raynaud line bundles $\CL_\tau$ on $S'$,
together with morphisms $s_\tau:\CL_\tau^{p} \to \CL_{\phi\circ\tau}$,
$t_\tau: \CL_{\phi\circ\tau}\to \CL_\tau^{p} $ such that
$s_\tau \circ t_\tau \in p\CO^\times$.   For sufficiently small $U'$ and $V'$, and $g \in G'(\A_\f^{(p)})$
such that $g^{-1}U'g \subset V'$, the canonical quasi-isogeny 
$\pi_g$ from $A'_{1,U'}$ to $\tilde{\rho}_g^* A'_{1,V'}$
induces isomorphisms $I_{U'} \to \tilde{\rho}_g^* I_{V'}$ on 
$\widetilde{Y}_{U'_0(p)}(G')$, hence isomorphisms
$\pi_g^*:\tilde{\rho}_g^*\CL_{V',\theta} \to \CL_{U',\theta}$ satisfying the
usual compatibility $\pi_{gh}^* = \pi_g^*\circ\tilde{\rho}_g^*(\pi_h^*)$.
Furthermore under the 
additional assumption that $\alpha - 1 \in p\CO_E$ for all $\alpha \in U' \cap E^\times$,
the $\CO_E/p\CO_E$-module scheme $I$, line bundles $\CL_\tau$ and morphisms
$s_\tau$, $t_\tau$ all descend to objects, which we denote by the same
symbol, on the quotient $Y_{U'_0(p)}(G')_\CO$.  If $V'$ also satisfies this
condition, then $\pi_g^*$ descends to an isomorphism
$\rho_g^*\CL_{V',\theta} \to \CL_{U',\theta}$, and these satisfy
$\pi_{gh}^* = \pi_g^*\circ\rho_g^*(\pi_h^*)$.

Finally we note that for sufficiently small 
$U \subset G(\A_\f)$ and $U' \subset G'(\A_\f)$ such that $U'$ contains the image of $U$,
the identification $\tilde{i}^*I = H \otimes_{\CO_F} \CO_E$ allows us to identify
$\tilde{i}^*\CL_\tau$ with $\CL_\theta$, where $\theta = \tau|_F$, under which
$\tilde{i}^*s_\tau = s_\theta$ and $\tilde{i}^*t_\tau = t_\theta$.  Furthermore the
identifications are compatible with the descent data to $Y_{U'_0(p)}$ defined
by the action of $\CO_{F,(p)}^\times$, and with the morphisms $\pi_g^*$ for suitable
$g \in G(\A_\f^{(p)})$, $V\subset G(\A_\f)$ and $V' \subset G'(\A_\f)$.

\subsection{Stratifications}  \label{sec:strata}
We now define a stratification on the special fibre of $Y_{U_0(p)}(G)$, as in \cite{GK}.
As usual, we first consider $\widetilde{Y}_{U_0(p)}(G)$ and then descend to $Y_{U_0(p)}(G)$.

\subsubsection{Closed strata} \label{sss:strata.def}
Let $U$ be a sufficiently small open compact subgroup of $G(\A_\f) = \GL_2(\A_{F,\f})$
of level prime to $p$, and let $\F$ be a sufficiently large extension of $\F_p$ as in \S\ref{sss:avb.hmf}.
Let $S = \widetilde{Y}_{U_0(p)}(G)_{\F}$, and let $(\underline{A}_1,\underline{A}_2,f)$ be the
universal object over $S$.  We thus have morphisms
of line bundles $\Lie(f)_\theta : \Lie(A_1/S)_\theta \to \Lie(A_2/S)_\theta$ and
$\Lie(f^\vee)_\theta : \Lie(A_2^\vee/S)_\theta \to \Lie(A_1^\vee/S)_\theta$.

For $I,J \subset \Theta$,
we let $S_{\phi(I),J}$ denote the subscheme of $S$ defined by the vanishing of the
sections\footnote{The reason for $\phi(I)$ ($=\{\,\theta\,|\,\phi^{-1}\circ\theta\in I\,\}$) instead of $I$ in the notation is for consistency
with the notation in \cite{GK}.}
$$\{ \,\Lie(f)_\theta\,|\,\theta \in I\,\} \cup \{\, \Lie(f^\vee)_\theta\,|\,\theta \in J\,\}.$$
We write simply $S_J$ for $S_{\phi(I),J}$ if $I$ is the complement of $J$ in $\Theta$.

Note that if $I \subset I'$ and $J \subset J'$, then $S_{\phi(I'),J'}$ is a closed
subscheme of $S_{\phi(I),J}$, and that for any $I,J,I',J'$, we have
$$ S_{\phi(I\cup I'),J\cup J'} = S_{\phi(I),J}  \cap S_{\phi(I'),J'}.$$
Furthermore since $\Lie(\lambda_2)$ is an isomorphism and 
$\Lie(f^\vee\circ\lambda_2\circ f) = 0$  on $S$,  we have
\begin{equation} \label{eqn:bisectS}
 S_{\phi(I),J} = S_{\phi(I\cup\{\theta\}),J}  \cup S_{\phi(I),J\cup\{\theta\}}.\end{equation}
for $\theta \not\in I \cup J$.

\subsubsection{Local structure} \label{sss:strata.loc}
For a closed point $Q \in S = \widetilde{Y}_{U_0(p)}(G)_{\F}$, define 
$$I_Q=\{\theta \in \Theta: \Lie(f)_{\theta}(Q)=0\}\quad\mbox{and}\quad
J_Q=\{\theta \in \Theta: \Lie(f^\vee)_{\theta}(Q)=0\}.$$
Note that $I_Q \cup J_Q=\Theta$. %
The following description of the completed local ring $\widehat{\CO}_{S,Q}$ is
an immediate consequence of the proof of Theorem 2.4.1 of \cite{GK}.
(Note that although the moduli problems considered here and in \cite{GK} are slightly different, the same
arguments apply since prime-to-$p$ polarizations, quasi-polarizations and level structures deform uniquely.)
In the statement, we view $\langle \Lie(f)_\theta\rangle$ and $\langle \Lie(f^\vee)_\theta\rangle$ via 
trivializations of the relevant line bundles in a neighborhood of $Q$.

\begin{theorem}\label{thm: coordinates}  Let $Q$ be a closed point of $S = \widetilde{Y}_{U_0(p)}(G)_\F$. 
There is an isomorphism 
$$\widehat{\mathcal{O}}_{S,Q} \cong  
\widehat{\bigotimes_{\theta\in \Theta}}
k_Q[[x_\theta,y_\theta]]/\langle c_\theta d_\theta \rangle,$$
where the completed tensor product is over the residue field $k_Q$ at $Q$,
$$\begin{array}{rcccl}
\langle c_\theta \rangle & = & \langle \Lie(f)_\theta \rangle & = &
  \left\{\begin{array}{cl} 
      \langle x_\theta \rangle, & \mbox{if $\theta \in I_Q$,} \\
      \langle 1 \rangle, & \mbox{if $\theta \not\in I_Q$,} \end{array}\right. \\
\mbox{and}\quad \langle d_\theta \rangle & = & \langle \Lie(f^\vee)_\theta \rangle & = &
  \left\{\begin{array}{cl} 
      \langle y_\theta \rangle, & \mbox{if $\theta \in J_Q$,} \\
      \langle 1 \rangle, & \mbox{if $\theta \not\in J_Q$.} \end{array}\right.
 \end{array}$$
\end{theorem}

\begin{corollary}  \label{cor:lci}
The scheme $S_{\phi(I),J}$ over $\F$ is a reduced local complete intersection
of constant dimension $d - |I \cap J|$, and is smooth if $I \cup J = \Theta$.
\end{corollary}

\begin{proof}  If $Q$ is a closed point of $S_{\phi(I),J}$, then $I \subset I_Q$ and $J\subset J_Q$. Theorem~\ref{thm: coordinates} and
the definition of $S_{\phi(I),J}$ therefore imply that
$$\widehat{\mathcal{O}}_{S_{\phi(I),J},Q} \cong \widehat{\bigotimes_{\theta\in \Theta}} R_\theta, \quad \mbox{where}\quad
R_\theta =
 \left\{\begin{array}{cl} 
      k_Q[[x_\theta,y_\theta]]/\langle c_\theta d_\theta \rangle, & \mbox{if $\theta \not\in I \cup J$,} \\
      k_Q[[x_\theta]], & \mbox{if $\theta \in J - I$,}\\
      k_Q[[y_\theta]], & \mbox{if $\theta \in I - J$,}\\
      k_Q,& \mbox{if $\theta \in I \cap J$.} \end{array}\right.$$
It follows that the completed local ring of $S_{\phi(I),J}$ at every closed point is a reduced complete intersection of dimension $d-|I \cap J|$,
and is regular if $I \cup J = \Theta$.
\end{proof}

\subsubsection{Descent and Hecke action}\label{sss:strata.dha}
Note that the natural action of $\CO_{F,(p),+}^\times$ on the line bundles 
$$\Shom_{\CO_S}(\Lie(A_1/S)_\theta,\Lie(A_2/S)_\theta) \quad\mbox{and}\quad
\Shom_{\CO_S}(\Lie(A_2^\vee/S)_\theta,\Lie(A_1^\vee/S)_\theta)$$
factors through the quotient $\CO_{F,(p),+}^\times/(F^\times \cap U)^2$,
so the line bundles descend to $\overline{Y}_0(p)$, as do the sections
$\Lie(f)_\theta$ and $\Lie(f^\vee)_\theta$.  We may thus similarly define
closed subschemes $\overline{Y}_0(p)_{\phi(I),J}$ of 
$\overline{Y}_0(p) := Y_{U_0(p)}(G)_\F$ by the vanishing of
$\Lie(f)_\theta$ for $\theta \in I$, and $\Lie(f^\vee)_\theta$ for $\theta \in J$,
or equivalently as the quotient of $S_{\phi(I),J}$ by the action of
$\CO_{F,(p),+}^\times$.  We again simply write $\overline{Y}_0(p)_J$ for
$\overline{Y}_0(p)_{\phi(I),J}$ where $I = \Theta - J$.  Analogous relations carry
over with the $S_{\phi(I),J}$ replaced by $\overline{Y}_0(p)_{\phi(I),J}$, as does
Corollary~\ref{cor:lci}.  In particular it follows that each irreducible component
of $\overline{Y}_0(p)$, with its reduced induced structure, is smooth of dimension
$d$ and is contained in $\overline{Y}_0(p)_J$ for a unique $J \subset \Theta$,
and that these irreducible subschemes are precisely the connected components
of $\overline{Y}_0(p)_J$.

The stratification\footnote{With respect to the obvious partial ordering on the set of pairs $(I,J)$, the locally closed subschemes 
$$\ol{Y}_0(p)_{\phi(I),J} - \bigcup_{(I',J') > (I,J)} \ol{Y}_0(p)_{\phi(I'),J'}$$
define a stratification of $\ol{Y}_0(p)$ in the usual sense, but we will make no direct use of this fact.
Our interest is mainly in the subschemes $\ol{Y}_0(p)_J$, but we will use the
term {\em stratification} as shorthand for the collections of closed subschemes indexed by $(I,J)$.}
 is compatible with the Hecke action for varying $U$ in the sense
that if $U$ and $V$ are sufficiently small of level prime to $p$ and $g \in \GL_2(\A_{F,\f}^{(p)})$
is such that $g^{-1}Ug \subset V$, then the morphisms denoted $\widetilde{\rho}_g$ and
$\rho_g$ restrict to ones on the corresponding closed subschemes for each $I,J \subset \Theta$.

\subsubsection{The unitary setting}\label{sss:strata.usv}
We may similarly define closed subschemes of $S' := \widetilde{Y}_{U_0'(p)}(G')_{\F}$
for each $I,J \subset \Theta$, where
$G' = G'_\emptyset$ and $U' \subset G'(\A_\f)$ are as in \S\ref{sss:usv.U0p}. 
Recall that we have fixed a subset $\widetilde{\Theta} \subset \Theta_E$
mapping bijectively to $\Theta$ under $\tau \mapsto \tau|_F$ and denoted the corresponding
extension of each $\theta \in \Theta$ by $\widetilde{\theta}$.  We then define $S'_{\phi(I),J}$
by the vanishing of the sections
$$\{ \,\Lie(f')^0_{\widetilde{\theta}}\,|\,\theta \in I\,\} 
\cup \{\, \Lie({f'}^\vee)^0_{\widetilde{\theta}^c}\,|\,\theta \in J\,\},$$
where $f'$ is the universal isogeny on $S'$, and as usual
${}^0$ denotes the restriction to the direct summand obtained by applying
the idempotent $e_0 = e_0^* \in \CO_D$.
We then obtain the same relations as above with $S$ replaced by $S'$, and
again the sections descend to the quotient $\overline{Y}'_0(p) : = Y_{U_0'(p)}(G')_{\F}$ 
to define closed subschemes $\overline{Y}'_0(p)_{\phi(I),J}$ satisfying
analogous relations.  Furthermore the morphisms $\widetilde{\rho}_g$ and
$\rho_g$ for $g \in G'(\A_{\f}^{(p)})$ induce ones on the corresponding
subschemes $S'_{\phi(I),J}$ and $\overline{Y}'_0(p)_{\phi(I),J}$ for $U$ and
$V$ satisfying the usual conditions.  Again we write simply $S_J'$ and 
$\overline{Y}'_0(p)_J$ when $I$ is the complement of $J$.

For sufficiently small $U \subset G(\A_\f)$ and $U' \subset G'(\A_\f)$ such that $U'$
contains the image of $U$, the morphism $\tilde{i}: S \to S'$ defined by applying 
$\otimes_{\CO_F}\CO_E^2$ yields an identification 
$\tilde{i}^*\Lie(A_j'/S')^0_{\widetilde{\tau}} =  \Lie(A_j/S)_\theta$ for $j =1,2$, 
under which $\Lie(f')^0_{\widetilde{\theta}}$ pulls back to $\Lie(f)_\theta$.
We have an analogous identification for the Lie algebras of the dual abelian varieties
(taking into account the isomorphism $\CO^2_E \otimes \F \to
\hom_{\CO_F}(\CO^2_E,\CO_F) \otimes \F$ induced by $(\alpha,\beta)
\mapsto \Tr_{E/F}(\alpha\overline{\beta})$) 
under which $\Lie({f'}^\vee)^0_{\widetilde{\theta}^c}$ pulls back to $\Lie(f^\vee)_\theta$.
It follows that $\tilde{i}$ respects the stratifications in the sense that
$S_{\phi(I),J} = S \times_{S'} S'_{\phi(I),J}$, and hence that we may
similarly identify 
$$\overline{Y}_0(p)_{\phi(I),J} = \overline{Y}_0(p) \times_{ \overline{Y}'_0(p)} \overline{Y}'_0(p)_{\phi(I),J}$$
under the morphism $i:  \overline{Y}_0(p) \to \overline{Y}'_0(p)$.
Furthermore under the hypotheses of Lemma~\ref{lem:cartesian}, it follows from
Lemma~\ref{lem:cart2} that $\tilde{i}$ and $i$ induce identifications
$S_{\phi(I),J} = S'_{\phi(I),J} \times_{C'} C$ and
$\overline{Y}_0(p)_{\phi(I),J} = \overline{Y}'_0(p)_{\phi(I),J} \times_{C'} C$.

Finally, we remark that the vanishing locus of $\Lie(f')^0_{\widetilde{\theta}}$ is the same
as that of $\Lie(f')^0_{\widetilde{\theta}^c}$, and the same holds with $f'$ replaced by $(f')^\vee$.
To see this, note that by descent we can replace $U'$ by a smaller open compact subgroup
to reduce to the setting of Lemma~\ref{lem:cart2}, where it is immediate that the claim
holds on $S = S' \times_{C'} C$.  It follows that it holds on $S'$ since for any connected
component, we may choose $h \in (\A_{E,\f}^{(p)})^\times$ such that $\rho_h$ maps
the component to $S' \times_{C'} C$.  Similarly the conclusions of Corollary~\ref{cor:lci}
carry over with $S_{\phi(I),J}$ replaced by $S'_{\phi(I),J}$, and hence by
$\overline{Y}'_0(p)_{\phi(I),J}$.

\section{A Jacquet--Langlands relation}  \label{sec:JL}

\subsection{Splicing}  The key new ingredient in obtaining our geometric relation will be the
introduction of abelian varieties whose Dieudonn\'e modules are described by ``splicing'' those
of the source and target of the universal isogeny at each closed point of the Iwahori level
moduli space.

\subsubsection{Preliminaries} \label{sss:splice.pre}
Fix a subset $J \subset \Theta$ and a sufficiently small level $U' \subset G'(\A_\f)$ of level
prime to $p$.  Recall that the scheme $S'_J$ is defined in \S\ref{sss:strata.usv}
as the closed subscheme of $\widetilde{Y}_{U'_0(p)}(G')_{\F}$ defined by the
vanishing of the sections
$$\{ \,\Lie(f')^0_{\widetilde{\theta}}\,|\,\theta \not\in J\,\} 
\cup \{\, \Lie({f'}^\vee)^0_{\widetilde{\theta}^c}\,|\,\theta \in J\,\}$$
where $f':A'_1 \to A'_2$ is the universal isogeny.  

Since $G$ will play no (direct) role in this section, we will write simply $S$ for $S_J'$
and $f:A_1 \to A_2$ for the universal isogeny on $S$, and let $H' = \ker(f)$ and
$H = e_0H'$.  Recall from part (1) of Proposition~\ref{prop:raynaud} that
$H$ is a Raynaud $\CO_E/p\CO_E$-module scheme on $S$.  This means that
we have invertible $\CO_S$-modules $\CL_\tau$ for $\tau \in \Theta_E$
and morphisms $s_\tau:\CL_\tau^{p} \to \CL_{\phi\circ\tau}$,
$t_\tau: \CL_{\phi\circ\tau}\to \CL_\tau^{ p} $ such that
$s_\tau \circ t_\tau = 0$, in terms of which we may write 
$H = \SPEC \CR_H$ where $\CR_H$ is the sheaf of
$\CO_S$-algebras
$$(\sym_{\CO_S}\CL)/\CI  \cong \bigoplus \left(\bigotimes_{\tau \in \Theta_E} \CL_\tau^{ m_\tau}\right)$$
on $S$, $\CI$ is the sheaf of ideals generated by the $\CO_{S}$-submodules 
$(s_\tau - 1)\CL_\tau^{ p}$ for $\tau \in \Theta_E$, and the action of
$\alpha \in \CO_E$ on $H$ is defined by
multiplication by $\tau(\alpha)$ on $\CL_\tau$
Similarly the Cartier dual $H^\vee$ is identified with
$$ \SPEC ((\sym_{\CO_S}\CM)/\CJ),$$
where $\CM = \Shom_{\CO_S}(\CL,\CO_S) = \bigoplus_{\tau\in \Theta_E}\CL_\tau^{-1}$
and $\CJ$ is generated by
$(t_\tau - 1)\CL_\tau^{-p}$ for $\tau \in \Theta_E$.

\begin{lemma}  \label{lem:vanishing}
Suppose that $\tau \in \Theta_E$.  If $\tau|_F \not\in J$, then $s_{\phi^{-1}\circ\tau} = 0$,
and if $\tau|_F \in J$, then $t_{\phi^{-1}\circ\tau} = 0$.
\end{lemma}
\begpf  Suppose first that $\tau|_F \not\in J$, so that $\tau = \widetilde{\theta}$ or $\widetilde{\theta}^c$
for some $\theta \not\in J$, and hence $\Lie(f)^0_\tau = 0$ (see the discussion
at the end of \S\ref{sss:strata.usv}).  Since $\Lie$ is left exact for morphisms of group schemes,
we have that $\Lie(H/S)_\tau = \Lie(A_1/S)^0_\tau$ is invertible.  In terms of the Raynaud data
for $H$, the sheaf $\CA$ of augmentation ideals is generated over $\CR_H$ by 
$\bigoplus_\tau \CL_\tau$, so that $\CA/\CA^2 = 
\bigoplus_\tau (\CL_\tau/s_{\phi^{-1}\circ\tau}\CL_{\phi^{-1}\circ\tau}^{p})$ as an
$\CO_E \otimes \CO_S$-module, and therefore
$$\Lie(H/S)_\tau = 
\Shom_{\CO_S}(\CL_\tau/s_{\phi^{-1}\circ\tau}\CL_{\phi^{-1}\circ\tau}^{p},\CO_S).$$
Since $\CL_\tau$ is invertible, it follows that $s_{\phi^{-1}\circ\tau} = 0$.

Suppose now that $\tau|_F \in J$, so that $\Lie(f^\vee)^0_\tau = 0$.
Since $H^\vee \cong \ker(f^\vee)^0$ as an $(\CO_E/p\CO_E)$-module scheme, we
see as above that $\Lie(H^\vee/S)_\tau \cong \Lie(A_2^\vee/S)^0_\tau$ is invertible,
and it similarly follows that $t_{\phi^{-1}\circ\tau} = 0$.
\epf

\subsubsection{Universal splices}\label{sss:splice.univ}
In view of Lemma~\ref{lem:vanishing}, we may apply Lemma~\ref{lem:slicing} to the Raynaud $(\CO_F/p)$-module scheme
$\prod_v H[\widetilde{v}]$ with $I = \phi^{-1}(J)$ to obtain a finite flat subgroup scheme
$C_J$ over $S$, which we view as a finite flat $\CO_E/p\CO_E$-submodule scheme of $H$.
Thus $C_J$ has rank $p^{|J|}$ over $S$, and
$\Lie(C_J^\vee/S) = \bigoplus_{\theta\in J} \CL_{\phi^{-1}\circ\widetilde{\theta}}$ as
an $\CO_E\otimes\CO_S$-module.
Having defined $C_J \subset H = e_0H'$, we let $C_J'$ denote the finite flat
$\CO_D$-submodule scheme $C_J \otimes_{\CO_E} \CO_E^2 \subset H'$.
Note that in fact $C_J' \subset \prod_v H'[\widetilde{v}] \subset A_1[\prod_v\widetilde{v}]$
and hence $\lambda_1(C_J') \subset A_1^\vee[\prod_v\widetilde{v}^c]$ (where $\lambda_1$ is the
quasi-polarization on $A_1$).
Taking the dual of the natural projection $A_1^\vee \to A_1^\vee/\lambda_1(C_J')$
yields an abelian scheme $A_J' := (A_1^\vee/\lambda_1(C_J'))^\vee$ with an action
of $\CO_D$, together with an $\CO_D$-linear isogeny $\pi:A_J' \to A_1$ whose
kernel is isomorphic to $(C_J')^\vee$ (the Cartier dual of $C_J'$), compatibly with the $\CO_D$-action under
the anti-involution $\alpha \mapsto \alpha^* = (\alpha^c)^t$ of $\CO_D$.
Equivalently we may define $A_J' = A_1/(C_J')^\perp$ with the isogeny $\pi$
induced by multiplication by $p$ on $A_1$.  It is immediate from either description
that $\pi$ restricts an isomorphism $A_J'[\prod_v\widetilde{v}]
 \stackrel{\sim}{\longrightarrow} A_1[\prod_v \widetilde{v}]$.
We then define the {\em universal $J$-splice} to be the abelian scheme 
 $A_J = A_J'/C_J''$, where $C_J'' = \pi^{-1}(C_J'))[\prod_v\widetilde{v}]$. 
 More generally for a triple $(\underline{A}_1,\underline{A}_2,f)$
 corresponding to an $S$-point of $S_J'$ for an arbitrary base $S$, we define the
 {\em $J$-splice} of $f:A_1 \to A_2$ to be the pull-back of the universal $J$-splice.
 
 Note that the universal $J$-splice $A_J$ inherits an $\CO_D$-action from
 the $\CO_D$-action on $A_J'$.  Furthermore the prime-to-$p$ quasi-polarization
 $\lambda_1$ induces one on $A_J$, which we denote $\lambda_J$.  To see this,
 consider (over each connected component of $S$) the commutative diagram:
 $$\xymatrix{A_1 \ar[d]_{n\lambda_1}
 & A_J'  \ar[l]_{\pi} \ar[r]^{\psi} \ar[d]_{n\lambda_J'} & A_J \ar@{-->}[d]^{n\lambda_J} \\
 A_1^\vee \ar[r]^{\pi^\vee} & (A_J')^\vee & A_J^\vee \ar[l]_{\psi^\vee}
 }$$
where $n$ is a positive integer prime to $p$ such that $n\lambda_1$
is a polarization of degree prime to $p$, $\lambda_J' = \pi^\vee\circ\lambda_1\circ \pi$
and $\psi$ is the natural projection.  By construction, $C_J''$ is contained in
$\ker(n\lambda_J')$, so $n\lambda_J'$ factors through an $\CO_E$-antilinear isogeny
$\xi:A_J \to (A_J')^\vee$.  Furthermore since $n\lambda_J' = \psi^\vee\circ \xi^\vee$,
$\psi^\vee$ induces an isomorphism $A_J^\vee[\prod_v \widetilde{v}^c] \to
(A_J')^\vee[\prod_v \widetilde{v}^c]$ and $\xi^\vee$ is $\CO_E$-antilinear, it follows
that $C_J''$ is contained in the kernel of $\xi^\vee$, and so
$\xi^\vee = \lambda\circ \psi$ for some polarization $\lambda$ of $A_J$.
Since $\pi$ and $\psi$ each have degree $p^{2|J|}$ and $n\lambda_1$ has
degree prime to $p$, it follows that $\lambda$ has degree prime to $p$, and
we let $\lambda_J = n^{-1}\lambda$.  Note that the associated Rosati
involution is compatible with the anti-involution $*$ on $D$.

\subsubsection{Relation of Dieudonn\'e modules}\label{sss:splice.dieu}
We now explain the relation between the Dieudonn\'e modules of $A_1$, $A_2$
and $A_J$ at closed points of $S_J'$ which motivates the construction (and
name) of the $J$-splice.  First note that since $H' = \ker(f)$ is totally isotropic
with respect to $\lambda_1$ (by part (2) of Proposition~\ref{prop:raynaud})
 the projection $A_1 \to A_J'$ induced by multiplication
by $p$ factors through $f:A_1 \to A_2$, so we obtain a diagram of isogenies
\begin{equation} \label{eqn:splice}\xymatrix{&& A_1 \ar[rd]&&\\
&A_J' \ar[ru]^{\pi}\ar[rd]_{\psi}&&  A_1/C_J' \ar[rd]&\\
A_2 \ar[ru] && A_J \ar[ru] && A_2 } \end{equation}
such that the square commutes, the composite 
down the top right is $f:A_1 \to A_2$, and
along the top left is the morphism $A_2 \to A_1$
whose composite with $f$ is multiplication by $p$.
For any $\overline{s}:\Spec{\Fpbar} \to S_J'$,
we have the morphisms of Dieudonn\'e modules induced by the
diagram (\ref{eqn:splice}) over $S=\Spec(\Fpbar)$:
\begin{equation}\label{eqn:splice2}\xymatrix{&& \D(A_1[p^\infty]) \ar[ld]_{\pi^*}&&\\
&\D(A_J'[p^\infty]) \ar[ld]&&  \D(A_1[p^{\infty}]/C_J')\ar[ld]\ar[lu]&\\
\D(A_2[p^\infty]) && \D(A_J[p^{\infty}])\ar[lu]^{\psi^*} && \D(A_2[p^\infty]).\ar[lu]}\end{equation}
All the morphisms become isomorphisms after inverting $p$,
so $(\pi^*)^{-1}\circ\psi^*$ identifies $\D(A_J[p^{\infty}])$ with an $\CO_D$-stable
$W$-lattice in $\D(A_1[p^\infty])\otimes \Q$ whose components switch between
those arising from $A_1$ and $A_2$ as dictated by the first two formulas in
following proposition:
\begin{proposition}  \label{prop:splice}  If $\theta \in \Theta$, then
$$
(\pi^*)^{-1}\psi^*(\D(A_J[p^{\infty}])_{\widetilde{\theta}} ) = \left\{  \begin{array}{cl}
\D(A_1[p^\infty])_{\widetilde{\theta}} , &  \mbox{if $\theta \not\in J$,}\\
f^*\D(A_2[p^\infty])_{\widetilde{\theta}} , &  \mbox{if $\theta \in J$;} \end{array} \right.$$
$$(\pi^*)^{-1}\psi^*(\D(A_J[p^{\infty}])_{\widetilde{\theta}^c} ) = \left\{  \begin{array}{cl}
\D(A_1[p^\infty])_{\widetilde{\theta}^c},  &  \mbox{if $\theta \not\in J$,}\\
p^{-1}f^*\D(A_2[p^\infty])_{\widetilde{\theta}^c} , &  \mbox{if $\theta \in J$;} \end{array}\right.$$
$$
\dim_{\Fpbar}  (\Lie(A_J)_{\widetilde{\theta}}) = 
 \left\{  \begin{array}{cl}
4, & \mbox{if $\theta \in J$ and $\phi\circ\theta \not\in J$,}\\
0, & \mbox{if $\theta \not\in J$ and $\phi\circ\theta \in J$,}\\
2,  &  \mbox{otherwise;}\end{array}  \right. $$
$$\mbox{and}\quad
\dim_{\Fpbar}  (\Lie(A_J)_{\widetilde{\theta}^c})  = 4 - \dim_{\Fpbar}  (\Lie(A_J)_{\widetilde{\theta}}).$$  
\end{proposition}
\begpf  Comparing dimensions shows the natural inclusion $\Lie(C_J^\vee)^{(p)} \to \D(C_J)$
is an isomorphism, so $\dim_{\Fpbar} \D(C_J)_\tau = 1$ if $\tau = \widetilde{\theta}$ for some
$\theta \in J$, and $\D(C_J)_\tau = 0$ otherwise; therefore $\dim_{\Fpbar} \D(C'_J)_\tau = 2$
if $\tau = \widetilde{\theta}$ for some $\theta \in J$, and $\D(C'_J)_\tau = 0$ otherwise.
By part (3) of Proposition~\ref{prop:raynaud}, we have $\dim_{\Fpbar}\D(H')_\tau = 2$ for all $\tau$,
so if $\theta \in J$, then $\D(H'/C'_J)_{\widetilde{\theta}} = 0$, and hence
$$\D(A_2[p^\infty])_{\widetilde{\theta}}  \longrightarrow  \D(A_1[p^{\infty}]/C_J')_{\widetilde{\theta}} $$
is an isomorphism.  On the other hand, if $\tau \not\in \{\,\widetilde{\theta}\,|\,\theta \in J\,\}$,
then 
$$\D(A_1[p^{\infty}]/C_J')_\tau \longrightarrow \D(A_1[p^\infty])_\tau \quad\mbox{and}\quad
\D(A_J[p^{\infty}])_\tau \longrightarrow \D(A_J'[p^\infty])_\tau$$
are isomorphisms.  A similar analysis shows that if $\theta \in J$, then
$$\D(A_J'[p^\infty])_{\widetilde{\theta}^c}  \longrightarrow  \D(A_2[p^{\infty}])_{\widetilde{\theta}^c} $$
is an isomorphism, but if $\tau \not\in \{\,\widetilde{\theta}^c\,|\,\theta \in J\,\}$,
then 
$$\D(A_1[p^{\infty}])_\tau \longrightarrow \D(A_J'[p^\infty])_\tau \quad\mbox{and}\quad
\D(A_1[p^{\infty}]/C_J')_\tau \longrightarrow \D(A_J[p^\infty])_\tau$$
are isomorphisms.  The first two formulas in the proposition then follow from composing
isomorphisms arising in each case.

For $\tau \in \Theta_E$  and $A = A_1$, $A_2$ or $A_J$, we have
$$\begin{array}{rl}
\dim_{\Fpbar}(\Lie(A)_\tau) &= \dim_{\Fpbar}(\Lie(A[p])_\tau)
 = \dim_{\Fpbar}({\Lie(A[p])^{(p)}})_{\phi\circ\tau} \\
 &=  \dim_{\Fpbar}(\hom_{\Fpbar}({\Lie(A[p])^{(p)}},\Fpbar)_{\phi\circ\tau})
 = \dim_{\Fpbar}(D_{\phi\circ\tau}/\Phi(D_\tau))\end{array}$$
where $D = \D(A[p^\infty])$.  This dimension is two for $A= A_1$, $A_2$,
and the morphisms $f^*$, $\pi^*$ and $\psi^*$ commute with $\Phi$.
Thus if $\tau|_F = \theta$ with $\theta,\phi\circ\theta \not\in J$, it follows
from the first two formulas of the proposition that 
$$\D(A_J[p^\infty])_{\phi\circ\tau}/\Phi \D(A_J[p^\infty])_\tau
    \cong \D(A_1[p^\infty])_{\phi\circ\tau}/\Phi \D(A_1[p^\infty])_\tau$$
has dimension two.  Similarly using $A_2$ instead of $A_1$,
we find the dimension is two if $\theta,\phi\circ\theta \in J$.

Suppose now that $\theta \in J$ and $\phi\circ\theta\not\in J$.
Applying $(\pi^*)^{-1}\psi^*$ to the inclusion
$p\D(A_J[p^\infty])_{\phi\circ\tau} \subset \Phi(\D(A_J[p^\infty])_\tau)$
implies that
$$p\D(A_1[p^\infty])_{\phi\circ\widetilde{\theta}}
  \subset \Phi(f^*\D(A_2[p^\infty])_{\widetilde{\theta}}),$$
but both spaces have codimension two in 
$f^*\D(A_2[p^\infty])_{\phi\circ\widetilde{\theta}}$, so
equality holds, and it follows that
$$\D(A_J[p^\infty])_{\phi\circ\widetilde{\theta}}/\Phi \D(A_J[p^\infty])_{\widetilde{\theta}}
    \cong \D(A_1[p^\infty])_{\phi\circ\widetilde{\theta}}/p \D(A_1[p^\infty])_{\phi\circ\widetilde{\theta}}$$
has dimension four.  On the other hand comparing codimensions
arising from the inclusions
$$\Phi(p^{-1}f^*\D(A_2[p^\infty])_{\widetilde{\theta}^c}) \subset
\D(A_1[p^\infty])_{\phi\circ\widetilde{\theta}^c}
  \subset p^{-1}f^*\D(A_2[p^\infty])_{\phi\circ\widetilde{\theta}^c})$$
shows that $\Phi(p^{-1}f^*\D(A_2[p^\infty])_{\widetilde{\theta}^c}) =
\D(A_1[p^\infty])_{\phi\circ\widetilde{\theta}^c}$,
and hence that 
$$\D(A_J[p^\infty])_{\phi\circ\widetilde{\theta}^c}/\Phi \D(A_J[p^\infty])_{\widetilde{\theta}^c}
   = 0.$$
(Alternatively, this can be deduced using duality from the formula in the case of 
$\tau = \widetilde{\theta}$.)

We omit the proof in the case $\theta \not\in J$ and $\phi\circ\theta\in J$,
which is similar.

\epf

We have the following immediate
consequence for the de Rham cohomology and Lie algebra sheaves of the
universal  $J$-splice over $S_J'$, and hence for the pull-back to any base $S$:
\begin{corollary}  \label{cor:splice}  Let $A_J$ denote the $J$-splice of an isogeny $f:A_1 \to A_2$
corresponding to an $S$-point of $S_J'$.
\begin{enumerate}
\item If $\theta\in J$, then the isogenies $A_2 \to A_J$ and  $A_J \to A_2$ (along
the bottom of diagram (\ref{eqn:splice})) induce isomorphisms:
$$\CH^1_{\dr}(A_J/S)_{\widetilde{\theta}^c} \stackrel{\sim}{\longrightarrow}  
\CH^1_{\dr}(A_2/S)_{\widetilde{\theta}^c}   \quad\mbox{and} \quad
\CH^1_{\dr}(A_J/S)_{\widetilde{\theta}} \stackrel{\sim}{\longleftarrow}  
\CH^1_{\dr}(A_2/S)_{\widetilde{\theta}}.$$
\item If $\tau \in \Theta_E$ and $\tau|_F \not\in J$, then $\pi$ and $\psi$ induce isomorphisms
$$\CH^1_{\dr}(A_J/S)_\tau \stackrel{\sim}{\longrightarrow}  
\CH^1_{\dr}(A_J'/S)_\tau\stackrel{\sim}{\longleftarrow}  
\CH^1_{\dr}(A_1/S)_\tau.$$
\item If $\tau \in \Theta_E$, then $\Lie(A_J/S)_{\tau}$ is locally free of rank $2s_\tau$
over $\CO_S$, where 
$$s_{\widetilde{\theta}} = \left\{  \begin{array}{cl}
2, & \mbox{if $\theta \in J$ and $\phi\circ\theta \not\in J$,}\\
0, & \mbox{if $\theta \not\in J$ and $\phi\circ\theta \in J$,}\\
1,  &  \mbox{otherwise,}\end{array}  \right. $$
and $s_{\widetilde{\theta}^c} = 2 - s_{\widetilde{\theta}}$.
\end{enumerate}

\end{corollary}

\subsection{Unitary analogue of the theorem}
We continue to work with a fixed $J \subset \Theta$.  In this section we will use the
$J$-splice to construct an isomorphism analogous to the one in Theorem~\ref{thm:JLiso},
but in the context of the related unitary Shimura varieties.

\subsubsection{Preliminaries}\label{sss:JL.prel}
Let $\Sigma = \Sigma_J = \{\theta\in J\,|\,\phi\circ\theta\not\in J\} \cup \{\theta\not\in J\,|\,\phi\circ\theta\in J\}$,
and let\footnote{We apologize that in the exceptional situation where $\Sigma_J = \Theta$ (i.e., $\phi(J)$ is the
complement of $J$, which can occur only when all primes over $p$ have even degree), we have different
meanings for $\widetilde{\Theta}$ and $\widetilde{\Sigma}$.}
 $$\widetilde{\Sigma} = \{\widetilde{\theta}^c\,|\,\theta\in J,\phi\circ\theta\not\in J\} \cup
 \{\widetilde{\theta}\,|\,\theta\not\in J,\phi\circ\theta\in J\}.$$
Note that the cardinality of $\Sigma$ is even.

Recall from \S\ref{sec:usv} that $B_\Sigma$ denotes
the quaternion algebra over $F$ ramified at precisely the places in $\Sigma$, $G_\Sigma$
is the algebraic group over $\Q$ defined by $G_\Sigma(R) = (B_\Sigma\otimes R)^\times$,
$G_\Sigma'$ denotes $(G_\Sigma \times T_E)/T_F$, and $D_\Sigma = B_\Sigma \otimes_F E$.
We continue to write $B$ for $B_\emptyset = \GL_2(F)$ and similarly
$G = G_\emptyset = \Res_{F/\Q}\GL_2$, $G' = G'_\emptyset$ and $D = D_\emptyset = M_2(E)$.

We need to make some choices of algebraic data in order to define
the Shimura variety to which we will relate $\overline{Y}'_0(p)_J$.
We first choose an $\A_{F,\f}$-algebra isomorphism 
$$\xi:  B_\Sigma \otimes\widehat{\Z}   \stackrel{\sim}{\longrightarrow}B \otimes \widehat{\Z} = M_2(\A_{F,\f}),$$
and we let $\CO_{B_\Sigma}$ denote the corresponding maximal order in $B_\Sigma$,
i.e., $B_\Sigma \cap \xi(M_2(\widehat{\CO}_F))$, and let 
$\CO_{D_\Sigma} = \CO_{B_\Sigma} \otimes_{\CO_F} \CO_E$.
Thus $\CO_{D_\Sigma}$ is a maximal order in $D_\Sigma$, and
we choose an $\CO_E$-algebra isomorphism
$$\vartheta:   \CO_{D_\Sigma}\stackrel{\sim}{\longrightarrow}\CO_D = M_2(\CO_E) .$$
Finally, by Lemma~5.4 of \cite{TX}, we may choose an element $\delta_\Sigma \in D_\Sigma^\times$
satisfying the conditions in \S\ref{sss:usv.mp} with $D$ replaced by $D_\Sigma$ (for the
above choice of $\widetilde{\Sigma}$) and such that
the resulting anti-involution is compatible with $*$ under $\vartheta$;
more precisely:
\begin{itemize}
\item $\delta_\Sigma \in \CO_{D_\Sigma,p}^\times$,
\item $\overline{\delta}_\Sigma = - \delta_\Sigma$,
\item the bilinear form on $D_\Sigma \otimes \R$ defined by
  $$(v,w) \mapsto \Tr_{E/\Q}(\tr_{D_\Sigma/E}(v\overline{h}'_{\widetilde{\Sigma}}(i)  \overline{w} \delta_\Sigma))$$
is positive definite,
\item $\vartheta (\delta_{\Sigma}^{-1}\overline{\alpha}\delta_\Sigma )
  = \delta^{-1}\overline{\vartheta(\alpha)}\delta$ for all $\alpha \in D_\Sigma$.
\end{itemize}
Recall that we have chosen $\delta = \begin{psmallmatrix}0&1 \\ -1 & 0 \end{psmallmatrix}$,
and that our notation differs slightly from that of \cite{TX}, but the preceding assertion
is nonetheless immediate from their Lemma~5.4.  Having fixed such a choice of $\delta_\Sigma$,
we will abuse notation by writing $*$ for the corresponding anti-involution
 $\delta_{\Sigma}^{-1}\overline{\vartheta(\alpha)}\delta_\Sigma$ on $D_\Sigma$.

Note that the choices above give rise to two $\widehat{\CO}_E$-algebra isomorphisms
$$\xi_E,\,\widehat{\vartheta}:  \widehat{\CO}_{D_\Sigma}  \stackrel{\sim}{\longrightarrow}  \widehat{\CO}_{D},$$
so there exists $h \in \widehat{\CO}_D^\times$ such that $h\xi_E(u)h^{-1} = \widehat{\vartheta}(u)$
for all $u \in  \widehat{\CO}_{D_\Sigma}$.  Using the fact that $\xi_E(\overline{u}) = \overline{\xi_E(u)}$
for all $u$, it is straightforward to check that the element
$$\varepsilon := \overline{h}\delta h \xi_E(\delta_\Sigma^{-1})  \in \GL_2(\A_{E,\f})$$
is central and satisfies $\varepsilon = \overline{\varepsilon}$, so in fact 
$\varepsilon \in \A_{F,\f}^\times$; moreover $\varepsilon_p \in \CO_{F,p}^\times$.

\subsubsection{Construction of the morphism}\label{sss:JL.con}
For an open compact subgroup $U'$ of $G'(\A_\f)$, we let $U_\Sigma'$ denote the 
corresponding open compact subgroup of $G_\Sigma'(\A_\f)$ under the
isomorphism $G'(\A_\f) \cong G_\Sigma'(\A_\f)$ induced by $\xi$.
We assume that $U'$ is sufficiently small for both $U'$ and $U_\Sigma'$ to be sufficiently small
in the usual sense, and that $U'$ is of level prime to $p$, so that $U'_\Sigma$ is necessarily so as well.

In \S\ref{sss:splice.univ}, we defined the universal $J$-splice $A_J$ on the closed
subscheme $S_J'$ of $\widetilde{Y}_{U_0'(p)}(G')_\F$.  We will now endow it
with the additional structure needed to define a morphism 
$$\widetilde{\Psi}_J:  S_J' \to \widetilde{Y}_{U_\Sigma'}(G_\Sigma')_\F.$$
Firstly we define the action $\iota_J$ of $\CO_{D_\Sigma}$ by composing
$\vartheta:\CO_{D_\Sigma} \to \CO_D$ with the unique $\CO_D$-action on $A_J$
compatible with the quasi-isogeny $\psi\circ\pi^{-1} \in \hom(A_1,A_J) \otimes \Z[1/p]$.
Part (3) of Corollary~\ref{cor:splice}
then implies that for $\alpha \in \CO_{D_\Sigma}$, the action of $\iota_J(\alpha)$
on $\Lie(A_J/S_J')$ has characteristic
polynomial $\prod (x - \overline{\tau}(\alpha))^{2s_\tau} \in \F[x]$ as required.
Next recall that we also defined the prime-to-$p$ quasi-polarization $\lambda_J$ on $A_J$
to be compatible with $\psi\circ\pi^{-1}$, so the condition on the associated Rosati involution
follows from the compatibility of $\vartheta$ with the anti-involutions denoted $*$.
Finally we define $\eta_J$ as the composite of the isomorphisms
$$\widehat{\CO}_{D_\Sigma}^{(p)} \stackrel{\widehat{\xi}_E^{(p)}}{\longrightarrow}
 \widehat{\CO}_{D}^{(p)} \stackrel{h^{(p)}\cdot}{\longrightarrow}
 \widehat{\CO}_{D}^{(p)} \stackrel{\eta_{1,\overline{s}_i}}{\longrightarrow}
 \widehat{T}^{(p)}(A_{1,\overline{s}_i}) \stackrel{\psi\circ\pi^{-1}}{\longrightarrow}
 \widehat{T}^{(p)}(A_{J,\overline{s}_i})$$
for each $\overline{s}_i$, and we let $\epsilon_J = \varepsilon^{(p)} \epsilon_1$.
The fact that $(\eta_J,\epsilon_J)$ is a level $U_\Sigma'$-structure on
$(A_J,\iota_J,\lambda_J)$ then follows from the relations between $\vartheta$,
$\xi$, $h$ and $\varepsilon$.

Recall that in \S\ref{sss:avb.umf} we defined the rank two vector bundle $\widetilde{\CV}_\tau^0$
on $S=\widetilde{Y}_{U_\Sigma'}(G_\Sigma')_\F$ for  each $\tau \in \Theta_E$ as $\CH^1_\dr(A/S)_\tau^0$,
where $A$ is the universal abelian scheme on $S$ and $\cdot^0$ denotes the image under pull-back by
the idempotent $e_0$, defined as the element of
$\CO_{D_\Sigma} \otimes \F_p$ corresponding to the usual
$e_0 = \begin{psmallmatriks} 1 & 0 \\ 0 & 0 \end{psmallmatriks} \in M_2(\CO_E/p\CO_E)
= \CO_D \otimes \F_p$
under an isomorphism which we may take to be $\xi_E\otimes \F_p$.

We will now define an $S$-morphism $S_J' \to \P_S(\widetilde{\CV}_{\widetilde{\theta}}^0)$
for each $\theta \in \Sigma$.  To do so amounts to defining a surjective morphism from
$\widetilde{\Psi}_J^*\widetilde{\CV}_{\widetilde{\theta}}^0$ to a line bundle on $S_J'$, but this sheaf is canonically
identified with $\CH^1_\dr(A_J/S_J')^0$.  Recall from Part~(1) of Corollary~\ref{cor:splice} that if $\theta\in J$,
then the isogeny $A_J \to A_2$ induces an isomorphism 
$\CH^1_\dr(A_2/S_J')_{\widetilde{\theta}} \to \CH^1_\dr(A_J/S_J')_{\widetilde{\theta}}$
which is compatible with the actions of $\CO_D \cong \CO_{D_\Sigma}$ under the
isomorphism $\vartheta$.  Precomposing with the morphism induced by 
$h_p^{-1} \in \GL_2(\CO_E/p\CO_E)$ yields an isomorphism compatible with
the actions of $\CO_D \otimes \F_p \cong \CO_{D_\Sigma} \otimes \F_p$
under the isomorphism induced by $\xi$, and hence restricts to an isomorphism
$$\CH^1_\dr(A_2/S_J')^0_{\widetilde{\theta} }\to \CH^1_\dr(A_J/S_J')^0_{\widetilde{\theta}}
       = \widetilde{\Psi}_J^*\widetilde{\CV}_{\widetilde{\theta}}^0.$$
If $\theta \in J$ and $\phi\circ\theta \not\in J$, then we
define $S_J' \to \P_S(\widetilde{\CV}_{\widetilde{\theta}}^0)$ corresponding to
the surjective morphism in the exact sequence
$$0 \to (s_* \Omega^1_{A_2/S_J'})_{\widetilde{\theta}}^0 \to
            \CH^1_\dr(A_2/S_J')_{\widetilde{\theta}}^0 \to (R_1s_*\CO_{A_2})_{\widetilde{\theta}}^0 \to 0$$
arising from the Hodge filtration on $\CH^1_\dr(A_2/S_J')$, where $s:A_2 \to S_J'$ is the
structure morphism.
Similarly if $\theta \not\in J$ and $\phi\circ\theta \in J$, then we define 
 $S_J' \to \P_S(\widetilde{\CV}_{\widetilde{\theta}}^0)$ using Part~(2) of Corollary~\ref{cor:splice}
 and the Hodge filtration on $\CH^1_\dr(A_1/S_J')$.
 Taking the fibre product over $S$ of the morphisms just defined for $\theta \in \Sigma$, we
 obtain the morphism
 $$\widetilde{\Xi}_J:  S_J' \to \prod_{\theta \in \Sigma}  \P_S(\widetilde{\CV}_{\widetilde{\theta}}^0).$$
 
 \subsubsection{Proof of isomorphism}  \label{sss:JL.proof}
 We now prove the key result:
 \begin{theorem}  \label{thm:bigJL}  The morphism $\widetilde{\Xi}_J$ is an isomorphism.
\end{theorem}
 \begpf  
 Since the schemes are smooth of the same (constant) dimension over $\F$, it suffices to prove
 that the morphism is bijective on $\Fpbar$-points and injective on tangent spaces at such points.
 To see this, note that the condition on tangent spaces implies that 
 $\Omega^1_{S_J' /\prod_{\theta \in \Sigma}  \P_S(\widetilde{\CV}_{\widetilde{\theta}}^0)} = 0$.
 Furthermore since the schemes are regular of the same dimension and the fibres at all closed
 points have dimension zero, it follows that $\widetilde{\Xi}_J$ is flat, so in fact it is \'etale.  
 By \cite[18.2.8]{ega4}, we have that $\widetilde{\Xi}_J$ is finite, so the condition at closed points
 implies that it is in fact an isomorphism.

 To prove bijectivity on points, we will construct an inverse map.  Let $\underline{A} = 
 (A,\iota,\lambda,\eta,\epsilon)$ correspond to an $\Fpbar$-point of $\widetilde{Y}_{U_\Sigma'}(G_\Sigma')$
 and for each $\theta \in \Sigma$, let $L^0_\theta$ be a line in $\P(H^1_\dr(A/\Fpbar)^0_{\widetilde{\theta}})$,
 so the datum $(\underline{A},\{L^0_\theta\}_{\theta\in \Sigma})$  corresponds to an $\Fpbar$-point of 
 the target of $\widetilde{\Xi}_J$.  Let $\Delta$ denote the Dieudonn\'e module  $\D(A[p^\infty])$, with
 its right $M_2(\CO_{E,p})$-action obtained from the left action of $\CO_{D_\Sigma}$ on $A$
 and the isomorphism $\xi_p$.   Decomposing $\Delta = \bigoplus_{\tau\in \Theta_E} \Delta_\tau$
 where $\CO_{E,p}$ acts on $\Delta_\tau$ via $\tau$, each $\Delta_\tau$ is a right $M_2(W)$-module
 of rank $4$ over $W$ and the endomorphism $\Phi$ of $\Delta$ restricts to define injective
maps $\Delta_\tau \to \Delta_{\phi\circ\tau}$ which are $\phi$-semilinear
 with respect to the actions of $M_2(W)$.  Considering the inclusions
 $$p\Delta_{\phi\circ\tau}  \subset \Phi(\Delta_\tau) \subset \Delta_{\phi\circ\tau},$$
 the condition on $\Lie(A)$ means that the successive quotients are two-dimensional
 over $\Fpbar$ unless $\theta = \tau|_F \in \Sigma$, in which case:
 \begin{itemize}
 \item $\Phi(\Delta_{\widetilde{\theta}}) = \Delta_{\phi\circ\widetilde{\theta}}$ and
    $\Phi(\Delta_{\widetilde{\theta}^c}) = p\Delta_{\phi\circ\widetilde{\theta}^c}$
     if $\theta\not\in J$, $\phi\circ\theta \in J$;
 \item  $\Phi(\Delta_{\widetilde{\theta}}) = p\Delta_{\phi\circ\widetilde{\theta}}$ and
    $\Phi(\Delta_{\widetilde{\theta}^c}) = \Delta_{\phi\circ\widetilde{\theta}^c}$
     if $\theta\in J$, $\phi\circ\theta\not\in J$.
  \end{itemize}
 
 For each $\theta \in \Sigma$, we let $L_\theta = L_\theta^0  \otimes_{\Fpbar} \Fpbar^2$,
 viewed as a right $M_2(\Fpbar)$-submodule of $H^1_{\dr}(A/\Fpbar)_{\widetilde{\theta}}$,
 and let $\Lambda_{\widetilde{\theta}}$ denote
 the preimage of $L_\theta$ under the natural projection
 $$\Delta_{\widetilde{\theta}} \longrightarrow
    \D(A[p])_{\widetilde{\theta}}  \cong H^1_\dr(A/\Fpbar)_{\widetilde{\theta}}.$$
Recall that the prime-to-$p$ quasi-polarization $\lambda$ gives rise to a perfect
$W$-linear pairing on $\Delta$, inducing an isomorphism
$$\Delta_{\widetilde{\theta}} \longrightarrow \hom_W(\Delta_{\widetilde{\theta}^c},W),$$
and we let $\Lambda_{\widetilde{\theta}^c}$ denote the dual lattice to 
$\Lambda_{\widetilde{\theta}}$ in $\Delta_{\widetilde{\theta}^c}\otimes_{\Z_p} \Q_p$
under the pairing.  We thus have the $M_2(W)$-equivariant inclusions
 $$p\Delta_{\widetilde{\theta}} 
   \subset \Lambda_{\widetilde\theta}  \subset \Delta_{\widetilde{\theta}}\quad\mbox{and}\quad
    \Delta_{\widetilde{\theta}^c} 
   \subset \Lambda_{\widetilde{\theta}^c}  \subset p^{-1}\Delta_{\widetilde{\theta}^c}$$
 for each $\theta \in \Sigma$, where all the successive quotients are two-dimensional
 over $\Fpbar$.
 
 For each $\tau \in \Theta_E$, we define a $W$-lattice $\Delta_{1,\tau} \subset p^{-1}\Delta_\tau$
 as follows:
  \begin{itemize}
\item if $\theta\not\in J$, then $\Delta_{1,\widetilde{\theta}} =  \Delta_{\widetilde{\theta}}$
  and  $\Delta_{1,\widetilde{\theta}^c} =  \Delta_{\widetilde{\theta}^c}$;
\item if $\phi^{-1}\circ\theta\in J$ and $\theta\in J$, then $\Delta_{1,\widetilde{\theta}} =      
     p^{-1} \Phi(\Delta_{\phi^{-1}\circ\widetilde{\theta}})$ and $\Delta_{1,\widetilde{\theta}^c} =      
     \Phi(\Delta_{\phi^{-1}\circ\widetilde{\theta}^c})$;
\item if $\phi^{-1}\circ\theta\not\in J$ and $\theta\in J$, then $\Delta_{1,\widetilde{\theta}} =      
     p^{-1} \Phi(\Lambda_{\phi^{-1}\circ\widetilde{\theta}})$ and $\Delta_{1,\widetilde{\theta}^c} =      
     \Phi(\Lambda_{\phi^{-1}\circ\widetilde{\theta}^c})$;
\end{itemize}
and a $W$-lattice $\Delta_{2,\tau} \subset \Delta_{\tau}$ as follows:
\begin{itemize}
\item if $\theta\in J$, then $\Delta_{2,\widetilde{\theta}} =  \Delta_{\widetilde{\theta}}$
  and  $\Delta_{2,\widetilde{\theta}^c} =  p\Delta_{\widetilde{\theta}^c}$;
\item if $\phi^{-1}\circ\theta\not\in J$ and $\theta\not\in J$, then $\Delta_{2,\widetilde{\theta}} =      
     \Phi(\Delta_{\phi^{-1}\circ\widetilde{\theta}})$ and $\Delta_{2,\widetilde{\theta}^c} =      
     \Phi(\Delta_{\phi^{-1}\circ\widetilde{\theta}^c})$;
\item if $\phi^{-1}\circ\theta\in J$ and $\theta\not\in J$, then $\Delta_{2,\widetilde{\theta}} =      
     p^{-1} \Phi(\Lambda_{\phi^{-1}\circ\widetilde{\theta}})$ and $\Delta_{2,\widetilde{\theta}^c} =      
     p\Phi(\Lambda_{\phi^{-1}\circ\widetilde{\theta}^c})$.
\end{itemize}
Setting $\Delta_j = \bigoplus_{\tau\in \Theta_E} \Delta_{j,\tau}$ for $j = 1,2$, it is straightforward to
check that 
$$p\Delta_j \subset \Phi(\Delta_j) \subset \Delta_j.$$ It follows that $\Delta_j$ can be
identified with $\D(A_j[p^\infty])$ for an abelian variety $A_j$ isogenous to $A$.  More precisely,
since $\Delta_2$ contains $p\Delta$ and is stable under the action of $\Phi$ and $V = p\Phi^{-1}$,
the quotient $\Delta/\Delta_2$ of $\Delta/p\Delta = \D(A[p])$ corresponds to a finite flat subgroup
scheme $C$ of $A[p]$.  Letting $A_2 = A/C$, the projection $g:A \to A_2$ induces an inclusion
$\D(A_2[p^{\infty}]) \to \D(A[p^\infty]) = \Delta$ with image $\Delta_2$.  Similarly since 
$\Delta_2 \subset \Delta_1 \subset p^{-1}\Delta_2$, the quotient $\Delta_2/p\Delta_1$
corresponds to a finite flat subgroup scheme $C'$ of $A_2[p]$. Letting $A_1 = A_2/C'$,
the projection $A_2 \to A_1$ induces an inclusion $\D(A_1[p^{\infty}]) \to \D(A_2[p^\infty]) 
= \Delta_2$ with image $p\Delta_1$, and the isogeny $f:A_1 \to A_2$ induced by multiplication
by $p$ identifies $\D(A_1[p^\infty])$ with $\Delta_1$.

Since the $\Delta_j$ are, by construction, stable under the action of 
$\CO_{D_\Sigma} \otimes {\Z_p}$, the abelian varieties $A_j$ inherit an action
of $\CO_{D_\Sigma}$ from the action on $A$.  (Note for example that $C$ is 
stable, so $\CO_{D_\Sigma}$ acts on $A_2$, and that $C'$ is stable
under the resulting action, which therefore carries over to $A_1$.)
We then define the actions $\iota_j$ of $\CO_D$ on $A_j$ by composing
with the isomorphism $\vartheta^{-1}: \CO_D \to \CO_{D^\Sigma}$.
The resulting action of $\CO_{D,p} = M_2(\CO_{E,p})$ on 
$\Delta_j = \D(A_j[p^\infty])$ is obtained by composing the conjugation
$u \mapsto h_p^{-1}u h_p$ with the action inherited from the inclusion
$\Delta_j \subset p^{-1}\Delta$.  In particular, the inclusion is compatible
with the action of $\CO_{E,p}$ induced by $\iota_j$, so this action yields
the same decomposition $\Delta_j = \bigoplus_{\tau\in \Theta_E} \Delta_{j,\tau}$
as in the definition of $\Delta_j$.  It follows that
$\dim_{\Fpbar} \Lie(A_j)_\tau = \dim_{\Fpbar} (\Delta_{j,\phi\circ\tau}/\Phi(\Delta_{j,\tau}))$,
which it is straightforward to check is equal to two for all $\tau \in \Theta_E$.

Recall that the pairing on $\Delta$ induced by the quasi-polarization $\lambda$ 
decomposes over the $\tau \in \Theta_E$ to define isomorphisms
$$\Delta_\tau \stackrel{\sim}{\longrightarrow} \hom_W(\Delta_{\tau^c},W)$$
under which $\Phi$ is ($\phi$-semi-linearly) adjoint to $V = p\Phi^{-1}$.
It is straightforward to check from the definitions of the $\Delta_{j,\tau}$ that
$\Delta_{1,\tau}$ is the dual lattice to $\Delta_{1,\tau^c}$ and that $\Delta_{2,\tau}$
is the dual lattice to $p^{-1}\Delta_{2,\tau^c}$ for all $\tau \in \Theta_E$.  It follows
that the quasi-isogenies $\lambda_j \in \hom(A_j,A_j^\vee)\otimes \Q$ defined by 
$$\lambda_1 = (g^{-1}\circ f)^\vee \circ \lambda \circ (g^{-1}\circ f)\quad\mbox{and}
\quad\lambda_2 = p(g^{-1})^\vee \circ \lambda \circ g^{-1}$$
induce isomorphisms $A_j[p^{\infty}] \cong A_j^\vee[p^{\infty}]$,
and therefore define prime-to-$p$ quasi-polarizations.
The compatibility under $\iota_j$ between the anti-involution $*$ on $D$ and
the Rosati involution induced by $\lambda_i$ then follows from the compatibility
of $*$ on $D_\Sigma$ with the $\lambda$-Rosati involution (under $\iota$)
and that of $\vartheta$ with the anti-involutions.

For $j=1,2$, we define $\eta_j$ as the composite of the isomorphisms
$$\widehat{\CO}_{D}^{(p)} \stackrel{(h^{(p)}\cdot)^{-1}}{\longrightarrow}
 \widehat{\CO}_{D}^{(p)} \stackrel{(\widehat{\xi}_E^{(p)})^{-1}}{\longrightarrow}
\widehat{\CO}_{D_\Sigma}^{(p)} 
 \stackrel{\eta}{\longrightarrow}
 \widehat{T}^{(p)}(A) {\longrightarrow}
 \widehat{T}^{(p)}(A_j),$$
 where the last isomorphism is induced by $f^{-1}\circ g$ for $j=1$
 and by $g$ for $j=2$.  Letting $\epsilon_1 = (\varepsilon^{(p)})^{-1}\epsilon$
 and $\epsilon_2 = p\epsilon_1$, it is straightforward to check that
 $(\eta_j,\epsilon_j)$ defines a level $U'$-structure on $A_j$ for $j=1,2$.
 
 We have now shown that the data $\underline{A}_j = (A_j,\iota_j,\lambda_j,\eta_j,\epsilon_j)$
 define $\Fpbar$-points of $\widetilde{Y}_{U'}(G')$ for $j=1,2$.  Furthermore the isogeny $f:A_1 \to A_2$
 is compatible with the $\CO_D$-action, and the equations $p\lambda_1 = f^\vee \circ \lambda_2\circ f$,
 $\eta_2= f\circ \eta_1$ and $\epsilon_2 = p\epsilon_1$ are immediate from the definitions.
 It is also straightforward to check from the definitions of the $\Delta_{j,\tau}$ that
 $\Delta_{1,\tau}/\Delta_{2,\tau}$ is two-dimensional over $\Fpbar$ for all $\tau \in \Theta_E$,
 so $\D(\ker f) \cong \Delta_1/\Delta_2$ is free of rank two over $\CO_E \otimes \Fpbar$.
 It follows that the triple $(\underline{A}_1,\underline{A}_2,f)$ defines an $\Fpbar$-point of
 $\widetilde{Y}_{U'_0(p)}(G')$.
 
Next we show that the ($\Fpbar$-point defined by the)
triple $(\underline{A}_1,\underline{A}_2,f)$ lies in the closed subscheme
$S_J'$, i.e., that $\Lie(f)_{\widetilde{\theta}} = 0$ (or equivalently $\Lie(f)_{\widetilde{\theta}^c} = 0$)
if $\theta \not\in J$, and $\Lie(f^\vee)_{\widetilde{\theta}} = 0$ (or equivalently 
$\Lie(f^\vee)_{\widetilde{\theta}^c} = 0$) if $\theta \in J$.  Identifying $\D(A_j[p^\infty])$
with $\Delta_j$ for $j=1,2$, we see that the vanishing of $\Lie(f)_\tau$ is equivalent to
the inclusion $V(\Delta_{2,\phi\circ\tau}) \subset p\Delta_{1,\tau}$, i.e., 
$\Delta_{2,\phi\circ\tau} \subset \Phi(\Delta_{1,\tau})$.
In fact the equality $\Delta_{2,\phi\circ\tau} = \Phi(\Delta_{1,\tau})$ if $\theta = \tau|_F\not\in J$ is
immediate from the definitions of the $\Delta_{i,\tau}$, using the equations
$\Phi(\Delta_{\widetilde{\theta}}) = \Delta_{\phi\circ\widetilde{\theta}}$ and 
$\Phi(\Delta_{\widetilde{\theta}^c}) = p\Delta_{\phi\circ\widetilde{\theta}^c}$
if also $\phi\circ\theta \in J$.  Similarly the vanishing of $\Lie(f^\vee)_\tau$ is equivalent to
the inclusion $\Delta_{2,\tau} \subset V(\Delta_{1,\phi\circ\tau})$, i.e.,
$\Phi(\Delta_{2,\tau}) \subset p\Delta_{1,\phi\circ\tau}$, and if $\theta = \tau|_F\in J$,
then in fact the equality $\Phi(\Delta_{2,\tau}) = p\Delta_{1,\phi\circ\tau}$ follows from
the definitions, using that
$\Phi(\Delta_{\widetilde{\theta}}) = p\Delta_{\phi\circ\widetilde{\theta}}$ and 
$\Phi(\Delta_{\widetilde{\theta}^c})= \Delta_{\phi\circ\widetilde{\theta}^c}$
if also $\phi\circ\theta \not\in J$. 

We now check that $(\underline{A},\{L^0_\theta\}_{\theta\in \Sigma}) \mapsto 
(\underline{A}_1,\underline{A}_2,f)$ defines an inverse of $\widetilde{\Xi}_J$
on $\Fpbar$-points.  Suppose first that $(\underline{A},\{L^0_\theta\}_{\theta\in \Sigma})$
and $(\underline{A}_1,\underline{A}_2,f)$ are as above, and let $A_J$ denote the
$J$-splice of $f$.  Letting $\pi$ and $\psi$ be as in the construction of $A_J$,
note that Proposition~\ref{prop:splice} and the formulas
 \begin{itemize} 
 \item $\Delta_{1,\widetilde{\theta}} = \Delta_{\widetilde{\theta}}$ and
 $\Delta_{1,\widetilde{\theta}^c} = \Delta_{\widetilde{\theta}^c}$ if $\theta\not\in J$,
\item $\Delta_{2,\widetilde{\theta}} = \Delta_{\widetilde{\theta}}$ and
 $\Delta_{2,\widetilde{\theta}^c} = p\Delta_{\widetilde{\theta}^c}$ if $\theta\in J$,
\end{itemize} 
imply that the lattice $(\pi^*)^{-1}\psi^*(\D(A_J[p^{\infty}]) \subset 
\D(A_1[p^\infty]) \otimes \Q = \Delta_1 \otimes \Q$
coincides with $\Delta$.  It follows that there is an isomorphism 
$\sigma:A_J \stackrel{\sim}{\longrightarrow} A$ compatible with
the quasi-isogenies $\psi\circ\pi^{-1}$ and $g^{-1}\circ f$ from $A_1$.
Unravelling the definitions of the auxiliary data for $\underline{A}_J$ in terms of
that for $\underline{A}_1$, and in turn in terms of that for $\underline{A}$, one finds
that these are compatible with $\sigma$.  Furthermore if $\theta \not\in J$ and $\phi\circ\theta\in J$,
then $\Lambda_{\widetilde{\theta}} \subset \Delta_{\widetilde{\theta}}$ coincides
with $V(\Delta_{1,\phi^{-1}\circ\widetilde{\theta}})$, so that 
$L_\theta \subset \Delta_{\widetilde{\theta}}/p\Delta_{\widetilde{\theta}}
   = H^1_{\dr}(A/\Fpbar)_{\widetilde{\theta}}$
corresponds to $H^0(A_1,\Omega^1_{A_1/\Fpbar})$ under the isomorphism
 $H^1_{\dr}(A/\Fpbar)_{\widetilde{\theta}} \cong   H^1_{\dr}(A_1/\Fpbar)_{\widetilde{\theta}}$
 induced by the quasi-isogeny $g^{-1}\circ f$.   It follows that 
 $\sigma^*L_\theta \subset H^1_{\dr}(A_J/\Fpbar)$ corresponds to 
 $H^0(A_1,\Omega^1_{A_1/\Fpbar})$ under the isomorphism
  $H^1_{\dr}(A_J/\Fpbar)_{\widetilde{\theta}} \cong   H^1_{\dr}(A_1/\Fpbar)_{\widetilde{\theta}}$
  of Corollary~\ref{cor:splice}(2), and hence that $\sigma^*L_\theta^0$ defines the same line
  in $H^1_{\dr}(A_J/\Fpbar)_{\widetilde{\theta}}^0$ as 
  $\widetilde{\Xi}_J(\underline{A}_1,\underline{A}_2,f)$.
If $\theta \in J$ and $\phi\circ\theta\not\in J$, then one similarly finds that the same
conclusion holds, now using that 
$\Lambda_{\widetilde{\theta}} = V(\Delta_{2,\phi^{-1}\circ\widetilde{\theta}})$
and that the isogeny $g \circ \sigma$ induces the isomorphism 
$H^1_{\dr}(A_J/\Fpbar)_{\widetilde{\theta}} \cong   H^1_{\dr}(A_2/\Fpbar)_{\widetilde{\theta}}$
of Corollary~\ref{cor:splice}(1).  We now shown that
 $\widetilde{\Xi}_J(\underline{A}_1,\underline{A}_2,f)$ defines the same
 point as the data $(\underline{A},\{L^0_\theta\}_{\theta\in \Sigma})$.
 
Suppose now that $(\underline{A}_1,\underline{A}_2,f) \in S_J'(\Fpbar)$, and let
 $(\underline{A},\{L^0_\theta\}_{\theta\in \Sigma}) = \widetilde{\Xi}_J(\underline{A}_1,\underline{A}_2,f)$,
 so in particular $A = A_J$ is the $J$-splice of $f$ and the $L_\theta^0$ are lines in 
 $H^1_\dr(A/\Fpbar)_{\widetilde{\theta}}^0$.  We must show that the triple
 $(\underline{A}_1',\underline{A}_2',f')$ arising from the above construction is isomorphic
 to $(\underline{A}_1,\underline{A}_2,f)$.
 
 With $\pi$ and $\psi$ as in diagram (\ref{eqn:splice}), we prove that
 $\Delta_1= (\psi^*)^{-1}\pi^*\D(A_1)$ as lattices in $\Delta\otimes \Q = \D(A)\otimes\Q$ by showing  that
 $$\Delta_{1,\tau}= (\psi^*)^{-1}\pi^*\D(A_1)_\tau$$
 for all $\tau \in \Theta_E$.
 For $\tau$ such that $\theta = \tau|_F \not\in J$, this is immediate from the definition of $\Delta_{1,\tau}$
 and Proposition~\ref{prop:splice}.  Suppose now that $\theta$ is such that $\phi^{-1}\circ\theta,\theta \in J$.
 The vanishing of $(\Lie f^\vee)_\tau$ means that $\Phi(f^*\D(A_2)_{\phi^{-1}\circ\tau}) = p\D(\A_1)_\tau$,
 so Proposition~\ref{prop:splice} implies that
 $$(\psi^*)^{-1}\pi^*\D(A_1)_\tau = p^{-1}\Phi((\psi^*)^{-1}\pi^*f^*\D(A_2)_{\phi^{-1}\circ\tau}) 
  = p^\delta \Phi(\D(A)_{\phi^{-1}\circ\tau}) = \Delta_{1,\tau},$$
 where $\delta =  0$ (resp.~$-1$) if $\tau = \widetilde{\theta}$ (resp.~$\widetilde{\theta}^c$).
 Finally if $\phi^{-1}\circ\theta \not\in J$ and $\theta \in J$, then the line $L_{\phi^{-1}\circ\theta}^0$
 given by the construction of $\widetilde{\Xi}_J$ is $((\psi^*)^{-1}\pi^*H^0(A_1,\Omega_{A_1/\Fpbar}))^0$,
 so that $L_{\phi^{-1}\circ\theta}=(\psi^*)^{-1}\pi^*H^0(A_1,\Omega_{A_1/\Fpbar})$.  It follows that
 $\Lambda_{\phi^{-1}\circ\widetilde{\theta}} = (\psi^*)^{-1}\pi^*V(\D(A_1)_{\widetilde{\theta}})$, and the compatibility
 of $\pi\circ\psi^{-1}$ with the quasi-polarizations $\lambda_1$ and $\lambda=\lambda_J$ inducing
 the perfect pairings on $\D(A_1)$ and $\Delta$ then implies that its dual lattice
 $\Lambda_{\phi^{-1}\circ\widetilde{\theta}^c}$ is $\Phi^{-1}((\psi^*)^{-1}\pi^*\D(A_1)_{\widetilde{\theta}^c})$.
 We then conclude the desired equation from the definition of $\Delta_{1,\tau}$ 
  for $\tau = \widetilde{\theta}$ and $\tau = \widetilde{\theta}^c$.  It follows that there
  is an isomorphism $\sigma_1: A_1 \to A_1'$ compatible with the quasi-isogenies
  $\psi\circ\pi^{-1}$ and $g^{-1}\circ f'$, where $g:A \to A_2'$ is the quotient map in
  the construction of  $(\underline{A}_1',\underline{A}_2',f')$.  
  Similarly one checks
  that $\Delta_2= (\psi^*)^{-1}\pi^*f^*\D(A_2)$ and deduces that there is an
  isomorphism $\sigma_2: A_2 \to A_2'$ compatible with $g$ and the isogeny
  $A \to A_2$ of diagram (\ref{eqn:splice}).  It follows that $\sigma_2\circ f = 
  f'\circ \sigma_1$, and it is straightforward to check that the $\sigma_i$ respect
  the auxiliary data for $\underline{A}_i$ for $i=1,2$, and hence define an
  isomorphism  between $(\underline{A}_1,\underline{A}_2,f)$ and $(\underline{A}'_1,\underline{A'}_2,f')$.
  
Recall that crystalline deformation theory, in particular the Grothendieck--Messing Theorem, gives an  equivalence of categories between 
 abelian varieties over $T=\Fpbar[\epsilon]/\epsilon^2$ and pairs $(A,\widetilde{L})$ where $A$ is an abelian variety over $\Fpbar$
 and $\widetilde{L}$ is a free $T$-submodule of $H^1_\dr(A/\Fpbar)\otimes_{\Fpbar} T$ such that
 $\widetilde{L} \otimes_T \Fpbar = H^0(A,\Omega^1_{A/\Fpbar})$. 
The equivalence is defined
 by the functor sending $\widetilde{A}$ to $(A,\widetilde{L})$, where
 $A = \widetilde{A}\otimes_T\Fpbar$ and $\widetilde{L} = H^0(\widetilde{A},\Omega^1_{\widetilde{A}/T})$
 is identified with a submodule of  $H^1_\dr(A/\Fpbar)\otimes_{\Fpbar} T$
 via its canonical isomorphisms with $H^1_\crys(A/T) \cong H^1_\dr(\widetilde{A}/T)$; in particular
 the functor $\cdot \otimes_T \Fpbar$ is faithful.  
 
Suppose that $(\underline{A}_1,\underline{A}_2,f)$ is an element
 of $S_J'(\Fpbar)$ with image $(\underline{A},\{L_\theta^0\}_{\theta \in \Sigma})$
 under $\widetilde{\Xi}_J$, and that 
 $(\underline{\widetilde{A}}_1,\underline{\widetilde{A}}_2,\widetilde{f})$
 and $(\underline{\widetilde{A}}'_1,\underline{\widetilde{A}}'_2,\widetilde{f}')$ are
 two elements of $S_J'(T)$ lifting $(\underline{A}_1,\underline{A}_2,f)$ and
having the same image $(\underline{\widetilde{A}},\{\widetilde{L}_\theta^0\}_{\theta\in \Sigma})$
under $\widetilde{\Xi}_J$.   Letting $\widetilde{L}_j = H^0(\widetilde{A}_j,\Omega^1_{\widetilde{A}_j/T})$
and $\widetilde{L}'_j = H^0(\widetilde{A}'_j,\Omega^1_{\widetilde{A}'_j/T})$, it suffices to prove
that $\widetilde{L}_j = \widetilde{L}_j'$ as submodules of 
$H^1_\dr(A_j/\Fpbar) \otimes_{\Fpbar} T$ for $j=1,2$.  Since $\widetilde{L}_j$
is stable under the action of $\CO_D\otimes T \cong \displaystyle\prod_{\tau \in \Theta_E} M_2(T)$
induced by $\iota_j$, we may decompose  $\widetilde{L}_j = \oplus_{\tau \in \Theta_E} \widetilde{L}_{j,\tau}$,
and similarly $\widetilde{L}'_j = \oplus_{\tau \in \Theta_E} \widetilde{L}'_{j,\tau}$.
Furthermore since $\widetilde{L}_j$ is totally isotropic with respect to the perfect $T$-valued pairing on
$H^1_\dr(A_j/\Fpbar)\otimes_{\Fpbar} T$ induced by $\lambda_j$, we have $\widetilde{L}_{j,\tau^c} = 
\widetilde{L}_{j,\tau}^\perp$, and similarly $\widetilde{L}'_{j,\tau^c} = 
(\widetilde{L}'_{j,\tau})^\perp$, so it suffices to prove that $\widetilde{L}_{j,\widetilde{\theta}}
= \widetilde{L}'_{j,\widetilde{\theta}}$ for all $\theta \in \Theta$ and $j=1,2$.

 If $\theta\not\in J$, then note that
 $H^0(\widetilde{A}_1,\Omega^1_{\widetilde{A}_1/T})_{\widetilde{\theta}}$ and 
 $H^0(\widetilde{A}_1',\Omega^1_{\widetilde{A}_1'/T})_{\widetilde{\theta}}$ have the
 same image in $H^1_{\dr}(\widetilde{A}/T)_{\widetilde{\theta}}$ under the isomorphism of
 part (2) of Corollary~\ref{cor:splice}(2), this being  
 $H^0(\widetilde{A},\Omega^1_{\widetilde{A}/T})_{\widetilde{\theta}}$
 if $\theta\not\in \Sigma$ and $\widetilde{L}_\theta^0 \otimes_T T^2$
 if $\theta \in \Sigma$, and it follows that 
 $\widetilde{L}_{1,\widetilde{\theta}} = \widetilde{L}'_{1,\widetilde{\theta}}$ in this case.
 On the other hand since $\Lie(\widetilde{f})_{\widetilde{\theta}} = 0$ if
 $\theta \not\in J$,  we have that $\widetilde{L}_{2,\widetilde{\theta}}$ is the kernel of 
 $H^1_{\dr}(\widetilde{A}_2/T)_{\widetilde{\theta}} \to H^1_{\dr}(\widetilde{A}_1/T)_{\widetilde{\theta}}$,
 and this is identified with $H^0(A_2,\Omega_{A_2/\Fpbar})_{\widetilde{\theta}}  \otimes_{\Fpbar}T$.
 Similarly the fact that $\Lie(\widetilde{f}')_{\widetilde{\theta}} = 0$ implies
 that $\widetilde{L}'_{2,\widetilde{\theta}} = 
 H^0(A_2,\Omega_{A_2/\Fpbar})_{\widetilde{\theta}}  \otimes_{\Fpbar}T$ and we conclude that
  $\widetilde{L}_{2,\widetilde{\theta}} = \widetilde{L}'_{2,\widetilde{\theta}}$.
 
 If $\theta \in J$, then we similarly find that 
 $\widetilde{L}_{2,\widetilde{\theta}} = \widetilde{L}'_{2,\widetilde{\theta}}$ coincide
 since they have the same images under the isomorphism of
 Corollary~\ref{cor:splice}(1).  Finally the vanishing of $\Lie(\widetilde{f}^\vee)_{\widetilde{\theta}}$
 implies that $\widetilde{L}_{2,\widetilde{\theta}}$ is the image of 
 $H^1_{\dr}(\widetilde{A}_2/T)_{\widetilde{\theta}}  \to H^1_{\dr}(\widetilde{A}_1/T)_{\widetilde{\theta}} $,
 and this is $H^0(A_1,\Omega_{A_1/\Fpbar})_{\widetilde{\theta}} \otimes_{\Fpbar}T$, and the same holds for
 $\widetilde{L}'_{2,\widetilde{\theta}}$.
 
 We thus have in all cases that $\widetilde{L}_{j,\widetilde{\theta}}
= \widetilde{L}'_{j,\widetilde{\theta}}$, from which it follows that 
$(\underline{\widetilde{A}}_1,\underline{\widetilde{A}}_2,\widetilde{f})$
 and $(\underline{\widetilde{A}}'_1,\underline{\widetilde{A}}'_2,\widetilde{f}')$
 define the same point of $S_J'(T)$.  This completes the proof of injectivity on
 tangent spaces, and hence of the theorem.

 \epf

\subsubsection{Descent and Hecke equivariance}\label{sss:JL:descent}
Recall from \S\ref{sss:Iw.dha} and \S\ref{sss:usv.dha}
that the reductions mod $p$ of the Shimura varieties associated to $G'$ (of level $U'_0(p)$)
and $G_\Sigma'$ (of level $U_\Sigma'$) are obtained as quotients of the schemes considered above
by the action of $\CO_{F,(p),+}^\times$ on the quasi-polarizations $\lambda$ (and multipliers $\epsilon$)
appearing in the moduli problems.  Furthermore for sufficiently small $U_\Sigma'$, the vector bundles
$\widetilde{\CV}_\tau^0$ on $\widetilde{Y}_{U_\Sigma'}(G_\Sigma')_\F$ descend to the vector bundles
denoted $\CV_\tau^0$ on the quotient $Y_{U_\Sigma'}(G_\Sigma')_\F$, which we denote
as $\overline{Y}_\Sigma'$.  The evident compatibility of the isomorphism $\widetilde{\Xi}_J$ 
with the action on $\CO_{F,(p),+}^\times$ therefore implies that it descends to an isomorphism\footnote{We remark
that since $(U_\Sigma' \cap F)^2$ acts via scalars on $\widetilde{\CV}_\tau^0$, the additional condition
imposed on $U_\Sigma'$ to ensure the vector bundles descend to the quotient by 
$\CO_{F,(p),+}/(U_\Sigma' \cap F)^2$ is not actually needed in order to descend the associated
projective bundles.}
$$\Xi'_J:  \overline{Y}_0(p)'_J   \longrightarrow \prod_{\theta\in \Sigma} 
 \P_{\overline{Y}'_\Sigma}  (\CV_{\widetilde{\theta}}^0),$$
where the product is a fibre product over $\overline{Y}'_{\Sigma}$.

Furthermore the isomorphisms $\Xi'_J$ for varying $U'$ are compatible with the Hecke action
in a sense we make precise as follows.   Recall that for  sufficiently small open
compact subgroups $U_1'$ and $U_2'$
of $G'(\A_\f)$ of level prime to $p$ and $g \in G'(\A_\f^{(p)})$
such that $g^{-1}U_1'g \subset U_2'$, we defined the finite \'etale morphism 
$\widetilde{\rho}_g : \widetilde{Y}_{U_{1,0}'(p)}(G')_\F \to \widetilde{Y}_{U_{2,0}'(p)}(G')_\F$,
restricting to such a morphism of the closed subschemes $S'_{1,J} \to S_{2,J}'$.
Letting $g_\Sigma = \xi_E^{-1}(g)$, we have $g_\Sigma^{-1}U_{1,\Sigma}'g_\Sigma \subset U'_{2,\Sigma}$,
hence also morphisms 
$$\widetilde{\rho}_{g_\Sigma} : \widetilde{Y}_{U_{1,\Sigma}}(G_\Sigma')_\F
\to \widetilde{Y}_{U_{2,\Sigma}}(G_\Sigma')_\F  \quad\mbox{and}\quad
\pi^*_{g_\Sigma}: \widetilde{\rho}_{g_\Sigma}^*\widetilde{\CV}^0_{2,\tau} 
\stackrel{\sim}{\longrightarrow} \widetilde{\CV}^0_{1,\tau},$$
where $\widetilde{\CV}^0_{i,\tau}$ denotes the automorphic bundle $\widetilde{\CV}^0_\tau$ on
$S_i := \widetilde{Y}_{U_{i,\Sigma}}(G_\Sigma')_\F$ for $i=1,2$.  These in turn induce a morphism
$$\prod_{\theta\in \Sigma} \P_{S_1} (\widetilde{\CV}_{1,\widetilde{\theta}}^0)  \to 
\prod_{\theta\in \Sigma} \P_{S_2} (\widetilde{\CV}_{2,\widetilde{\theta}}^0)$$
which we also denote by $\widetilde{\rho}_{g_\Sigma}$, and it is straightforward to check
that the resulting diagram
$$\xymatrix{
S_{1,J}' \ar[r]^-{\widetilde{\Xi}_{1,J}} \ar[d]_{\widetilde{\rho}_g} &
\displaystyle\prod_{\theta \in \Sigma}  \P_{S_1}(\widetilde{\CV}_{1,\widetilde{\theta}}^0)
\ar[d]^{\widetilde{\rho}_{g_\Sigma}} \\
S_{2,J}' \ar[r]^-{\widetilde{\Xi}_{2,J}} &
\displaystyle\prod_{\theta \in \Sigma}  \P_{S_2}(\widetilde{\CV}_{2,\widetilde{\theta}}^0) }
$$
commutes.  Taking quotients by the action of $\CO_{F,(p),+}$ then yields a commutative
diagram
\begin{equation}\label{eqn:heckeJL1}
\xymatrix{
\overline{Y}_{1,0}(p)'_J \ar[r]^-{{\Xi}'_{1,J}} \ar[d]_{{\rho}_g} &
\displaystyle\prod_{\theta \in \Sigma}  \P_{\overline{Y}'_{1,\Sigma}}({\CV}_{1,\widetilde{\theta}}^0)
\ar[d]^{{\rho}_{g_\Sigma}} \\
\overline{Y}_{2,0}(p)'_J \ar[r]^-{{\Xi}'_{2,J}} &
\displaystyle\prod_{\theta \in \Sigma}  \P_{\overline{Y}'_{2,\Sigma}}({\CV}_{2,\widetilde{\theta}}^0) }
\end{equation}

We have now proved the analogue of Theorem~\ref{thm:JLiso} with $G$ replaced by $G'$:
\begin{theorem} \label{thm:uJL}
For each sufficiently small open compact subgroup $U'$ of $G'(\A_\f)$ of level prime to $p$,
there is an isomorphism 
$$\Xi'_J:  \overline{Y}_0(p)'_J   \longrightarrow \prod_{\theta\in \Sigma} 
 \P_{\overline{Y}'_\Sigma}  (\CV_{\widetilde{\theta}}^0).$$
The morphisms $\Xi'_J$ are compatible with the Hecke action in the sense that 
if $g \in G'(\A_\f^{(p)})$ is such that $g^{-1}U_1'g \subset U_2'$,
then the diagram (\ref{eqn:heckeJL1}) commutes.
\end{theorem}

\subsection{Proof of Theorem~\ref{thm:JLiso}}
We now explain how to deduce Theorem~\ref{thm:JLiso} from Theorem~\ref{thm:uJL}.
The task at hand is similar to that of proving that Theorem~5.2 of \cite{TX} follows from their Theorem~5.8 and Corollary~5.9,
but we were not able to supply a proof of this
using the ingredients provided in \cite{TX}\footnote{The proof of the implication in \cite{TX} seems to be non-trivial, but is omitted from the paper.
In any case, our situation is different in that we also need to transfer the vector bundles whose projectivizations are involved.},
so instead we give an argument based on our
Lemmas~\ref{lem:restriction} and~\ref{lem:cart2} (similar to Construction~2.12 of \cite{TX2}), 

\subsubsection{Compatibility on components}\label{sss:trans.comp}
Note that Lemmas~\ref{lem:restriction} and~\ref{lem:cart2} (and the discussion at the end of \S\ref{sss:strata.usv})
describe both varieties in Theorem~\ref{thm:JLiso}
(over $\F = \Fpbar$) as fibre products of those of Theorem~\ref{thm:uJL} with $C$ over $C'$, where $C = C_{\det(U)} = C_{\det(U_\Sigma)}$
indexes the set of components of $\overline{Y} = Y_U(G)_{\Fpbar}$ (and $\overline{Y}_\Sigma = Y_{U_\Sigma}(G_\Sigma)_{\Fpbar}$
if $\Sigma \neq \emptyset$) and $C' = C_{\nu'(U')} = C_{\nu'(U_\Sigma')}$ indexes the set of components of 
$\overline{Y}' = Y'_{U'}(G')_{\Fpbar}$ (and $\overline{Y}'_\Sigma = Y'_{U'_\Sigma}(G'_\Sigma)_{\Fpbar}$
if $\Sigma \neq \emptyset$).  Thus if the isomorphism of Theorem~\ref{thm:uJL} were compatible with the natural projections to $C'$,
we could obtain Theorem~\ref{thm:JLiso} by pulling back along the inclusion $C\to C'$.  Unfortunately we do not know
this compatibility, but it will suffice for our purposes to prove that the diagram
\begin{equation}\label{eqn:components}  \xymatrix{
\overline{Y}_{0}(p)'_J   \ar[d]_{{\Xi}'_{J}} \ar[rr]  && \overline{Y}_0(p)' \ar[rr] &&  \overline{Y}' \ar[rr] && C'  \\
\displaystyle\prod_{\theta \in \Sigma}  \P_{\overline{Y}'_{\Sigma}}({\CV}_{\widetilde{\theta}}^0)
\ar[rrr] &&&  \overline{Y}'_\Sigma \ar[rrr]  &&& C'  \ar[u]  }
\end{equation}
commutes for some automorphism of $C'$
defined by multiplication by an element $t \in T'(\A_\f^{(p)})$ independent of $U'$.
Let $\pi$ denote the composite along the top row
of the diagram, let $\Psi_J$ denote the composite $\overline{Y}_0(p)'_J \to \overline{Y}'_\Sigma$,
and let $\pi_\Sigma$ denote the projection $ \overline{Y}'_\Sigma \to C'$.

Suppose first that $\Sigma \neq \Theta$.  In this case $\pi_\Sigma$ has connected fibres, and since
$\Psi_J$ is closed and has connected fibres, it follows that $\pi_\Sigma\circ \Psi_J$ has connected fibres.
Thus if $y_1,y_2 \in \overline{Y}_0(p)'_J(\Fpbar)$ have the same image under $\pi_\Sigma \circ \Psi_J$,
then $\pi(y_1) = \pi(y_2)$.  Since $\pi_\Sigma\circ \Psi_J$ is surjective, it follows that there is a unique
endomorphism $\delta = \delta_{U'}$ of $C'$ such that $\pi = \delta\circ \pi_\Sigma \circ \Psi_J$.
Furthermore if $U_1', U_2' \subset G'(\A_\f)$ and $g \in G'(\A_\f^{(p)})$
are such that $g^{-1}U_1'g \subset U_2'$, then the compatibility of $\Psi_J$, $\pi$ and $\pi_\Sigma$
with the Hecke action (via $\xi_E$ and $\nu'$) implies that the diagram
$$\xymatrix{C_1 \ar[r]^{\delta_1}\ar[d]_{\nu'(g)} & C_1 \ar[d]^{\nu'(g)} \\
C_2 \ar[r]^{\delta_2} & C_2}$$
commutes, where $C_j = C_{\nu'(U'_j)}$ and $\delta_j = \delta_{U_j'}$ for $j=1,2$.
We may therefore take the limit over sufficiently small $(U')^p\subset G'(\A_\f^{(p)})$ to obtain
an endomorphism $\delta$ of
$$\varprojlim C_{\nu'(U')}
  = \varprojlim  T'(\Q)^{(p)}_+\backslash T'(\A_\f^{(p)}) / \nu'(U')^p
   = \left.\overline{T'(\Q)^{(p)}_+}\right\backslash T'(\A_\f^{(p)}),$$
 compatible with the action of $T'(\A_\f^{(p)}) = \nu'(G'(\A_\f^{(p)}))$
 (where $T'(\Q)^{(p)}$ denotes the subgroup of $p$-integral elements of $T'(\Q)$
 and $\overline{\cdot}$ denotes its closure).
 It follows that $\delta$ is defined by multiplication by $t$ for some
 $t \in T'(\A_\f^{(p)})$ independent of $U'$.  

Suppose on the other hand that $\Sigma = \Theta$, so that $B_\Sigma$ is a totally definite quaternion algebra,
$\overline{Y}'_\Sigma$ is the finite set 
$$G'_\Sigma(\Q)\backslash G'_\Sigma(\A_\f) / U_\Sigma' = G'_\Sigma(\Q)^{(p)}\backslash G'_\Sigma(\A_\f^{(p)}) / (U'_\Sigma)^p$$
viewed as a scheme over $\Fpbar$  (where $G_\Sigma'(\Q)^{(p)}$, and the morphism
$$\pi_\Sigma: \overline{Y}'_\Sigma \longrightarrow C' = T'(\Q)_+^{(p)}\backslash T'(\A_\f^{(p)}) / (\nu'(U'_\Sigma)^p)$$
is induced by $\nu':G_\Sigma' \to T'$.  The same argument as above still yields a unique morphism
 $\delta = \delta_{U'}: \overline{Y}'_\Sigma \longrightarrow C'$ such that $\pi = \delta \circ \Psi_J$.
 Furthermore if $U_1', U_2' \subset G'(\A_\f)$ and $g \in G'(\A_\f^{(p)})$
are such that $g^{-1}U_1'g \subset U_2'$, then the diagram
$$\xymatrix{\overline{Y}'_{1,\Sigma}  \ar[r]^{\delta_1}\ar[d]_{\rho_{g_\Sigma}} & C_1 \ar[d]^{\nu'(g)} \\
\overline{Y}'_{2,\Sigma} \ar[r]^{\delta_2} & C_2}$$
commutes (where now $\overline{Y}'_{j,\Sigma} = G'_\Sigma(\Q)\backslash G'_\Sigma(\A_\f) / U_{j,\Sigma}'$ for $j=1,2$).
Taking the limit over sufficiently small $(U')^p\subset G'(\A_\f^{(p)})$ now yields a map
$$\delta:  \left.\overline{G'_\Sigma(\Q)^{(p)}}\right\backslash G'_\Sigma(\A_\f^{(p)}) 
\longrightarrow
    \left.\overline{T'(\Q)^{(p)}_+}\right\backslash T'(\A_\f^{(p)})$$
such that $\delta(yg) = \nu'(g)\delta(y)$ for all $g \in G'_\Sigma(\A_\f^{(p)})$,
and it follows just the same that $\delta(y) = t\pi_\Sigma(y)$ for some 
$t \in T'(\A_\f^{(p)})$ independent of $U'$.  

\subsubsection{Construction of the isomorphism}\label{sss:trans.iso}
We have now shown that the diagram (\ref{eqn:components}) commutes, where the
right vertical arrow is multiplication by $t^{-1}$
for some $t \in T'(\A_\f^{(p)})$ independent of $U$.  (Note from the proof that such a $t$ is only
unique up to multiplication by an element of $\overline{T'(\Q)^{(p)}_+}$.)
Recall that $T' = (T_F \times T_E)/T_F$ where $T_F$ is embedded in the product via
$x \mapsto (x^2,x^{-1})$, so we may choose $u \in T_E(\A_\f^{p)})$ such that 
$(1,u)t^{-1}$ is in the image of the embedding $T_F(\A_\f^{(p)}) \to T'(\A_\f^{(p)})$
defined by $v \mapsto (v,1)$.  We then have a commutative diagram
\begin{equation}\label{eqn:tweaked} \xymatrix{\overline{Y}_{0}(p)'_J   \ar[r]^-{\rho_{u}}  \ar[d] &
\overline{Y}_{0}(p)'_J   \ar[r]^-{{\Xi}'_{J}}  \ar[d] &
\displaystyle\prod_{\theta \in \Sigma}  \P_{\overline{Y}'_{\Sigma}}({\CV}_{\widetilde{\theta}}^0) \ar[d] \\
C' \ar[r]_{\cdot (1,u)} & C' \ar[r]_{\cdot t^{-1}} & C'}
\end{equation}
where the downward arrows are the natural maps defined by the rows of (\ref{eqn:components}).

Suppose now that $U$ is an open compact subgroup of $G(\A_\f)$ of level prime to $p$,
and let $U_\Sigma = \xi^{-1}(U) \subset G_\Sigma(\A_\f)$.  Suppose that $V_E$ is an
open compact subgroup of $\A_{E,\f}^\times$ of level prime to $p$, sufficiently small
relative to $U$ (and hence $U_\Sigma$) in the sense of \S\ref{sss:qsv.tori}, and let $U'$
denote the image of $U \times V_E$ in $G'(\A_\f)$; note that $U'_\Sigma = \xi^{-1}(U')$
is also the image of $U_\Sigma \times V_E$ in $G'_\Sigma(\A_\f)$.  We assume that
$U$ is sufficiently small that so are $U_\Sigma$, $U'$ and $U'_\Sigma$.

By Lemma~\ref{lem:restriction}, the natural maps give rise to a Cartesian diagram
$$\begin{array}{ccc} \overline{Y}_\Sigma & \stackrel{i}{\longrightarrow} & \overline{Y}'_\Sigma \\
\downarrow && \downarrow \\
C_{\det(U)}  & \longrightarrow & C_{\nu'(U')}, \end{array}$$
where $\overline{Y}_\Sigma = Y_{U_\Sigma}(G_\Sigma)_{\Fpbar}$ and 
where $\overline{Y}'_\Sigma = Y_{U'_\Sigma}(G'_\Sigma)_{\Fpbar}$.
Furthermore the automorphic vector bundle $\CV_\theta$ is defined as
$i^*\CV^0_{\widetilde{\theta}}$, so we obtain from this a Cartesian diagram
$$\begin{array}{ccc} {\displaystyle\prod_{\theta\in \Sigma}} \P_{\overline{Y}_\Sigma}(\CV_\theta)
& \longrightarrow &{\displaystyle\prod_{\theta\in \Sigma}} \P_{\overline{Y}'_\Sigma}(\CV^0_{\widetilde{\theta}}) \\
\downarrow && \downarrow \\
C_{\det(U)}  & \longrightarrow & C_{\nu'(U')} .\end{array}$$
On the other hand, by the discussion at the end of \S\ref{sss:strata.usv},
Lemma~\ref{lem:cart2} gives rise to a Cartesian diagram
$$\begin{array}{ccc} \overline{Y}_0(p)_J
& \longrightarrow &\overline{Y}_0(p)'_J \\
\downarrow && \downarrow \\
C_{\det(U)}  & \longrightarrow & C_{\nu'(U')} .\end{array}$$
Note that the bottom row of (\ref{eqn:tweaked}) restricts to
multiplication by $v$ on $C_{\det(U)}$ for some $v \in T_F(\A_\f^{(p)})$ independent of $U$.
It follows that the fibre product of the isomorphisms 
$$\Xi'_J\circ \rho_{u}:  \overline{Y}_0(p)'_J \to {\displaystyle\prod_{\theta\in \Sigma}}\P_{\overline{Y}'_\Sigma}(\CV^0_{\widetilde{\theta}})
\quad\mbox{and}\quad \cdot{v}: C_{\det(U)} \to C_{\det(U)}$$
over $\cdot v: C_{\nu'(U')} \to C_{\nu'(U')}$ defines an isomorphism
$$\Xi_J :  \overline{Y}_0(p)_J \to {\displaystyle\prod_{\theta\in \Sigma}} \P_{\overline{Y}_\Sigma}(\CV_{{\theta}}).$$

Furthermore if $U_1$ and $U_2$ are sufficiently small open compact subgroups of
$g \in G(\A_\f^{(p)})$ is such that $g^{-1} U_1 g \subset U_2$, then letting $g_\Sigma = \xi^{-1}(g) \in G_\Sigma(\A_\f^{(p)})$
and assuming $V_{E,1} \subset V_{E,2}$,  the morphisms
$$ \overline{Y}_{1,0}(p)_J \stackrel{\rho_g}{\longrightarrow} \overline{Y}_{2,0}(p)_J
\quad\mbox{and}\quad
\P_{\overline{Y}_{1,\Sigma}}(\CV_{{1,\theta}}) 
\stackrel{\rho_{g_\Sigma}}{\longrightarrow} \P_{\overline{Y}_{2,\Sigma}}(\CV_{{2,\theta}})$$
are the restrictions of the ones obtained from the images of $g$ in $G'(\A_\f^{(p)})$
and $g_\Sigma \in G'_\Sigma(\A_\f^{(p)})$.  Since $u$ is central in $G'(\A_\f^{(p)})$,
$\rho_u$ commutes with $\rho_g$, and the commutativity of (\ref{eqn:heckeJL1})
implies that of
\begin{equation} \label{eqn:heckeJL2}
\xymatrix{
\overline{Y}_{1,0}(p)_J \ar[r]^-{{\Xi}_{1,J}} \ar[d]_{{\rho}_g} &
\displaystyle\prod_{\theta \in \Sigma}  \P_{\overline{Y}_{1,\Sigma}}({\CV}_{1,\theta})
\ar[d]^{{\rho}_{g_\Sigma}} \\
\overline{Y}_{2,0}(p)_J \ar[r]^-{{\Xi}_{2,J}} &
\displaystyle\prod_{\theta \in \Sigma}  \P_{\overline{Y}_{2,\Sigma}}({\CV}_{2,\theta}). }
\end{equation}
In particular taking $U_1 = U_2$ and $g=1$, we see that $\Xi_J$ is independent of the
choice of $V_E$.  (Note however that several choices, namely $\vartheta$, $\xi$ and $u$,
were made in the construction of $\Xi_J$, and the choice of $\widetilde{\Theta}$ is even
implicit in the definition of $\CV_\theta$, but these were all independent of $U$.)

This completes the proof of Theorem~\ref{thm:JLiso}:
\begin{theorem} \label{thm:qJL}
For each sufficiently small open compact subgroup $U$ of $G(\A_\f)$ of level prime to $p$,
there is an isomorphism 
$$\Xi_J:  \overline{Y}_0(p)_J   \longrightarrow \prod_{\theta\in \Sigma} 
 \P_{\overline{Y}_\Sigma}  (\CV_{{\theta}}).$$
The morphisms $\Xi_J$ are compatible with the Hecke action in the sense that 
if $g \in G(\A_\f^{(p)})$ is such that $g^{-1}U_1g \subset U_2$,
then the diagram (\ref{eqn:heckeJL2}) commutes.
\end{theorem}

\begin{remark}  It is natural to expect that, like $\overline{Y}_0(p)_J$, the product
 $\prod_{\theta\in \Sigma} \P_{\overline{Y}_\Sigma}  (\CV_{{\theta}})$ admits canonical
 descent data to $\F_J$ (and even to $\F_\Sigma$), where $\F_J$ denotes
 the fixed field of the stabilizer of $J$ in $\gal(\Fpbar/\F_p)$.
 However one easily sees that if $\Sigma \neq \emptyset$, then $\Xi_J$ is not
 compatible with the natural Galois action on the sets of components, and hence
 does not descend to a morphism over $\F_J$.  One can then consider whether there
 is a natural description of the obstruction, for example viewed
 as a class in $H^1(\gal(\Fpbar/\F_J), \Aut_{\Fpbar}(\overline{Y}_0(p)_J))$.
 \end{remark}

\subsection{Comparison of vector bundles}  \label{sec:compbuns}
In this section, we will relate the Raynaud bundles on $\overline{Y}_0(p)_J$ to the tautological line
bundles on $\prod_{\theta \in \Sigma}  \P_{\overline{Y}_{\Sigma}}({\CV}_{\theta})$
and automorphic vector bundles on $\overline{Y}_\Sigma$.

\subsubsection{The unitary setting}\label{sss:buns.usv}
We first prove an analogous result in the context of Theorem~\ref{thm:uJL}.
As in the construction of the morphisms $\widetilde{\Psi}_J$ and $\widetilde{\Xi}_J$ in \S\ref{sss:JL.con}, we let
$A_J$ denote the $J$-splice of the universal isogeny
$f:A_1 \to A_2$ over the closed subscheme $S_J'$ of $\widetilde{Y}_{U'_0(p)}(G')_\F$.
We write $H$ for the Raynaud group scheme $e_0\ker(f)$ on
$S_J'$, and let $\pi:H \to S_J'$, $s_J:A_J \to S_J'$ and $s_j:A_j \to S_J'$ for $j=1,2$ denote
the structure morphisms.    

Recall that for $\tau \in \Theta_E$, the rank two automorphic bundle $\widetilde{\CV}_\tau^0$
on $ \widetilde{Y}_{U'_\Sigma}(G_\Sigma')_\F$ is defined in \S\ref{sss:avb.umf} in terms of the
de Rham cohomology of the universal abelian scheme; its determinant bundle is denoted 
$\widetilde{\delta}_\tau$, and if $\tau|_F \not\in \Sigma$, then the Hodge filtration on $\widetilde{\CV}_\tau^0$
is given by the line bundles $\widetilde{\omega}_{\widetilde{\theta}}^0$ and $\widetilde{\upsilon}_{\widetilde{\theta}}^0$.

The Raynaud bundles $\CL_\tau$ on $S_J'$ are defined by the identification
$$H = \SPEC\left(\left(\sym_{\CO_{S_J'}} \left( \oplus_{\tau\in \Theta_E}  \CL_\tau  \right) \right) /
 \langle (s_\tau - 1)\CL_\tau^p \rangle_{\tau \in \Theta_E} \right).$$
Recall also from Lemma~\ref{lem:vanishing} and its proof that if $\theta \not\in J$, then
$s_{\phi^{-1}\circ{\widetilde{\theta}}} = 0$ and $\Lie(H/S_J')_{\widetilde{\theta}} = \Lie(A_1/S_J')_{\widetilde{\theta}}^0$, so we have
canonical isomorphisms
$$\CL_{\widetilde{\theta}} = (\pi_*\Omega^1_{H/S_J'})_{\widetilde{\theta}} = (s_{1,*} \Omega^1_{A_1/S_J'})_{\widetilde{\theta}}^0.$$
Similarly if $\theta \in J$, then $t_{\phi^{-1}\circ{\widetilde{\theta}}} = 0$ and $\Lie(H^\vee/S_J')_{\widetilde{\theta}} = \Lie(A_2^\vee/S_J')_{\widetilde{\theta}}^0$,
giving canonical isomorphisms
$$\CL_{\widetilde{\theta}} = \Lie(H^\vee/S_J')_{\widetilde{\theta}} = (R^1s_{2,*} \CO_{A_2})_{\widetilde{\theta}}^0.$$

Next recall from the construction of the morphisms
$$\xymatrix{ S_J'  \ar[r]^-{\widetilde{\Xi}_J}  \ar[rd]_-{\widetilde{\Psi}_J} & 
  {\displaystyle\prod_{\theta\in \Sigma}  \P_S (\widetilde{\CV}^0_{\widetilde{\theta}} )}\ar[d]\\
& S =  \widetilde{Y}_{U'_\Sigma}(G_\Sigma')_\F,}$$
that Corollary~\ref{cor:splice} yields an isomorphism
\begin{equation}\label{eqn:DRcomp}
\CH^1_{\dr}(A_j/S_J')^0_{\widetilde{\theta}}
  \longrightarrow \CH^1_{\dr}(A_J/S_J')^0_{\widetilde{\theta}} = 
   \widetilde{\Psi}_J^*\widetilde{\CV}^0_{\widetilde{\theta}},\end{equation}
where $j=2$ or $1$ according to whether or not $\theta \in J$.
 Furthermore, if $\theta \not\in \Sigma$, then the isomorphism
is compatible with the Hodge filtration and so induces isomorphisms
$$(s_{j,*} \Omega^1_{A_j/S_J'})_{\widetilde{\theta}}^0  \cong \widetilde{\Psi}_J^*\widetilde{\omega}^0_{\widetilde{\theta}}
     \quad\mbox{and} \quad (R^1s_{j,*} \Omega^1_{A_j/S_J'})_{\widetilde{\theta}}^0  \cong \widetilde{\Psi}_J^*\widetilde{\upsilon}^0_{\widetilde{\theta}}
              \cong \widetilde{\Psi}_J^*\widetilde{\delta}_{\widetilde{\theta}} \otimes_{\CO_{S_J'}} \widetilde{\Psi}_J^*(\widetilde{\omega}^0_{\widetilde{\theta}})^{-1}  .$$
On the other hand if $\theta \in \Sigma$, then the Hodge filtration on $\CH^1_{\dr}(A_j/S_J')^0_{\widetilde{\theta}}$
defines the morphism $S_J' \to \P_S(\widetilde{\CV}^0_{\widetilde{\theta}} )$
and hence induces isomorphisms 
$$(s_{j,*} \Omega^1_{A_j/S_J'})_{\widetilde{\theta}}^0  \cong \widetilde{\Psi}_J^*\widetilde{\delta}_{\widetilde{\theta}} \otimes_{\CO_{S_J'}} \widetilde{\Xi}_J^*\CO(-1)_\theta
     \quad\mbox{and} \quad (R^1s_{j,*} \Omega^1_{A_j/S_J'})_{\widetilde{\theta}}^0  \cong \widetilde{\Xi}_J^*\CO(1)_\theta.$$
We thus obtain an isomorphism
\begin{equation}\label{eqn:bundling}  \CL_{\widetilde{\theta}}  \stackrel{\sim}{\longrightarrow}  \widetilde{\Xi}^*_J\CM_\theta,\quad \mbox{$\CM_\theta :=$}
 \left\{ \begin{array}{cl}  \widetilde{\omega}_{\widetilde{\theta}},&  \mbox{if $\theta \not\in J$, $\theta \not\in \Sigma$;} \\
\widetilde{\delta}^0_{\widetilde{\theta}} (-1)_\theta ,&  \mbox{if $\theta \not\in J$, $\theta \in \Sigma$;} \\
\widetilde{\delta}^0_{\widetilde{\theta}}(\widetilde{\omega}_{\widetilde{\theta}})^{-1} ,& 
         \mbox{if $\theta \in J$, $\theta \not\in \Sigma$;} \\
\CO(1)_\theta,&  \mbox{if $\theta \in J$, $\theta \in \Sigma$,}\end{array}
  \right.\end{equation}
where we write $\widetilde{\delta}_{\widetilde{\theta}}$ and $\widetilde{\omega}^0_{\widetilde{\theta}}$ for their pull-back to
$\prod_{\theta \in \Sigma} \P_S (\widetilde{\CV}^0_{\widetilde{\theta}} )$, $\CO(1)_\theta$ for the pull-back of the twisting sheaf
on the $\theta$-component of the product, and $\CF(n)_\theta$ for the twist of $\CF$ by $\CO(n)_\theta = \CO(1)_\theta^n$.

The isomorphism of (\ref{eqn:bundling}) is Hecke-equivariant in the following sense.  Suppose as usual
that $U'_1, U_2' \subset G'(\A_\f)$ are as above and $g \in G'(\A_\f^{(p)})$ is such that $g^{-1}U'_1 g \subset U'_2$,
It is then straightforward to check that the resulting diagram
\begin{equation}\label{eqn:JLHbundle0} \xymatrix{
\widetilde{\rho}_g^*\CL_{2,\widetilde{\theta}}  \ar[r]^-{\widetilde{\rho}_g^*(\sigma_2)} \ar[d]_{\pi_g^*} &
*+[r]{ \widetilde{\rho}_g^* \widetilde{\Xi}^*_{2,J} \CM_{2,\theta} = }& 
  \widetilde{\Xi}^*_{1,J} \widetilde{\rho}_{g_\Sigma}^* \CM_{2,\theta} \ar[d]^{  \widetilde{\Xi}^*_{1,J}( \pi_{g_\Sigma}^* )} \\
\CL_{1,\widetilde{\theta}}  \ar[rr]_-{\sigma_1}  && 
  \widetilde{\Xi}^*_{1,J} \CM_{1,\theta},}\end{equation}
  where $\sigma_i: \CL_{i,\widetilde{\theta}} \to 
  \widetilde{\Xi}^*_{i,J} \CM_{\theta,i}$ is the isomorphism of (\ref{eqn:bundling}) with $U' = U_i'$ for $i=1,2$,
$\pi_g^*: \widetilde{\rho}_g^*\CL_{2,\widetilde{\theta}}   \stackrel{\sim}{\to}  \CL_{1,\widetilde{\theta}}$ is 
defined at the end of \S\ref{sss:Iw.dha}, and $\pi^*_{g_\Sigma}:
\widetilde{\rho}_{g_\Sigma}^*\CM_{2,\widetilde{\theta}}   \stackrel{\sim}{\to} \CM_{1,\widetilde{\theta}}$
is induced by the isomorphisms 
$\widetilde{\rho}_{g_\Sigma}^*\widetilde{\CV}_{2,\widetilde{\theta}}^0 \to 
  \widetilde{\CV}_{1,\widetilde{\theta}}^0$, with determinant 
  $\widetilde{\rho}_{g_\Sigma}^*\widetilde{\delta}_{2,\widetilde{\theta}} \to 
  \widetilde{\delta}_{1,\widetilde{\theta}}$ and restriction
  $\widetilde{\rho}_{g_\Sigma}^*\widetilde{\omega}_{2,\widetilde{\theta}}^0 \to 
  \widetilde{\omega}_{1,\widetilde{\theta}}^0$ if $\theta\not\in \Sigma$, all
  denoted $\pi^*_{g_\Sigma}$ in \S\ref{sss:avb.umf}.

The isomorphism of (\ref{eqn:bundling}) is also compatible with the descent data associated to 
the action of $\CO_{F,(p),+}$ (for sufficiently small $U'$), and hence descends to an isomorphism
$\sigma: \CL_{\widetilde{\theta}} \to (\Xi_J')^* \CM_\theta$ on $\overline{Y}_0(p)'_J$ in the context of 
Theorem~\ref{thm:uJL}, where the line
bundle $\CM_\theta$ on $\prod_{\theta\in \Sigma} \P_{\overline{Y}_\Sigma'} (\CV_{\widetilde{\theta}}^0)$
is defined by the same formula, but using the automorphic bundles $\omega_{\widetilde{\theta}}^0$
and $\delta_{\widetilde{\theta}}$ on $\overline{Y}_\Sigma'$.  Furthermore the resulting isomorphism
is Hecke equivariant in the sense that the diagram analogous to (\ref{eqn:JLHbundle0}) commutes.

\subsubsection{The Hilbert setting}  \label{sss:buns.hmv}
To obtain the desired relation on $\overline{Y}_0(p)_J$, recall that the isomorphism $\Xi_J$ of
Theorem~\ref{thm:qJL} is defined as the restriction of the composite $\Xi_J'\circ \rho_u$
for suitably chosen $u \in (\A_{E,\f}^{(p)})^\times$ (independent of $U$) and $U'$.
We thus obtain an isomorphism
\begin{equation}\label{eqn:bundling2}  \CL_{{\theta}}  \stackrel{\sim}{\longrightarrow}  {\Xi}^*_J\CM_\theta,\quad \mbox{$\CM_\theta :=$}
 \left\{ \begin{array}{cl}  \omega_{{\theta}},&  \mbox{if $\theta \not\in J$, $\theta \not\in \Sigma$;} \\
{\delta}_{{\theta}} (-1)_\theta ,&  \mbox{if $\theta \not\in J$, $\theta \in \Sigma$;} \\
{\delta}_{{\theta}}{\omega}_{{\theta}}^{-1} ,& 
         \mbox{if $\theta \in J$, $\theta \not\in \Sigma$;} \\
\CO(1)_\theta,&  \mbox{if $\theta \in J$, $\theta \in \Sigma$}\end{array}\right.\end{equation}
as the restriction to $\overline{Y}_0(p)_J$ of the composite of the isomorphisms
$$\CL_\theta \stackrel{(\pi_u^*)^{-1}}{\longrightarrow} \rho_u^* \CL_\theta 
\stackrel{\rho_u^*\sigma}{\longrightarrow } \rho_u^* (\Xi_J')^* \CM_\theta.$$
Furthermore if $U_1$ and $U_2$ are sufficiently small open compact subgroups of $G(\A_\f)$ and $g \in G(\A_\f^{(p)})$
is such that $g^{-1}U_1 g \subset U_2$, then it is straightforward to check that the resulting diagram
\begin{equation}\label{eqn:JLHbundles} \xymatrix{
\rho_g^*\CL_{2,{\theta}}  \ar[r]^-{{\rho}_g^*(\sigma_2)} \ar[d]_{\pi_g^*} &
*+[r]{ {\rho}_g^* {\Xi}^*_{2,J} \CM_{2,\theta} = }& 
  {\Xi}^*_{1,J} {\rho}_{g_\Sigma}^* \CM_{2,\theta} \ar[d]^{  {\Xi}^*_{1,J}( \pi_{g_\Sigma}^*) } \\
\CL_{1,{\theta}}  \ar[rr]_-{\sigma_1}  && 
  {\Xi}^*_{1,J} \CM_{1,\theta}}\end{equation}
commutes.

\subsubsection{Relation of determinant bundles}\label{sss:buns.delta}
We also relate the bundles denoted $\delta_\theta$ on the varieties $\ol{Y}$ and $\ol{Y}_\Sigma$.
Again we first make the comparison in the context of the unitary Shimura varieties $\ol{Y}'$ and
$\ol{Y}'_\Sigma$.   Note that (\ref{eqn:DRcomp}) yields an isomorphism
$$\wedge^2_{\CO_{S_J'}}\CH^1_\dr(A_j/S_J')^0_{\wt{\theta}}
  \longrightarrow \wt{\Psi}_J^*\wt{\delta}_{\wt{\theta}}$$
where $j=2$ or $1$ according to whether or not $\theta\in J$, and that the first line bundle
is simply the pull-back of $\wt{\delta}_{\wt{\theta}}$ on $\wt{Y}_{U'}(G')_{\F}$ under the
forgetful morphism $\wt{\pi}_j$ sending $(\underline{A}_1,\underline{A}_2,f)$ to $\underline{A}_j$.
Furthermore the isomorphism descends to an isomorphism on $\ol{Y}_0(p)_J'$, restricting to
an isomorphism
$$\pi_j^*\delta_\theta \stackrel{\sim}{\longrightarrow} \Psi_J^*\delta_\theta$$
on $\ol{Y}_0(p)_J$, where the first $\delta_\theta$ is on $\ol{Y}$, the second is on $\ol{Y}_\Sigma$,
and $\pi_j$ is the morphism $\ol{Y}_0(p)_J$ obtained from $\wt{\pi}_j$.

On the other hand, letting $g$ denote the transpose of $f$, so $f\circ g = p$, we have the exact sequences
$$\begin{array}{cccccccccc}
&0 &\to& \ker(f^*)& \to & \CH^1_\dr(A_2/S_J')& \to &\ker(g^*) &\to& 0\\
\mbox{and}& 0 &\to& \ker(g^*)& \to &\CH^1_\dr(A_1/S_J')& \to& \ker(f^*) &\to& 0,\end{array}$$
of right $M_2(\CO_E)\otimes \CO_{S_J'}$-modules, to which we may apply idempotents to obtain
the exact sequences
$$\begin{array}{cccccccccc}
&0 &\to& \ker(f^*)_{\wt{\theta}}^0 & \to & \CH^1_\dr(A_2/S_J')_{\wt{\theta}}^0& \to &\ker(g^*)_{\wt{\theta}}^0 &\to& 0\\
\mbox{and}
& 0 &\to& \ker(g^*)_{\wt{\theta}}^0& \to &\CH^1_\dr(A_1/S_J')_{\wt{\theta}}^0& \to& \ker(f^*)_{\wt{\theta}}^0 &\to& 0.
\end{array}$$
Since $\ker(f^*)_{\wt{\theta}}^0$ and $\ker(g^*)_{\wt{\theta}}^0$ are invertible $\CO_{S_J'}$-modules, this
yields an isomorphism
$$\wt{\pi}_2^*\wt{\delta}_{\wt{\theta}}\quad \cong \quad
\ker(f^*)_{\wt{\theta}}^0 \otimes_{\CO_{S_J'}} \ker(g^*)_{\wt{\theta}}^0
\quad \cong \quad \wt{\pi}_1^*\wt{\delta}_{\wt{\theta}},$$
which in turn descends to $\ol{Y}_0(p)_J'$ and restricts to an isomorphism
$$\pi_2^*\delta_\theta \cong \pi_1^*\delta_\theta$$
on $\ol{Y}_0(p)_J$.   Again abusing notation and writing $\delta_\theta$ for its pull-back
to $\prod_{\theta\in \Sigma} \P_{\ol{Y}_\Sigma}(\CV_\theta)$, we have now constructed isomorphisms
$$\pi_j^*\delta_\theta \stackrel{\sim}{\longrightarrow} \Xi_J^*\delta_\theta$$
for all $\theta \in \Theta$,  $J \subset \Theta$ and $j=1,2$.  
Furthermore it is straightforward to check that the isomorphisms are Hecke-equivariant in the 
same sense as (\ref{eqn:JLHbundles}).

Finally, we note that just as in \cite[\S3.4]{DS}, the line bundle $\bigotimes_{\theta\in \Theta} \delta_\theta$
on $\wt{Y}_U(G)_\CO$  is equipped with a canonical trivialization which descends to $Y_U(G)_\CO$ and
transforms by $||\det g||^{-1}$ under $\pi_g^*$ (for $g \in G(\A_\f^{(p)})$ and varying $U$).  It follows that the
same holds for
$\bigotimes_{\theta\in \Theta} \delta_\theta$ on $\ol{Y}$, and this in turn pulls back to such
trivializations of $\pi_j^*\left(\bigotimes_{\theta\in\Theta}\delta_\theta\right)$ for $j=1$ and $2$,
which one can check  are compatible with the isomorphism
$\pi_2^*\left(\bigotimes_{\theta\in\Theta}\delta_\theta\right) \cong 
\pi_1^*\left(\bigotimes_{\theta\in\Theta}\delta_\theta\right)$
constructed above.  Applying $\Xi_{J,*}$, we obtain a trivialization of 
$\Psi_J^*\left(\bigotimes_{\theta\in\Theta}\delta_\theta\right)$ which
transforms by $||\det g_\Sigma||^{-1}$ under $\pi_{g_\Sigma}^*$ (for $g \in G_\Sigma(\A_\f^{(p)})$
varying $U_\Sigma$).  Since the natural map $\CO_{\ol{Y}_\Sigma} \to \Psi_{J,*}\CO_X$ is an
isomorphism for $X = \prod_{\theta\in \Sigma} \P_{\ol{Y}_\Sigma}(\CV_\theta)$, this in fact
yields a trivialization of $\bigotimes_{\theta\in\Theta}\delta_\theta$ on $\ol{Y}_\Sigma$
which transforms by $||\det g_\Sigma||^{-1}$ under $\pi_{g_\Sigma}^*$.

\section{The Serre filtration}  \label{sec:fil}
\subsection{Dualizing sheaves} \label{sec:dualizing}
Our aim now is to analyze the dualizing sheaf on the special fibre of $Y_{U_1(p)}(G)$.

\subsubsection{Generalities}\label{sss:dual.gen}
We begin with some general properties of dualizing sheaves, all of which are
recorded in \cite[Chapter~48]{stacks}.

First recall that if $f:Y \to S$ is a Cohen--Macaulay morphism of constant relative dimension $n$,
then $f$ admits a relative dualizing sheaf, which we denote $\CK_{Y/S}$.  We will write this as $\CK_{Y/R}$
if $S = \Spec R$, and simply as $\CK_Y$ if $R$ is a field evident from the context.  Then $\CK_{Y/S}$
is a coherent sheaf on $Y$, flat over $S$.  If moreover $f$ is Gorenstein (i.e., all fibres are Gorenstein),
then $\CK_{Y/S}$ is invertible.  If moreover $f$ is smooth, then $\CK_{Y/S}$ is identified with the canonical
sheaf $\Omega_{Y/S}^n = \wedge^n_{\CO_Y} \Omega^1_{Y/S}$.  We recall also that formation of the
dualizing sheaf commutes with base-change; i.e., if $f:Y \to S$ is Cohen-Macaulay of (constant relative)
dimension $n$, and $S' \to S$ is any morphism, then letting $Y'$ denote $Y \times_S S'$ and $\pi:Y' \to Y$
the projection, we have a canonical isomorphism $\pi^*\CK_{Y/S} \cong \CK_{Y'/S'}$.

If $h:X \to Y$ and $f:Y \to S$ are both Cohen-Macaulay of constant relative dimension
and either $f$ or $h$ is Gorenstein, then $\CK_{X/S}$ is canonically isomorphic to 
$\CK_{X/Y} \otimes_{\CO_X} h^*\CK_{Y/S}$. In particular if $h$ is \'etale, then
$\CK_{X/S} = h^*\CK_{Y/S}$.

Suppose now that $h:X \to Y$ is finite and that $f:Y \to S$ and $f\circ h: X \to S$ are Cohen--Macaulay
of the same relative dimension.  Then $h_*\CK_{X/S}$ is canonically isomorphic to 
$\Shom_{\CO_Y}(h_*\CO_X,\CK_{Y/S})$.
This can be seen as an immediate consequence of the Duality Theorem, or indeed as a special
case of the construction of dualizing sheaves.
We will apply this in \S\ref{sss:dice.dual} in the case where $h$ is a closed immersion and both
$f$ and $f\circ h$ are local complete intersections.  However we will first apply it in \S\ref{sss:dual.wt2}
in a setting where $h$ is finite flat and $f$ is Cohen--Macaulay; note that in this case $f\circ h$
is automatically Cohen--Macaulay of the same relative dimension as $f$.   

\subsubsection{The Kodaira--Spencer isomorphism}\label{sss:dual.KS}
Recall from \S\ref{sec:KS} that the Kodaira--Spencer isomorphism on Shimura varieties associated to
$G_\Sigma$ relates their dualizing sheaves to certain automorphic bundles.  We defined this isomorphism
(and indeed the integral models and automorphic bundles) via the Shimura varieties associated to the unitary
group $G_\Sigma'$, but for $G = \Res_{F/\Q}\GL_2$ (i.e., $\Sigma = \emptyset$), we can establish the
relation directly from the moduli problem considered in \S\ref{sss:hmv.mp}.  We briefly recall this now, but working
integrally instead of modulo $p$ as we did in \S\ref{sec:KS}.

Suppose that $L \subset \Qbar$ contains the Galois closure of $F$, and let $\CO$ be the localization of $\CO_L$
at the prime over $p$ determined by our fixed embedding $\Qbar \to \Qpbar$, and let $A$ be the universal
abelian variety over $S = \widetilde{Y}_U(G)_\CO$ where $U$ is a sufficiently small open compact
subgroup of $\GL_2(\A_{F,\f})$ of level prime to $p$.  Proceeding exactly as in \S\ref{sss:avb.hmf}
we can define line bundles $\widetilde{\omega}_\theta$, $\widetilde{\upsilon}_\theta$ and 
$\widetilde{\delta}_\theta$ on $S$ by the Hodge filtration and determinant on the $\theta$-component
of $\CH^1_\dr(A/S)$ for each $\theta \in \Theta$.  For $k,\ell \in \Z^\Theta$, we define
$$\widetilde{\CA}_{k,l} = \bigotimes_{\theta \in \Theta}  \wt{\delta}_\theta^{\ell_\theta} \wt{\omega}_\theta^{k_\theta}.$$
If $(k,\ell)$ is paritious in the sense of \cite[Def.~3.2.1]{DS}, i.e., $k_\theta + 2\ell_\theta$ is independent of $\theta$,
then $\widetilde{\CA}_{k,\ell}$ descends to a line bundle $\CA_{k,\ell}$ on $Y = Y_U(G)_\CO$;
in particular the line bundles $\wt{\upsilon}_\theta^{-1}\wt{\omega}_\theta  = \wt{\delta}_\theta^{-1}\wt{\omega}^2_\theta$
descend to $Y = Y_U(G)_\CO$.  The same construction as in \S\ref{sec:KS} then produces isomorphisms
$$ \bigoplus_{\theta \in \Theta}  \wt{\delta}_\theta^{-1}\wt{\omega}^2_\theta
\stackrel{\sim}{\longrightarrow}\Omega^1_{S/\CO}
\quad\mbox{and}\quad
\bigotimes_{\theta \in \Theta}  \wt{\delta}_\theta^{-1} \wt{\omega}^2_\theta 
\stackrel{\sim}{\longrightarrow}\CK_{S/\CO}$$
which descend to isomorphisms on $Y$.\footnote{In fact identifying $\oplus \widetilde{\upsilon}_\theta^{-1}\widetilde{\omega}_\theta$
with $\Shom_{\CO_S\otimes \CO_F}(R^1s_*\CO_A, s_*\Omega^1_{A/S})$, the vector bundles and isomorphisms 
descend to $Y_U(G)$ over $\Z_{(p)}$.}
We thus obtain the Kodaira--Spencer isomorphism $\CA_{2,-1} \cong \CK_{Y/\CO}$.
Using the commutativity of the diagram
$$\begin{array}{ccccccccc}
0 & \to & (s')^*\Omega^1_{S/\CO} & \to & \Omega^1_{A'/\CO}  & \to & \Omega^1_{A'/S} & \to & 0 \\
&&\downarrow&&\downarrow&&\downarrow&& \\
0 & \to & \Delta_*s^*\Omega^1_{S/\CO} & \to & \Delta_*\Omega^1_{A/\CO}  & \to & \Delta_*\Omega^1_{A/S} & \to & 0 \end{array}
$$
where $A' = A\otimes_{\CO_F} \CO_E^2$, $s:A\to S$, $s':A' \to S$ are the structure morphisms, $\Delta:A \to A'$ is the diagonal embedding
and the vertical arrows are adjoint to the canonical morphisms, one sees that the base-change to $\F = \Fpbar$ of the
Kodaira--Spencer isomorphism just defined coincides with the one already defined on $\overline{Y} = Y_U(G)_\F$ in \S\ref{sec:KS}.
On the other hand, working over $L$ instead of $\CO$, one can dispense with the hypothesis that the open compact subgroup has level prime to $p$.
For any sufficiently small open compact subgroup $U$, the analogous constructions yield automorphic bundles
$\CA_{k,\ell}$ on $Y_{U}(G)_L$ for paritious $(k,\ell)$ and the Kodaira--Spencer isomorphism $\CA_{2,-1} \cong \CK_{Y_{U}(G)_L}$.

\subsubsection{Hecke action}\label{sss:dual.Hecke}
Suppose now that $U_1$ and $U_2$ are sufficiently small open compact subgroups of $\GL_2(\A_{F,\f})$ of level prime to $p$
and that $g \in \GL_2(\A_{F,\f}^{(p)})$ is such that $g^{-1}U_1g \subset U_2$.  Letting $Y_j = Y_{U_j}(G)_\CO$ for $j=1,2$,
we have a morphism $\rho_g:Y_1 \to Y_2$ and a canonical isomorphism $\rho_g^*\CK_{Y_1/\CO} \to \CK_{Y_2/\CO}$,
satisfying the usual compatibility with $\rho_h^*\CK_{Y_2/\CO} \to \CK_{Y_3/\CO}$ if $h^{-1}U_2 h \subset U_3$.
We thus obtain a natural action of $\GL_2(\A_{F,\f}^{(p)})$ on
$$\varinjlim H^0(Y_{U}(G)_\CO , \CK_{Y_U(G)_\CO/\CO}),$$
where the direct limit is over open compact subgroups $U \subset \GL_2(\A_{F,\f})$ of level prime to $p$.

For any paritious $(k,\ell)$, we similarly define an action of $\GL_2(\A_{F,\f}^{(p)})$ on the space of Hilbert modular forms
(over $\CO$) of weight $(k,\ell)$ and level prime to $p$, i.e., the direct limit over such $U$ of the $\CO$-modules
$$M_{k,\ell}(U,\CO) := H^0(Y_{U}(G)_\CO , \CA_{k,\ell}),$$
incorporating a twist by the character $||\det||$ as in \cite[\S4.2]{DS}
for consistency with standard conventions.  As recalled in \S\ref{sss:buns.delta}, the line bundle 
$\widetilde{\CA}_{0,1} = \otimes_{\theta} \wt{\delta}_\theta$ on $\wt{Y}_U(G)_\CO$ 
is equipped with a canonical trivialization which descends to $Y_U(G)_\CO$ and transforms by
$||\det g||^{-1}$ under $\pi_g^*$.   We thus obtain an isomorphism $\CA_{2,0} \cong \CA_{2,-1}$
whose composite with the Kodaira--Spencer isomorphism $\CA_{2,-1} \cong \CK_{Y/\CO}$
is compatible with the Hecke action in the sense that the resulting diagram
$$ \begin{array}{ccc}  M_{2,0}(U_2,\CO)   & \stackrel{\sim}{\longrightarrow} &    H^0(Y_2,\CK_{Y_2/\CO}) \\
\downarrow && \downarrow \\
      M_{2,0}(U_1,\CO)  & \stackrel{\sim}{\longrightarrow} &      H^0(Y_1,\CK_{Y_1/\CO}) \end{array} $$
commutes, and hence the isomorphism
$$ \varinjlim_U M_{2,0}(U,\CO)  \cong  \varinjlim_U  H^0(Y_{U}(G)_\CO , \CK_{Y_U(G)_\CO/\CO})$$
is $\GL_2(\A_{F,\f}^{(p)})$-equivariant.

Similarly for any $k,\ell\in \Z^\Theta$ (not necessarily paritious) and $\F$ containing the residue field of $\CO$, 
we have an action of $\GL_2(\A_{F,\f}^{(p)})$ on the space
$$\varinjlim_U M_{k,\ell}(U,\F) = \varinjlim_U H^0(Y_{U}(G)_\F , \CA_{k,\ell})$$
 of Hilbert modular forms over $\F$ of weight $(k,\ell)$ and level prime to $p$.
Note that for any paritious $(k,\ell)$, the natural map
$$M_{k,\ell}(U,\CO) \otimes_{\CO} \F \to M_{k,\ell}(U,\F)$$
is Hecke-equivariant and injective, but not obviously surjective, the cokernel being
isomorphic to $\mathrm{Tor}_1^\CO(H^1(Y_U(G)_\CO, \CA_{k,\ell}),\F)$.
The Kodaira--Spencer isomorphism in this context gives an isomorphism
 $$ \varinjlim_U M_{2,0}(U,\F)  \cong  \varinjlim_U  H^0(Y_{U}(G)_\F , \CK_{Y_U(G)_\F}).$$
compatible with the one over $\CO$ and the action of $\GL_2(\A_{F,\f}^{(p)})$.

Similarly for paritious $(k,\ell)$ and fields $K$ containing $L$, we have an action of 
 $\GL_2(\A_{F,\f})$ on the space
$$\varinjlim_{U} M_{k,\ell}(U,K) = \varinjlim_{U} H^0(Y_{U}(G)_K , \CA_{k,\ell})$$
 of Hilbert modular forms over $K$ of weight $(k,\ell)$, where the limit is now over all sufficiently small
open compact $U \subset \GL_2(\A_{F,\f})$, and $M_{k,\ell}(U,\C)$ is identified with the usual space
of holomorphic Hilbert modular forms of weight $(k,\ell)$ and level $U$.  
The Kodaira--Spencer isomorphism in this context gives a $\GL_2(\A_{F,\f})$-equivariant isomorphism
 $$ \varinjlim_U M_{2,0}(U,K)  \cong  \varinjlim_U  H^0(Y_{U}(G)_K , \CK_{Y_U(G)_K}).$$
 If $U$ has level prime to $p$, then the natural map
$$M_{k,\ell}(U,\CO) \otimes_{\CO} K \to M_{k,\ell}(U,K)$$
is a Hecke-equivariant isomorphism, compatible with the Kodaira--Spencer isomorphism
for $(k,\ell) = (2,0)$.

\subsubsection{Forms of parallel weight $2$, level $U_1(p)$}\label{sss:dual.wt2}
Our main object of study for the rest of the paper will be the dualizing sheaf on the special
fibre of the Pappas model $Y_{U_1(p)}(G)$, defined in \S\ref{sss:U1p}.  Recall that
$U_1(p) = U^p U_{1,p}$ where $U^p$ is a sufficiently small open compact subgroup of
$\GL_2(\A_{F,\f}^{(p)})$ and $U_{1,p} \subset U_{0,p} \subset \GL_2(\CO_{F,p})$ with
$U_0(p)/U_1(p) \cong (\CO_F/p)^\times$.    We will write simply $Y_i(p)$ for 
$Y_{U_i(p)}(G)_\CO$ for $i=0,1$.  Since $Y_0(p)$ is a flat local complete
intersection over $\CO$ and $Y_1(p)$ is finite flat over $Y_0(p)$, 
it follows that $Y_1(p)$ is Cohen--Macaulay over $\CO$.  Therefore $Y_0(p)$
and $Y_1(p)$ admit dualizing sheaves $\CK_{Y_0(p)/\CO}$ and $\CK_{Y_1(p)/\CO}$,
and $\CK_{Y_0(p)/\CO}$ is invertible.

If $U$ and $U'$ are sufficiently small open compact subgroups of $\GL_2(\A_{F,\f})$ of level prime to $p$
and that $g \in \GL_2(\A_{F,\f}^{(p)})$ is such that $g^{-1}Ug \subset U'$, then we have the finite \'etale
morphisms $\rho_g:Y_i(p) \to Y'_i(p) =Y_{U'_i(p)}(G)_\CO$ and canonical isomorphisms
$\rho_g^*\CK_{Y_i'(p)/\CO} \to \CK_{Y_i(p)/\CO}$ for $i=0,1$ satisfying the usual compatibilities.
We thus obtain a natural action of $\GL_2(\A_{F,\f}^{(p)})$ on
$$\varinjlim H^0(Y_i(p) , \CK_{Y_i(p)/\CO}),$$
where the direct limit is over open compact subgroups $U \subset \GL_2(\A_{F,\f})$ of level prime to $p$.
Since
$$H^0(Y_i(p) , \CK_{Y_i(p)/\CO}) \otimes_{\CO} L
  = H^0(Y_{U_i(p)}(G)_L , \CK_{Y_{U_i(p)}(G)_L}) \cong M_{2,0}(U_i(p),L),$$
 we may identify $H^0(Y_i(p) , \CK_{Y_i(p)/\CO})$ with an $\CO$-lattice
 in $M_{2,0}(U_i(p),L)$ which we denote $M_{2,0}(U_i(p),\CO)$; the identification is Hecke-equivariant
 in the sense that it is compatible with the $\GL_2(\A_{F,\f}^{(p)})$-action on direct limits.
 Similarly we may consider the dualizing sheaf $\CK_{\overline{Y}_i(p)}$ on 
 $\overline{Y}_i(p) := Y_{U_i(p)}(G)_\F$ for $\F$ containing the residue field of $\CO$.  Letting 
 $$M_{2,0}(U_i(p),\F) = H^0(\overline{Y}_i(p), \CK_{\overline{Y}_i(p)}),$$
 we obtain a natural action of $\GL_2(\A_{F,\f}^{(p)})$ on
 $\varinjlim M_{2,0}(U_i(p),\F)$ and Hecke-equivariant injective homomorphisms
 $$M_{2,0}(U_i(p),\CO) \otimes_\CO \F  \to M_{2,0}(U_i(p),\F).$$

Recall also that we have a natural action of $(\CO_F/p)^\times \cong U_0(p)/U_1(p)$ on
$Y_1(p)$, hence on $M_{2,0}(U_1(p),\CO) = H^0(Y_1(p), \CK_{Y_1(p)/\CO})$.
Furthermore the action is compatible with the Hecke action (in the usual sense) and with
the natural action of $(\CO_F/p)^\times$ on $M_{2,0}(U_1(p),L)$.  If $L$ is
sufficiently large as in \S\ref{sss:U1p} (i.e., $L$ contains the $(q-1)$-roots of unity
where $q-1$ is divisible by the exponent of $(\CO_F/p)^\times$), then we may decompose
$$M_{2,0}(U_1(p),\CO) = \bigoplus_{\chi}  M_{2,0}(U_1(p),\CO)^\chi$$
as the direct sum of $\chi$-eigenspaces over the characters 
$\chi:(\CO_F/p)^\times \to \CO^\times$.    Thus $M_{2,0}(U_1(p),\CO)^\chi$
is an $\CO$-lattice in the space $M_{2,0}(U_1(p),L)^\chi$ of Hilbert modular
forms of parallel weight $2$, level $U_1(p)$ and character $\chi$, and $\GL_2(\A_{F,\f}^{(p)})$
acts on $\varinjlim M_{2,0}(U_1(p),\CO)^\chi$ where the direct limit is over sufficiently small
open compact $U$ of level prime to $p$.  Similarly we may decompose 
$$M_{2,0}(U_1(p),\F)= \bigoplus_{\chi}  M_{2,0}(U_1(p),\F)^\chi,$$
and we obtain an action of $\GL_2(\A_{F,\f}^{(p)})$ on $\varinjlim M_{2,0}(U_1(p),\F)^\chi$
compatible with the injective maps
$$M_{2,0}(U_1(p),\CO)^\chi  \otimes_{\CO} \F \to M_{2,0}(U_1(p),\F)^\chi.$$

Since the morphism $h: Y_1(p) \to Y_0(p)$ is finite, we have 
$$h_*\CK_{Y_1(p)/\CO} = \Shom_{\CO_{Y_0(p)}}(h_*\CO_{Y_1(p)},\CK_{Y_0(p)/\CO})$$
and $R^jh_*\CK_{Y_1(p)/\CO} = 0$ for $j > 0$.  Recall that the explicit description of $Y_1(p)$
 in terms of a Raynaud $(\CO_F/p)$-module scheme on $Y_0(p)$ shows that
 $$h_*\CO_{Y_1(p)} = (\sym_{\CO_{Y_0(p)}} \CL)/\CI'$$
 where $\CL$ is the direct sum of the Raynaud line bundles $\CL_\theta$ on $Y_0(p)$ and
 $\CI'$ is the ideal generated by the $\CO_S$-submodules $(s_\theta - 1)\CL_\theta^p$ for
 $\theta \in \Theta$ and $(s_v - 1) \otimes_{\theta \in \Theta_v} \CL_\theta^{p-1}$ for $v|p$.
 We therefore have
 $$h_*\CO_{Y_1(p)} = \bigoplus \left(\bigotimes_{\theta\in \Theta} \CL_\theta^{m_\theta} \right)$$
 where the direct sum is over $m \in  \Z^\Theta$ such $0 \le m_\theta \le p-1$ for all $\theta$ and
 $m_\theta < p-1$ for some $\theta$ in each $\Theta_v$.  Since $(\CO_F/p)^\times$ acts on $\CL_\theta$
 via the Teichmuller lift of $\overline{\theta}$ and each character $\chi$ of $(\CO_F/p)^\times$
 is the Teichmuller lift of $\prod_{\theta \in \Theta} \overline{\theta}^{m_{\chi,\theta}}$
 for a unique such $m = m_\chi$, we may rewrite this decomposition as
$$h_*\CO_{Y_1(p)} = \bigoplus_\chi  \CL_\chi$$
where $(\CO_F/p)^\times$ acts on the line bundle $\CL_\chi : = \bigotimes_{\theta\in \Theta} \CL_\theta^{m_{\chi,\theta}}$
via $\chi$.  We therefore have
$$ h_*\CK_{Y_1(p)/\CO} =  \bigoplus_\chi  \CL_{\chi^{-1}}^{-1}\CK_{Y_0(p)/\CO} $$
so that $H^i(Y_1(p),\CK_{Y_1(p)/\CO})^\chi = H^i (Y_0(p), \CL_{\chi^{-1}}^{-1}\CK_{Y_0(p)/\CO})$ for $i \ge 0$,
and in particular
$$M_{2,0}(U_1(p),\CO)^\chi = H^0(Y_0(p), \CL_{\chi^{-1}}^{-1}\CK_{Y_0(p)/\CO}).$$
Furthermore these identifications are compatible with the natural Hecke action arising
from the identifications $\rho_g^*\CK_{Y'_j(p)/\CO} = \CK_{Y_j(p)/\CO}$ for $j=0,1$
and the isomorphism $\pi_g^*:\rho_g^*\CL_\theta' \cong \CL_\theta$ defined at the end of \S\ref{sss:U1p}
(where $g \in \GL_2(\A_{F,\f}^{(p)})$, $g^{-1}U'g \subset U$, $Y'_j(p) =Y_{U'_j(p)}(G)_\CO$ and 
$\CL_\theta'$ is the Raynaud bundle on $Y'_0(p)$).

Similarly if we let $\overline{Y}_j(p) = Y_{U_j(p)}(G)_\F$ for $j=0,1$ and $\overline{h}:\ol{Y}_1(p) \to \ol{Y}_0(p)$,
then we obtain
\begin{equation}
\label{eqn:chars}\ol{h}_*\CK_{\ol{Y}_1(p)} =  \bigoplus_\chi  \CL_{\chi^{-1}}^{-1} \CK_{\ol{Y}_0(p)}
\end{equation}
where $\CL_\chi = \bigotimes_{\theta\in \Theta} \CL_\theta^{m_{\chi,\theta}}$ as above, but now $\chi$ is viewed as a character
$(\CO_F/p\CO_F)^\times \to \F^\times$ and $\CL_\theta$ is the Raynaud
bundle on $\ol{Y}_0(p)$.  It follows that
$H^i(\ol{Y}_1(p),\CK_{\ol{Y}_1(p)})^\chi = H^i (\ol{Y}_0(p), \CL_{\chi^{-1}}^{-1}\CK_{\ol{Y}_0(p)})$
and in particular
$$M_{2,0}(U_1(p),\F)^\chi = H^0(\ol{Y}_0(p), \CL_{\chi^{-1}}^{-1}\CK_{\ol{Y}_0(p)}),$$
compatibly with the Hecke action and the corresponding identifications over $\CO$.
Our analysis of the space $M_{2,0}(U_1(p),\F)$, and more generally the cohomology
of $\CK_{\ol{Y}_1(p)}$, thus reduces to the study of the line bundles
$\CL_{\chi^{-1}}^{-1}\CK_{\ol{Y}_0(p)}$ on $\ol{Y}_0(p)$.

\subsection{Dicing} \label{sec:dicing}
In this section, we describe the dualizing sheaf $\CK_{\ol{Y}_0(p)}$ on $\ol{Y}_0(p)$ in terms of
those on the subschemes $\ol{Y}_0(p)_J$.  Recall that $\ol{Y}_0(p)_J$
is defined as the vanishing locus of
$$\{\,\Lie(f)_\theta\,|\,\theta\not\in J \,\}  \cup \{\,\Lie(f^\vee)_\theta\,|\,\theta\in J \,\}$$
where $f:A_1 \to A_2$ is the universal isogeny on $\widetilde{Y}_{U_0(p)}(G)_\F$;  the schemes $\ol{Y}_0(p)_J$
are smooth over $\F$, and every irreducible component of $\ol{Y}_0(p)$ is a connected
component of $\ol{Y}_0(p)_J$ for a unique $J$.  We will define a filtration on $\CK_{\ol{Y}_0(p)}$
whose associated graded is a direct sum of sheaves supported on the subschemes $\ol{Y}_0(p)_J$.
We refer to the process of analyzing the sheaf by dividing the scheme into smaller ones of the
same dimension as {\em dicing}.

\subsubsection{Dicing the structure sheaf}\label{sss:dice.str}
We first define a filtration on $\CO_{\ol{Y}_0(p)}$ by the sheaves of ideals
$$\CI_j := \langle \,  \prod_{\theta \in J}  \Lie(f)_\theta  \, \rangle_{|J|=j}$$
for $j \ge 0$, where $\Lie(f)_\theta$ is viewed locally as a section of $\CO_{\ol{Y}_0(p)}$ via
any choice of trivialization of the descent of 
$$\Shom_{\CO_S}(\Lie(A_1/S)_\theta,\Lie(A_2/S)_\theta)$$
to $\ol{Y}_0(p)$.  Note that
$$\CO_{\ol{Y}_0(p)} = \CI_0 \supset  \CI_1 \supset \cdots \supset \CI_d \supset \CI_{d+1} = 0.$$

For each $J$, let $i_J$ denote the closed immersion $\ol{Y}_0(p)_J \to \ol{Y}_0(p)$.  We let
$\CP_J$ denote the ideal sheaf on $\ol{Y}_0(p)$ defining $\ol{Y}_0(p)_J$, and let
$\CI_J$ denote the ideal sheaf on $\ol{Y}_0(p)_J$ generated by $\prod_{\theta \in J} \Lie(f)_\theta$.  
We see from Theorem~\ref{thm: coordinates} that the ideal generated by
$\prod_{\theta \in J} \Lie(f)_\theta$ is non-zero in the regular local ring obtained by completing 
$\ol{Y}_0(p)_J$ at each closed point,
and it follows that $\CI_J$ is an invertible sheaf.
Similarly we let $\CJ_J$ denote the invertible sheaf of ideals on $\ol{Y}_0(p)_J$ generated by 
$\prod_{\theta \not\in J} \Lie(f^\vee)_\theta$.

Note that if $|J| = |J'|$ and $J' \neq J$, then the section $\prod_{\theta \in J'}  \Lie(f)_\theta$ vanishes on 
$\ol{Y}_0(p)_J$, so we have a natural
map $\CI_j \to \bigoplus_{|J| = j} i_{J,*} \CI_J$.
Furthermore if $|J'| = j+1$, then $\prod_{\theta \in J'}  \Lie(f)_\theta$ vanishes on $\ol{Y}_0(p)_J$, so 
$\CI_{j+1}$ is contained in the kernel.  

\begin{lemma} \label{lem:diceO}
The natural map $\gr^j\CO_{\ol{Y}_0(p)} \to \bigoplus_{|J|= j} i_{J,*} \CI_J$ is an isomorphism.
\end{lemma}
\begpf  It suffices to prove the morphism is an isomorphism on completions of stalks at each closed point
$Q$ of $\ol{Y}_0(p)$.  We do this using the description
$$\widehat{\mathcal{O}}_{S,Q} \cong  
\widehat{\bigotimes_{\theta\in \Theta}}
k_Q[[x_\theta,y_\theta]]/\langle c_\theta d_\theta \rangle$$
provided by Theorem~\ref{thm: coordinates}.  Using the notation of that theorem
and writing $c_J$ for $\prod_{\theta \in J} c_\theta$ and $\overline\cdot$ for images in the quotient
$k_Q[[x_\theta,y_\theta]]/\langle c_\theta d_\theta \rangle$, the above
morphism corresponds to the natural map
\begin{equation}\label{eqn:dicingmap} \begin{array}{ccc}
I_j/I_{j+1}  &\to &\bigoplus_{|J|=j}  (\langle \overline{c}_J \rangle + P_J)/P_J\\
\left(\displaystyle\sum_{|J|=j} r_J\overline{c}_J \right) + I_{j+1} & \mapsto & \left(r_J \overline{c}_J + P_J \right)_J,
\end{array}\end{equation}
where $I_m = \langle \overline{c}_J\rangle_{|J| = m}$ and 
$P_J = \langle \overline{c}_\theta \rangle_{\theta\not\in J} + \langle \overline{d}_\theta \rangle_{\theta\in J}$.
Note that $Q \in \overline{Y}_0(p)_J$ if and only if $J \subset J_Q$ and $I_Q \cup J = \Theta$, in which case
$P_J$ corresponds to a minimal prime of $\widehat{\mathcal{O}}_{S,Q}$; otherwise
$P_J$ corresponds to $\widehat{\mathcal{O}}_{S,Q}$.

Since (\ref{eqn:dicingmap}) is obviously surjective, we just need to prove that if $r_J \overline{c}_J \in P_J$,
then $r_J\overline{c}_J \in I_{j+1}$.  If $\theta \not\in I_Q$ for some $\theta \not\in J$, then 
$\overline{c}_\theta$ is a unit, so $r_J\overline{c}_J \in \langle \overline{c}_{J\cup \{\theta\}} \rangle \subset I_{j+1}$,
and if $\theta \not\in J_Q$ for some $\theta \in J$, then $\overline{c}_\theta = 0$, so
$r_J\overline{c}_J = 0 \in I_{j+1}$.  We may therefore assume that $P_J$ 
corresponds to a prime ideal of  $\widehat{\mathcal{O}}_{\overline{Y}_0(p),Q}$.  Note that 
since $\overline{c}_\theta \not\in P_J$ for each $\theta \in J$, we have $\overline{c}_J \not\in P_J$,
and therefore $r_J \in P_J$.  Furthermore $\overline{c}_J \overline{d}_\theta = 0$ for $\theta\in J$, so
$$r_J \overline{c}_J \in \overline{c}_J P_J = \overline{c}_J \langle \overline{c}_\theta \rangle_{\theta\not\in J}
   \subset I_{j+1}.$$
This completes the proof that (\ref{eqn:dicingmap}) is an isomorphism.
\epf

Note that the proof that (\ref{eqn:dicingmap}) is an isomorphism
shows that $I_j = \bigcap_{|J| < j} P_J$, which translates to the statement that $\CI_j = \bigcap_{|J|<j} \CP_J$,
hence defines $\bigcup_{|J|<j} \overline{Y}_0(p)_J$ with its reduced induced subscheme structure.  Note in
particular that $\langle \overline{c}_\Theta \rangle = I_d = \bigcap_{J\neq \Theta} P_J$;
furthermore the same argument shows that if $J \subset \Theta$, then
\begin{equation}\label{eqn:ann} \langle \overline{c}_J \overline{d}_{\Theta - J} \rangle = \bigcap_{J' \neq J}  P_{J'}.\end{equation}
Since the ring is reduced and all its minimal primes are of the form $P_J$, it follows
that $\langle \overline{c}_J \overline{d}_{\Theta - J} \rangle$ is the annihilator of $P_J$ and
that $\langle \overline{c}_J \overline{d}_{\Theta - J} \rangle \cap P_J = \{0\}$.

\subsubsection{Dicing the dualizing sheaf}\label{sss:dice.dual}
Recall from \S\ref{sss:dual.gen} that since $i_J:\overline{Y}_0(p)_J \to \ol{Y}_0(p)$ is a finite morphism of Cohen--Macaulay
schemes of the same dimension, we have a canonical isomorphism
$$
i_{J,*}  \CK_{\ol{Y}_0(p)_J} \cong\Shom_{\CO_{\ol{Y}_0(p)}} ( i_{J,*}  \CO_{\ol{Y}_0(p)_J}, \CK_{\ol{Y}_0(p)} ).$$
Moreover since $\overline{Y}_0(p)_J$ and $\ol{Y}_0(p)$ are local complete intersections, and hence Gorenstein,
their dualizing sheaves are line bundles; in particular we have
$$ \Shom_{\CO_{\ol{Y}_0(p)}} ( i_{J,*}  \CO_{\ol{Y}_0(p)_J}, \CK_{\ol{Y}_0(p)} )
= \Shom_{\CO_{\ol{Y}_0(p)}} ( i_{J,*}  \CO_{\ol{Y}_0(p)_J}, \CO_{\ol{Y}_0(p)} ) \otimes_{ \CO_{\ol{Y}_0(p)}  } \CK_{\ol{Y}_0(p)}.$$
The surjective morphism $ \CO_{\ol{Y}_0(p)}  \to  \CO_{\ol{Y}_0(p)}/\CP_J = i_{J,*}  \CO_{\ol{Y}_0(p)_J}$ induces an
injective morphism
$$ \Shom_{\CO_{\ol{Y}_0(p)}} ( i_{J,*}  \CO_{\ol{Y}_0(p)_J}, \CO_{\ol{Y}_0(p)} ) \to \CO_{\ol{Y}_0(p)}$$
whose image is the sheaf of ideals $\mathcal{A}nn_{\CO_{\ol{Y}_0(p)}}\CP_J$.  It follows from (\ref{eqn:ann}) and
the subsequent discussion that $\mathcal{A}nn_{\CO_{\ol{Y}_0(p)}}\CP_J$ is generated by
$\prod_{\theta\in J} \Lie(f)_\theta \prod_{\theta\not\in J} \Lie(f^\vee)_\theta$, so that
$$i_J^* \mathcal{A}nn_{\CO_{\ol{Y}_0(p)}} \CP_J= \CI_J\CJ_J;$$
moreover since $\CP_J \cap \mathcal{A}nn_{\CO_{\ol{Y}_0(p)}}\CP_J = 0$, the natural map
$$\mathcal{A}nn_{\CO_{\ol{Y}_0(p)}} \CP_J \to i_{J,*}i_J^* \mathcal{A}nn_{\CO_{\ol{Y}_0(p)}} = i_{J,*}(\CI_J\CJ_J)$$
is an isomorphism.    Tensoring with  $\CK_{\ol{Y}_0(p)}$ therefore yields a canonical isomorphism
\begin{equation}
\label{eqn:pushK}
 i_{J,*} \CK_{\ol{Y}_0(p)_J} \stackrel{\sim}{\longrightarrow} i_{J,*}(\CI_J \CJ_J) \otimes_{\CO_{\ol{Y}_0(p)}}   \CK_{\ol{Y}_0(p)}.
\end{equation}

Now consider the {\em dicing filtration} on $\CK_{\ol{Y}_0(p)}$ defined by 
$$\fil^j \CK_{\ol{Y}_0(p)} = \CI_j \CK_{\ol{Y}_0(p)} = \CI_j \otimes_{\CO_{\ol{Y}_0(p)}}  \CK_{\ol{Y}_0(p)},$$
where the latter equality follows from the invertibility of $\CK_{\ol{Y}_0(p)}$.
Lemma~\ref{lem:diceO} thus yields a canonical isomorphism
$$\gr^j\CK_{\ol{Y}_0(p)} \stackrel{\sim}{\longrightarrow}  \bigoplus_{|J|= j} \, i_{J,*} \CI_J \otimes_{\CO_{\ol{Y}_0(p)}}  \CK_{\ol{Y}_0(p)}.$$
We may then apply (\ref{eqn:pushK}) to compute the summands as
$$i_{J,*} \CJ_J^{-1} \otimes_{\CO_{\ol{Y}_0(p)}} i_{J,*}(\CI_J\CJ_J) \otimes_{\CO_{\ol{Y}_0(p)}}   \CK_{\ol{Y}_0(p)}
    =  i_{J,*}( \CJ_J^{-1}\CK_{\ol{Y}_0(p)_J}).$$
We therefore have a canonical isomorphism
\begin{equation} \label{eqn:diceK}
\gr^j\CK_{\ol{Y}_0(p)}  \cong  \bigoplus_{|J|= j} i_{J,*}( \CJ_J^{-1}\CK_{\ol{Y}_0(p)_J}).
\end{equation}
It is straightforward to check that the dicing filtration is compatible with the Hecke action in the usual sense,
as is the isomorphism (\ref{eqn:diceK}).  Combining this isomorphism with the conclusion of \S\ref{sss:dual.wt2},
we see that the analysis of the cohomology of $\CK_{\overline{Y}_1(p)}$ reduces to that of the line bundles
$\CJ_J^{-1}\CK_{\ol{Y}_0(p)_J} i_J^*\CL_{\chi^{-1}}^{-1}$ on the smooth varieties $\ol{Y}_0(p)_J$.

More precisely, we define a filtration on each component of $\overline{h}_*\CK_{\ol{Y}_1(p)}$
under the decomposition (\ref{eqn:chars}) by setting 
$$\fil^j(\CK_{\ol{Y}_0(p)}\CL_{\chi^{-1}}^{-1}) = \CI_j\CK_{\ol{Y}_0(p)}\CL_{\chi^{-1}}^{-1}.$$
The filtration induces one on $H^i(\ol{Y}_1(p),\CK_{\ol{Y}_1(p)})^\chi = H^i(\ol{Y}_0(p),\CK_{\ol{Y}_0(p)}\CL_{\chi^{-1}}^{-1})$
defined by $\fil^j (H^i(\ol{Y}_0(p),\CK_{\ol{Y}_0(p)}\CL_{\chi^{-1}}^{-1}) = $
$$\im\left(H^i(\ol{Y}_0(p),\fil^j(\CK_{\ol{Y}_0(p)}\CL_{\chi^{-1}}^{-1})) \to H^i(\ol{Y}_0(p),\CK_{\ol{Y}_0(p)}\CL_{\chi^{-1}}^{-1})\right),$$
for which the graded pieces $E_\infty^{j,i} = E_{d+1}^{j,i} = \gr^j \left(H^{i+j}(\ol{Y}_0(p),\CK_{\ol{Y}_0(p)}\CL_{\chi^{-1}}^{-1})\right)$
are computed by the spectral sequence
$$
E_1^{j,i} = H^{i+j}(\ol{Y}_0(p),\gr^j(\CK_{\ol{Y}_0(p)}\CL_{\chi^{-1}}^{-1})) \Longrightarrow H^{i+j}(\ol{Y}_0(p),\CK_{\ol{Y}_0(p)}\CL_{\chi^{-1}}^{-1}).
$$
Note that (\ref{eqn:diceK}) gives a canonical isomorphism
$$\gr^j(\CK_{\ol{Y}_0(p)}\CL_{\chi^{-1}}^{-1})  \cong  \bigoplus_{|J|= j}\left( i_{J,*}( \CJ_J^{-1}\CK_{\ol{Y}_0(p)_J}) \otimes_{\CO_{\ol{Y}_0(p)}}\CL_{\chi^{-1}}^{-1}\right),$$
from which it follows that
$$E_1^{j,i} \cong \bigoplus_{|J|= j} H^{i+j}(\ol{Y}_0(p)_J, \;\CJ_J^{-1}\CK_{\ol{Y}_0(p)_J}i_J^*\CL_{\chi^{-1}}^{-1}).$$

\subsubsection{Hecke equivariance}\label{sss:dice.Hecke}
The above spectral sequence and isomorphisms are compatible with the Hecke action in the following sense.
Suppose as usual that $U$ and $U'$ are sufficiently small open compact subgroups of $G(\A_\f)$ of level prime to $p$,
and $g  \in G(\A_\f^{(p)})$ is such that $g^{-1}Ug \subset U'$.  Using $'$ to denote the relevant objects defined
with $U$ replaced by $U'$, we have the corresponding spectral sequence
$$E_1'^{j,i} = H^{i+j}(\ol{Y}'_0(p),\gr^j(\CK_{\ol{Y}'_0(p)}\CL_{\chi^{-1}}'^{-1})) \Longrightarrow H^{i+j}(\ol{Y}'_0(p),\CK_{\ol{Y}'_0(p)}\CL_{\chi^{-1}}'^{-1})$$
and isomorphism
$$E_1'^{j,i} \cong \bigoplus_{|J|= j} H^{i+j}(\ol{Y}'_0(p)_J, \;\CJ_J'^{-1}\CK_{\ol{Y}'_0(p)_J}i_J'^*\CL_{\chi^{-1}}'^{-1}).$$
The canonical isomorphism $\rho_g^*\CK_{\ol{Y}'_1(p)} \cong \CK_{\ol{Y}_1(p)}$ is compatible with the decompositions
(\ref{eqn:chars}) and corresponding isomorphisms of summands (with $\rho_g^*\CL_{\chi^{-1}}'^{-1} \cong \CL_{\chi^{-1}}^{-1}$
defined by $\pi_g^*$).  As $\rho_g^*\CI'_j = \CI_j$, the isomorphisms of summands preserve the filtrations and
hence induce morphisms $E_r'^{j,i} \to E_r^{j,i}$ compatible with the differentials $d^{j,i}_r:E_r^{j,i} \to E_r^{j+r,i-r+1}$,
$d'^{j,i}_r:E_r'^{j,i} \to E_r'^{j+r,i-r+1}$ and identifications $E_{r+1}^{j,i} = \ker(d_r^{j,i})/\im(d_r^{j-r,i+r-1})$,
$E_{r+1}'^{j,i} = \ker(d_r'^{j,i})/\im(d_r'^{j-r,i+r-1})$.
Furthermore the descriptions of $\gr^j\CK_{\ol{Y}_0(p)}$ and $\gr^j\CK_{\ol{Y}'_0(p)}$ are compatible with $\rho_g^*$
in the obvious sense, so the resulting diagram
$$\begin{array}{ccc}
E_1'^{j,i} & \stackrel{\sim}{\longrightarrow} & \displaystyle\bigoplus_{|J|= j} H^{i+j}(\ol{Y}'_0(p)_J, \;\CJ_J'^{-1}\CK_{\ol{Y}'_0(p)_J}i_J'^*\CL_{\chi^{-1}}'^{-1})\\
\downarrow&&\downarrow \\
E_1^{j,i} & \stackrel{\sim}{\longrightarrow} &\displaystyle\bigoplus_{|J|= j} H^{i+j}(\ol{Y}_0(p)_J, \;\CJ_J^{-1}\CK_{\ol{Y}_0(p)_J}i_J^*\CL_{\chi^{-1}}^{-1})\end{array}$$
commutes, where the arrow on the right is induced by the canonical isomorphisms
$\rho_g^*\CJ'_J \cong \CJ_J$, $\rho_g^*\CK_{\ol{Y}'_0(p)_J} \cong \CK_{\ol{Y}_0(p)_J}$ and
$\pi_g^*:\rho_g^*\CL'_{\chi^{-1}} \stackrel{\sim}{\to} \CL_{\chi^{-1}}$.  To sum up so far, we have
constructed a Hecke-equivariant spectral sequence with
\begin{equation}\label{eqn:SS1}
E_1^{j,i} \cong \bigoplus_{|J|= j} H^{i+j}(\ol{Y}_0(p)_J, \;\CJ_J^{-1}\CK_{\ol{Y}_0(p)_J}i_J^*\CL_{\chi^{-1}}^{-1})
\Longrightarrow H^{i+j}(\ol{Y}_1(p),\CK_{\ol{Y}_1(p)})^\chi.\end{equation}

\subsection{Proof of Theorem~\ref{thm:JLss}}  \label{sec:fun}
We now proceed to prove Theorem~\ref{thm:JLss}, relating the cohomology of $\CK_{\overline{Y}_1(p)}$ to
that of automorphic sheaves on the quaternionic Shimura varieties $\overline{Y}_\Sigma$.  We do this by
determining the line bundles corresponding to the $\CJ_J^{-1}\CK_{\ol{Y}_0(p)_J} i_J^*\CL_\chi^{-1}$ under
the isomorphism $\Xi_J$ of Theorem~\ref{thm:qJL}, and in turn their (higher) direct images on the $\ol{Y}_\Sigma$.
(Recall that $\CJ_J$ and $i_J$ were defined in \S\ref{sss:dice.str}, and $\CL_\chi$ in \S\ref{sss:dual.wt2}.)
We now assume $\F = \Fpbar$ to ensure that $\Xi_J$ is defined.

\subsubsection{The factors $\CJ$, $\CK$ and $\CL$}\label{sss:fun.JKL}
We first describe the line bundle corresponding to $\CK_{\ol{Y}_0(p)_J}$ under 
$$\Xi_J:  \overline{Y}_0(p)_J \stackrel{\sim}{\longrightarrow} \prod_{\theta\in \Sigma}\P_{\overline{Y}_\Sigma}(\CV_\theta).$$
We will then describe the factors $i_J^*\CL_{\chi^{-1}}^{-1}$ and $\CJ_J^{-1}$
in terms of the Raynaud bundles $\CL_\theta$ on
$\overline{Y}_0(p)_J$, for which we already determined the corresponding bundles under $\Xi_J$
in \S\ref{sss:buns.hmv}.

Note that $\Xi_{J,*}\CK_{\ol{Y}_0(p)_J}$ is canonically isomorphic to the dualizing sheaf on
$\prod_{\theta\in \Sigma}\P_{\overline{Y}_\Sigma}(\CV_\theta)$.
Furthermore if $X = \P_S(\CV)$ where $\CV$ is a rank two vector bundle on a scheme $S$, then
the dualizing sheaf $\CK_{X/S}$ relative to the projection $\psi:X \to S$ is canonically isomorphic to
\begin{equation}\label{eqn:calculus}\Omega^1_{X/S}  \cong \psi^*(\wedge_{\CO_S}\CV)(-2).\end{equation}
It follows that if $X = \prod_{\theta\in \Sigma}\P_{\overline{Y}_\Sigma}(\CV_\theta)$ and 
$\Psi_J: X \to \ol{Y}_\Sigma$ denotes the projection, then $\CK_X$ is
canonically isomorphic to
$$\Psi_J^*\CK_{\ol{Y}_\Sigma}  \bigotimes \left(\bigotimes_{\theta\in \Sigma}
        \Psi_J^*(\delta_\theta)(-2)_\theta\right),$$
 where $\delta_\theta = \wedge^2_{\CO_{\ol{Y}_\Sigma}}\CV_\theta$, all tensor products are over $\CO_X$,
 and as usual, $(n)_\theta$ denotes the twist by $\CO(n)_\theta = \CO(1)^n_\theta$ in the $\theta$-component.
 Combining this with the Kodaira--Spencer isomorphism from \S\ref{sec:KS}, we conclude that
 $\Xi_{J,*}\CK_{\ol{Y}_0(p)_J}$ is isomorphic to
\begin{equation}\label{eqn:Kfactor}
\CK_X \cong
\left(\bigotimes_{\theta\not\in \Sigma}
        \delta_\theta^{-1}\omega_\theta^2\right)\bigotimes \left(\bigotimes_{\theta\in \Sigma}
        \delta_\theta(-2)_\theta\right),\end{equation}
 where we have again written $\delta_\theta$ and $\omega_\theta$ for their pull-back via $\Psi_J$.
 Furthermore the compatibility of (\ref{eqn:calculus}) with isomorphisms of vector bundles and
 that of the Kodaira--Spencer isomorphism with the Hecke action ensures that
 (\ref{eqn:Kfactor}) is compatible with the Hecke action in the usual sense.

Turning now to $i_J^*\CL_{\chi^{-1}}$, note that we may write $\chi = \prod_{\theta\in \overline{\Theta}_p}  \theta^{m_\theta}$ for a unique
$m = (m_\theta)_\theta \in \Z^{\Theta}$ such that $0 \le m_\theta \le p-1$ for all $\theta$ and $m_\theta > 0$
for some $\theta$ in each $\Theta_v$.  (Recall that we identify  $\overline{\Theta}_p$ with
$\Theta = \coprod_{v|p} \Theta_v$ via our fixed choices of embeddings, and note the slight
difference with the choice made in \S\ref{sss:dual.wt2} in that we take $m_\theta = p-1$, instead
of $0$, for all $\theta \in \Theta_v$ if $\chi$ is trivial on $(\CO_F/v)^\times$.)   We then
have $\chi^{-1} = \prod_{\theta\in \overline{\Theta}_p}  \theta^{p-1-m_\theta}$, so that
\begin{equation}  \label{eqn:Lfactor}
\CL_{\chi^{-1}}^{-1}  = \bigotimes_{\theta\in \Theta}  \CL_\theta^{m_\theta - p + 1}.
\end{equation}

Finally we write the factor $\CJ_J^{-1}$ in terms of the Raynaud bundles on $\ol{Y}_0(p)_J$.
More precisely, we will define a Hecke-equivariant isomorphism
\begin{equation}\label{eqn:Jfactor}
\CJ_J^{-1} \stackrel{\sim}{\to} \bigotimes_{\theta\not\in J}  i_J^* (\CL_\theta^{-1}\CL^p_{\phi^{-1}\circ\theta}),\end{equation}
under which the canonical section of $\CJ_J^{-1}$ corresponds to the restriction to $\ol{Y}_0(p)_J$ of the morphism
$t_{\phi^{-1}\circ\theta}:\CL_\theta \to \CL_{\phi^{-1}\circ\theta}^p$ induced by $\Ver:H^{(p)} \to H = \ker(f)$.

Recall that $\CJ_J$ is the ideal sheaf on $\overline{Y}_0(p)_J$ defined by 
$\prod_{\theta\not\in J} \Lie(f^\vee)_\theta$.  Though we make no direct use of the following
description, we remark that $\CJ_J^{-1}$ can be identified with the line bundle
$\CO_{\overline{Y}_0(p)_J}(D)$ where
$$D =  \sum_{\theta\not\in J}\left(\overline{Y}_0(p)_{J} \cap \overline{Y}_0(p)_{J \cup \{\theta\}}\right)
 = \sum_{\theta\not\in J}  \overline{Y}_0(p)_{\Theta - \phi(J), J \cup \{\theta\}},$$
viewed as a Weil divisor on $\overline{Y}_0(p)_J$.  Recall that $\Lie(f^\vee)_\theta$ is a section
of the line bundle defined by 
$$\Shom_{\CO_S}(\Lie(A_2^\vee/S)_\theta,\Lie(A_1^\vee/S)_\theta)$$
with its canonical descent data, where $f:A_1 \to A_2$ is the universal isogeny on 
$S = \widetilde{Y}_{U_0(p)}(G)_\F$.  Letting $s_j:A_j \to S$ denote the structure morphism for $j=1,2$,
we have canonical isomorphisms
$$\Lie(A_j^\vee/S)_\theta \cong   \wt{\upsilon}_{j,\theta} \cong \wt{\delta}_{j,\theta}\wt{\omega}_{j,\theta}^{-1},$$
where $\wt{\upsilon}_{j,\theta} = R^1s_{j,*}\CO_{A_j}$, $\wt{\omega}_{j,\theta} = s_{j,*}(\Omega^1_{A_j/S})_\theta$
and $\wt{\delta}_{j,\theta}= \wedge^2_{\CO_S}\CH^1_{\dr}(A_j/S)_\theta$. As in \S\ref{sss:buns.delta}, the exact sequences
$$\begin{array}{cccccccccc}
&0 &\to& \ker(f^*)& \to & \CH^1_\dr(A_2/S)& \to &\ker(g^*) &\to& 0\\
\mbox{and}& 0 &\to& \ker(g^*)& \to &\CH^1_\dr(A_1/S)& \to& \ker(f^*) &\to& 0\end{array}$$
(where $g:A_2 \to A_1$ denotes the transpose of $f$, i.e., $f\circ g = p$)
yield an isomorphism $\wt{\delta}_{1,\theta} \cong \wt{\delta}_{2,\theta}$.  We thus obtain an isomorphism
$$\begin{array}{rl} \Shom_{\CO_S}(\wt{\upsilon}_{2,\theta},\wt{\upsilon}_{1,\theta})
 &\cong \Shom_{\CO_S}(\wt{\delta}_{2,\theta}\wt{\omega}_{2,\theta}^{-1},\wt{\delta}_{1,\theta}\wt{\omega}_{1,\theta}^{-1})\\
& \cong  \Shom_{\CO_S}(\wt{\delta}_{1,\theta}\wt{\omega}_{2,\theta}^{-1},\wt{\delta}_{1,\theta}\wt{\omega}_{1,\theta}^{-1})  \cong 
 \Shom_{\CO_S}(\wt{\omega}_{1,\theta},\wt{\omega}_{2,\theta})\end{array}$$
under which $\Lie(f^\vee)$ corresponds to $g^*$.

Suppose now that $\theta\not\in J$, let $\tilde{i}_J$ denote the closed immersion $S_J \to S$,
and consider the restriction of $g^*$ to $S_J$.  Recall from the proof of Lemma~\ref{lem:vanishing}
that the inclusion $H = \ker(f) \subset A_1$ induces an isomorphism $\Lie(H/S_J)_\theta \cong \Lie(A_1/S_J)_\theta$
and hence 
$$\tilde{i}_J^*\wt{\omega}_{1,\theta}  \cong \tilde{i}_J^* \CL_\theta.$$
Furthermore if also $\phi^{-1}\circ\theta \not\in J$, then we similarly have 
$\tilde{i}_J^*\wt{\omega}_{1,\phi^{-1}\circ\theta}  \cong \tilde{i}_J^* \CL_{\phi^{-1}\circ\theta}$, and since
$\Ver^*(f^*\CH^1_\dr(A_2/S_J)_\theta) = f^*(\tilde{i}_J^*\wt{\omega}^p_{2,\phi^{-1}\circ\theta}) = 0$ in this case, we
have the exact sequence
$$0 \longrightarrow f^*\CH^1_\dr(A_2/S_J)_\theta  \longrightarrow \CH^1_\dr(A_1/S_J)_\theta  \stackrel{\Ver^*}{\longrightarrow} \tilde{i}_J^*\wt{\omega}^p_{1,\phi^{-1}\circ\theta} \longrightarrow 0,$$
which combined with the exact sequence
$$0 \longrightarrow f^*\CH^1_\dr(A_2/S_J)_\theta  \longrightarrow \CH^1_\dr(A_1/S_J)_\theta  \stackrel{g^*}{\longrightarrow} \tilde{i}_J^*\wt{\omega}_{2,\theta} \longrightarrow 0$$
yields an isomorphism
$$  \tilde{i}_J^*\wt{\omega}_{2,\theta}  \cong   \tilde{i}_J^*\wt{\omega}^p_{1,\phi^{-1}\circ\theta}  \cong \tilde{i}_J^* \CL^p_{\phi^{-1}\circ\theta}.$$
On the other hand if $\phi^{-1}\circ\theta \in J$, then the proof of Lemma~\ref{lem:vanishing} shows that 
the inclusion $H^\vee \subset A_2^\vee$ induces an isomorphism $\Lie(H^\vee/S_J)_{\phi^{-1}\circ\theta} \cong \Lie(A_2^\vee/S_J)_{\phi^{-1}\circ\theta}$, and
hence $\tilde{i}_J^*\CL_{\phi^{-1}\circ\theta} \cong \tilde{i}_J^*\wt{\upsilon}_{2,\phi^{-1}\circ\theta}$.  In this case however, the fact that
$\phi^{-1}\circ \theta \in J$ and $\theta\not\in J$ implies that
$$\Phi(\D(A_{2,\overline{s}})_{\phi^{-1}\circ\theta})  = g^*(\D(A_{1,\overline{s}}))_\theta = V(\D(A_{2,\overline{s}})_{\phi\circ\theta})$$
for all $\overline{s} \in S_J(\Fpbar)$.  It follows that $\Frob^*: \CH^1_\dr(A_2/S_J)^{(p)}_{\phi^{-1}\circ\theta} \to \CH^1_\dr(A_2/S_J)_\theta$
induces an isomorphism $\tilde{i}_J^*\wt{\upsilon}^p_{2,\phi^{-1}\circ\theta}  \to \tilde{i}_J^*\wt{\omega}_{2,\theta}$, so that now
$$\tilde{i}_J^*\CL^p_{\phi^{-1}\circ\theta} \cong \tilde{i}_J^*\wt{\upsilon}^p_{2,\phi^{-1}\circ\theta} \cong  \tilde{i}_J^*\wt{\omega}_{2,\theta}.$$
Therefore in both cases we obtain an isomorphism
$$ \tilde{i}_J^*(\wt{\omega}_{1,\theta}^{-1}\wt{\omega}_{2,\theta}) \cong \tilde{i}_J^*(\CL_{\theta}^{-1}\CL_{\phi^{-1}\circ\theta}^p).$$

We now show that the section corresponding to $g^*$ under this isomorphism is simply the one
induced on the Raynaud line bundles by $\Ver: H^{(p)} \to H$; i.e., the diagram
$$\xymatrix{ \tilde{i}_J^*\wt{\omega}_{1,\theta} \ar[rr]^-{\sim} \ar[d]_{g^*} && \tilde{i}_J^* \CL_\theta \ar[d]_{\Ver^*} \\
\tilde{i}_J^*\wt{\omega}_{2,\theta} &\cong& \tilde{i}_J^* \CL_{\phi^{-1}\circ\theta}^p}$$
commutes, where the top arrow is induced by the inclusion $H \subset A_1$ and bottom arrow is the isomorphism defined above.
It suffices to check the commutativity on fibres at all $\overline{s} \in S_J(\Fpbar)$, which we do by writing the resulting diagram
in terms of the Dieudonn\'e modules of $H_{\overline{s}}$, $A_{1,\overline{s}}[p]$ and $A_{2,\overline{s}}[p]$.    More precisely,
under the canonical descriptions of cotangent and tangent spaces in terms of Dieudonn\'e modules, the diagram on fibres becomes
$$\xymatrix{ \D(A^{(p^{-1})}_{1,\ol{s}}[p])_{\theta}/\Phi(\D(A^{(p^{-1})}_{1,\ol{s}}[p])_{\phi^{-1}\circ\theta} )
\ar[rrrr]^-{\sim}  \ar[d]_{g^*} &&&& \D(H^{(p^{-1})}_{\ol{s}})_{\theta} \ar[d]^{\Ver^*} \\
 \D(A^{(p^{-1})}_{2,\ol{s}}[p])_{\theta}/\Phi(\D(A^{(p^{-1})}_{2,\ol{s}}[p])_{\phi^{-1}\circ\theta})  & \cong & \Delta & \cong & \D(H_{\ol{s}})_{\theta}},$$
where the top arrow corresponds to the isomorphism\footnote{Note that $\Lie(H_{\ol{s}})_\theta$ and $\D(H^{(p^{-1})}_{\ol{s}})_{\theta}$ are
both one-dimensional, so $\Phi(\D(H^{(p^{-1})}_{\ol{s}})_{\phi^{-1}\circ\theta}) = 0$.}
$\Lie(H_{\ol{s}})_\theta \cong \Lie(A_{1,\ol{s}}[p])_\theta$ given by the inclusion $H_{\ol{s}} \subset A_{1,\ol{s}}[p]$
and the bottom isomorphisms are the ones induced
by those defined above, where
\begin{itemize}
\item $\Delta = \D(A_{1,\ol{s}}[p])_{\theta}/\Phi(\D(A_{1,\ol{s}}[p])_{\phi^{-1}\circ\theta})$ if $\phi^{-1}\circ\theta \not\in J$, and
\item $\Delta = \ker(\Phi) = \ker(V)$ on $\D(A_{2,\ol{s}}[p])_\theta$ if $\phi^{-1}\circ\theta \in J$.
\end{itemize}
In the first case, the desired compatibility is immediate from the fact that 
the isomorphism $\Delta \to \D(H_{\ol{s}})_{\theta}$ is again defined by the inclusion $H_{\ol{s}} \subset A_{1,\ol{s}}[p]$
and  the isomorphism $\D(A^{(p^{-1})}_{2,\ol{s}}[p])_{\theta}/\Phi(\D(A^{(p^{-1})}_{2,\ol{s}}[p])_{\phi^{-1}\circ\theta}) \cong  \Delta$ is 
the one arising from the commutative diagram
$$\xymatrix{ &  H^0(A_{1,\ol{s}},\Omega^1_{A_{1,\ol{s}}/\Fpbar})_\theta\ar[ld]_{g^*}\ar[rd]^{\Ver^*}              &\\
H^0(A_{2,\ol{s}},\Omega^1_{A_{2,\ol{s}}/\Fpbar})_\theta & \cong & H^0(A^{(p)}_{1,\ol{s}},\Omega^1_{A^{(p)}_{1,\ol{s}}/\Fpbar})_\theta}$$
and the canonical isomorphisms $\D(A[p])/\Phi(\D(A[p]))  \cong H^0(A^{(p)},\Omega^1_{A^{(p)}/\Fpbar})$ induced by $\Ver^*$ on
$\D(A[p]) \cong H^1_\dr(A/\Fpbar)$ for $A = A^{(p^{-1})}_{1,\ol{s}}$, $A^{(p^{-1})}_{2,\ol{s}}$ and $A_{1,\ol{s}}$.
In the second case, one finds that $\D(H_{\ol{s}})_{\theta}\stackrel{\sim}{\to} \Delta$ is induced by
$g:A_{2,\ol{s}}[p] \to H_{\ol{s}} \subset A_{1,\ol{s}}[p]$, whereas
$ \Delta \stackrel{\sim}{\to} \D(A^{(p^{-1})}_{2,\ol{s}}[p])_{\theta}/\Phi(\D(A^{(p^{-1})}_{2,\ol{s}}[p])_{\phi^{-1}\circ\theta})$
is the composite of
\begin{itemize}
\item the inverse of the canonical isomorphism 
$$H^1(A_{2,\ol{s}}^{(p)},\CO_{A_{2,\ol{s}}^{(p)}})_\theta \cong \Delta \subset \D(A_{2,\ol{s}}[p])_\theta \cong H^1_\dr(A_{2,\ol{s}}/\Fpbar)_\theta$$
induced by $\Frob^*:H^1_\dr(A_{2,\ol{s}}^{(p)}/\Fpbar)_\theta \to H^1_\dr(A_{2,\ol{s}}/\Fpbar)_\theta$,
\item  the isomorphism $H^1(A_{2,\ol{s}}^{(p)},\CO_{A_{2,\ol{s}}^{(p)}})_\theta \to H^0(A_{2,\ol{s}},\Omega^1_{A_{2,\ol{s}}/\Fpbar})_\theta$
induced by $\Frob^*$ (in view of the fact that this is indeed the image),
\item and the inverse of the canonical isomorphism 
$$\D(A^{(p^{-1})}_{2,\ol{s}}[p])_{\theta}/\Phi(\D(A^{(p^{-1})}_{2,\ol{s}}[p])_{\phi^{-1}\circ\theta})
 \cong H^0(A_{2,\ol{s}},\Omega^1_{A_{2,\ol{s}}/\Fpbar})_\theta$$
induced by $\Ver^*$ on $\D(A_{2,\ol{s}}^{(p^{-1})}[p]) \cong H^1_\dr(A_{2,\ol{s}}^{(p^{-1})}/\Fpbar)$.
\end{itemize}
It follows that this composite is simply the inverse of an isomorphism induced by $\Ver:A_{2,\ol{s}}[p] \to A^{(p^{-1})}_{2,\ol{s}}[p]$,
and the desired compatibility is an immediate consequence.

The isomorphisms
$$\Shom_{\CO_{S_J}}(\Lie(A_2^\vee/S_J)_\theta,\Lie(A_1^\vee/S_J)_\theta)
\cong \tilde{i}_J^*(\wt{\omega}_{1,\theta}^{-1}\wt{\omega}_{2,\theta}) \cong \tilde{i}_J^*(\CL_{\theta}^{-1}\CL_{\phi^{-1}\circ\theta}^p)$$
are compatible with the canonical descent data relative to the cover $S_J \to \overline{Y}_0(p)_J$, and hence descend to
isomorphisms of line bundles on $\overline{Y}_0(p)_J$.  Furthermore the section defined by the morphism $\Ver^*$
(or equivalently $\Lie(f^\vee)_\theta$ or $g^*$) descends to a section over $\overline{Y}_0(p)_J$, and taking the product over
$\theta\not\in J$ yields, by the very definition of $\CJ_J$, the desired isomorphism (\ref{eqn:Jfactor}).
Furthermore the isomorphism is compatible with the Hecke action in the usual sense:  if $U$ and $U'$ are sufficiently small
open compact subgroups of $G(\A_\f)$ of level prime to $p$ and $g \in G(\A_\f^{(p)})$ is such that $g^{-1}Ug \subset U'$,
then the resulting diagram
$$\xymatrix{ \rho_g^*\CJ_{J}'^{-1} \ar[r] \ar[d] & \rho_g^*i_{J}'^* (\CL_{\theta}'^{-1}\CL'^p_{\phi^{-1}\circ\theta})  \ar[d]^{\pi_g^*}\\
\CJ_{J}^{-1} \ar[r]  & i_{J}^* (\CL^{-1}_{\theta}\CL^p_{\phi^{-1}\circ\theta})}$$
with the obvious notation commutes.

\subsubsection{Jordan-H\"older factors and automorphic bundles}\label{sss:fun.JH}
Before completing our description of the bundles $\Xi_{J,*}(\CJ_J^{-1}\CK_{\ol{Y}_0(p)_J} i_J^*\CL_\chi^{-1})$, 
we recall the explicit formula of Bardoe and Sin~\cite{BS} (as presented in \cite{BP}) for the Jordan-H\"older factors of the right representation
$\Ind_P^{\GL_2(\CO_F/p\CO_F)}(1\otimes\chi)$, where $P$ is the subgroup of upper-triangular matrices and $1\otimes\chi$
denotes the character $\begin{psmallmatriks} a & b \\ 0 & d \end{psmallmatriks} \mapsto \chi(d)$. 
\begin{theorem} \label{thm:JHfactors}   There exists a decreasing filtration
$$V_\chi = \fil^0V_\chi \supset \fil^1V_\chi \supset \cdots \supset \fil^{d-1}V_\chi \supset \fil^dV_\chi \supset \fil^{d+1}V_\chi = 0$$
on $V_\chi = \Ind_P^{\GL_2(\CO_F/p\CO_F)}(1\otimes\chi)$ such that $\gr^j V_\chi \cong \oplus_{|J| = j} V_{\chi,J}$ with
$$V_{\chi,J} = \bigotimes_{\theta \in \Theta} (\theta\circ\det)^{\ell_{J,\theta}} \sym^{n_{J,\theta}} (V_\st \otimes_{\CO_F,\theta} \F),$$
where $V_\st$ denotes the standard representation of $\GL_2(\CO_F/p\CO_F)$ on $(\CO_F/p\CO_F)^2$ and
$$(\ell_{J,\theta},n_{J,\theta})  = \left\{\begin{array}{ll}  
(0,m_\theta),& \mbox{if $\theta\not\in J$, $\phi\circ\theta\not\in J$;}\\
(0,m_\theta-1),& \mbox{if $\theta\in J$, $\phi\circ\theta\not\in J$;}\\
(m_\theta+1,p-2-m_\theta),& \mbox{if $\theta\not\in J$, $\phi\circ\theta\in J$;}\\
(m_\theta,p-1-m_\theta),& \mbox{if $\theta\in J$, $\phi\circ\theta\in J$.}
\end{array}\right.$$
\end{theorem}
Note that the integers $n_{J,\theta}$ in the statement of the theorem satisfy $-1 \le n_{J,\theta} \le p-1$;
thus each non-zero $V_{\chi,J}$ is irreducible.  (We adopt the usual convention that $\sym^{-1}\F^2 = 0$.)
Theorem~\ref{thm:JHfactors} is immediate from Lemma~2.3 and Theorem~2.4 of \cite{BP}, which treats similarly
defined left representations of each $\GL_2(\CO_F/v)$.  To obtain Theorem~\ref{thm:JHfactors}, simply write
$\chi = \prod_{v|p} \chi_v$ where $\chi_v$ is a character of $(\CO_F/v)^\times$, so that
$V^\iota_\chi \cong \bigotimes_{v|p} \Ind_{P_v}^{\GL_2(\CO_F/v)}(\chi_v\otimes 1)$ as a left representation of 
$\GL_2(\CO_F/p\CO_F) = \prod_{v|p} \GL_2(\CO_F/v)$, where $P_v$ denotes the subgroup of upper-triangular
matrices of $\GL_2(\CO_F/v)$, and $\cdot^\iota$ is the equivalence of categories between right and left
representations of $\GL_2(\CO_F/p\CO_F)$ associated to its standard anti-involution.
Furthermore the filtration in the theorem is in fact the socle, as well as co-socle, filtration on $V_\chi$
provided $m_\theta < p-1$ for some $\theta$ in each $\Theta_v$.  (More generally, the socle and co-socle
filtrations of $V_\chi$ coincide and have length $1 + |\Theta_\chi|$ where $\Theta_\chi$ is the union of the $\Theta_v$
for which $\chi_v$ is non-trivial, and the constituents of the $j^\th$ graded piece of the co-socle filtration are the 
non-zero $V_{\chi,J}$ such that $|J \cap \Theta_\chi| = j$.)

We let $\CA_{\chi,J}$ denote the automorphic vector bundle
\begin{equation}\label{eqn:AchiJ}
\left(\displaystyle\bigotimes_{\theta\not\in \Sigma}\delta_\theta^{\ell_{J,\theta}}\omega_\theta^{n_{J,\theta}+2}\right)
\left(\displaystyle\bigotimes_{\theta\in \Sigma}\delta_\theta^{\ell_{J,\theta}+1}\sym_{\CO_{\ol{Y}_\Sigma}}^{n_{J,\theta}}\CV_\theta \right)
\end{equation}
on $\ol{Y}_\Sigma$, where $\ell_J$, $n_J \in \Z^\Theta$
are as in the statement of Theorem~\ref{thm:JHfactors}, the tensor products are over $\CO_{\ol{Y}_\Sigma}$, and
we suppress the central tensor product symbols here and below.   We view $H^0(\ol{Y}_\Sigma,\CA_{\chi,J})$ as the
space of automorphic forms for $G_\Sigma$ over $\F$ of weight $(n_J+2,\ell_J)$ and level $U_\Sigma$; indeed
for paritious $(n_J,\ell_J)$, the same definition with $\F$ replaced by $\C$ yields the usual space of automorphic
forms for $G_\Sigma$ of weight $(n_J+2,\ell_J)$ and level $U_\Sigma$.

Recall that if $U_{\Sigma}$ and $U'_{\Sigma}$ are sufficiently small
open compact subgroups of $G_\Sigma(\A_\f)$ of level prime to $p$, and $g_\Sigma \in G_\Sigma(\A^{(p)}_\f)$ is such that
$g_\Sigma^{-1} U_{\Sigma} g_\Sigma \subset U_{\Sigma}'$, then we have morphisms
$$\rho_{g_\Sigma}:\overline{Y}_{\Sigma} \to \overline{Y}'_{\Sigma}\quad\mbox{and}\quad
\pi_{g_\Sigma}^*: \rho_{g_\Sigma}^*\CA_{\chi,J}' \stackrel{\sim}{\to} \CA_{\chi,J}$$
satisfying the usual compatibilities, where $\CA'_{\chi,J}$ is the automorphic vector
bundle defined on $\overline{Y}_{\Sigma}' := Y_{U_{\Sigma}'}(G_\Sigma)_\F$ by (\ref{eqn:AchiJ}).
Letting $g_\Sigma$ act as the composite
$$\xymatrix{H^i(\ol{Y}_{\Sigma}',\CA'_{\chi,J}) \ar[r]^{\rho_{g_\Sigma}^*} &
H^i(\ol{Y}_{\Sigma},\rho_{g_\Sigma}^*\CA_{\chi,J}')\ar[rr]^{||\det(g_\Sigma)|| \pi_{g_\Sigma}^*}&&
H^i(\ol{Y}_{\Sigma},\CA_{\chi,J})}$$
we obtain a left action of $G_\Sigma(\A_\f^{(p)})$ on $\varinjlim  H^i(\ol{Y}_\Sigma,\CA_{\chi,J})$,
where the direct limit is over open compact $U_\Sigma \subset G_\Sigma(\A_\f)$ of level prime to $p$.

\subsubsection{Completion of the proof}\label{sss:fun.proof}
We now return to the task of computing the bundles $\Xi_{J,*}(\CJ_J^{-1}\CK_{\ol{Y}_0(p)_J} i_J^*\CL_\chi^{-1})$, with a view
to relating them to the bundles $\CA_{\chi,J}$.  First note that the isomorphisms
(\ref{eqn:Lfactor}) and (\ref{eqn:Jfactor}) immediately give
$$\begin{array}{c} \CJ_J^{-1}i_J^*\CL_{\chi^{-1}}^{-1}\cong
\left(\displaystyle\bigotimes_{\theta\in\Theta}i_J^*\CL_\theta^{m_\theta-p+1}\right)\left(\displaystyle\bigotimes_{\theta\not\in J}i_J^*\CL_\theta^{-1}\right)
\left(\displaystyle\bigotimes_{\phi\circ\theta\not\in J}i_J^*\CL_\theta^{p}\right)\\
=
\left(\displaystyle\bigotimes_{\theta\not\in J,\phi\circ\theta\not\in J}i_J^*\CL_\theta^{m_\theta}\right)
\left(\displaystyle\bigotimes_{\theta\not\in J,\phi\circ\theta\in J}i_J^*\CL_\theta^{m_\theta-p}\right)
\left(\displaystyle\bigotimes_{\theta\in J,\phi\circ\theta\in J}i_J^*\CL_\theta^{m_\theta+1}\right)
\left(\displaystyle\bigotimes_{\theta\in J,\phi\circ\theta\in J}i_J^*\CL_\theta^{m_\theta-p+1}\right).
\end{array}$$
Applying (\ref{eqn:bundling2}) then gives an isomorphism
$$\begin{array}{ll} \Xi_{J,*}(\CJ_J^{-1} i_J^*\CL_\chi^{-1}) \cong  &
\left(\displaystyle\bigotimes_{\theta\not\in J,\phi\circ\theta\not\in J}\omega_\theta^{m_\theta}\right)
\left(\displaystyle\bigotimes_{\theta\not\in J,\phi\circ\theta\in J}\delta_\theta^{m_\theta -p}(p-m_\theta)_\theta \right)\\
&\ \quad\cdot\left(\displaystyle\bigotimes_{\theta\in J,\phi\circ\theta\in J}\CO(m_\theta+1)_\theta\right)
\left(\displaystyle\bigotimes_{\theta\in J,\phi\circ\theta\in J}\delta_\theta^{m_\theta-p+1}\omega_\theta^{p-1-m_\theta}\right).\end{array}$$
Combining this with (\ref{eqn:Kfactor}) and the isomorphism $\delta_\theta \cong \delta_{\phi^{-1}\circ\theta}^p$ of \S\ref{sec:frob}
therefore yields an isomorphism
\begin{align}
\nonumber \Xi_{J,*}(\CJ_J^{-1} \CK_{\ol{Y}_0(p)_J} i_J^*\CL_\chi^{-1}) \cong  &
\left(\displaystyle\bigotimes_{\theta\not\in J,\phi\circ\theta\not\in J}\delta_\theta^{-1}\omega_\theta^{m_\theta+2}\right)
\left(\displaystyle\bigotimes_{\theta\not\in J,\phi\circ\theta\in J}\delta_\theta^{m_\theta -p+1}(p-2-m_\theta)_\theta \right)\\
\label{eqn:JKL} &
\ \quad\cdot\left(\displaystyle\bigotimes_{\theta\in J,\phi\circ\theta\in J}\delta_\theta(m_\theta-1)_\theta\right)
\left(\displaystyle\bigotimes_{\theta\in J,\phi\circ\theta\in J}\delta_\theta^{m_\theta-p}\omega_\theta^{p+1-m_\theta}\right)\\
\cong  &
\nonumber \left(\displaystyle\bigotimes_{\theta\not\in \Sigma}\delta_\theta^{\ell_{J,\theta}-1}\omega_\theta^{n_{J,\theta}+2}\right)
\left(\displaystyle\bigotimes_{\theta\in \Sigma}\delta_\theta^{\ell_{J,\theta}}(n_{J,\theta})_\theta \right)
\end{align}
of line bundles on $\prod_{\theta\in \Sigma}\P_{\overline{Y}_\Sigma}(\CV_\theta)$, where $\ell_J$, $n_J \in \Z^\Theta$
are as in the statement of Theorem~\ref{thm:JHfactors}.  Furthermore the isomorphism is
Hecke-equivariant in the sense that the diagram analogous to (\ref{eqn:JLHbundles}) commutes.
More precisely, suppose that $U$ and $U'$ are sufficiently small open compact subgroups of $G(\A_\f)$ of level prime to $p$
and $g \in G(\A_\f^{(p)})$ is such that $g^{-1}U g \subset U'$.  Let
$$\CF = \CJ_J^{-1} \CK_{\ol{Y}_0(p)_J} i_J^*\CL_\chi^{-1},\quad
\CG = \left(\displaystyle\bigotimes_{\theta\not\in \Sigma}\delta_\theta^{\ell_{J,\theta}-1}\omega_\theta^{n_{J,\theta}+2}\right)
\left(\displaystyle\bigotimes_{\theta\in \Sigma}\delta_\theta^{\ell_{J,\theta}}(n_{J,\theta})_\theta \right)$$
and similarly define $\CF'$ and $\CG'$ with $U$ replaced by $U'$.  The isomorphisms
$\sigma:\Xi_{J,*}\CF \stackrel{\sim}{\to} \CG$ and $\sigma':\Xi'_{J,*}\CF' \stackrel{\sim}{\to} \CG'$ of (\ref{eqn:JKL}) are then
compatible with the isomorphisms $\pi_g^*: \rho_g^*\CF' \stackrel{\sim}{\to} \CF$
and $\pi_{g_\Sigma}^*: \rho_{g_\Sigma}^*\CG' \stackrel{\sim}{\to} \CG$, meaning that the diagram
$$\xymatrix{
\Xi_{J,*}\rho_g^*\CF' \ar[d]_{\Xi_{J,*}(\pi_g^*)} &
*+[l]{= \rho_{g_\Sigma}^*\Xi'_{J,*}\CF' } \ar[r]^-{\rho_{g_\Sigma}^*(\sigma')} & 
\rho_{g_\Sigma}^*\CG'  \ar[d]^{  \pi_{g_\Sigma}^* } \\
\Xi_{J,*}\CF  \ar[rr]_-{\sigma}  && 
\CG}$$
commutes.

Recall that for a rank two vector bundle $\CV$ on a scheme $S$, if $\psi:X = \P_S(\CV) \to S$ denotes the natural projection and $n \ge -1$,
then $R^i\psi_*\CO(n) = 0$ for all $i \ge 1$ and $\psi_*\CO(n)$ is canonically isomorphic to $\sym^n_{\CO_S}\CV$
(where $\sym^{-1}_{\CO_S}(V) = 0$ as usual).  Since $n_{J,\theta} \ge -1$ for all $\theta$, it follows from (\ref{eqn:JKL}) that
$$R^i(\Psi_{J,*}\Xi_{J,*})(\CJ_J^{-1} \CK_{\ol{Y}_0(p)_J} i_J^*\CL_\chi^{-1}) 
   = R^i\Psi_{J,*}(\Xi_{J,*}(\CJ_J^{-1} \CK_{\ol{Y}_0(p)_J} i_J^*\CL_\chi^{-1})) = 0$$
for all $i \ge 1$, and that
$$\Psi_{J,*}\Xi_{J,*}(\CJ_J^{-1} \CK_{\ol{Y}_0(p)_J} i_J^*\CL_\chi^{-1})  \cong 
\left(\displaystyle\bigotimes_{\theta\not\in \Sigma}\delta_\theta^{\ell_{J,\theta}-1}\omega_\theta^{n_{J,\theta}+2}\right)
\left(\displaystyle\bigotimes_{\theta\in \Sigma}\delta_\theta^{\ell_{J,\theta}}\sym^{n_{J,\theta}}_{\CO_{\ol{Y}_\Sigma}}(\CV_\theta)\right).$$
Note that the resulting vector bundle on $\ol{Y}_\Sigma$ is $\left(\bigotimes_{\theta\in \Theta} \delta^{-1}_\theta\right)\CA_{\chi,J}$.
Multiplying by the trivialization of $\bigotimes_{\theta\in \Theta} \delta_\theta$ defined in \S\ref{sss:buns.delta}, we obtain an
isomorphism
$$\varsigma: \Psi_{J,*}\Xi_{J,*}(\CJ_J^{-1} \CK_{\ol{Y}_0(p)_J} i_J^*\CL_\chi^{-1})  \stackrel{\sim}{\longrightarrow} \CA_{\chi,J}$$
whose compatibility with the Hecke action is expressed by the commutativity of the resulting diagram
$$\xymatrix{
\Psi_{J,*}\Xi_{J,*}\rho_g^*\CF' \ar[d]_{\Psi_{J,*}\Xi_{J,*}(\pi_g^*)} &
*+[l]{= \rho_{g_\Sigma}^*\Psi'_{J,*}\Xi'_{J,*}\CF' } \ar[r]^-{\rho_{g_\Sigma}^*(\varsigma')} & 
\rho_{g_\Sigma}^*\CA_{\chi,J}'  \ar[d]^{ ||\det g_\Sigma|| \pi_{g_\Sigma}^* } \\
\Psi_{J,*}\Xi_{J,*}\CF  \ar[rr]_-{\varsigma}  && 
\CA_{\chi,J}.}$$
In view of the vanishing of $R^i(\Psi_{J,*}\Xi_{J,*})(\CJ_J^{-1} \CK_{\ol{Y}_0(p)_J} i_J^*\CL_\chi^{-1})$ for all $i>0$,
we conclude that $\varsigma$ induces a Hecke-equivariant isomorphism
$$H^i(\ol{Y}_0(p)_J, \,\CJ_J^{-1} \CK_{\ol{Y}_0(p)_J} i_J^*\CL_\chi^{-1})  \stackrel{\sim}{\longrightarrow}  
    H^i(\ol{Y}_\Sigma, \CA_{\chi,J})$$
for all $i \ge 0$.  Combining this with (\ref{eqn:SS1}) completes the proof of Theorem~\ref{thm:JLss}.

\subsubsection{Hecke equivariance}\label{sss:fun.Hecke}
For clarity and convenience, we review the sense in which the resulting spectral sequence
\begin{equation}\label{eqn:SS2} E_1^{j,i} = \bigoplus_{|J| = j} H^{i+j}(\Ybar_{\Sigma} , \CA_{\chi,J})
         \Longrightarrow H^{i+j}(\Ybar_1(p), \CK_{\ol{Y}_1(p)})^\chi\end{equation}
is Hecke-equivariant.  First recall that for each $\Sigma = \Sigma_J$ appearing in the direct sum
we have fixed an isomorphism
$$G(\A_\f)= \GL_2(\A_{F,\f}) \cong G_\Sigma(\A_\F),$$
and for open compact subgroups $U$ (resp.~elements $g$) of $G(\A_\f)$, we write
$U_\Sigma$ (resp.~$g_\Sigma$) for its image in $G_\Sigma(\A_\f)$.
Thus if $U$ and $U'$ are sufficiently small open compact subgroups of $G(\A_\f)$ of level prime to $p$
and $g \in G(\A_\f^{(p)})$ is such that $g^{-1}Ug \subset U'$, we have $\F$-linear maps we denote
$$\begin{array}{lclll}
&[g] = [g]_{U,U'} &:H^{i+j}(\Ybar'_1(p), \CK_{\ol{Y}'_1(p)})^\chi &\to &H^{i+j}(\Ybar_1(p), \CK_{\ol{Y}_1(p)})^\chi \\
\mbox{and} & [g_\Sigma] = [g_\Sigma]_{U_\Sigma,U'_\Sigma} &:  
H^{i+j}(\Ybar'_\Sigma,\CA'_{\chi,J})&\to& H^{i+j}(\Ybar_\Sigma,\CA_{\chi,J}),\end{array}$$
where as usual $\cdot'$ denotes the object defined with $U$ replaced by $U'$, and
if $g'$ and $U''$ are as above with $g'^{-1}U'g' \subset U''$, then 
$[g]\circ [g']= [gg']$ and $[g_\Sigma]\circ [g'_\Sigma]= [g_\Sigma g'_\Sigma]$.

The existence of the spectral sequence in (\ref{eqn:SS2}) means that
\begin{itemize}
\item there are $\F$-linear differentials $d_r^{j,i}: E_r^{j,i}  \to E_r^{j+r,i-r+1}$ 
for all $r \ge 1$, $j \ge 0$, $i \ge -j$, where $E_{r+1}^{j,i}$ is defined inductively for $r \ge 1$ by
$$E_{r+1}^{j,i} = \ker(d_r^{j,i})/\im(d_r^{j-r,i+r-1});$$
\item there is a decreasing filtration of length $d+1$ on
 $H^{i+j}(\Ybar_1(p), \CK_{\Ybar_1(p)})^\chi$ and $\F$-linear isomorphisms
 $$\alpha^{j,i} : E_{\infty}^{j,i} = E_{d+1}^{j,i} \stackrel{\sim}{\longrightarrow}\gr^j\left(H^{i+j}(\Ybar_1(p), \CK_{\Ybar_1(p)})^\chi\right).$$
\end{itemize}

The Hecke-equivariance of the spectral sequence means that for $g$, $U$, $U'$ as above
(again using $\cdot'$ to denote the corresponding objects with $U$ replaced by $U'$), we have
$$[g]\left(\fil^j\left(H^{i+j}(\Ybar'_1(p), \CK_{\ol{Y}'_1(p)})^\chi)\right)\right) \subset \fil^j\left(H^{i+j}(\Ybar_1(p), \CK_{\ol{Y}_1(p)})^\chi)\right),$$
and for all $r,j,i$ as above there are $\F$-linear $g_r^{j,i}: E_r'^{j,i} \to E_r^{j,i}$ such that
\begin{itemize}
\item $g_1^{j,i} = ([g_{\Sigma_J}])_{|J|=j} : \displaystyle\bigoplus_{|J|=j} H^{i+j}(\Ybar'_\Sigma,\CA'_{\chi,J}) \to  \bigoplus_{|J|=j} H^{i+j}(\Ybar_\Sigma,\CA_{\chi,J})$,
\item the diagram
$$\xymatrix{  E_r'^{j,i} \ar[d]_{g_r^{j,i}} \ar[r]^-{d_r'^{j,i}} & E_r'^{j+r,i-r+1}\ar[d]^{g_r^{j+r,i-r+1}} \\
E_r^{j,i} \ar[r]^-{d_r^{j,i}}  & E_r^{j+r,i-r+1} }$$
commutes,
\item $g_{r+1}^{j,i}$ is induced by $g_r^{j,i}$,
\item and the diagram
$$\xymatrix{  E_\infty'^{j,i} =E_{d+1}'^{j,i} \ar@<3ex>[d]_{g_{d+1}^{j,i}} \ar[r]^-{\alpha'^{j,i}} & \gr^j\left(H^{i+j}(\Ybar'_1(p), \CK_{\Ybar'_1(p)})^\chi\right) \ar[d]^{\gr^j([g])}\\
E_\infty^{j,i} = E_{d+1}^{j,i} \ar[r]^-{\alpha^{j,i}}  &\gr^j\left(H^{i+j}(\Ybar_1(p), \CK_{\Ybar_1(p)})^\chi\right) }$$
commutes.
\end{itemize}

We remark that we have not given an intrinsic description of the differentials $d_1^{j,i}$
in terms of the spaces $H^{i+j}(\ol{Y}_\Sigma,  \CA_{\chi,j})$;
rather they are instead defined via the isomorphisms induced by $\varsigma$ for varying $J$.

Taking the direct limit over sufficiently small $U$ of level prime to $p$
of the spectral sequences in Theorem~\ref{thm:JLss} and defining the
action of $\GL_2(\A_{F,\f}^{(p)})$ on each $\varinjlim H^{i+j}(\ol{Y}_\Sigma,  \CA_{\chi,j})$ via the isomorphism
with $G_\Sigma(\A_\f)$ gives the following:
\begin{corollary} \label{cor:JLss}
There is a spectral sequence of smooth $\F$-representations of $\GL_2(\A_{F,\f}^{(p)})$
$$ E_1^{j,i} = \bigoplus_{|J| = j} \varinjlim \left(H^{i+j}(\Ybar_{\Sigma_J} , \CA_{\chi,J})\right)
         \Longrightarrow \varinjlim\left( H^{i+j}(\Ybar_1(p), \CK_{\Ybar_1(p)})^\chi\right).$$
\end{corollary}

\subsubsection{The Serre filtration}\label{sss:fun.Serre}
Specializing the preceding corollary to the case of $i+j=0$ immediately gives:
\begin{corollary} \label{cor:JLss2}
There is a filtration on $\varinjlim H^0(\ol{Y}_1(p),  \CK_{\Ybar_1(p)})^\chi$ of length $d+1$ by 
$\GL_2(\A_{F,\f}^{(p)})$-subrepresentations such that for each $j = 0,1,\ldots,d+1$, there is a 
$\GL_2(\A_{F,\f}^{(p)})$-equivariant injection
$$\gr^j\left(\varinjlim H^0(\ol{Y}_1(p),  \CK_{\Ybar_1(p)})\right)^\chi \longrightarrow
\bigoplus_{|J| = j} \varinjlim \left(H^0(\Ybar_{\Sigma_J} , \CA_{\chi,J})\right).$$
\end{corollary}

We immediately obtain Corollary~\ref{cor:JLfil} by specializing Theorem~\ref{thm:JLss} (or alternatively by taking 
$U^p$-invariants in the preceding corollary).  With the notation from the discussion above, the Hecke-equivariance
in the statement of Corollary~\ref{cor:JLfil} means that 
$$[g]\fil^j\left(H^0(\ol{Y}'_1(p),  \CK_{\Ybar'_1(p)})^\chi\right)\subset \fil^j\left(H^0(\ol{Y}_1(p),  \CK_{\Ybar_1(p)})^\chi\right)$$
and that the diagram
$$\xymatrix{   \gr^j\left(H^0(\Ybar'_1(p), \CK_{\Ybar'_1(p)})^\chi\right) \ar[d]^{\gr^j([g])} \ar[r] &
\displaystyle\bigoplus_{|J| = j}H^0(\Ybar'_{\Sigma_J} , \CA'_{\chi,J})
 \ar[d]^{([g_{\Sigma_J}])_{|J|=j}} \\
\gr^j\left(H^0(\Ybar_1(p), \CK_{\Ybar_1(p)})^\chi\right)  \ar[r]
 &
\displaystyle\bigoplus_{|J| = j} H^0(\Ybar_{\Sigma_J} , \CA_{\chi,J})  }$$
commutes.

\begin{remark}  The representations in Corollary~\ref{cor:JLss2} are furthermore admissible, provided $F \neq \Q$.
For $F = \Q$, working instead with the compactified modular curves yields admissible representations, and taking 
$U_1(N)$-invariants recovers the exact sequence described in the Introduction.  We note also that in the case
$F = \Q$, the spectral sequences of Theorems~\ref{thm:JLss} and Corollary~\ref{cor:JLss} (and slightly less
obviously their analogues for the compact curves) degenerate at $E_1$.
\end{remark}

\begin{remark} The obstruction to the injection in Corollary~\ref{cor:JLfil} being an isomorphism is measured
by the image of the differentials $d_r^{j,-j}$, and these take values in subquotients of direct sums of spaces
of the form $H^1(\Ybar_{\Sigma_J}, \CA_{\chi,J})$ for $|J| = j + r > j$.  We note that these spaces need not
vanish.  For example, suppose $d=3$ and $p\CO_F = v_1v_2$ with $|\CO_F/v_i| = p^i$, and write
$\ol{\Theta}_p = \Theta_{v_1} \coprod \Theta_{v_2}$ where $\Theta_{v_1} = \{\theta_1\}$ and
$\Theta_{v_2} = \{\theta_2,\theta_3\}$.  Note that for $J = \{\theta_1,\theta_2\}$, we have 
$\Sigma = \Theta_{v_2}$, so $\ol{Y}_{\Sigma}$ is a curve.  If $\chi = \theta_2\theta_3^{p-2}$,
then we have
$$m_{\theta_1} = p-1,\quad m_{\theta_2} = 1,\quad m_{\theta_3} = p-2,$$
so the formula for $(\ell_J,n_J)$ and the Kodaira--Spencer isomorphism give
$$\CA_{\chi,J} = \omega_{\theta_1}^2\delta_{\theta_2}\delta_{\theta_3}^{p}
         \cong \delta_{\theta_1}\delta_{\theta_2}^2\CK_{\ol{Y}_\Sigma}.$$
 For sufficiently small $U$, the line bundles $\delta_{\theta_i}$ on $\ol{Y}$ are
 (non-canonically)  trivializable (see \cite[\S4.5]{DS}), and arguing as we did in \S\ref{sss:buns.delta}
 for $\otimes_{\theta\in\Theta}\delta_\theta$, we see that the $\delta_{\theta_i}$ on $\ol{Y}_\Sigma$ are
 (non-canonically) trivializable, from which it follows that $H^1(\Ybar_{\Sigma_J}, \CA_{\chi,J}) \neq 0$.  Note that
 this does not preclude the vanishing of the differentials $d_r^{j,-j}$.  We remark also that even
 if the differentials do not vanish, so the injections are not isomorphisms, it may be the case that their
 cokernels are in some sense ``Eisenstein.''
 \end{remark}

\section{Degeneracy fibres}   \label{sec:fibre}

\subsection{Frobenius factorization} \label{sec:frobfactor}

In the next two sections we study the fibres of the natural degeneracy map $\overline{\psi}:\overline{Y}_0(p) \to \overline{Y}$ over $\F = \Fpbar$,
where $\overline{Y}_0(p) = Y_{U_0(p)}(G)_\F$, $\overline{Y} = Y_U(G)_\F$ and $\overline{\psi}$ is induced
by the forgetful morphism $\widetilde{Y}_{U_0(p)}(G) \to \widetilde{Y}_U(G)$.  In particular, we will improve on results  of
\cite[\S2.6]{GK} and \cite[\S4.9]{ERX}\footnote{There is however a serious gap in the argument in \cite[\S4.9]{ERX}, specifically 
in the claim about rearranging choices of local parameters in the paragraph after (4.9.3).}
(see also \cite[\S5]{HU2} in the unitary setting) with a view to proving Theorem~\ref{thm:H1}. 

\subsubsection{Preliminaries}\label{sss:frob.prel}
First recall that $\overline{Y}$ is equipped with a stratification defined by the vanishing of the partial Hasse invariants
$h_\theta \in H^0(\overline{Y}, \overline{\omega}_\theta^{-1}\overline{\omega}_{\phi^{-1}\circ\theta}^p)$.  Recall (see for example~\cite[\S5.1]{DS})
that these are defined by descent from sections $\widetilde{h}_\theta$ on $T := \widetilde{Y}_U(G)_\F$ induced by the Verschiebung morphism
$A^{(p)} \to A$, where $A$ is the universal abelian scheme over $T$, and that for each $J \subset \Theta$, the closed subscheme
$T_J$ of $T$ (resp.~$\overline{Y}_J$ of $\overline{Y}$) defined by the vanishing of the $\widetilde{h}_\theta$ (resp.~$h_\theta$)
for $\theta\in J$ is smooth of dimension $d - |J|$.  By \cite[Thm.~2.6.4]{GK}, the restriction of the forgetful morphism in 
$S = \widetilde{Y}_{U_0(p)}(G)_\F \to T$ to $S_J$ factors through $T_{J'}$ where $J' := \{\,\theta \in J \,|\,\phi^{-1}\circ \theta \not\in J\,\}$.
Furthermore letting $J'' := \{\,\theta \not\in J \,|\,\phi^{-1}\circ \theta \in J\,\}$, the morphisms
$$\CH^1_\dr(A_2/S)_\theta \to \CH^1_\dr(A_1/S)_\theta$$
induced by the universal isogeny $f:A_1 \to A_2$ on $S_J$ 
give rise to a morphism
\begin{equation}\label{eqn:xiJ} \widetilde{\xi}_J:  S_J \to  P_J:= \prod_{\theta \in J''}  \P_{T_{J'}}(\CH^1_\dr(A/T_{J'})_\theta)\end{equation}
which is bijective on geometric closed points (the injectivity is clear and the surjectivity follows from the calculation in
\cite[Lemma~2.6.6]{GK}).  Since $S \to T$ is projective, so is $\widetilde{\xi}_J$, and since
$S_J$ and $P_J$ are locally of finite type of the same dimension over $\F$, it follows that $\widetilde{\xi}_J$
is finite.  Since the two schemes are smooth over $\F$, it follows by ``miracle flatness'' that $\widetilde{\xi}_J$ is flat,
and hence $S_J \to T_{J'}$ is Cohen--Macaulay.  

Recall that $\CH^1_\dr(A/T_{J'})_\theta$ is the (restriction to $T_{J'}$) of the vector bundle
$\widetilde{\CV}_\theta$ defined in \S\ref{sss:avb.hmf}, and that these descend
to the automorphic bundles on $\overline{Y}_{J'}$ denoted $\CV_\theta$.
Taking quotients in (\ref{eqn:xiJ}) by the action of $\CO_{F,(p),+}^\times$,
we see that the restriction of $\overline{\psi}$ to $\overline{Y}_0(p)_J$ similarly factors
as a composite
$$ \overline{Y}_0(p)_J \stackrel{{\xi}_J}{\longrightarrow} \prod_{\theta \in J''}  \P_{T_{J'}}(\CV_\theta)
    \to \overline{Y}_{J'}  \to \overline{Y}$$
where $\xi_J$ is finite flat (and bijective on geometric closed points) and the next two maps are the canonical
projection and closed immersion, so that $\overline{Y}_0(p)_J \to \overline{Y}_{J'}$ is Cohen--Macaulay.

\subsubsection{Relative Dieudonn\'e theory}\label{sss:frob.dieu}
The study of the fibres of $\overline{Y}_0(p)_J \to \overline{Y}$ at closed points reduces to understanding the
fibres of $S_J \to T_{J'}$, whose sets of closed points are seen from the above discussion to be in bijection with
products of projective lines; however we must deal with the fact that these fibres are not in general reduced.
In order to proceed, we will first analyze the morphism $\widetilde{\xi}: S_J \to P_J$ more carefully, for which
we will make use of Dieudonn\'e theory for finite flat group schemes over more general bases than perfect fields,
as developed in \cite{BBM}.  More precisely, if $H$ is a finite flat group scheme over a scheme $S$ 
of characteristic $p$, then \cite[D\'ef.~3.1.5]{BBM}  associates to $H$ a Dieudonn\'e crystal $\D(H)$ over $S$,
i.e., a crystal of $\CO_{S/\Z_p}$-modules equipped with morphisms\footnote{Recall we use $\Phi$ instead of $F$
since $F$ denotes our fixed totally real field; here however we use the linearized versions of $\Phi$ and $V$.} 
$\Phi:\D(H)^{(p)} \to \D(H)$ and
$V:\D(H) \to \D(H)^{(p)}$ such that $\Phi\circ V$ and $V\circ \Phi$ are multiplication by $p$.
The functor $\D$ is contravariant in $H$ and the morphisms $\Phi$ and $V$ are induced by 
$\Frob_H:H \to H^{(p)}$ and $\Ver_H:H^{(p)} \to H$; furthermore $\D$ is compatible
with base change $S \to S'$ in the obvious sense, and reduces to the usual Dieudonn\'e theory when
$S = \Spec k$ for a perfect field $k$ (see\cite[Thm.~4.2.14]{BBM}; note that we have already used
$\D(H)$ to denote $\varprojlim_n \Gamma(\Spec(W_n),\D(H))$ when $k = \Fpbar$, and this coincides
with $\Gamma(S,\D(H))$ when $H$ is killed by $p$).  
We shall in fact only consider group schemes $H$ killed by $p$, on which we may view $\D$ as an
exact functor to the category of crystals of locally free $\CO_{S/\F_p}$-modules equipped with $\Phi$ and
$V$ such that $\Phi \circ V$ and $V \circ \Phi$ are trivial (Prop.~4.3.1 and Lemma 4.3.5 of \cite{BBM}).

Recall also from \cite[4.3.4]{BBM} that if $\CF$ is a quasi-coherent sheaf of $\CO_S$-modules, then pull-back via Frobenius
endomorphisms on divided power thickenings over $\F_p$ yields a crystal of $\CO_{S/\F_p}$-modules, which we denote
$\Pi^*\CF$ (rather than $\Phi^*\CF$).  The proof of \cite[Prop.~4.3.6]{BBM} yields a natural morphism
$$\Pi^*\Lie(H^\vee) \to \D(H)$$
which is an isomorphism if $\Ver_H = 0$.  On the other hand \cite[Prop.~4.3.10]{BBM} yields a natural morphism
$$\D(H) \to \Pi^*\omega_H$$
where $\omega_H$ denotes the sheaf of invariant differentials, and this is an isomorphism if $\Frob_H = 0$.
Furthermore recall that $\omega_H$ is canonically isomorphic to $e^*\Omega^1_{H/S}$ where
$e:S \to H$ is the zero section, and we have a natural map $\Lie(H^\vee) \to \omega_H$,
defined for example by identifying $\Lie(H^\vee)$ with $\Shom_S(H,{\bf G}_a)$ and sending
a section $\eta$ to the invariant differential $\eta^*(dX)$.  We remark also that in the case $S=\Spec k$
for a perfect field $k$, the two propositions in \cite{BBM} yield the isomorphisms
$$\Lie(H^\vee)^{(p)} = \ker(V)_S\qquad\mbox{and}\qquad \omega_H^{(p)} = \coker (\Phi)_S$$
which appeared in our previous discussion in the case $k = \Fpbar$.

\begin{lemma}  \label{lem:BBMcomp} Suppose that $k$ is a field and $H$ is a finite flat group scheme 
over $S = \Spec k$ killed by $p$, then the composite 
$$\Pi^*\Lie(H^\vee) \to \D(H) \to \Pi^*\omega_H$$
is $\Pi^*\delta_H$, where $\delta_H$ is the natural map $\Lie(H^\vee) \to \omega_H$.
\end{lemma}
\begpf  First note that the desired compatibility amounts to the claim that 
the composite $\Lie(H^\vee)^{(p)} \to \D(H)_S \to \omega_H^{(p)}$
is $\delta_H^{(p)}$, and we may extend scalars so as to assume $k$ is perfect.  
Letting $H_0 = \ker(\Frob_H)$, the commutative
diagram
$$\xymatrix{\Lie(H^\vee)^{(p)} \ar@{^{(}->}[r]\ar[d] & \D(H)_S \ar@{->>}[r]\ar@{->>}[d] & \omega_H^{(p)}  \ar[d]^-{\wreath} \\
\Lie(H_0^\vee)^{(p)} \ar@{^{(}->}[r]  & \D(H_0)_S \ar[r]^-{\sim} & \omega_{H_0}^{(p)}  }$$
shows that the compatibility for $H_0$ implies it for $H$, so we may assume $\Frob_H = 0$.
Similarly letting $H_1= \coker(\Ver_H)$, the diagram 
$$\xymatrix{\Lie(H_1^\vee)^{(p)}  \ar[r]^-{\sim} \ar[d]^-{\wreath} & \D(H_1)_S \ar[r]^-{\sim} \ar@{_{(}->}[d]& \omega_{H_1}^{(p)}   \ar@{_{(}->}[d] \\
\Lie(H^\vee)^{(p)} \ar@{^{(}->}[r]  & \D(H)_S \ar[r]^-{\sim} & \omega_{H}^{(p)}  }$$
shows we may also assume $\Ver_H = 0$.  Since $\Frob_H$ and $\Ver_H$
are both trivial, we have $H \cong \alpha_p^m$ for some $m \ge 0$, so we may further assume
$H = \alpha_p$.  Finally in the case $H = \alpha_p = \Spec k[Y]/(Y^p)$ with $\mu(Y) = Y \otimes 1 + 1 \otimes Y$,
we can compute the isomorphisms
$\Lie(H^\vee)^{(p)} \cong \D(H)_S$ and $\D(H)_S \cong \omega_H^{(p)}$ of Prop.~4.3.6 and
Prop.~4.3.10 of \cite{BBM} using their descriptions in \cite[4.3.7, 4.3.12]{BBM}.  More precisely,
take the basis element of $\D(H)_S \cong \Ext_S^\natural(H,{\bf G}_a)$ defined by the
extension $E$ associated in \cite[4.3.12(i)]{BBM} to the pair
$(f,\omega)$ where 
$$f = - \sum_{i=1}^{p-1} \left(\begin{array}{c} p \\ i \end{array}\right) Y^i \otimes Y^{p-i} \quad\mbox{and}\quad
\omega = Y^{p-1}dY.$$
 According to \cite[4.3.7]{BBM}, the image of the class in $\Lie(H^\vee)^{(p)}$
is defined by the homomorphism $H^{(p)} \to {\mathbf G}_a = \Spec k[X]$ through which $\Ver_E:E^{(p)} \to E$
factors, which one can check directly from the definition in \cite[Exp.~VII, 4.3]{SGA3} is given by $X \mapsto Y$.
On the other hand according to \cite[4.3.12(ii)]{BBM}, its image in $\omega_{H^{(p)}}$ is 
given by $C(\omega) = dY$.
\epf

\begin{remark}  The preceding lemma would also follow easily from knowing that the relations between
(co)tangent spaces and Dieudonn\'e modules provided by Propositions~4.3.6 and~4.3.10 of \cite{BBM}
were compatible (via their Theorem~4.2.14) with the ones given by Propositions~III.3.2 and~III.4.3 of \cite{fontaine}.
However we were not able to verify the commutativity of the resulting diagram.
\end{remark}

We now determine the Dieudonn\'e crystal of a Raynaud $(\CO_F/p)$-module scheme $H$ over a suitable
base $S$ in characteristic $p$.  Let $q$ be a power of $p$ such that $q-1$ is divisible by the exponent of 
$(\CO_F/p)^\times$.
\begin{proposition} \label{prop:DRaynaud}
Suppose that $S$ is a smooth scheme over a field containing $\F_q$ and $H$ is the Raynaud $(\CO_F/p)$-module scheme over
$S$ associated to the data $(\CL_\theta,s_\theta,t_\theta)_{\theta\in \Theta}$ (see~\S\ref{sss:U1p}).  Then $\D(H)$ is canonically isomorphic to
$\Pi^*\CL$ with $\Phi = \Pi^*(s)$ and $V = \Pi^*(t)$, where $\CL = \bigoplus_{\theta \in \Theta} \CL_\theta$,
$s = \bigoplus_{\theta \in \Theta} s_\theta$ and $t = \bigoplus_{\theta\in \Theta} t_\theta$.
\end{proposition}
\begpf First note that the action of $\CO_F$ on $H$ induces an action of $\CO_F \otimes \F_q$ on
the crystal, and hence a decomposition $\D(H) = \bigoplus_{\theta \in \Theta} \D(H)_\theta$
under which $\Phi$ and $V$ restrict to morphisms
$$\Phi:\D(H)_\theta^{(p)} \to \D(H)_{\phi\circ\theta} \qquad \mbox{and}\qquad  V:\D(H)_{\phi\circ \theta} \to \D(H)_\theta^{(p)}.$$
To prove the proposition, we may assume $S$ is connected and hence integral, so the equation $s_\theta t_\theta = 0$
implies that either $s_\theta = 0$ or $t_\theta = 0$, and we may therefore choose a set $I$ such
that $s_\theta = 0$ (resp.~$t_\theta = 0$) for all $\theta \not\in I$ (resp.~$\theta\in I$).
Applying Lemma~\ref{lem:slicing} to $H$ and $H^\vee$ now yields an exact sequence
of finite flat $(\CO_F/p)$-module schemes
$$0 \to C \to H \to C' \to 0$$
such that $\Ver_C = 0$, $\Frob_{C'} = 0$, and the natural maps
$\Lie(H^\vee) \to \CL \to \omega_H$ induce isomorphisms
$$\Lie(C^\vee) \cong \bigoplus_{\theta \in I} \CL_\theta\quad\mbox{and}\quad \bigoplus_{\theta\not\in I} \CL_\theta \cong \omega_{C'} .$$
The $(\CO_F/p)$-equivariant exact sequence 
$$0 \to \D(C') \to \D(H) \to \D(C) \to 0$$
and isomorphisms $\D(C) \cong \Pi^*\Lie(C^\vee)$, $\D(C') \cong \Pi^*\omega_{C'}$ thus yield
an $(\CO_F/p)$-equivariant isomorphism
$$\alpha:  \D(H) \stackrel{\sim}{\longrightarrow}   \bigoplus_{\theta \in I} \Pi^*\CL_\theta$$
of crystals of locally free $\CO_{S/\F_p}$-modules, which we write as $\bigoplus \alpha_\theta$.

We must prove that $\alpha$ is compatible with $\Phi$ and $V$.  
Since $S$ is smooth and the constructions are compatible with base-change, we
can apply \cite[Cor.~1.3.4]{BM3} to replace $S$ by its generic point and hence
assume $S = \Spec k$ for a field $k$.  We just give the argument for $\Phi$; the proof
for $V$ is similar\footnote{As will be clear from the argument, the compatibility of $\Phi$
(resp.~$V$) on $\D(H)_{\theta}^{(p)}$ (resp.~$\D(H)_{\phi\circ\theta}$) is straightforward if
$\phi^{-1}\circ\theta\not\in I$ or $\theta\in I$ (resp.~$\phi^{-1}\circ\theta \in I$ or $\theta\not\in I$); the appeal to
faithfulness under base-change is only needed to address the remaining case.}.

First note that since $\Frob_{C'} = 0$, it follows that
$\Phi$ annihilates $\D(C')^{(p)}$ and $\Frob_H$ factors through $C^{(p)}$.
In particular if $\phi^{-1}\circ\theta \not\in I$, then $\Phi$ is trivial on 
$$\D(H)_\theta^{(p)} \cong \D(C')_\theta^{(p)} \cong \Pi^*\CL^p_{\phi^{-1}\circ\theta},$$
as is $\Pi^*s_{\phi^{-1}\circ\theta}$.  On the other hand if $\phi^{-1}\circ\theta \in I$,
then $t_{\phi^{-1}\circ\theta} =0$, so that $\omega_{H^\vee,\theta} = \CL_\theta^{-1}$
and $\Lie(H^\vee)_\theta = \CL_\theta$.  Furthermore the right-most square of the diagram
\begin{equation} \label{eqn:crystalPhi}
\xymatrix{\D(H)_\theta^{(p)} \ar[rd]^-{\Phi}  \ar[r]^-{\sim} & \D(C)_\theta^{(p)}\ar[d] & \Pi^*\Lie(C^\vee)_{\phi^{-1}\circ\theta}^{(p)} \ar[d]
 \ar[r]^-{\sim}  \ar[l]_-{\sim} & \Pi^*\CL_{\phi^{-1}\circ\theta}^p \ar[d]^-{\Pi^*s_{\phi^{-1}\circ\theta}} \\
 &\D(H)_{\phi\circ\theta} & \Pi^*\Lie(H^\vee)_{\theta}  
 \ar[r]^-{\sim}  \ar[l] & \Pi^*\CL_{\theta}  }\end{equation}
commutes since $\Frob_H: H \to C^{(p)} \to H^{(p)}$ is defined by the morphism of
$\CO_S$-algebras induced by the $s_{\theta}$, and the rest of the diagram commutes
by functoriality.   Since the top row is $\alpha_\theta^{(p)}$, it just remains to show that
the bottom row of (\ref{eqn:crystalPhi}) is $\alpha_{\phi\circ\theta}^{-1}$.  If $\theta \in I$, this is clear
from the definitions.
On the other hand if $\theta\not\in I$ and $\phi^{-1}\circ\theta \in I$, then we find that $\Lie({C'}^\vee)_\theta \cong \Lie(H^\vee)_\theta = \CL_\theta$,
and by functoriality and Lemma~\ref{lem:BBMcomp} the squares in the diagram
$$\xymatrix{ \D(H)_{\phi\circ\theta}   &    \ar[l]_-{\sim} \D(C')_{\phi\circ\theta}   \ar[r]^-{\sim} &  \Pi^*\omega_{C',\theta} \\
\Pi^*\Lie(H^\vee)_{\theta} \ar[u] &    \ar[l]_-{\sim} \Pi^*\Lie({C'}^\vee)_{\theta}  \ar[r]^-{\sim}  \ar[u] &  \Pi^*\CL_\theta  \ar[u]_-{\wreath} 
   }$$ 
commute.  The composite along the top and right is $\alpha_{\phi\circ\theta}$, and along the bottom and left is 
the bottom row of (\ref{eqn:crystalPhi}).

Finally we note that the isomorphism $\alpha$ is independent of choice of $I$.  To prove this we can again apply
\cite[Cor.~1.3.4]{BM3} to reduce to the case of $S = \Spec k$ for a field $k$, and then the above argument shows
that if $\phi^{-1}\circ\theta \in I$, then $\alpha^{-1}_{\phi\circ\theta}$ is the composite
 $\Pi^*\CL_\theta \stackrel{\sim}{\longrightarrow}
 \Pi^*\Lie(H^\vee)_\theta \to \D(H)_{\phi\circ\theta}$, and similarly if $\phi^{-1}\circ\theta \not\in I$, then
 $\alpha_{\phi\circ\theta}$ is the composite
$\D(H)_{\phi\circ\theta} \to  \Pi^*\omega_{H,\theta}  \stackrel{\sim}{\longrightarrow}
  \Pi^*\CL_\theta$.  If $s_{\phi^{-1}\circ\theta}$ and $t_{\phi^{-1}\circ\theta}$ are
both trivial, then Lemma~\ref{lem:BBMcomp} implies that the composite of these maps is the identity,
so $\alpha_{\phi\circ\theta}$ is independent of whether or not $\phi^{-1}\circ\theta \in I$.
\epf
\begin{remark}  It is natural to expect that the proposition in fact holds over an arbitrary (locally Noetherian) $\F_q$-scheme $S$.
In fact our proof only required that locally $\CO_S$ have a $p$-basis in the sense of \cite[1.1.1]{BM3},
but the result stated above will suffice for our purpose and avoids having to recall the definition of a $p$-basis.
(Note however that the notion of a $p$-basis in \cite{BM3} is
more restrictive than the one in \cite[21.1.9]{ega4}; in particular any $\F_p$-algebra having a $p$-basis is formally smooth
over $\F_p$.)
\end{remark}

\subsubsection{Construction of an isogeny over $P_J$}\label{sss:frob.isog}
We first construct Raynaud data on the variety $P_J$ (defined in \S\ref{sss:frob.prel})
to which we will apply Proposition~\ref{prop:DRaynaud}.  
Let $u:A\to P_J$ denote the pull-back of the universal abelian scheme over $T_{J'}$ and
consider the $\CO_F \otimes \CO_S$-linear morphisms
$$\Frob^*:\CH^1_\dr(A^{(p)}/P_J) \to \CH^1_\dr(A/P_J)\quad\mbox{and}\quad 
    \Ver^*: \CH^1_\dr(A/P_J) \to \CH^1_\dr(A^{(p)}/P_J).$$
For each $\theta \in \Theta$ we define line bundles $\CM_\theta$ and $\CL_\theta$ fitting in an exact sequence
\begin{equation}  \label{eqn:filoverPJ}  0 \to \CM_\theta \to \CH^1_\dr(A/P_J) \to \CL_\theta \to 0\end{equation}
as follows:
\begin{itemize}
\item if $\theta \in J$, then $\CM_\theta = (u_*\Omega^1_{A/P_J})_\theta$, $\CL_\theta = (R^1u_*\CO_A)_\theta$ and
 (\ref{eqn:filoverPJ}) is the Hodge filtration;
\item if $\phi^{-1}\circ\theta \not\in J$, then $\CM_\theta = \Frob^*(\CH^1_\dr(A^{(p)}/P_J)_\theta) \cong (R^1u_*\CO_A)_{\phi^{-1}\circ\theta}^p$ and
 $$\CL_\theta = \CH^1_\dr(A/P_J)_\theta /\CM_\theta \cong \Ver^*(\CH^1_\dr(A/P_J)_\theta) = (u_*\Omega^1_{A/P_J})_{\phi^{-1}\circ\theta}^p;$$
\item if $\theta \in J''$, then (\ref{eqn:filoverPJ}) is the tautological filtration defining the projection $P_J \to \P_{T_{J'}}(\CH^1_\dr(A/P_J)_\theta)$,
so $\CL_\theta = \CO(1)_\theta$ and $\CM_\theta \cong \wedge^2(\CH^1_\dr(A/P_J))_\theta(-1)_\theta$.
\end{itemize}
Note the first two conditions are not exclusive, but if $\theta \in J$ and $\phi^{-1}\circ\theta \not\in J$, then $\theta \in J'$,
in which case $\Frob^*(\CH^1_\dr(A^{(p)}/P_J)_\theta) = (u_*\Omega^1_{A/P_J})_\theta$ by the definition of $T_{J'}$.
Furthermore note that $\Frob^*(\CM_{\phi^{-1}\circ\theta}^p) \subset \CM_\theta$ for all $\theta$ (since 
$\Frob^*(\CM_{\phi^{-1}\circ\theta}^p) = 0$ if $\phi^{-1}\circ\theta \in J$), and 
similarly $\Ver^*(\CM_\theta) \subset \CM_{\phi^{-1}\circ\theta}^p$ for all $\theta$ (since
$\Ver^*(\CH^1_\dr(A/P_J)_\theta) = \CM_{\phi^{-1}\circ\theta}^p$ if $\phi^{-1}\circ\theta \in J$, and
$\Ver^*(\CM_\theta) = 0$ if $\phi^{-1}\circ\theta \not\in J$).  We can therefore define 
$s_\theta:\CM_\theta^p \to \CM_{\phi\circ\theta}$ and $t_\theta: \CM_{\phi\circ\theta} \to \CM_\theta^p$
as the restrictions of $\Frob^*$ and $\Ver^*$ and so obtain Raynaud data $(\CM_\theta,s_\theta,t_\theta)$
on $P_J$, and we let $C$ denote the corresponding $(\CO_F/p)$-module scheme.  Similarly the morphisms
induced by $\Frob^*$ on $\CL_\theta^p$ and $\Ver^*$ on $\CL_{\phi\circ\theta}$
give rise to Raynaud data on $P_J$, and we let $H$ denote the corresponding $(\CO_F/p)$-module scheme.

Letting $\CM = \bigoplus \CM_\theta$ and $\CL = \bigoplus \CL_\theta$, we obtain
from (\ref{eqn:filoverPJ}) an exact sequence of Dieudonn\'e crystals
$$0 \to \Pi^*\CM  \to \Pi^* \CH^1_\dr(A/P_J) \to \Pi^*\CL \to 0,$$
where $\Phi$ (resp.~$V$) is defined on the terms by $\Pi^*\Frob^*$ (resp.~$\Pi^*\Ver^*$),
and according to Proposition~\ref{prop:DRaynaud}, we have canonical isomorphisms
$\Pi^*\CM \cong \D(C)$ and $\Pi^*\CL \cong \D(H)$.   On the other hand we also have
the canonical isomorphisms 
$$\Pi^*\CH^1_\dr(A/P_J) \cong \Pi^*\D(A[p])_{P_J}  \cong \D(A[p])^{(p)} \cong \D(A^{(p)}[p])$$
of Dieudonn\'e crystals provided by \cite[(3.3.7.3),(4.3.7.1)]{BBM}, allowing us to interpret the above exact sequence as
$$0 \to \D(C)   \to \D(A^{(p)}[p]) \to \D(H) \to 0.$$
Since $P_J$ is smooth over $\F$, the functor $\D$ is fully faithful by \cite[Thm.~4.1.1]{BM3}, so we obtain
an exact sequence
$$0 \to H \to A^{(p)}[p] \to C \to 0$$
of finite flat group schemes on $P_J$.
Since $\D(H)_{P_J,\theta}$ is locally free of rank one for each $\theta \in \Theta$, we see as in the proof of
Lemma~\ref{lem:cart2} that $H$ is totally isotropic with respect to the $\lambda$-Weil pairing on $A^{(p)}$,
and hence the pair $(\underline{A}^{(p)},H)$ defines a morphism $P_J \to S = \widetilde{Y}_{U_0(p)}(G)_\F$.
Letting $(\underline{A}_1,\underline{A}_2,f)$ denote the universal triple on $S_J$, note that
the pull-back of $\underline{A}^{(p)}$ to $S_J$ via $\widetilde{\xi}_J$ is $\underline{A}_1^{(p)}$.
Furthermore, the definition of $\CM$ ensures that $\widetilde{\xi}^*_J\CM_\theta$ is the image of
$\CH^1_\dr(A_2/S_J)_\theta$ for each $\theta$.  (If $\theta \in J$, this is immediate from the
vanishing of $\Lie(f^\vee)_\theta$, if $\theta \in J''$, this is part of the definition of $\widetilde{\xi}_J$,
and if $\phi^{-1}\circ\theta \not\in J$, note that 
$$\Ver_{A_1}^*(f^*\CH^1_\dr(A_2/S_J)_\theta) = 
    f^{(p),*}(\Ver_{A_2}^*\CH^1_\dr(A_2/S_J)_\theta) =  0$$
since $\Ver_{A_2}^*\CH^1_\dr(A_2/S_J)_\theta = (u_{2,*}\Omega^1_{A_2/S_J})_{\phi^{-1}\circ\theta}^p$.)
It follows that the composite 
$$\Pi^*\D(A_2[p])_{S_J} \to \Pi^*\D(A_1[p])_{S_J} \to \Pi^*(\widetilde{\xi}_J^*\CL)$$
is trivial, but this is the same as
$$\D(A^{(p)}_2[p])_{S_J} \to \D(A^{(p)}_1[p])_{S_J} \to \D(\widetilde{\xi}_J^*H).$$
Therefore $H$ is contained in $\ker(f^{(p)})$, and comparing ranks, we see that
$H = \ker(f^{(p)})$.  It follows that the composite 
$S_J \stackrel{\widetilde{\xi}_J}{\longrightarrow} P_J \to S$ is the Frobenius morphism on $S_J$,
or more precisely the composite
$$S_J \to S_J^{(p)} \cong S_{\phi(J)} \subset S$$
where $S_J \to S_J^{(p)}$ is the Frobenius morphism relative to $\Spec \F$ and
$S_J^{(p)} \cong S_{\phi(J)}$ is the canonical isomorphism induced by the Frobenius
automorphism of $\F$.  We record this as a lemma:
\begin{lemma} \label{lem:frobfactor}  There exists a morphism $P_J \to S_J^{(p)}$
whose composite with $\xi_J:S_J \to P_J$ is the Frobenius morphism $S_J \to S_J^{(p)}$.
\end{lemma}

\subsubsection{Local structure}\label{sss:frob.loc}
We now study the local structure of the degeneracy fibres by computing the effect of the
morphisms $S_J \to P_J \to T_{J'}$ on tangent spaces at closed points.    As in the proof of Theorem~\ref{thm:bigJL},
we apply the equivalence of categories between 
 abelian schemes over $T = \F[\epsilon]/(\epsilon^2)$ and pairs $(A,\widetilde{L})$ where $A$ is an abelian variety over $\F$
 and $\widetilde{L}$ is a free $T$-submodule of $H^1_\dr(A/\F)\otimes_{\F} T$ such that
 $\widetilde{L} \otimes_T \F = H^0(A,\Omega^1_{A/\F})$. 
 
Let $(\underline{A}_1,\underline{A}_2,f)$ be (the data corresponding to) an element
of $z \in S_J(\F)$, and let $y$ (resp.~$x$) denote its image in $P_J(\F)$ (resp.~$T_{J'}(\F)$).
Thus $x$ corresponds to $\underline{A}_1$ and $y$ corresponds to 
$(\underline{A}_1, (L_\theta)_{\theta \in J''})$ where $L_\theta$ is the image of 
$H^1_\dr(A_2/\F)_\theta \to H^1_\dr(A_1/\F)_\theta$.  Note that if
$(\underline{\widetilde{A}}_1,\underline{\widetilde{A}}_2,\widetilde{f})$ is a lift of $z$ to $S_J(T)$, then
$H^0(\widetilde{A}_2,\Omega^1_{\widetilde{A}_2/T})_\theta$ is the kernel of
$$\widetilde{f}_\theta^*: H^1_\dr(\widetilde{A}_2/T)_\theta \to H^1_\dr(\widetilde{A}_1/T)_\theta$$
for $\theta\not\in J$,  and this corresponds to
$H^0(A_2,\Omega^1_{A_2/\F})_\theta \otimes_\F T$ under the isomorphisms with
$H^1_\crys(A_2/T)_\theta$.  Similarly if $\theta \in J$, then $H^0(\widetilde{A}_1,\Omega^1_{\widetilde{A}_1/T})_\theta$
is the image of $\widetilde{f}_\theta^*$, and this 
corresponds to $H^0(A_1,\Omega^1_{A_1/\F})_\theta \otimes_\F T$.  It follows from
a standard argument that the tangent space of $S_J$ at $z$ is in canonical bijection with
the set of tuples 
$$\left((\widetilde{L}_{1,\theta})_{\theta \not \in J} ,\,(\widetilde{L}_{2,\theta})_{\theta \in J}\right)$$
where  each $\widetilde{L}_{j,\theta}$  is a free $T$-submodule of $H^1_\dr(A_j/\F)\otimes_{\F} T$ lifting
$H^0(A_j,\Omega^1_{A_j/\F})$.  Similarly if $\underline{\widetilde{A}}_1$ is a lift of $x$ to $T_{J'}(T)$, then
we find that $H^0(\widetilde{A}_1,\Omega^1_{\widetilde{A}_1/T})_\theta$ corresponds to
$H^0(A_1,\Omega^1_{A_1/\F})_\theta \otimes_\F T$ if $\theta \in J'$, so the tangent space of $T_{J'}$ at 
$x$ is identified with 
$$(\widetilde{L}_{1,\theta})_{\theta \not \in J'}$$
where  each $\widetilde{L}_{1,\theta}$ is a lift of $H^0(A_i,\Omega^1_{A_1/\F})$ to
$H^1_\dr(A_1/\F)\otimes_{\F} T$.  Finally it follows from the definition of $P_J$ that 
its tangent space at 
$y$ is identified with 
$$\left((\widetilde{L}_{1,\theta})_{\theta \not \in J'} ,\,(\widetilde{L}_{\theta})_{\theta \in J''}\right)$$
where each $\widetilde{L}_{1,\theta}$ (resp.~$\widetilde{L}_{\theta}$) is a lift of $H^0(A_1,\Omega^1_{A_1/\F})$ 
(resp.~$L_\theta$) to $H^1_\dr(A_1/\F)\otimes_{\F} T$.  Furthermore the map on tangent spaces
induced by $P_J \to T_{J'}$ corresponds to the natural projection under the above descriptions.
As for the map on tangent spaces induced by $\xi_J$, suppose that $\left((\widetilde{L}_{1,\theta})_{\theta \not \in J} ,\,(\widetilde{L}_{2,\theta})_{\theta \in J}\right)$
corresponds to $(\underline{\widetilde{A}}_1,\underline{\widetilde{A}}_2,\widetilde{f})$ and let 
$\left((\widetilde{L}_{1,\theta})_{\theta \not \in J'} ,\,(\widetilde{L}_{\theta})_{\theta \in J''}\right)$
correspond to its image in $P_J(T)$.    Note that if $\theta \in J$ but $\theta \not\in J'$ (i.e., $\phi^{-1}\circ\theta$ is also in $J$),
then $\widetilde{L}_{1,\theta}$, being the image of $\widetilde{f}_\theta^*$, is
$H^0(A_1,\Omega^1_{A_1/F})_\theta \otimes_\F T$, and similarly if $\theta\in J''$, then 
$\widetilde{L}_\theta = L_\theta \otimes_\F T$.

Recall that letting $\Spec(R_1)$ denote the first infinitesimal neighborhood of $x$ in $T_{J'}$ and choosing trivializations
$R^2 \cong H^1_\dr(A_1^\univ/R)_\theta$ in a Zariski neighborhood $\Spec R$ compatible with the canonical isomorphism
$$H^1_\dr(A_1^\univ/R_1)_\theta \cong H^1_\crys(A_1/R_1)_\theta \cong H^1_\dr(A_1/\F)_\theta \otimes_\F R_1$$
yields a regular system of parameters $\{\,t_\theta\,|\,\theta\not\in J'\,\}$ at $x$, so that
$$\F[[t_\theta]]_{\theta\not\in J'} \stackrel{\sim}{\longrightarrow} \widehat{\CO}_{T_{J'},x}.$$
More precisely, for a sufficiently small neighborhood of $x$ the trivialization may be chosen so that
$H^0(A_1^\univ,\Omega^1_{A_1^\univ/R})_\theta \subset H^1_\dr(A_1^\univ/R)_\theta$ corresponds to 
$R(1,t_\theta) \subset R^2$ for some $t_\theta \in R$ vanishing at $x$.  The resulting parameters
are then compatible with the above description of the tangent space in the sense that the span of each $t_{\sigma}$ in the
cotangent space $\gm_x/\gm_x^2$ is orthogonal to the set of $(\widetilde{L}_{1,\theta})_{\theta \not \in J'}$
such that $\widetilde{L}_{1,\sigma} = L_{1,\sigma} \otimes_\F T$.  We similarly obtain isomorphisms
$$\begin{array}{rcl}  \F[[t_\theta]]_{\theta\not\in J'}\, \widehat{\otimes}\, \F[[u_\theta]]_{\theta \in J''}
&\stackrel{\sim}{\longrightarrow} &\widehat{\CO}_{P_{J},y}\\
\mbox{and}\quad
  \F[[t_\theta]]_{\theta\not\in J}\, \widehat{\otimes}\, \F[[v_\theta]]_{\theta \in J}
&\stackrel{\sim}{\longrightarrow} &\widehat{\CO}_{S_{J},z}\end{array}$$
for which the parameters are compatible with the above descriptions of the tangent spaces
of $P_J$ at $y$ and $S_J$ at $z$.  The homomorphisms of completed local rings induced
by the morphisms $S_J \to P_J \to T_{J'}$ then take the form
$$\F[[t_\theta]]_{\theta\not\in J'} \hookrightarrow
  \F[[t_\theta]]_{\theta\not\in J'}\, \widehat{\otimes}\, \F[[u_\theta]]_{\theta \in J''}  \hookrightarrow
    \F[[t_\theta]]_{\theta\not\in J}\, \widehat{\otimes}\, \F[[v_\theta]]_{\theta \in J},$$
where the first arrow is the natural inclusion and the second sends $t_\theta$ to $t_\theta$
for $\theta \not\in J$ and the remaining parameters to elements of $\gm_z^2$.
Furthermore Lemma~\ref{lem:frobfactor} implies that the image of the second map
contains $v_\theta^p$ for all $\theta \in J$.  Taking the quotient of each ring by the ideal
generated by the $t_\theta$ for $\theta\not\in J$, the resulting homomorphism
$$ \F[[t_\theta]]_{\theta\in J - J'}\, \widehat{\otimes}\, \F[[u_\theta]]_{\theta \in J''}  \rightarrow
     \F[[v_\theta]]_{\theta \in J}$$
is trivial on the cotangent space at the closed point and has image containing $v_\theta^p$
for all $\theta \in J$.  It therefore follows from \cite[Cor.~2]{KN} that the image is precisely
$\F[[v^p_\theta]]_{\theta \in J}$, so we may replace the parameters $v_\theta$ by parameters
$w_\theta$ in $\widehat{\CO}_{S_{J},z}$ such that 
$$\begin{array}{rcll}  t_\theta &\equiv &w_{\phi^{-1}\circ\theta}^p \bmod I & \quad \mbox{if $\theta\in J$, $\phi^{-1}\circ\theta \in J$}\\
\qquad \mbox{and}\quad u_\theta &\equiv& w_{\phi^{-1}\circ\theta}^p \bmod I&\quad  \mbox{if $\theta\not\in J$, $\phi^{-1}\circ\theta \in J$,}
\end{array}$$
where $I$ is the ideal generated by the $t_\theta$ for $\theta \not\in J$.  
We conclude that the quotient of $\widehat{\CO}_{S_{J},z}$ by the ideal generated by the $t_\theta$ for $\theta\not\in J'$
is isomorphic to
$$\F[[w_\theta]]_{\theta\in J} /  \langle w_\theta^p \rangle_{\theta,\phi\circ\theta \in J}.$$
\begin{lemma}  \label{lem:fibrecompletions}  Let $Z$ be the fibre of $S_J \to T_{J'}$ at a closed point
$x$ of $T_{J'}$, and let $Z_\red$ be its reduced subscheme.  If $z$ is a closed point of $Z$, then
$$\widehat{\CO}_{Z,z}  \cong \F[[w_\theta]]_{\theta\in J} /  \langle w_\theta^p \rangle_{\theta,\phi\circ\theta \in J}
\quad\mbox{and}\quad
\widehat{\CO}_{Z_\red,z}  \cong \F[[w_\theta]]_{\theta\in J,\phi\circ\theta \not\in J};$$
in particular $Z_\red$ is smooth of dimension $|J'| = |J''|$ over $\F$.
\end{lemma}
\begpf  The assertion concerning $\widehat{\CO}_{Z,z}$ is immediate from the above discussion;
the assertions for $Z_\red$ follow since $Z$ is locally of finite type over $\F$, and in particular Nagata,
so that $\widehat{\CO}_{Z_\red,z}  = \widehat{\CO}^{\,\red}_{Z,z}$.
\epf

\subsection{Crystallization}  \label{sec:fibres}
We maintain the notation of the preceding section.

\subsubsection{A crystallization lemma}\label{sss:crys.lem}
Note that at each closed point of $S_J$ and $\theta \in J$, the map 
$$H^1_\dr(A_2/\F)_\theta \to H^1_\dr(A_1/\F)_\theta$$
induced by the fibre $A_1 \to A_2$ of the universal isogeny has image $H^0(A_1,\Omega^1_{A_1/\F})_\theta$,
which coincides with the image of the map induced by $\Ver:A_1 \to A_1^{(p^{-1})}$.
It follows that the maps 
$$\D(A_2[p^\infty])_\theta \to \D(A_1[p^\infty])_\theta \quad\mbox{and} \quad
\D(A_1^{(p^{-1})}[p^\infty])_\theta \to \D(A_1[p^\infty])_\theta$$
have the same image, and hence we obtain an isomorphism
$$\D(A_2[p^\infty])_\theta  \cong \D(A_1^{(p^{-1})}[p^\infty])_\theta,$$
which in turn induces an isomorphism 
$$H^1_\dr(A_2/\F)_\theta \to H^1_\dr(A_1^{(p^{-1})}/\F)_\theta.$$
We wish to prove this in fact arises from an isomorphism of sheaves
on each fibre of $S_J \to T_{J'}$, for which we appeal to
the following crystallization lemma:

\begin{lemma}     \label{lem:crys}  Suppose that $X$ is a smooth scheme over $\F$, 
$A$, $B_1$ and $B_2$ are abelian schemes over $X$
with $O_F$-action, and $\theta \in \Theta$.  Let $\alpha_i:A \to B_i$ for $i=1,2$ be $O_F$-linear isogenies such that
\begin{itemize}
\item $\ker(\alpha_i) \cap A[p^\infty] \subset A[p]$ for $i=1,2$, and
\item $\alpha_{1,x}^* \D(B_{1,x}[p^\infty])_\theta = \alpha_{2,x}^* \D(B_{2,x}[p^\infty])_\theta$ for all $x \in X(\F)$.
\end{itemize}
Then there is a unique isomorphism $\CH^1_{\dr}(B_1/X)_\theta \cong \CH^1_{\dr}(B_2/X)_\theta$
of coherent sheaves on $X$ whose fibres are compatible with the isomorphisms
$$\D(B_{1,x}[p])_\theta \cong \D(B_{2,x}[p])_\theta$$
 induced by $(\alpha_{2,x}^*)^{-1}\alpha_{1,x}^*$
for all $x \in X(\F)$.  Furthermore letting $j:\Spec(R_1) = X_1 \to X$ denote the first infinitesimal neighborhood of $x$,
the isomorphism is also compatible with the isomorphisms
$$H^1_\dr(B_{i,x}/\F) \otimes_\F R_1 \cong H^1_\dr(j^*B_i/R_1)$$
induced by their canonical isomorphisms with $H^1_\cris(B_{i,x}/R_1)$ for $i=1,2$.
\end{lemma}
\begpf  First note that $\CH^1_{\dr}(A/X)_\theta$, being a direct summand of $\CH^1_{\dr}(A/X)$, is locally free,
and similarly for $\CH^1_{\dr}(B_i/X)_\theta$ for $i=1,2$.  Furthermore the rank of the cokernel of
$\alpha_{i,x}^*:H^1_{\dr}(B_{i,x}/\F) \to H^1_\dr(A_x/\F)$ is the $p$-part of the degree of $\alpha_i$,
hence is locally constant, and the rank of the cokernel of each summand
$H^1_\dr(B_{i,x}/\F)_{\theta'} \to H^1_\dr(A_x/\F)_{\theta'}$
is upper semi-continuous, so it follows that these ranks are also locally constant, and
hence that the cokernel of $\alpha_i^*: \CH^1_{\dr}(B_i/X)_\theta \to \CH^1_{\dr}(A/X)_\theta$
is locally free (since $X$ is reduced and its closed points are dense).  The hypothesis that
$\alpha_{1,x}^* \D(B_{1,x})_\theta = \alpha_{2,x}^* \D(B_{2,x})_\theta$
implies that the fibre at $x$ of the composite 
$$\CH^1_{\dr}(B_1/X)_\theta  \stackrel{\alpha_1^*}{\longrightarrow}  \CH^1_{\dr}(A/X)_\theta 
        \longrightarrow  \coker(\alpha_2^*)$$
is trivial.  Since $\CO_X(U)$ has trivial (Jacobson) radical for all open $U\subset X$,
morphisms to $\CO_X$, hence to any locally free coherent sheaf on $X$, are
determined by their fibres, so in fact the above composite is trivial.
Therefore $\im(\alpha_2^*)  \subset \im(\alpha_1^*)$ and similarly the
opposite inclusion, and hence equality, holds.

Again using the fact that morphisms to locally free coherent sheaves on $X$
are determined by fibres, we see that the desired isomorphism, if it exists, is
unique.  Furthermore we may work locally on $X$ and hence assume that there
is a smooth scheme $\widetilde{X}$ over $W = W(\F)$ such that $\widetilde{X}\times_W \F = X$
(see for example \cite[Lemma~10.135.19]{stacks}).

Letting $s: A \to X$ and $t_i:B_i \to X$ (for $i=1,2$) denote the structure morphisms, consider the
crystals of locally free $\CO_{X/\Z_p}$-modules $\CC:= R^1s_{\cris,*}\CO_{A/\Z_p}$ and $\CE_i := R^1_{t_i,\cris,*}\CO_{B_i/\Z_p}$
(see \cite[Cor.~2.5.5]{BBM}).
Since the functor sending an abelian scheme $u:C \to X$ to $R^1u_{\cris,*}\CO_{C/\Z_p}$
is additive, the crystals $\CC$ and $\CE_i$ inherit an action of $\CO_F\otimes W$, and hence decompose as
direct sums of crystals of locally free $\CO_{X/\Z_p}$-modules indexed by $\Theta$.
Since the isogenies $\alpha_i$ are $\CO_F$-linear,
they induce morphisms of crystals $\CE_{i,\theta} \to \CC_\theta$ for $i = 1,2$.

We claim that there is an isomorphism  $\CE_{1,\theta} \stackrel{\sim}{\longrightarrow} \CE_{2,\theta}$
such that the diagram 
\begin{equation} \label{eqn:3crystals}  \xymatrix{ \CE_{1,\theta}  \ar[rr] \ar[dr] && \CE_{2,\theta} \ar[dl] \\
&\CC_\theta&}
\end{equation}
commutes.
To prove this we use the equivalence of categories provided by \cite[Cor.~6.8]{BO} (with $S=W$ and $S_0 = \F$).
If $\CF$ denotes the quasi-coherent sheaf (with integrable quasi-nilpotent connection)
on $\widetilde{X}$ corresponding to $\CC_\theta$, then $\CF$ is locally free and its restriction to $X$ is
canonically isomorphic to $(\CC_{\theta})_X \cong \CH^1_\dr(A/X)_\theta$ (\cite[(3.3.7.3)]{BBM}).  
Similarly (for $i=1,2$) the underlying quasi-coherent
sheaf $\CG_i$ corresponding to $\CE_{i,\theta}$ is a locally free sheaf on $\widetilde{X}$ whose restriction to $X$
is identified with $\CH^1_\dr(B_i/X)_\theta$.  The claim is thus equivalent to the existence of an isomorphism
$\CG_1 \stackrel{\sim}{\longrightarrow}  \CG_2$ which is compatible with connections and makes the diagram
$$  \xymatrix{ \CG_{1}  \ar[rr] \ar[dr]_-{\widetilde{\alpha}_1^*} && \CG_{2} \ar[dl]^-{\widetilde{\alpha}_2^*} \\
&\CC_\theta&}$$
commute, where each $\widetilde{\alpha}_i^*$ is the morphism corresponding to $\CE_{i,\theta} \to \CC_\theta$.
Since the $\widetilde{\alpha}_i^*$ are compatible with connections, it suffices to prove that they are
injective and have the same image.

Now let $\beta_i: B_i \to A$ be an isogeny such that $\alpha_i \circ \beta_i = d$ with $p||d$, and let $\widetilde{\beta}_i^*$
be the morphism $\CF \to \CG_i$ corresponding to the one induced on crystals.  
Then the composite $\widetilde{\beta}_i^*\circ \widetilde{\alpha}_i^*$ is also multiplication by $d$, since it is so for the
corresponding morphism of crystals by additivity.  The composite is therefore injective, and therefore so is
$\widetilde{\alpha}_i^*$.  On the other hand considering the composite $\beta_i \circ \alpha_i = d$, we see that
the image of $\widetilde{\alpha}_i^*$ contains $p\CF$.  Therefore to conclude $\widetilde{\alpha}_1^*$ and
$\widetilde{\alpha}_2^*$ have the same image, it suffices to prove their reductions modulo $p$, or equivalently
their restrictions to $X$, have the same image, but this is precisely the fact that $\im(\alpha_1^*) = \im(\alpha_2^*)$.

The isomorphism of crystals now yields an isomorphism
$$\CH^1_\dr(B_1/X)_\theta \cong (\CE_{1,\theta})_X   \stackrel{\sim}{\longrightarrow} (\CE_{2,\theta})_X \cong \CH^1_\dr(B_2/X)_\theta$$
of coherent sheaves on $X$.  The desired compatibility with homomorphisms of Dieudonn\'e modules follows from the commutativity of
(\ref{eqn:3crystals}) and the fact
that $\alpha_{i,x}^*:\D(B_{i,x}[p^\infty])_\theta \to \D(A_x[p^\infty])_\theta$ is realized as the inverse limit of the morphisms
$$\CE_{i,\theta}(\{x\},\Spec(W_n)) \longrightarrow \CC_{\theta}(\{x\},\Spec(W_n))$$
under the canonical isomorphisms provided by \cite[(3.3.7.2), Thm.~4.2.14]{BBM}.
Finally, since the isomorphisms
$H^1_\dr(B_{i,x}/\F) \otimes_\F R_1 \cong H^1_\dr(j^*B_i/R_1)$
are realized as the composite of the crystalline transition maps
$$\CE_{i,\theta}(\{x\},\{x\}) \otimes_\F R_1 \stackrel{\sim}{\longrightarrow}
\CE_{i,\theta}(\{x\},X_1)  \stackrel{\sim}{\longleftarrow} \CE_{i,\theta}(X_1,X_1) 
$$
under \cite[(3.3.7.3)]{BBM}, the second desired compatibility follows from the fact that 
$\CE_{1,\theta} \to \CE_{2,\theta}$ is an isomorphism of crystals.  \epf

\subsubsection{The reduced degeneracy fibre}\label{sss:crys.red}
We now return to the setup of Lemma~\ref{lem:fibrecompletions}, so $Z$ is the 
fibre of $S_J \to T_{J'}$ at a closed point $x$ of $T_{J'}$, $Z_\red$ is its reduced subscheme,
and $z$ is a closed point of $Z$.  
Let $\underline{A}_0$ be the abelian scheme over $\F$ corresponding to $x$ and consider
$$A = A_0 \times_\F Z_\red  = A_1 \times_{S_J} Z_\red  \quad\mbox{and}\quad B = A_2 \times_{S_J} Z_\red,$$
where as usual $f:A_1 \to A_2$ is the universal isogeny on $S_J$.    We have already observed
that if $\theta \in J$, then
$$f^*: H^1_\dr(A_{2,z}/\F)_\theta \to H^1_\dr(A_0/\F)_\theta \quad
\mbox{and}\quad \Ver^*:  H^1_\dr(A_0^{(p^{-1})}/\F)_\theta \to H^1_\dr(A_0/\F)_\theta$$
have the same image, so we may apply Lemma~\ref{lem:crys} with $B_1 = B$, $\alpha_1$ as the restriction of $f$,
$B_2 = A_0^{(p^{-1})} \times_\F Z_\red$ and $\alpha_2$ as the pull-back of $\Ver$ to obtain an isomorphism
\begin{equation} \label{eqn:good}
\CH^1_\dr(B/Z_\red)_\theta \stackrel{\sim}{\longrightarrow} H^1_\dr(A_0^{(p^{-1})}/\F)_\theta \otimes_\F \CO_{Z_\red}
\end{equation}
whose fibre at each closed point $z$ is compatible with the isomorphism in the top row of the 
commutative diagram
\begin{equation} \label{eqn:Dieugood} \xymatrix{ \D(B_z[p^\infty])_\theta \ar[rr]^-{\sim} \ar@{_{(}->}[dr]_-{f^*} && \D(A_0^{(p^{-1})}[p^\infty])_\theta \ar@{^{(}->}[dl]^-{\Ver^*}\\
&\D(A_0[p^{\infty}])_{\theta}.&}
\end{equation}
Furthermore letting $j:\Spec R_1 \to Z_\red$ denote the  first infinitesimal neighborhood of $z$, the diagram
\begin{equation}\label{eqn:crysgood} 
\xymatrix{ H^1_\dr(B_z /\F)_\theta \otimes_\F R_1\ar[rr] \ar[dr] && H^1_\dr(j^*B/R_1)_\theta \ar[dl]\\
&   H^1_\dr(A_0^{(p^{-1})}/\F)_\theta \otimes_\F R_1&}
\end{equation}
also commutes, where the top arrow is given by the canonical isomorphisms with $H^1_\cris(B_z/R_1)$ and
the diagonal arrows by the restrictions of (\ref{eqn:good}) to $z$ and $\Spec R_1$.
Moreover note that if also $\phi\circ\theta \in J$, then the commutativity of (\ref{eqn:Dieugood}) implies that of
$$
\xymatrix{ \D(B_z[p^\infty])_\theta^{(p)} \ar[r]^-{\sim}\ar[d]^-{\wreath}
 &  \D(B_z^{(p)}[p^\infty])_{\phi\circ\theta} \ar[r]^-{\Frob^*} 
  &  \D(B_z[p^\infty])_{\phi\circ\theta} \ar[d]^{\wreath}  \\
\D(A_0^{(p^{-1})}[p^\infty])_\theta^{(p)}  \ar[r]^-{\sim}  
&  \D(A_0[p^\infty])_{\phi\circ\theta}   \ar[r]^-{\Frob^*} 
&  \D(A_0^{(p^{-1})}[p^\infty])_{\phi\circ\theta},  }
$$
and hence that (\ref{eqn:good}) restricts to
\begin{equation} \label{eqn:filgood}
(s_*\Omega^1_{B/Z_\red})_\theta \stackrel{\sim}{\longrightarrow} H^0(A_0^{(p^{-1})},\Omega^1_{A_0^{(p^{-1})}/\F})_\theta \otimes \CO_{Z_\red}
\end{equation}
since it does so on fibres (where $s:B \to Z_\red$ is the structure morphism).
On the other hand if $\phi\circ\theta \not\in J$, i.e. $\phi\circ\theta \in J''$, then the Hodge filtration on $H^1_\dr(B/Z_\red)_\theta$ defines
a morphism $Z_\red \to \P(H^1_\dr(A^{(p^{-1})}_0/\F)_\theta)$ under which $\CO(1)$ pulls back via (\ref{eqn:good}) to
$(R^1s_*\CO_B)_\theta$.

\begin{theorem} \label{thm:reducedfibres}
The morphism
$$\psi = \psi_{J,x} :  Z_\red \to \prod_{\phi\circ\theta \in J''}  \P(H^1_\dr(A_0^{(p^{-1})}/\F)_\theta)$$
is an isomorphism.
\end{theorem}
\begpf
As in the proof of Theorem~\ref{thm:bigJL}, it suffices to prove
 that the morphism is bijective on $\Fpbar$-points and injective on tangent spaces at such points.

For the bijectivity on points, recall that the morphism $\widetilde{\xi}_J$ of (\ref{eqn:xiJ}) is bijective
on points, and therefore so is its restriction to fibres over $x$, and hence so is
$$\widetilde{\xi}_{J,x}:  Z_\red \to \prod_{\theta \in J''}  \P(H^1_\dr(A_0/\F)_\theta).$$
If the point $z \in Z_\red(\F)$ corresponds to the isogeny $f:A_0 \to B_z$, then the $\theta$-component 
of $\widetilde{\xi}_{J,x}(z)$ is the image of the $W$-submodule
$$f^*\D(B_z[p^\infty])_\theta \subset \D(A_0[p^\infty])_\theta$$
in $ \D(A_0[p^\infty])_\theta/p  \D(A_0[p^\infty])_\theta = \D(A_0[p])_\theta \cong H^1_\dr(A_0/\F)_\theta$.  
On the other hand, since $H^0(B_z,\Omega^1_{B_z/\F})_{\phi^{-1}\circ\theta}$
is the reduction mod $p$ of 
$$\Ver^*(\D(B_z^{(p^{-1})}[p^\infty])_{\phi^{-1}\circ\theta}  \subset \D(B_z[p^\infty])_{\phi^{-1}\circ\theta},$$
it follows from the commutativity of (\ref{eqn:Dieugood}) and the compatibility $f\circ \Ver = \Ver\circ f^{(p^{-1})}$
that the $\phi^{-1}\circ\theta$-component of $\psi(z)$ is the image of the  $W$-submodule
$$\left(f^{(p^{-1})}\right)^*\D(B_z^{(p^{-1})}[p^\infty])_{\phi^{-1}\circ\theta} \subset \D(A^{(p^{-1})}_0[p^\infty])_{\phi^{-1}\circ\theta}$$
in $H^1_\dr(A^{(p^{-1})}_0/\F)_{\phi^{-1}\circ\theta}= H^1_\dr(A/\F)_\theta^{(p^{-1})}$.  Thus if $\widetilde{\xi}_{J,x}(z) = (L_\theta)_{\theta\in J''}$, then
$\psi(z) = (L_\theta^{(p^{-1})})_{\theta\in J''}$, so $\psi$ is bijective on $\F$-points.

To prove the injectivity on tangent spaces, let $T = \F[\epsilon]/(\epsilon^2)$ and suppose
$(\underline{A}_{0,T},\widetilde{\underline{B}}_z,\widetilde{f})$ and 
$(\underline{A}_{0,T},\widetilde{\underline{B}}'_z,\widetilde{f}')$ are two lifts of $z \in Z_\red(\F)$
to $Z_\red(T)$ with the same image in $\prod_{\theta \in J''}  \P(H^1_\dr(A^{(p^{-1})}_0/\F)_\theta)(T)$.
This means that $L_\theta:= H^0(\widetilde{B}_z, \Omega^1_{\widetilde{B}_z/T})_\theta$
and $L_\theta':=H^0(\widetilde{B}'_z, \Omega^1_{\widetilde{B}'_z/T})_\theta$ have the same image
in $H^1_\dr(A^{(p^{-1})}_0/\F)_\theta\otimes_\F T$ under (\ref{eqn:good}) for all $\theta \in J$ such
that $\phi\circ\theta\not\in J$. On the other hand if $\theta,\phi\circ\theta \in J$, then the
same holds by (\ref{eqn:filgood}).  In either case, the commutativity of (\ref{eqn:crysgood}), tensored
over $R_1$ with $T$, implies that $L_\theta$ and $L_\theta'$
have the same image in $H^1_\dr(B_z/\F)_\theta \otimes_\F T$.
On the other hand if $\theta\not\in J$, then the same assertion follows from the local analysis
preceding Lemma~\ref{lem:fibrecompletions}.  Therefore the Grothendieck--Messing Theorem
implies that the identity on $B_z$ lifts to an isomorphism $\widetilde{B}_z \cong \widetilde{B}'_z$.
The compatibility of the isomorphism with the auxiliary data follows from the faithfulness of
$\cdot \otimes_T \F$ by a standard argument, so we conclude that 
$(\underline{A}_{0,T},\widetilde{\underline{B}}_z,\widetilde{f})$ and 
$(\underline{A}_{0,T},\widetilde{\underline{B}}'_z,\widetilde{f}')$ define the same element
of the tangent space.
\epf

\subsubsection{Thickening}\label{sss:crys.thick}
In order to extend this to obtain a complete description of the fibre $Z$, we recall the general version
of the Grothendieck--Messing Theorem (see \cite[Thm.~V.1.10]{mm}) in characteristic $p$ states that if $i:X_0\hookrightarrow X$ is a 
nilpotent thickening of $\F_p$-schemes whose defining ideal sheaf is equipped with a divided power structure,
then the category of abelian schemes over $X$ is equivalent (under the obvious functor) to the category of
pairs $(C,\CM)$ where $u:C\to X_0$ is an abelian scheme and 
$$\CM \subset (R^1u_{\cris,*}\CO_{C/\F_p})_{(X_0,X)}$$
is a locally free $\CO_X$-submodule such that $i^*\CM$ corresponds to the image of $u_*\Omega^1_{C/X_0}$
under the canonical isomorphism with $\CH^1_\dr(C/X_0)$.  

We apply this with $X = X_0 \times_{\F} X_1$, where $X_0 = Z_\red$
and $X_1$ is defined as the fibre product
$$\begin{array}{ccc}  X_1  & \longrightarrow &\displaystyle{ \prod_{\theta,\phi\circ\theta \in J}  \P(H^1_\dr(A_0^{(p^{-1})}/\F)_\theta)} \\
\downarrow && \downarrow \\
\Spec \F & \longrightarrow &\displaystyle{ \prod_{\theta,\phi\circ\theta\in J}  \P(H^1_\dr(A_0/\F)_{\phi\circ\theta}),} \end{array}$$
where the bottom arrow is defined by $H^0(A_0,\Omega^1_{A_0 /\F})_{\phi\circ\theta}$ 
and the right vertical arrow is the (relative Frobenius) morphism
sending $(M_\theta)_{\theta}$ to $(M_\theta^p)_\theta$, where we view
$$L_\theta^p \subset H^1_\dr(A_0^{(p^{-1})}/\F)_\theta^{(p)} \cong H^1_\dr(A_0/\F)_{\phi\circ\theta}.$$
Note that $X_1 \cong \prod \Spec \F[w_\theta]/(w_\theta^p)$, so the closed immersion $\Spec \F \hookrightarrow X_1$
is a nilpotent thickening with divided powers, and hence so is the resulting closed immersion $i:X_0 \hookrightarrow X$.

Letting $\pi:X \to X_0$ denote the projection, we have a canonical isomorphism
\begin{equation}\label{eqn:X0cris}   (R^1s_{\cris,*}\CO_{B/\F_p})_{(X_0,X)}  \cong \pi^*\CH^1_\dr(B/X_0),\end{equation}
and we let $\CM  = \bigoplus_{\theta\in \Theta} \CM_\theta$
where
\begin{itemize}
\item $\CM_\theta$ corresponds under (\ref{eqn:good}) to 
the pull-back of the tautological bundle on $\P(H^1_\dr(A_0^{(p^{-1})}/\F)_\theta)$ if $\theta,\phi\circ\theta\in J$;
\item $\CM_\theta$ corresponds under (\ref{eqn:X0cris}) to $\pi^*(s_*\Omega^1_{B/X_0})_\theta$ otherwise.
\end{itemize}
We then see  (using (\ref{eqn:filgood}) if $\theta,\phi\circ\theta \in J$) that $i^*\CM_\theta= (s_*\Omega^1_{B/X_0})_\theta$ for all $\theta$,
so that the data $(B,\CM)$ corresponds to an abelian scheme $\widetilde{B}$ over $X$ lifting $B$. 
Furthermore writing $f$ for the restriction to $Z_\red$ of the universal isogeny, and using that
$$f^*(\CH^1_\dr(B/X_0)_\theta) = H^0(A_0,\Omega^1_{A_0/\F})_\theta \otimes_\F \CO_{X_0}$$
if $\theta,\phi\circ\theta \in J$, we see that the image of $\CM$ under $f^*$ is contained
in $H^0(A_0,\Omega^1_{A_0/\F})_\theta \otimes_\F \CO_X$, from which it follows that
$f$ lifts (uniquely) to an isogeny $\widetilde{f}: A_0 \times_{\F}  X \to \widetilde{B}$ such that 
$\Lie (\widetilde{f})_\theta = 0$ if $\theta\not\in J$ and $\Lie (\widetilde{f}^\vee)_\theta = 0$
if $\theta \in J$.  Equipping $\widetilde{B}$ with the (unique) auxiliary data lifting that on $B$, we thus obtain
a triple $(\underline{A}_0\times_\F X, \widetilde{\underline{B}},\widetilde{f})$ corresponding to a morphism
$\widetilde{\psi}: X \to Z$ extending the inclusion $Z_\red \hookrightarrow Z$.

\begin{lemma} \label{lem:nilfibres} The morphism $\widetilde{\psi} : X \to Z$ is an isomorphism.
\end{lemma}
\begpf  Since the morphism is finite and bijective on closed points, it suffices to prove it induces
isomorphisms on the completed local rings at closed points.

Letting $y$ be a closed point of $X$ and $z = \widetilde{\psi}(y)$, 
we first note that $\widetilde{\psi}$ is injective on tangent spaces; indeed
the injectivity of the composite $T_y(X) \to T_z(Z) \subset T_z(S_J)$ is immediate from the construction
of $\widetilde{\psi}$ and the description of $T_z(S_J)$ preceding Lemma~\ref{lem:fibrecompletions}.
Therefore the homomorphism $\widetilde{\psi}^*: \widehat{\CO}_{Z,z}  \to \widehat{\CO}_{X,y} $
is surjective.  

Now we let $\ga$ and $\gb$ denote the respective nilradicals and show by induction that $\widetilde{\psi}^*$ induces an isomorphism
$\widehat{\CO}_{Z,z} /\ga^n \to \widehat{\CO}_{X,y} /\gb^n $ for all $n\ge 1$.  For $n=1$, this follows from the identification of 
each quotient with the regular local ring $\widehat{\CO}_{Z_\red,z}$.  For the induction step,
apply the snake lemma to the diagram
$$\xymatrix{ 0 \ar[r]   & \ga^n/\ga^{n+1}   \ar[r] \ar[d] & \widehat{\CO}_{Z,z} /\ga^{n+1}   \ar[r]\ar@{->>}[d] & \widehat{\CO}_{Z,z} /\ga^n  \ar[r]\ar[d]^-{\wreath} & 0   \\
0  \ar[r]  & \gb^n/\gb^{n+1}   \ar[r] & \widehat{\CO}_{X,y} /\gb^{n+1}    \ar[r]& \widehat{\CO}_{X,y} /\gb^n  \ar[r] & 0.}  $$
The definition of $X$ and the explicit description of $\widehat{\CO}_{Z,z}$ in Lemma~\ref{lem:fibrecompletions}
imply that $\ga^n/\ga^{n+1}$ and $\gb^n/\gb^{n+1}$ are finite free $\widehat{\CO}_{Z_\red,z}$-modules of the same rank,
so the surjectivity of the left vertical arrow implies its injectivity, and hence that of the middle arrow.
Taking $n$ sufficiently large (namely $p^{|J|-|J'|}$) yields the desired isomorphism, and hence the lemma.
\epf

Combining Theorem~\ref{thm:reducedfibres} and Lemma~\ref{lem:nilfibres} yields a canonical isomorphism
\begin{equation}  \label{eqn:fibre}  Z   \stackrel{\sim}{\longrightarrow} 
\prod_{\theta \in J,\phi\circ\theta\not\in J } \P(H^1_\dr(A_0^{(p^{-1})}/\F)_\theta)   \times 
 \prod_{\theta,\phi\circ\theta \in J}  T_\theta,\end{equation}
 where $T_\theta$ is a fibre of the Frobenius morphism on $\P(H^1_\dr(A_0^{(p^{-1})}/\F)_\theta)$
 defined by Frobenius on the relevant factors.     In particular, the degeneracy fibre $Z$ is isomorphic to a product of
 $|J'|$ copies of $\P^1$ and $|J|-|J'|$ copies of  $\Spec \F[x]/(x^p)$.  
 
 Note also that we may view (\ref{eqn:X0cris}) as defining an
 isomorphism
 $$\CH^1_\dr(\widetilde{B}/Z) \cong 
   (R^1s_{\cris,*}\CO_{B/\F_p})_{(Z_\red,Z)}  \cong \pi^*\CH^1_\dr(B/Z_\red)$$
(where we have written $s:\widetilde{B} \to Z$ for the restriction of the universal $A_2$ and
$\pi:Z \to Z_\red$ for $\pi\circ\widetilde{\psi}^{-1}$), which combined with (\ref{eqn:good}) yields an isomorphism\footnote{This isomorphism can in fact
be seen more directly as arising from the isomorphism of crystals obtained in the proof of Lemma~\ref{lem:crys}.}
\begin{equation}  \label{eqn:better} \CH^1_\dr(\widetilde{B}/Z)_\theta \cong H^1_\dr(A_0^{(p^{-1})}/\F)_\theta\otimes_\F \CO_Z \end{equation}
for $\theta \in J$.   Unravelling the definitions, one finds that (\ref{eqn:fibre}) is the morphism whose $\theta$-component is defined
by the Hodge filtration on $\CH^1_\dr(\widetilde{B}/Z)_\theta$ under (\ref{eqn:better}).  We record this as follows:
\begin{theorem}  \label{thm:fibres}  Let $x \in T_{J'}(\F)$ correspond to $\underline{A}_0$, let $Z$ be the fibre over $x$ of the projection
$S_J \to T_{J'}$, and let $s:\widetilde{B} \to Z$ denote the pull-back of $A_2 \to S_J$.  Then there is a canonical closed immersion
$$ j: Z   \longrightarrow  \prod_{\theta \in J} \P(H^1_\dr(A_0^{(p^{-1})}/\F)_\theta)$$
under which $j^*\CO(1)_\theta$ is identified with $(R^1s_*\CO_{\widetilde{B}})_\theta$ for all $\theta \in J$, and
$Z$ is identified with the fibre over $(H^0(A_0,\Omega^1_{A_0/F})_{\phi\circ\theta})_\theta$ of the morphism
$$\begin{array}{ccc} \displaystyle \prod_{\theta \in J} \P(H^1_\dr(A_0^{(p^{-1})}/\F)_\theta)  & \longrightarrow & \displaystyle
 \prod_{\theta,\phi\circ\theta \in J}  \P(H^1_\dr(A_0/\F)_{\phi\circ\theta}) \\
  (M_\theta)_\theta & \longmapsto & (M_\theta^p)_\theta.\end{array}$$
 \end{theorem}
 
 Finally we observe that the construction of the immersion $j$ and the identifications in the theorem are independent of
 the quasi-polarization in the data defining the closed point $x$ of $T_{J'}$, and hence is invariant under the action of $\CO_{F,(p),+}^\times$.
 It follows that if $x\in \overline{Y}_{J'}(\F)$, then the fibre over $x$ of the degeneracy map
 $\pi_1:\overline{Y}_0(p)_J \to \overline{Y}$ induced by $(\underline{A}_1,\underline{A}_2,f) \mapsto \underline{A}_2$
 is canonically isomorphic to a fibre of the morphism
 $$\begin{array}{ccc} \displaystyle \prod_{\theta \in J} \P(V_{\phi\circ\theta}^{(p^{-1})})  & \longrightarrow &
 \displaystyle\prod_{\theta,\phi\circ\theta \in J}  \P(V_{\phi\circ\theta}) \\ 
 (M_\theta)_\theta & \longmapsto & (M_\theta^p)_\theta,\end{array}
 $$
 where $V_{\phi\circ\theta}$ is the fibre at $x$ of the automorphic bundle $\CV_{\phi\circ\theta}$ on $\overline{Y}$.
 Furthermore the isomorphism identifies the pull-back of $\CO(1)_\theta$ for each $\theta \in J$ with the
 pull-back of the automorphic bundle $\upsilon_\theta = \omega_\theta^{-1}\delta_\theta$ on $\overline{Y}$ under the degeneracy map
$\pi_2:\overline{Y}_0(p)_J \to \overline{Y}$ induced by $(\underline{A}_1,\underline{A}_2,f) \mapsto \underline{A}_2$.

\subsection{Cohomological vanishing}
In this section we prove Theorem~\ref{thm:H1}.  First note that in the statement, we may replace $\Z_{(p)}$
by the faithfully flat extension $\CO$, so the task is to prove that $R^i\pi_*\CK_{Y_1(p)/\CO} = 0$ for all $i > 0$
and sufficiently small open compact subgroups $U \subset G(\A_\f)$ of level prime to $p$,
where $\pi$ is the natural projection from $Y_1(p) := Y_{U_1(p)}(G)_\CO$ to $Y := Y_U(G)_\CO$
and $G = \Res_{F/\Q}(\GL_2)$.  We proceed in a series of steps to reduce the proof to showing the vanishing of
the cohomology of certain line bundles on the fibres of $\pi$ which we have already described.

\subsubsection{Reduction steps}\label{sss:cv.red}
We first reduce the problem to the setting of the special fibres of $Y_1(p)$ and $Y$.  Note that 
since $\pi$ is projective, $R^i\pi_*\CK_{Y_1(p)/\CO}$ is a coherent sheaf on $Y$.  Since $\pi \otimes_{\CO} L$
is finite, it follows that $R^i\pi_*\CK_{Y_1(p)/\CO}$ is supported on $\overline{Y} = Y_\F$.  Letting $j_0: \overline{Y} \to Y$
and $j_1: \overline{Y}_1(p) \to Y_1(p)$ denote the immersions of special fibres and $\overline{\pi}$ the restriction of $\pi$,
the natural map
$$j_0^*R^i\pi_*\CK_{Y_1(p)/\CO} \longrightarrow R^i\overline{\pi}_*(j_1^*\CK_{Y_1(p)/\CO}) =  R^i\overline{\pi}_*\CK_{\overline{Y}_1(p)}$$
is injective, so it suffices to prove that $R^i\overline{\pi}_* \CK_{\overline{Y}_1(p)} = 0$ for all $i > 0$.
Furthermore we may extend scalars and assume $\F = \Fpbar$.

Next we reduce to the setting of the irreducible components of $\overline{Y}_0(p)$.
Recall that $\overline{\pi} = \overline{\psi}\circ\overline{h}$ where
$\overline{h}:\overline{Y}_1(p) \to \overline{Y}_0(p)$ is finite flat and $\overline{\psi}:\overline{Y}_0(p) \to \overline{Y}$
is the natural degeneracy map induced by $(\underline{A}_1,\underline{A}_2,f) \mapsto f$, so we have
$$R^i{\overline{\pi}_*} \CK_{\overline{Y}_1(p)} = R^i{\overline{\psi}_*}(\overline{h}_* \CK_{\overline{Y}_1(p)}).$$
Furthermore in \S\ref{sss:dice.dual}, we defined a filtration on
$\overline{h}_*\CK_{\overline{Y}_1(p)}$ whose graded pieces are direct sums of direct images of
line bundles on the strata $\overline{Y}_0(p)_J$.  Recall also from \S\ref{sss:frob.prel} that the restriction of $\psi$ to
$\overline{Y}_0(p)_J$ factors through a projective Cohen--Macaulay morphism $\psi_J:\overline{Y}_0(p)_J \to \overline{Y}_{J'}$,
where the stratum $\overline{Y}_{J'}$ is the smooth closed subscheme of $\overline{Y}$ defined
by the vanishing of the partial Hasse invariants $h_\theta$ for $\theta \in J' := \{\theta \in J\,|\,\phi^{-1}\circ\theta \not\in J\}$.
Thus from the description of $\gr(\overline{h}_*\CK_{\overline{Y}_1(p)})$ in \S\ref{sss:dice.dual}, we see that
it suffices to prove that
$$R^i\psi_{J,*}\left(\CJ_J^{-1}\CK_{\overline{Y}_0(p)_J}i_J^*\CL_{\chi^{-1}}^{-1}\right) = 0$$
for all characters $\chi:(\CO_F/p\CO_F)^\times \to \F^\times$, subsets $J \subset \Theta$ and 
integers $i > 0$.

Now we reduce to considering the fibres of the morphism from $\overline{Y}_0(p)_J$ to $\overline{Y}$ for
each $J$.  We let $s: \Spec \F \to \overline{Y}_{J'}$ be a closed point, let $Z$ be the fibre of $\psi_J$
at $s$, and let $j:Z \to \overline{Y}_0(p)_J$ be the inclusion.  Since $\psi_J$ is flat
and projective, it suffices to prove that 
$$H^i(Z, j^*(\CJ_J^{-1}\CK_{\overline{Y}_0(p)_J}i_J^*\CL_{\chi^{-1}}^{-1})) = 0$$
for all $\chi$, $J$, $s$ as above and $i > 0$.  Furthermore since $\psi_J$ is Cohen--Macaulay and $\overline{Y}_{J'}$ is smooth, we have
$$\CK_{\overline{Y}_0(p)_J} = \CK_{\overline{Y}_0(p)_J/\overline{Y}_{J'}} \otimes_{\CO_{{\overline{Y}_0(p)_J}}}\psi_J^*\CK_{\overline{Y}_{J'}},$$
so that $j^*\CK_{\overline{Y}_0(p)_J}$ is isomorphic to $\CK_Z$.  It therefore suffices to prove that
$$H^i(Z, \CK_Z \otimes_{\CO_Z} j^*(\CJ_J^{-1}i_J^*\CL_{\chi^{-1}}^{-1})) = 0$$
for all $\chi$, $J$, $s$ as above and $i > 0$.

\subsubsection{Completion of the proof}\label{sss:cv.end}
Recall that the fibre $Z$ was computed in Theorem~\ref{thm:fibres}, where it is shown to be isomorphic
to a product of projective lines and copies of $\Spec(\F[x]/(x^p))$.  More precisely, we have that $Z$
is isomorphic to 
\begin{equation}\label{eqn:thickprod}\prod_{\theta \in J \cap \Sigma} \P(V_{\phi\circ\theta}^{(p^{-1})})_\theta \times \prod_{\theta \in J - \Sigma} T_\theta\end{equation}
where $T_\theta$ is a finite subscheme of  $\P(V_{\phi\circ\theta}^{(p^{-1})})_\theta$.
Furthermore, this isomorphism identifies the pull-back of each $\CO(1)_\theta$ with the restriction to $Z$ of the
line bundle obtained by descent from $(R^1s_{2,*}\CO_{A_2})_\theta$, where
$f:A_1 \to A_2$ is the universal isogeny on $\widetilde{Y}_0(p)_{J,\F}$.
Recall also that for $\theta \in J$, we have a canonical identification of $(R^1s_{2,*}\CO_{A_2})_\theta$
with the Raynaud line bundle $\CL_\theta$ on $\widetilde{Y}_0(p)_{J,\F}$, and this identification
descends to $\overline{Y}_0(p)_J$. In particular, if $\theta \in J\cap \Sigma$,
then the restriction of $\CL_\theta$ to $Z$ corresponds under the above isomorphism
to $\CO(1)_\theta$.   Similarly if $\theta \in J - \Sigma$, then the restriction of $\CL_\theta$
corresponds to the restriction of $\CO(1)_\theta$, which is (non-canonically) isomorphic to $\CO_Z$.
On the other other hand, if $\theta\not\in J$, then we obtain a
trivialization of $\CL_\theta$ on $Z$ from its identification
with $s_{1,*}\Omega^1_{A_1/\F}$ on $\widetilde{Y}_0(p)_{J,\F}$.
Combining this with the formulas (\ref{eqn:Lfactor}) and (\ref{eqn:Jfactor}), we conclude that
$$j^*(\CJ_J^{-1}i_J^*\CL_{\chi^{-1}}^{-1})  \cong \bigotimes_{\theta \in J \cap \Sigma} \CO(n_\theta)_\theta$$
where $n_\theta = m_\theta + 1 > 0$ for all $\theta \in J \cap \Sigma$.
Since the dualizing sheaf of (\ref{eqn:thickprod}) is isomorphic to $\bigotimes_{\theta\in J \cap \Sigma} \CO(-2)_\theta$,
it follows that
$$H^i(Z,  \CK_{Z} \otimes_{\CO_{Z}} j^*(\CJ_J^{-1}i_J^*\CL_{\chi^{-1}}^{-1}))
 \cong H^i\left( \prod_{\theta \in J \cap \Sigma} \P^1_\F, \bigotimes_{\theta \in J \cap \Sigma} \CO(n_\theta - 2)_\theta\right)^{p^{|J-\Sigma|}} = 0$$
for all $i > 0$, and this completes the proof of Theorem~\ref{thm:H1}.

\bibliographystyle{amsalpha} 


\bibliography{MRrefs_JL}

\end{document}